\documentclass[11pt]{amsart}

\usepackage{a4wide}
\usepackage[dvipsnames]{xcolor}
\usepackage{mathrsfs}
\usepackage{amsmath, amscd, mathtools}
\usepackage{amsthm}
\usepackage{amsfonts}
\usepackage{amssymb}
\usepackage{enumitem}
\usepackage{graphicx}
\usepackage{palatino}
\usepackage{overpic}
\usepackage[psdextra]{hyperref}
\usepackage{comment}
\usepackage[noabbrev,capitalize,nameinlink]{cleveref}
\usepackage[normalem]{ulem}
\usepackage{subcaption}

\usepackage{tikz}
\usetikzlibrary{decorations.markings}
\usetikzlibrary{calc,arrows.meta}

\usepackage{thmtools}

\makeatletter
\renewcommand{\part}{%
  \clearpage 
  \thispagestyle{plain} 
  \@startsection{part}{0}%
    \z@{\linespacing\@plus\linespacing}{.5\linespacing}%
    {\normalfont\Large\bfseries\centering}%
}
\makeatother

\makeatletter 
 \def\l@subsection{\@tocline{2}{0pt}{2pc}{6pc}{}} \makeatother

\numberwithin{equation}{section}

\declaretheorem[name=Theorem, numberwithin=section]{theorem}
\declaretheorem[name=Proposition, sibling=theorem]{proposition}
\declaretheorem[name=Lemma, sibling=theorem]{lemma}
\declaretheorem[name=Corollary, sibling=theorem]{corollary}

\declaretheorem[name=Claim, sibling=theorem]{claim}

\theoremstyle{definition}
\declaretheorem[name=Definition, sibling=theorem]{definition}
\declaretheorem[name=Example, sibling=theorem]{example}

\theoremstyle{remark}
\declaretheorem[name=Remark, sibling=theorem]{remark}
\declaretheorem[name=Terminology, sibling=theorem]{terminology}
\declaretheorem[name=Notation, sibling=theorem]{notation}

\Crefname{theorem}{Theorem}{Theorems}
\Crefname{lemma}{Lemma}{Lemmas}
\Crefname{proposition}{Proposition}{Propositions}
\Crefname{conjecture}{Conjecture}{Conjectures}
\Crefname{definition}{Definition}{Definitions}
\Crefname{corollary}{Corollary}{Corollaries}
\Crefname{remark}{Remark}{Remarks}
\Crefname{example}{Example}{Examples}
\Crefname{question}{Question}{Questions}

\newcommand{\R}{\mathbb{R}}
\newcommand{\C}{\mathbb{C}}
\newcommand{\D}{\mathbb{D}}

\newcommand{\Z}{\mathbb{Z}}

\newcommand{\ve}{\varepsilon}

\DeclarePairedDelimiter\abs{\lvert}{\rvert}
\DeclarePairedDelimiter\norm{\lVert}{\rVert}

\DeclareMathOperator{\id}{id}

\DeclareMathOperator{\Sk}{Sk}

\DeclareMathOperator{\PL}{PL}
\DeclareMathOperator{\inward}{in}
\DeclareMathOperator{\outward}{out}

\newcommand{\Ical}{\mathcal{I}}
\newcommand{\Scal}{\mathcal{S}}
\newcommand{\Tcal}{\mathcal{T}}
\newcommand{\Zcal}{\mathcal{Z}}

\newcommand{\s}{\vskip.1in}
\newcommand{\n}{\noindent}
\newcommand{\p}{\partial}
\newcommand{\bdry}{\partial}

\newcommand{\be}{\begin{enumerate}}
\newcommand{\ee}{\end{enumerate}}
\newcommand{\op}{\operatorname}

\definecolor{darkgreen}{RGB}{0,153,0}
\definecolor{divgreen}{RGB}{0,180,0}
\definecolor{darkred}{RGB}{204,0,0}
\definecolor{darkblue}{RGB}{0,51,204}

\hypersetup{colorlinks,linkcolor={darkblue},citecolor={darkgreen},urlcolor={darkgreen}}

\begin{document}

\title{Convex hypersurface theory in contact topology}

\author{Joseph Breen}
\address{University of Alabama, Tuscaloosa, AL 35401}
\email{jjbreen@ua.edu} \urladdr{https://sites.google.com/view/joseph-breen}

\author{Austin Christian}
\address{California Polytechnic State University, San Luis Obispo, CA 93407}
\email{achris66@calpoly.edu} \urladdr{https://sites.google.com/view/austin-christian}

\author{Ko Honda}
\address{University of California, Los Angeles, Los Angeles, CA 90095}
\email{honda@math.ucla.edu} \urladdr{http://www.math.ucla.edu/\char126 honda}

\author{Yang Huang}
\address{Somewhere On Earth}
\email{hymath@gmail.com} \urladdr{https://sites.google.com/site/yhuangmath}

\thanks{JB was partially supported by NSF Grant DMS-2038103 and an AMS-Simons Travel Grant.  AC was partially supported by NSF Grant DMS-2532551 and an AMS-Simons Travel Grant. KH was partially supported by NSF Grants DMS-2003483 and DMS-2505876.  YH was partially supported by the grant KAW 2016.0198 from the Knut and Alice Wallenberg Foundation.}

\begin{abstract}
We lay the foundations of convex hypersurface theory in contact topology, extending the work of Giroux in dimension three. Specifically, we prove that any closed hypersurface in a contact manifold can be $C^0$-approximated by a convex one.  We also prove that a $C^0$-generic family of mutually disjoint closed hypersurfaces parametrized by $t\in[0,1]$ is convex except at finitely many times $t_1,\dots,t_N$, and that crossing each $t_i$ corresponds to a bypass attachment.  As an application, we prove the existence of compatible (relative) open book decompositions for contact manifolds.
\end{abstract}

\maketitle

\tableofcontents

\section{Introduction} \label{sec:intro}

\subsection{Convex contact structures}

Morse theory is a topologist's favorite tool for exploring the structure of manifolds. The significance of Morse theory --- here we mean the traditional finite-dimensional version, not Floer theory --- in contact and symplectic topology was advocated by Eliashberg and Gromov in \cite{EG91}. In particular, according to \cite[Definition 3.5.A]{EG91}, a contact manifold $(M,\xi)$ is \emph{convex} if there exists a Morse function, called a \emph{contact Morse function}, which admits a gradient-like vector field whose flow preserves $\xi$. Just as a manifold can be reconstructed from its Morse function by a sequence of handle attachments in traditional Morse theory, a contact manifold can be reconstructed from a contact Morse function by a sequence of contact handle attachments. The analogous theory in symplectic topology is known as the theory of Weinstein manifolds.

Eliashberg and Gromov asked in \cite{EG91} whether there exist non-convex contact manifolds. Around 2000 Giroux gave a negative answer to the question by showing that every closed contact manifold is convex; see \cite{Gi02}.  This can also be formulated as his celebrated correspondence between contact structures and open book decompositions. This is in sharp contrast to the theory of Weinstein manifolds, where it is relatively easy to see that any compatible Morse function cannot have critical points of index greater than half of the dimension of the manifold.

Giroux used two completely different sets of techniques in the $3$-dimensional and higher-dimensional cases to show that every contact manifold admits a compatible contact Morse function.

We first discuss the $3$-dimensional case.  In his thesis \cite{Gi91}, Giroux introduced what is now known as \emph{convex surface theory} into $3$-dimensional contact topology. It is an extremely powerful and efficient way of studying embedded surfaces in contact $3$-manifolds, and can recover most of the pioneering results of Bennequin \cite{Ben83} and Eliashberg \cite{Eli92}. Using convex surface theory, Giroux showed that for closed contact $3$-manifolds, there is a one-to-one correspondence between isotopy classes of contact structures and compatible open book decompositions up to positive stabilization.

Before moving onto higher dimensions, let us recall the definition of a convex hypersurface following \cite{Gi91}:

\begin{definition} \label{defn:convex hypersurface}
A hypersurface $\Sigma \subset (M,\xi)$ is \emph{convex} if there exists a contact vector field $v$, i.e., a vector field whose flow preserves $\xi$, which is transverse to $\Sigma$ everywhere.
\end{definition}

\noindent Observe that regular level sets of a contact Morse function are convex hypersurfaces.

The situation in dimensions $> 3$ is quite different. Besides the fact that convex hypersurfaces can be defined in any dimension, until now there has been no systematic \emph{convex hypersurface theory}.   Giroux's proof \cite{Gi02,Gi17} that every closed contact manifold is convex involves a completely different technology, i.e., Donaldson's \cite{Don96} technique of \emph{approximately holomorphic sections}, transplanted into contact topology independently by Ibort, Mart\'inez-Torres, and Presas \cite{IMP00}  and by Mohsen~\cite{Moh01,Moh19}.  Donaldson used the approximate holomorphic technology to construct real codimension $2$ symplectic hypersurfaces of a closed symplectic manifold as the zero locus of an approximately holomorphic section of a complex line bundle, while \cite{IMP00} and ~\cite{Moh01,Moh19} constructed certain codimension $2$ contact submanifolds of a closed contact manifold.  What Giroux realized is that \cite{Don96}, \cite{IMP00}, and \cite{Moh01,Moh19} could be used to produce compatible open book decompositions. Roughly speaking, given a closed contact manifold $(M,\xi=\ker\alpha)$, one considers the trivial line bundle $\underline{\C}$ on $M$ equipped with a suitable Hermitian connection determined by $\alpha$. Then there exists a section $s: M \to \underline{\C}$ whose zero locus $B \coloneqq s^{-1} (0)$ is a closed codimension $2$ contact submanifold called the \emph{binding}, and
$$\frac{s}{\abs{s}}: M \setminus B \to S^1$$
is a smooth fibration defining the compatible open book decomposition of $(M,\xi)$. As a consequence of using the approximate holomorphic technology, the higher-dimensional Giroux correspondence (see \cref{cor:OBD existence}) is a weaker statement compared to its $3$-dimensional counterpart.

\subsection{Main results}

The main goal of this paper is to systematically generalize Giroux's convex surface theory to all dimensions.  The main results of \emph{convex hypersurface theory} (CHT) are \cref{thm:genericity}, \cref{thm:family genericity}, and \cref{thm:BBC}. In fact, even in dimension $3$, our method (see \cref{sec:CST revisited}) differs somewhat from Giroux's original approach, is simpler, and is consistent with our more general approach in higher dimensions.

We first introduce some more terminology describing the anatomy of a convex hypersurface.

\begin{definition}
Let $\Sigma \subset (M,\xi=\ker \alpha)$ be a convex hypersurface with respect to a transverse contact vector field $v$. Define the \emph{dividing set} $\Gamma(\Sigma) \coloneqq \{\alpha(v)=0\}$ and $R_{\pm}(\Sigma) \coloneqq \{\pm \alpha(v)>0\}$ as subsets of $\Sigma$.
\end{definition}

It turns out that $\Gamma(\Sigma) \subset (M,\xi)$ is a codimension-$2$ contact submanifold, and $R_{\pm} (\Sigma)$ are (complete) Liouville manifolds with Liouville form given by a suitable rescaling of $\alpha|_{R_{\pm} (\Sigma)}$, respectively. Moreover, the isotopy classes of $\Gamma(\Sigma), R_{\pm} (\Sigma)$ are independent of the choices of $v$ and $\alpha$.

In dimensions $\geq 4$, there exist Liouville manifolds that are not Weinstein by McDuff \cite{McD91}, Geiges \cite{Gei94,Gei95}, Mitsumatsu \cite{Mit95}, and Massot, Nieder\-kr\"uger, and Wendl \cite{MNW13}. While these exotic Liouville manifolds work well as (counter-)examples, there currently is no systematic understanding of such non-Weinstein Liouville manifolds, partially because of the lack of an appropriate Morse theory on such manifolds.  This is slowly starting to change.  For example, it was recently shown in \cite{EOY20} that a stabilized Liouville manifold with the homotopy type of a half-dimensional CW-complex is symplectomorphic to a flexible Weinstein manifold and in \cite{breen2025torus} that the stabilizations of torus bundle Liouville domains constructed by Mitsumatsu and generalized by Huang \cite{huang2020dynamical} are Weinstein domains.\footnote{Note the latter involves Liouville {\em domains}, which is a distinct problem.} Additionally, there is a burgeoning understanding of certain non-Weinstein Liouville manifolds via the lens of Anosov flows, generalizing the work of Mitsumatsu; see \cite{cieliebak2022floertheoryanosovflows,hozoori2024symplectic,hozoori2025strongly,massoni2025symplectic}. This motivates the following definition:

\begin{definition}
 A convex hypersurface $\Sigma$ is \emph{Weinstein (resp.\ $1$-Weinstein) convex} if $R_+(\Sigma)$ and $R_-(\Sigma)$ are Weinstein (resp.\ $1$-Weinstein).
\end{definition}

\noindent Here a Liouville domain that admits a Liouville vector field which is gradient-like with respect to a ``$1$-Morse" function which also admits critical points of birth-death type will be called a \emph{$1$-Weinstein domain} (instead of a \emph{generalized Weinstein domain}) in this paper. Weinstein convex hypersurfaces admit a Morse-theoretic interpretation, given in \cref{prop:convex criterion}.

Now we are ready to state the foundational theorems of CHT.

\begin{theorem} \label{thm:genericity}
Any closed hypersurface in a contact manifold can be $C^0$-approxi\-mated by a  Weinstein convex one.
\end{theorem}

\begin{remark}
 The $C^\infty$-version (in fact, even the $C^2$-version) of \cref{thm:genericity} was recently shown to be false by Chaidez~\cite{Ch24}.  The prior candidate counterexample, due to Mori \cite{Mo11}, was shown in \cite{Br21} to actually admit a $C^\infty$-small approximation by a convex one. Work of Eliashberg-Pancholi \cite{EP2023} has provided additional perspective on convex approximation. 
\end{remark}

\begin{theorem} \label{thm:family genericity}
Let $\xi$ be a contact structure on $\Sigma \times [0,1]$ such that the hypersurfaces $\Sigma \times \{0,1\}$ are  Weinstein  convex. Then, up to a boundary-relative contact isotopy, there exists a finite sequence $0<t_1<\dots<t_N<1$ such that $\Sigma \times \{t\}$ is  $1$-Weinstein  convex if $t \neq t_i$ for any $1 \leq i \leq N$.  
\end{theorem}

\begin{theorem} \label{thm:BBC}
Let $\xi$ be a contact structure on $\Sigma \times [0,1]$ such that the hypersurfaces $\Sigma \times \{t\}$ are  $1$-Weinstein convex for $t\neq \frac{1}{2}$. Then $(\Sigma \times [0,1], \xi)$ is contactomorphic to a bypass attachment. 
\end{theorem}

An initial study of bypass attachments in higher dimensions first appeared in \cite{HH18}. In \cref{appendix} we reformulate and reprove all the necessary results, so that the present work does not depend on \cite{HH18}. A more precise and technical version of \cref{thm:BBC} will be stated later; c.f.\ \cref{prop:bb-correspondence}.

\begin{remark}
\cref{thm:family genericity} and \cref{thm:BBC} were conjectured by Paolo Ghiggini in the afternoon of April 10, 2015 in Paris.
\end{remark}

\subsection{Applications}

As an immediate application of the above foundational theorems, we can extend Giroux's $3$-dimensional approach to constructing compatible open book decompositions to higher dimensions. This is the content of the following three corollaries. We do not address the stabilization equivalence of the compatible open book decompositions in this paper, which is addressed in the sequel~\cite{BHH23}. 

\begin{corollary}[\cite{Gi02}] \label{cor:OBD existence}
Any closed contact manifold admits a compatible open book decomposition,  all of whose pages are $1$-Weinstein.
\end{corollary}

\begin{corollary} \label{cor:POBD existence}
Any compact contact manifold with convex boundary admits a compatible partial open book decomposition,  all of whose pages are $1$-Weinstein domains or $1$-Weinstein cobordisms. 
\end{corollary}

\begin{corollary} \label{cor: Legendrians on a page}
Given a possibly disconnected closed Legendrian submanifold $\Lambda$ in a closed contact manifold, there exists a compatible open book decomposition  such that all the pages are $1$-Weinstein and $\Lambda$ is contained in a page.
\end{corollary}

\begin{remark} 
Although \cref{cor: Legendrians on a page} does not appear in the literature, asymptotically holomorphic techniques in contact geometry from \cite{IMP00} and ~\cite{Moh01,Moh19} --- more specifically the combination of Giroux's existence theorem of compatible open book decompositions (see Presas~\cite{Pre14} for a published proof) and asymptotically holomorphic symplectic special position theorems for Lagran\-gians in symplectic manifolds \cite{AMP05} --- applied to the Lagrangianization of a Legendrian in the symplectization gives the result. 
\end{remark}

\s\n
{\em Organization.} \cref{part:CHT} is dedicated to proving the convex approximation results of \cref{thm:genericity} and \cref{thm:family genericity}. \cref{part:BBC} establishes the ``bypass-bifurcation correspondence'' of \cref{thm:BBC}, which connects the parametric genericity statement in \cref{thm:family genericity} with bypass attachments (and thus open book decompositions). \cref{part:CHT} and \cref{part:BBC} are entirely independent. Finally, \cref{part:OBD} proves the corollaries on existence of open book decompositions, combining the genericity results and the bypass-bifurcation correspondence of the previous parts.

The bypass-bifurcation correspondence naturally relies on an understanding of contact handles and bypass attachments in higher dimensions. An earlier version of this article appealed to many results from the temporal precursor \cite{HH18} in which this theory is developed; however, the present work should be considered the first in the series of papers establishing CHT in all dimensions (followed by \cite{HH18,BHH23}). As such, we have included \cref{appendix}, which reformulates and reproves the necessary results on contact handles and bypass attachments from \cite{HH18} in order to make this article the true introductory work.   

\s\n
{\em Acknowledgments.}  KH is grateful to Yi Ni and the Caltech Mathematics Department for their hospitality during his sabbatical. YH thanks the geometry group at Uppsala: Georgios Dimitroglou Rizell, Luis Diogo (his officemate), Tobias Ekholm, Agn\`es Gadbled, Thomas Kragh, Wanmin Liu, Maksim Maydanskiy and Jian Qiu for conversations about Everything during the period 2017--2019.  AC and JB thank John Etnyre for several helpful conversations. We thank Cheuk Yu Mak and Fran Presas for pointing out some errors and the anonymous referees for extensive comments during the revision of the paper.  

Finally, we thank Yasha Eliashberg, Fran\c cois-Simon Fauteux-Chapleau and Dishant Pancholi for finding a mistake in the previous version of the paper (i.e., partial mushrooms do not work), suggesting some simplifications which were implemented in the current version of the paper with their permission, and pointing out the reference \cite{Wi66}; and Emmanuel Giroux for providing a proof of \cref{lemma: Weinstein}.

\part{Convex hypersurface approximation}\label{part:CHT}

\section{A convexity criterion} \label{sec:convexity criterion}

Let $\Sigma \subset (M^{2n+1},\xi)$ be a closed cooriented hypersurface.  The goal of this section is to give a sufficient condition for the characteristic foliation $\Sigma_{\xi}$ on $\Sigma$ (see \cref{defn:char_foliation}) which guarantees the Weinstein convexity of $\Sigma$.

\subsection{Characteristic foliations} \label{subsec:char fol}

Let $\alpha$ be a contact form for $\xi$. Let $(-\epsilon,\epsilon) \times \Sigma$ be a {tubular} neighborhood of $\Sigma=\{0\}\times \Sigma \subset M$. Fix an orientation on $\Sigma$ such that the induced orientation on $(-\epsilon,\epsilon) \times \Sigma$ agrees with the orientation determined by $\alpha \wedge (d\alpha)^n$.  We now introduce the characteristic foliation $\Sigma_{\xi}$ on $\Sigma$.

\begin{definition} \label{defn:char_foliation}
The \emph{characteristic foliation} $\Sigma_{\xi}$ is an oriented singular line field on $\Sigma$ defined by $$\Sigma_{\xi} = \ker d\beta|_{\ker \beta},$$ where $\beta \coloneqq \alpha|_{\Sigma} \in \Omega^1(\Sigma)$. The orientation of $\Sigma_{\xi}$ is determined by requiring that the decomposition $T\Sigma = \Sigma_{\xi} \oplus \Sigma_{\xi}^{\bot}$ respect orientations, where the orthogonal complement $\Sigma_{\xi}^{\bot}$, taken with respect to an auxiliary Riemannian metric on $\Sigma$, is oriented by $\beta \wedge (d\beta)^{n-1} |_{\Sigma_{\xi}^{\bot}}$.
\end{definition}

\begin{remark}
The characteristic foliation depends only on the contact structure and the orientation of $\Sigma$, and \emph{not} on the choice of the contact form.
\end{remark}

Note that $x \in \Sigma$ is a singular point of $\Sigma_{\xi}$ if $T_x \Sigma = \xi_x$ as unoriented spaces. We say $x$ is \emph{positive} (resp.\ \emph{negative}) if $T_x \Sigma = \pm\xi_x$ as oriented spaces, respectively.

The significance of the characteristic foliation in $3$-dimensional contact topology is that it uniquely determines the germ of contact structures on any embedded surface. The corresponding statement for hypersurfaces in contact manifolds of dimension $>3$ is unlikely to hold, i.e., the characteristic foliation by itself is \emph{not} enough to determine the contact germ. Instead we have the following characterization of contact germs on hypersurfaces in any dimension. The proof is a standard application of the Moser technique and is omitted here.

\begin{lemma} \label{lem:char foliation}
Suppose $\xi_i = \ker\alpha_i, i=0,1$, are contact structures on $M$ such that $\beta_0 = g\, \beta_1 \in \Omega^1(\Sigma)$ for some $g: \Sigma \to \R_+$, where $\beta_i=\alpha_i|_{\Sigma}$. Then there exists an isotopy $\phi_s: M \stackrel\sim\to M, s \in [0,1]$, such that $\phi_0 = \id_M, \phi_s(\Sigma) = \Sigma$ and $(\phi_1)_{\ast} (\xi_0) = \xi_1$ on a neighborhood of $\Sigma$.
\end{lemma}

Generally speaking, $\Sigma_{\xi}$ can be rather complicated, even when $\Sigma$ is convex with Liouville $R_\pm(\Sigma)$. For our purposes of this paper, it is more convenient to regard $\Sigma_{\xi}$ as a vector field rather than an oriented line field. Of course there is no natural way to specify the magnitude of $\Sigma_{\xi}$ as a vector field, which motivates the following definition: Two vector fields $v_1,v_2$ on $\Sigma$ are \emph{conformally equivalent} if there exists a positive function $h: \Sigma \to \R_+$ such that $v_1 = hv_2$. This is clearly an equivalence relation among all vector fields, and we will \emph{not} distinguish conformally equivalent vector fields in the rest of the paper unless otherwise stated.

In order to state the convexity criterion, we need to prepare some generalities on gradient-like vector fields in the following subsection. Our treatment on this subject will be kept to a minimum. The reader is referred to the classical works of Cerf \cite{Cer70} and Hatcher-Wagoner \cite{HW73} for more thorough discussions. Indeed, the {adaptation} of the techniques of Cerf and Hatcher-Wagoner to CHT is carried out in \cite{BHH23}. Note that similar techniques in symplectic topology have been developed by Cieliebak-Eliashberg in \cite{CE12}.

\subsection{Morse and \(1\)-Morse vector fields}

Let $Y$ be a closed manifold of dimension $n$. A smooth function $f: Y \to \R$ is \emph{Morse} if all the critical points of $f$ (i.e., points $p\in Y$ such that $df(p)=0$) are \emph{nondegenerate}, i.e., there exist local coordinates $x_1,\dots,x_n$ about $p$ such that locally $f$ takes the form
\begin{equation} \label{eqn:morse local model}
-x_1^2-\dots-x_k^2+x_{k+1}^2+\dots+x_n^2.
\end{equation}
Here $k$ is called the \emph{Morse index}, or just the \emph{index}, of the critical point $p$.  We write $\op{ind}(p)=k$. 

A smooth function $f: Y \to \R$ is  \emph{$1$-Morse} if the critical points of $f$ are either nondegenerate or of \emph{birth-death type}. Here a critical point $p \in Y$ of $f$ is of \emph{birth-death type} (also called \emph{embryonic}) if there exist local coordinates $x_1,\dots,x_n$ about $p$ such that $f$ takes the form 
\begin{equation} \label{eqn:1-morse local model}
-x_1^2-\dots-x_k^2+x_{k+1}^2+\dots+x_{n-1}^2+x_n^3.
\end{equation}
Similarly, $k$ is defined to be the \emph{(Morse) index} of $p$. The birth-death type critical point fits into a $1$-parameter family of $1$-Morse functions
\begin{equation*}
-x_1^2-\dots-x_k^2+x_{k+1}^2+\dots+x_{n-1}^2+tx_n+x_n^3,
\end{equation*}
such that for $t<0$, there exist two nondegenerate critical points of indices $k$ and $k+1$; for $t=0$, there exists a birth-death type critical point; and for $t>0$, there are no critical points.

It is a well-known fact due to Morse that any smooth function can be $C^{\infty}$-approximated by a Morse function. Moreover, Cerf proved that any $1$-parameter family of smooth functions can be $C^{\infty}$-approximated by a family of $1$-Morse functions, where the birth-death type critical points as above occur only at isolated moments.

Given a $1$-Morse function $f: Y \to \R$, we say a vector field $v$ on $Y$ is \emph{gradient-like} for $f$ if the following two conditions are satisfied:
\begin{itemize}
	\item[(GL1)] Near each critical point of $f$, $v=\nabla f$ with respect to some Riemannian metric; and
	\item[(GL2)] $f$ is strictly increasing along (non-constant) flow lines of $v$.
\end{itemize}

\begin{definition} \label{defn:morse vector field} $\mbox{}$
\be
\item A vector field $v$ on $Y$ is \emph{Morse} (resp.\ {\em $1$-Morse}) if there exists a Morse (resp.\ $1$-Morse) function $f: Y \to \R$ such that $v$ is gradient-like for $f$.  
\item  A $1$-parameter family of vector fields $(v_t)_{t\in[0,1]}$ is  {\em a $1$-Morse family} if each $v_t$ is $1$-Morse and there exist $t_1<\dots< t_k\in (0,1)$ such that the birth-death type singularities occur only at $t_i$ and there is a single birth-death type singularity at each $t_i$. 
\ee
\end{definition}

 We also make the slightly nonstandard definition: 

\begin{definition}
 A Liouville domain is {\em $1$-Weinstein} if its Liouville vector field is gradient-like with respect to a $1$-Morse function.
\end{definition}

\begin{remark}
$1$-Morse functions will be sufficient for the purposes of this paper since we will only encounter $1$-parameter families of functions. In \cite{BHH23}, it is necessary to deal with generic $2$-parameter families of functions (called $2$-Morse functions) where new singularities, i.e., the swallowtails, appear. 
\end{remark}

\begin{definition} 
 A {\em flow line} $\ell: (a,b)\to Y$ of a vector field $v$ on $Y$ is assumed to be a {\em maximal} oriented smooth trajectory $\R\to Y$ that has been precomposed with an orientation-preserving reparametrization $(a,b)\stackrel\sim\to \R$.  A {\em partial flow line} is the restriction of a flow line to a subinterval. A {\em broken flow line} (resp.\ {\em possibly broken flow line}) of a vector field $v$ on $Y$ is a continuous map $\ell: [a,b] \to Y$ such that there exists an increasing sequence $a=a_0<a_1< \dots <a_m=b$ with $m> 1$ (resp.\ $m\geq 1$) such that $\ell(a_j)$, $j=0,\dots, m$, are zeros of $v$ and $\ell|_{(a_j,a_{j+1})}$, $j=0,\dots, m-1$, are flow lines of $v$. We may also replace $[a,b]$ by half-open or open intervals. 
\end{definition} 

In the rest of this subsection, we present a simple criterion for a vector field to be Morse which will be useful for our later applications. The corresponding version for $1$-Morse vector fields is left to the reader as an exercise.

\begin{proposition} \label{prop:morse criterion}
A vector field $v$ on a closed manifold $Y$ is Morse  (resp.\ $1$-Morse)  if and only if the following conditions are satisfied:
	\begin{enumerate}
		\item[(M1)] For any point $x  \in Y$ with $v(x)=0$, there exists a neighborhood of $x$ and a locally defined function $f$ of the form given by \eqref{eqn:morse local model}  (resp.\ \eqref{eqn:morse local model} or \eqref{eqn:1-morse local model})  such that $v=\nabla f$.
		\item[(M2)] For any point $x \in Y$ with $v(x) \neq 0$, the flow line of $v$ passing through $x$ converges to zeros of $v$ in both forward and backward time.
		\item[(M3)] There exist no  \emph{possibly broken loops}, i.e., a possibly broken flow line $\ell:[0,1]\to Y$ such that $\ell(0)=\ell(1)$.
	\end{enumerate}
\end{proposition}

\begin{proof}
The ``only if" direction is clear. To prove the ``if" direction, let $Z(v) = \{x_1, \dots, x_k \}$ be the finite set of zeros of $v$, where the finiteness is guaranteed by (M1) and the compactness of $Y$. Then we define a partial order on $Z(v)$ such that $x_i \prec x_j$ if there exists a flow line of $v$ from $x_i$ to $x_j$. The fact that $\prec$ is a partial order follows from (M3).

We then construct a handle decomposition of $Y$ starting from the minimal elements $Z_0$ of $Z(v)$ (note that a minimal element of $Z(v)$ has index $0$ by (M2))  and inductively attaching handles as follows: Starting with a standard neighborhood of $Z_0$, suppose we have already attached the handles corresponding to $Z_j$.  Then we attach the handles corresponding to the minimal elements of $Z(v)-Z_j$, and then let $Z_{j+1}$ be the union of $Z_j$ and the minimal elements of $Z(v)-Z_j$.
\end{proof}

\subsection{A convexity criterion}

The goal of this subsection is to give a sufficient condition for a hypersurface to be Weinstein convex.  To this end, we introduce the notions of \emph{Morse} and \emph{Morse}$^+$ hypersurfaces whose characteristic foliations have particularly simple dynamics.

\begin{definition} \label{defn:morse-type hypersurface} $\mbox{}$
\be
\item A hypersurface $\Sigma \subset (M,\xi)$ is \emph{Morse}  (resp.\ {\em $1$-Morse}) if there exists a representative $v$ in the conformal equivalence class of $\Sigma_{\xi}$  which is a Morse  (resp.\ {\em $1$-Morse})  vector field on $\Sigma$. We say $\Sigma$ is \emph{Morse}$^+$ (resp.\ {\em $1$-Morse$^+$}) if, in addition, there exist no flow trajectories from a negative singular point of $v$ to a positive one. 
\item  A $1$-parameter family of hypersurfaces $(\Sigma_t)_{t \in [0,1]}$ is a {\em $1$-Morse family}, if $((\Sigma_t)_{\xi})_{t\in [0,1]}$ is represented by a $1$-Morse family. 
\ee
\end{definition}

\begin{lemma} \label{lem:C^infty perturb}
If $\Sigma$ is a Morse hypersurface, then a $C^{\infty}$-small perturbation of $\Sigma$ is Morse$^+$.
\end{lemma}

\begin{proof}
Choose a contact form $\xi = \ker \alpha$. It suffices to observe that $d\alpha|_{\Sigma}$ is non-degenerate on a neighborhood of the singular points of $\Sigma_{\xi}$. It is a standard fact (see e.g. \cite[Proposition 11.9]{CE12}) that the Morse index $\text{ind}(x) \leq n$ if $x$ is a positive singular point of $\Sigma_\xi$, and $\text{ind}(x) \geq n$ if $x$ is negative. The claim therefore follows from the usual transversality argument.
\end{proof}

The following proposition gives a sufficient condition for convexity:

\begin{proposition} \label{prop:convex criterion} \mbox{} 
\be
\item A $1$-Morse$^+$ hypersurface $\Sigma$ is convex. 
\item A hypersurface $\Sigma$ is Weinstein convex if and only if it is Morse$^+$.
\ee
\end{proposition}

\begin{proof}
(1) is a straightforward generalization of the usual proof for surfaces due to Giroux that $\Sigma$ is convex if it has a Morse$^+$ characteristic foliation. 

Let $\mathbf{x} = \{x_1, \dots, x_m \}$ (resp.\ $\mathbf{y} = \{y_1, \dots, y_\ell \}$) be the positive (resp.\ negative) singular points of $\Sigma_{\xi}$. Then  $d\beta$ is nondegenerate on a small open neighborhood $U(\mathbf{x})$ of $\mathbf{x}$, where $\beta \coloneqq \alpha|_{\Sigma}$.  Let $W_{x_i}$ be the stable manifold of $x_i$ with respect to the gradient of the $1$-Morse function and let the $i$th skeleton $\op{Sk}_i$ be the closure of $W_{x_1}\cup\dots\cup W_{x_i}$. We order the points of $\mathbf{x}$ so that $W_{x_{i+1}}$ intersects the boundary of a small open neighborhood of $\op{Sk}_i$ along a sphere if $x_{i+1}$ is a Morse critical point and along a disk if $x_{i+1}$ is an embryonic point.  In particular we necessarily have $\text{ind}(x_1)=0$, but we do \emph{not} require $\text{ind}(x_i) \geq \text{ind}(x_j)$ for $i > j$. Such an arrangement is possible thanks to the assumption that there is no flow line of $\Sigma_{\xi}$ going from ${\bf y}$ to $\mathbf{x}$.

 There exists a  conformal modification $\beta\rightsquigarrow g \beta$, where $g$ is a positive function, so that it becomes Liouville on a tubular neighborhood $U(\op{Sk}_{m})$ of $\op{Sk}_{m}$ and $\p  U(\op{Sk}_m)$ is contact: Arguing by induction, suppose that $\beta$ is Liouville on $U(\op{Sk}_{i})$ such that $\p U(\op{Sk}_{i})$ is contact. {During the induction we will often reset notation, i.e., modify $\beta\rightsquigarrow g\beta$ and call the result the new $\beta$. We will explain the case where $x_{i+1}$ is Morse and $W_{x_{i+1}}\cap \p U(\op{Sk}_{i})$ is a Legendrian sphere $\Lambda \subset \p U(\op{Sk}_{i})$; the other cases are similar.} Using the flow of $\Sigma_{\xi}$, we may identify a tubular neighborhood of $W_{x_{i+1}} \setminus (U(\op{Sk}_{i}) \cup U(x_{i+1}))$ with $ [0,1]_r \times Y$, where $Y$ is an open neighborhood of the $0$-section in $J^1(\Lambda)$ such that:
	\begin{itemize}
		\item $ \{0\} \times Y \subset \p U(\op{Sk}_{i})$;
		\item $ \{1\}\times Y  \subset \p U(x_{i+1})$; and
		\item $\p_r$ is identified with $\Sigma_{\xi}$.
	\end{itemize}
It follows that one can write $\beta = g\lambda$ on $[0,1]\times Y$, where $\lambda$ is a contact form on $Y$ and $g$ is a positive function on $[0,1] \times Y$. Note that
$$d\beta = \p_r g \, dr \wedge \lambda + d_Y g \wedge \lambda +g \, d\lambda$$
is symplectic if $\p_r g>0$. By assumption we have $\p_r g>0$ when $r$ is close to 0 or 1. Rescaling {$\beta|_{U(x_{i+1})}$} by a large constant $K \gg 0$, we can extend $\beta|_{U(\op{Sk}_{i}) \cup U(x_{i+1})}$ to a Liouville form on $U(\op{Sk}_{{i+1}})$. Moreover, we can assume $\p U(\op{Sk}_{{i+1}})$ is transverse to $\Sigma_{\xi}$ by slightly shrinking $U(\op{Sk}_{{i+1}})$. Hence by induction we can arrange so that $\beta$ is a Liouville form on $U(\op{Sk}_{m})$.

The treatment of the negative singular points of $\Sigma_{\xi}$ is similar. Let $\op{Sk}'_{\ell}$ be the closure of the union of the unstable manifolds of $\mathbf{y}$. Then by the same argument we can assume that $\beta$ is a Liouville form on $-U(\op{Sk}'_{\ell})$, where the minus sign indicates the opposite orientation.

Using the flow of $\Sigma_{\xi}$, we can identify $\Sigma \setminus (U(\op{Sk}_{m}) \cup U(\op{Sk}'_{\ell}))$ with $\Gamma\times [-1,1]_s$ such that:
	\begin{itemize}
		\item $\Gamma\times \{-1\}$ is identified with $\p U(\op{Sk}_{m})$;
		\item $\Gamma \times \{1\}$ is identified with $\p U(\op{Sk}'_{\ell})$; and
		\item $\R\langle\p_s\rangle=\Sigma_{\xi}$.
	\end{itemize}
We can write $\beta = h\eta$ near $ \Gamma\times \{-1,1\}$, where $\eta$ is a contact form on $\Gamma$ and $h=h(s)$ is a positive function such that $h'(s) > 0$ near $\Gamma\times\{-1\}$ and $h'(s) < 0$ near $\Gamma\times\{1\}$. Extend $h$ to a positive function $\Gamma\times[-1,1]\to \R$ such that $h'(s)>0$ for $s<0$, $h'(0)=0$, and $h'(s)<0$ for $s>0$.
Let $f=f(s): \Gamma\times[-1,1]\to \R$ be {a nonincreasing function of $s$ such that $f(-1)=1$, $f(0)=0$, $f(1)=-1$, $f'(0)<0$, and $f^{(n)}(-1)=0=f^{(n)}(1)$ for all $n\geq 1$.} Then define $\rho= f\, dt + h\, \eta$ on $\R_t \times \Gamma\times[-1,1]$, $\rho=dt + \beta$ on $\R\times U(\op{Sk}_{m})$, and $\rho=-dt+\beta$ on $\R\times U(\op{Sk}'_{{\bf y}_\ell})$.  We leave it to the reader to check that $\rho$ is contact and that $\rho|_{\{0\}\times \Sigma}$ agrees with $\alpha|_{\Sigma}$ up to an overall positive function. (1) now follows from \cref{lem:char foliation}.

 (2) The ``if" direction follows from the proof of (1) and the ``only if" direction is clear.
\end{proof}

\section{Construction of mushrooms in dimension \(3\)} \label{sec:C-fold 3d}

In order to make a hypersurface $\Sigma \subset (M,\xi)$ Weinstein convex, we would like to modify the characteristic foliation $\Sigma_\xi$ so it is directed by a Morse vector field and then apply \cref{prop:convex criterion}.  (Note that going from Morse to Morse$^+$ is a $C^\infty$-generic  operation which can always be done by \cref{lem:C^infty perturb}.)  This will be achieved by certain $C^0$-small perturbations of $\Sigma$ which we call {\em mushrooms}. 
The mushrooms are most easily described in dimension $3$ and the general case will be constructed in \cref{sec:C-fold hd} using $3$-dimensional mushrooms. It turns out that mushrooms alone are enough to make any $\Sigma_\xi$ Morse if $\dim \Sigma=2$. If $\dim \Sigma>2$, then mushrooms are not quite sufficient and we will need an additional technical construction in \cref{sec:plug}.

The standard model of a mushroom will be constructed in a Darboux chart
\[
(\R^3_{z,s,t},\xi=\ker\alpha),\quad\alpha = dz + e^s \, dt.
\]
Let $\Sigma = \{z=0\}$ be the surface under consideration with normal orientation $\bdry_z$ and characteristic foliation $\Sigma_\xi$ directed by $\p_s$. The goal of this section is to ``fold" $\Sigma$ to obtain another surface $Z$ which coincides with $\Sigma$ outside of a compact set, and analyze the change in the dynamics of the characteristic foliations.

In \cref{subsec:PL fold} we construct a piecewise linear (PL) model $Z_{\PL}$ and then in \cref{subsec:smoothing PL fold} we round the corners of $Z_{\PL}$ to obtain a suitably generic smooth surface $Z$ such that the characteristic foliation $Z_{\xi}$ has the desired properties. 

\begin{remark}
 In an earlier version of the paper we constructed a mushroom whose base was smaller than the cap and discussed ``mushroom packing ratios."  In the current version they approximately have the same size and can be packed tightly.
\end{remark}

\subsection{Constructing the PL fold} \label{subsec:PL fold}

Choose a rectangle $\square = [0,s_0] \times [0,t_0] \subset \Sigma$, where $s_0,t_0>0$. We define $Z_{\PL}$ to coincide with $\Sigma$ outside of $\square$.

\begin{remark} \label{rmk: slightly more general}
 The more general case $[s_{-1},s_0]\times[t_{-1},t_0]$ can be computed analogously.  In what follows we can replace $1-e^{-s_0}$ by $e^{-s_{-1}}-e^{-s_0}$ and the $s$-width $\mathcal{S}(Z_{PL})$ becomes $s_0-s_{-1}$. 
\end{remark}

Choose $z_0>0$. We construct three rectangles $P_0,P_2,P_4$ and two {parallelograms} $P_1,P_3$ in $\R^3$, which, together with $\R^2_{s,t} \setminus \square$, glue to give $Z_{\PL}$, i.e., we define
\begin{itemize}
	\item $P_0 \coloneqq [0,s_0] \times   [-e^{-s_0/2}z_0, -e^{-s_0/2}z_0 + t_0]  \subset \{z=z_0\}$; 
	\item $P_i, i=1,\dots,4$, are the faces ($\not= P_0,\square$) of the convex hull of $P_0\cup \square$  (which is a parallelepiped), ordered counterclockwise so that $P_1 \subset \{s=0\}$.
\end{itemize}

\begin{figure}[ht]
	\begin{overpic}[scale=.5]{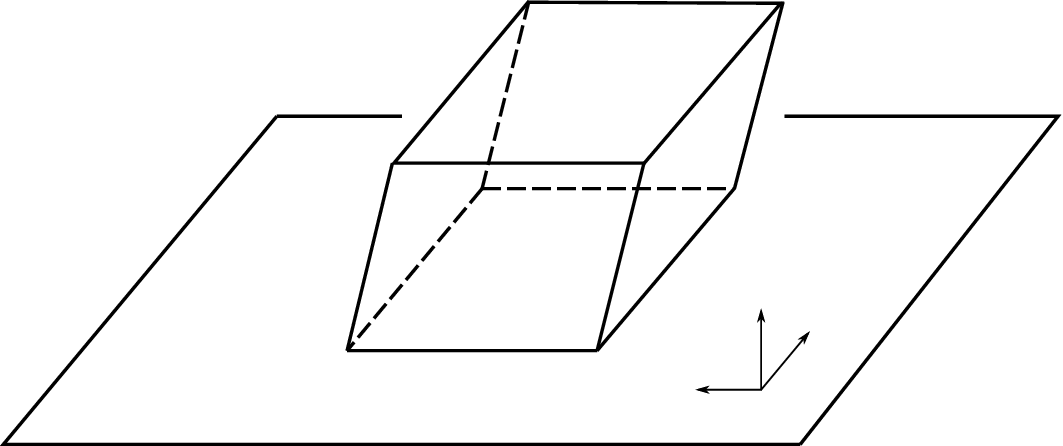}
		\put(50,15.5){$\square$}
		\put(53,33){$P_0$}
		\put(76.7,9){$s$}
		\put(65,1){$t$}
		\put(68,11){$z$}
	\end{overpic}
	\caption{The PL model $Z_{\PL}$. Here $P_0$ is the top face; $P_1$ and $P_3$ are the front and back faces, respectively; $P_2$ and $P_4$ are the right and left faces, respectively; and $\square$ is the bottom face which is \emph{not} part of $Z_{\PL}$.}
	\label{fig:PL_mushroom}
\end{figure}

\begin{definition} \label{defn:Z_PL fold}
We define the PL surface
$$Z_{\PL} \coloneqq (\Sigma \setminus \square) \cup {(\cup_{0 \leq i \leq 4} P_i).}$$
The rectangle $\square \subset \Sigma$ (resp.\ {$P_0$)} is called the  \emph{base (resp.\ cap) of the PL mushroom}  and $\cup_{0 \leq i \leq 4} P_i$ is called the {\em PL mushroom}. We refer to the modification $\Sigma\rightsquigarrow Z_{\PL}$ as ``growing a PL mushroom." 
\end{definition}

Observe that, away from the corners, the characteristic foliation $(Z_{\PL})_{\xi}$ on $Z_{\PL}$ satisfies
\begin{itemize}
	\item $(Z_{\PL})_{\xi} =\R\langle \p_s\rangle$ on $\Sigma \setminus \square$ and $P_0$;
	\item  on $P_2$ (resp.\ $P_4$), $(Z_{\PL})_{\xi}$ is directed by $\bdry_s$ (resp.\ $-\bdry_s$) for $s\in[0,\tfrac{s_0}{2})$, is singular along $s=\tfrac{s_0}{2}$, and is directed by $-\bdry_s$ (resp.\ $\bdry_s$) for $s\in(\tfrac{s_0}{2},s_0]$;  
	\item $(Z_{\PL})_{\xi}$ is the linear foliation on $P_1$ and $P_3$ with ``slopes" $-1$ and $-e^{-s_0}$, respectively, where ``slope" refers to the value of $dt/dz = -e^{-s}$. See \cref{fig:linear_foliation}.
\end{itemize}
 We refer to the singular line segments on $P_2$ and $P_4$ by $\mathcal{S}_-$ and $\mathcal{S}_+$ indicating their signs. 

\begin{figure}[ht]
	\begin{overpic}[scale=.5]{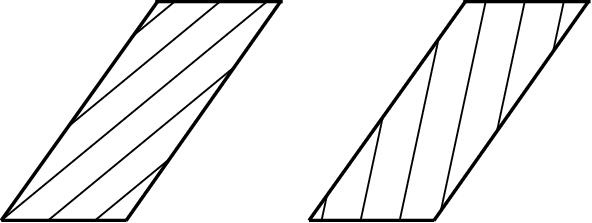}
		
	\end{overpic}
	\caption{The linear characteristic foliations on $P_1$ (left) and $P_3$ (right)  where $Z_{\PL}$ is sufficiently thin.}
	\label{fig:linear_foliation}
\end{figure}

We now analyze the dynamics of the PL flow on $Z_{\PL}$. Note that the flow lines are not necessarily uniquely determined by the initial conditions due to the presence of corners.

We begin by introducing a few quantities which characterize the various sizes of the mushrooms.

\begin{definition} \label{defn:parameters of fold}
Given $Z_{\PL}$ as above, its {\em $z$-height}, {\em $s$-width}, and {\em $t$-width} are given by:
$$\Zcal(Z_{\PL}) \coloneqq z_0, \quad \Scal(Z_{\PL}) \coloneqq s_0,\quad \Tcal(Z_{\PL}) \coloneqq t_0.$$
\end{definition}

The following lemma characterizes a key feature of $(Z_{\PL})_{\xi}$ when the parameters of the mushrooms are appropriately adjusted.

\begin{lemma} \label{lemma: char foliation on 3d PL-fold}
Fix $s_0,z_0>0$. If $t_0 < (1 - e^{-s_0}) z_0$, then
\be
\item  the unique  flow line of $(Z_{\PL})_{\xi}$ passing through $(-1,a)\in \R^2_{s,t}$, where $a \in (0,t_0)$,  either hits $P_0 \cap P_2$ or converges to $\mathcal{S}_-$  in forward time;
\item  the unique  flow line of $(Z_{\PL})_{\xi}$ passing through $(s_0+1,a), a \in (0,t_0)$,  either hits $P_0 \cap P_4$ or converges to $\mathcal{S}_+$  in backward time; and
\item  all the flow lines of $(Z_{\PL})_{\xi}$ passing through $(-1,a)$, $a\not\in [0,t_0]$, or $(s_0+1,a)$, $a\not\in [0,t_0]$, are unaffected.
\ee
\end{lemma}

\begin{proof}
(1) Since $t_0 - z_0 (1- e^{-s_0}) < 0$, it follows that the unique flow line passing through the point $(-1,a) \in \Sigma$ with $a \in (0,t_0)$  does one of three things in forward time:
\begin{itemize}
\item travels over $P_0$, enters $P_1$, and ends at $\mathcal{S}_-$;
\item travels over $P_0$ into $P_0\cap P_2$ (this happens with only one flow line); or
\item travels over $P_0,P_1,P_2$ in that order, enters $P_1$, and ends at $\mathcal{S}_-$.  
\end{itemize}
See \cref{fig:PL_mushroom_flow}. Claim (2) is similar and (3) is clear.
\end{proof}

\begin{figure}[ht]
	\begin{overpic}[scale=.5]{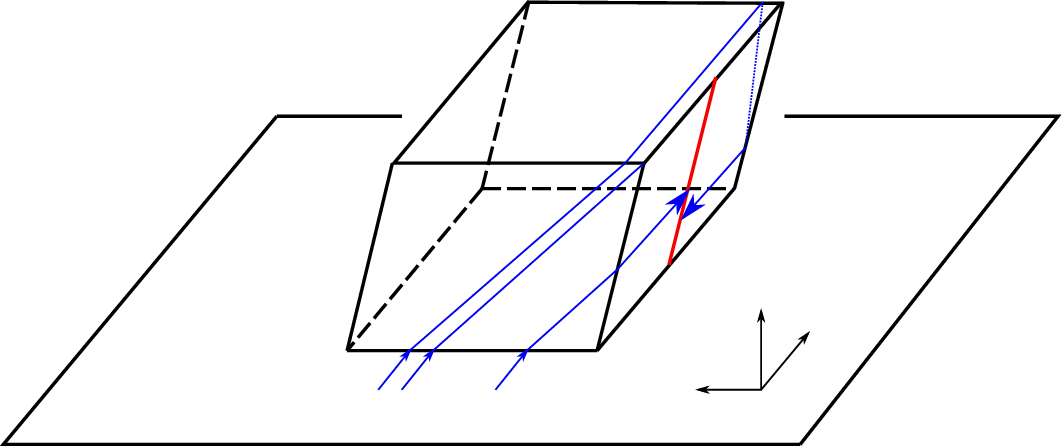}
	\end{overpic}
	\caption{Flow lines of $(Z_{\PL})_{\xi}$ limiting to $\mathcal{S}_-$ in red (the generic case) and one flow line limiting to $P_0 \cap P_2$.}
	\label{fig:PL_mushroom_flow}
\end{figure}

\subsection{Smoothing the PL fold} \label{subsec:smoothing PL fold}

In this subsection we construct a smoothing of $Z_{\PL}$.

\s\n
{\em Construction of the smoothing $Z$.}
Choose a  {\em small smoothing parameter} $\delta>0$ and a smooth ``profile function" $\phi: [0,z_0] \to [-\delta,\delta]$ such that $\phi(0)=\delta$, $\phi(z_0)= -\delta$ and has ``derivative $-\infty$'' at $z=0, z_0$.  

For each $z'\in (0,z_0)$, the slices $R_{z'}:=Z_{\PL}\cap \{z=z'\}$ are rectangles. When $z'=0$ or $z_0$, we take $R_{z'}=\bdry (Z_{\PL}\cap \{z=z'\})$.  For each  $\delta'\in \R$ with $|\delta'|$ small, let  $R_{z'}^{\delta'}\subset \{z=z'\}$ be the rectangle concentric to $R_{z'}$ and whose side lengths are $\delta'$ larger.  Then let $\widetilde R_{z'}^{\delta'}\subset \{z=z'\}$ be a smoothing of $R_{z'}^{\delta'}$ contained in the closure of the region between $R_{z'}^{\delta'}$ and $R_{z'}^{\delta'-\delta}$, which:
\be
\item[(R1)] agrees with $R_{z'}^{\delta'}$ on the complement of the open $\delta$-neighborhoods of the edges of $R_{z'}^{\delta'}$ parallel to $t=const$;
\item[(R2)] has nonzero curvature on these $\delta$-neighborhoods and is tangent to $R_{z'}^{\delta'}$ precisely at the midpoints of the edges of $R_{z'}^{\delta'}$ (i.e., when $s=s_0/2$); and
\item[(R3)] is smoothly varying with $z'$ and $\delta'$. 
\ee

The {\em $\phi$-smoothing $Z$ of $Z_{\PL}$ with smoothing parameter $\delta>0$ and profile function $\phi$} is obtained by modifying $Z_{\PL}$ as follows:
\be
\item replace $R_{z'}$ by $\widetilde R_{z'}^{\phi(z')}$ for $z\in (0,z_0)$;
\item remove the bounded component of  {$\{z=0\} \setminus \widetilde R_0^{\phi(0)}$;} and
\item adjoin the {closure of the} bounded component of {$\{z=z_0\}\setminus \widetilde R_{z_0}^{\phi(z_0)}$.}
\ee
The \emph{base of the mushroom} $\widetilde{\square} \subset \Sigma$ of $Z$ is the closure of $\Sigma\setminus Z$ and the {\em mushroom} is the closure of $Z\setminus \Sigma$. By construction $\widetilde{\square}$ converges to $\square$ when all the parameters tend to zero.

\s
The following proposition describes the key dynamical properties of $Z_{\xi}$. 

\begin{proposition} \label{prop:char foliation on 3d C-fold}
 Given $Z_{PL}$ with parameters $s_0,z_0,t_0,\epsilon$ satisfying $t_0 < (1-e^{-s_0}) z_0$, there exists a smoothing $Z$ of $Z_{PL}$ with small smoothing parameter $\delta>0$ and profile function $\phi: [0,z_0] \to [-\delta,\delta]$  whose vector field $Z_{\xi}$ satisfies the following {properties:}
\begin{itemize}
\item[(TZ0)]  $Z_{\xi}$ is gradient-like with respect to a Morse function $f_a: \widetilde Z\to \R$ which agrees with $s$ outside of $\widetilde\square$.
\item[(TZ1)] $Z_{\xi}$ has four nondegenerate singularities:  a positive source $e_+$ and a positive saddle $h_+$ near the midpoint of $\R^2_{s,t} \cap P_4$, and a negative sink $e_-$ and a negative saddle $h_-$ near the midpoint of $\R^2_{s,t}\cap P_2$.
\item[(TZ2)] {There is a unique flow line each from $e_+$ to $h_+$, from $e_+$ to $h_-$, from $h_+$ to $e_-$, and from $h_-$ to $e_-$, and the four flow lines bound a quadrilateral whose interior consists of flow lines from $e_+$ to $e_-$.}
\item[(TZ3)]  There exist $\kappa_1>\kappa_2>\kappa_3>0>\kappa_4>\kappa_5>\kappa_6$ such that all $\kappa_i\to 0$ as $\delta \to 0$ and the following hold:
\be
\item the stable manifold of $h_+$ (resp.\ $h_-$) intersects the line $\{s=-1\}$ at $(-1,t_0+\kappa_2)$ (resp.\ $(-1,\kappa_4)$),
\item the unstable manifold of $h_+$ (resp.\ $h_-$) intersects the line $\{s=s_0+1\}$ at $(s_0+1,t_0+\kappa_3)$ (resp.\ $(s_0+1,\kappa_5)$),
\item any flow line passing through $(-1,a)$, $a\in (\kappa_4,t_0+\kappa_2)$, converges to $e_-$ in forward time,
\item any flow line passing through $(s_0+1,a)$, $a\in (\kappa_5,t_0+\kappa_3)$, converges to $e_+$ in backward time,
\item a flow line passes through $(-1,a)$, $a\not\in [\kappa_6,t_0+\kappa_1]$, if and only if it passes through $(s_0+1,a)$, $a\not\in [\kappa_6,t_0+\kappa_1]$,
\item a flow line passes through $(-1,a)$, $a\in (t_0+\kappa_2,t_0+\kappa_1)$, if and only if it passes through $(s_0+1,a')$, $a'\in (t_0+\kappa_3,t_0+\kappa_1)$,
\item a flow line passes through $(-1,a)$, $a\in (\kappa_6,\kappa_4)$ if and only if it passes through $(s_0+1,a')$, $a'\in (\kappa_6,\kappa_5)$.
\ee 
\item[(TZ4)]  The flow lines described in (TZ2) and (TZ3) are all the flow lines that nontrivially intersect the mushroom. 
\end{itemize}
\end{proposition}

 In words, $Z_\xi$ blocks all flow lines that pass through an open interval that is close to $\{-1\}\times (0,t_0)$, slightly bends unblocked flow lines that pass through points close to $(-1,0)$ and $(-1,t_0)$ (with bending $\to 0$ as the smoothing parameter $\delta\to 0$), and leaves all other flow lines passing through $s=-1$ untouched. 

See \cref{fig:Z char foliation} for an illustration of the effect of a mushroom on the characteristic foliation  and also the values of $t$ where certain flow lines in (TZ3) intersect $s=-1$ or $s=s_0+1$.

\begin{figure}[ht]
	\begin{overpic}[scale=.55]{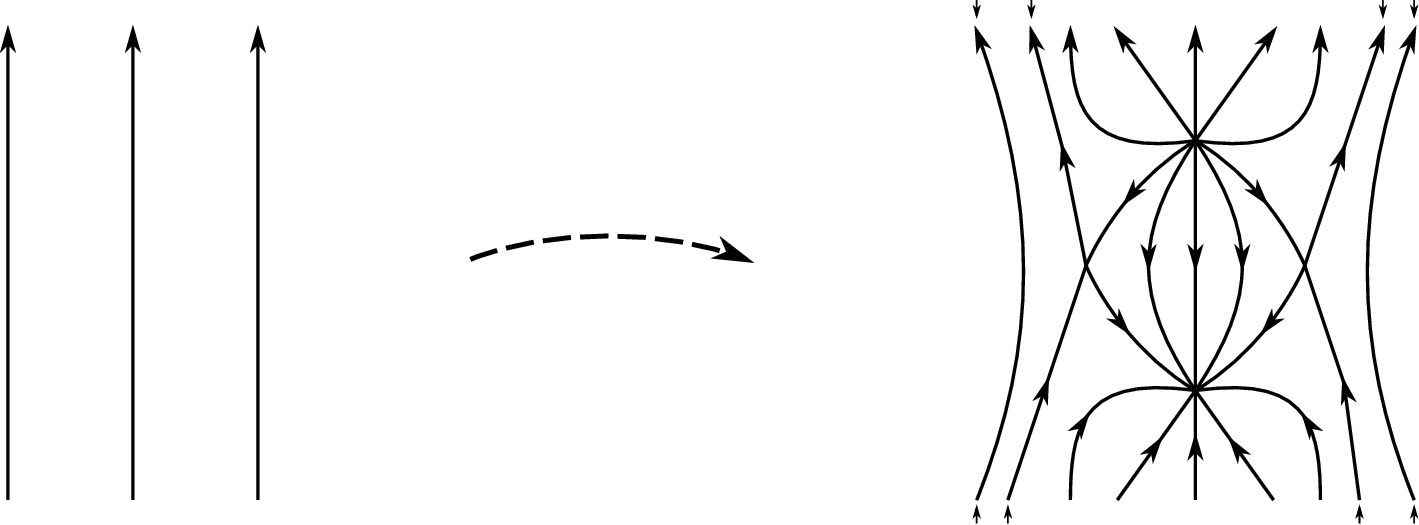}
		\put(36,21){\mbox{mushroom}}
		\put(73,17){\tiny{$h_+$}}
		\put(92.7,17){\tiny{$h_-$}}
		\put(86,28){\tiny{$e_+$}}
		\put(86,7.8){\tiny{$e_-$}}
	\put(62.7,-1.2){\tiny $t_0+\kappa_1$} \put(70.5,-1.2){\tiny $t_0+\kappa_2$} \put(95,-1.2){\tiny $\kappa_4$} \put(98.7,-1.2){\tiny $\kappa_6$}
	\put(62.7,37.7){\tiny $t_0+\kappa_1$} \put(72,37.7){\tiny $t_0+\kappa_3$} \put(96.5,37.7){\tiny $\kappa_5$} \put(98.7,37.7){\tiny $\kappa_6$}
	\end{overpic}
	\caption{The characteristic foliations before and after a mushroom.  The numbers $t_0+\kappa_1$ etc.\ are the $t$-coordinates where the indicated flow lines intersect $s=-1$ or $s=s_0+1$; see (TZ3). }
	\label{fig:Z char foliation}
\end{figure}

\begin{proof} 
 Let $Z$ be the smoothing of $Z_{PL}$ from above. 

(TZ1) We determine the singular points of $Z_\xi$ as follows: The first requirement for $\xi$ to be tangent to $Z$ is for $Z$ to be tangent to $\bdry_s$. Since $\widetilde R_{z'}^{\delta'}$ is tangent to $\bdry_s$ exactly at two points by (R2), i.e., when $s=s_0/2$, the singular points lie on the restriction of $Z$ to the slice $s=s_0/2$.  

We now choose $\phi$ such that $\phi'<0$ on $(0,c_0)$ and $(c_1,z_0)$ and $\phi'>0$ on $(c_0,c_1)$, where $0<c_0<c_1\ll z_0$. By the choice of $\phi$, there are four points where $Z\cap \{s=s_0/2\}$ is tangent to $\xi\cap \{s=s_0/2\}$. They occur as described in (TZ1) since $\phi'=0$ at $z=c_0,c_1$, which are both close to $z=0$. 

 (TZ0),  (TZ2)--(TZ5) then follow from \cref{lemma: char foliation on 3d PL-fold}.  The conditions $\kappa_2>\kappa_3$ and $\kappa_4>\kappa_5$ are required since the surface $Z$ was obtained from $\Sigma$ by pushing in the positive $z$-direction. 
\end{proof}

 The following remark will also be very useful later:

\begin{remark} \label{rmk: variant of prop:char foliation on 3d C-fold}
 \cref{prop:char foliation on 3d C-fold} also holds with (TZ1) replaced by:
\be
\item[(TZ1')] $Z_{\xi}$ has two singularities:  a positive birth-death singularity near the midpoint of $\R^2_{s,t} \cap P_4$ and a negative birth-death singularity near the midpoint of $\R^2_{s,t}\cap P_2$.
\ee
Moreover, there is a foliated $1$-parameter family of surfaces from $\{z=0\}$ to a slight upward translate of $Z$ (i.e., in the $z$-direction) whose characteristic foliations have no singularities except for $Z$.  For this $Z$ we take $\phi$ such that $\phi'<0$ for $z\not=c_0$ and $\phi'(c_0)=0$. Note that the singularities vanish if we take $\phi$ such that $\phi'<0$ for all $z$. 
\end{remark}

In view of \cref{prop:char foliation on 3d C-fold}, from now on we assume that all mushrooms satisfy $t_0 < (1-e^{-s_0}) z_0$.

\section{Convex surface theory revisited} \label{sec:CST revisited}

The goal of this section is  to give elementary, Morse-theoretic proofs  of \cref{thm:genericity} and \cref{thm:family genericity} in dimension $3$ using the folding techniques developed in \cref{sec:C-fold 3d}. In dimension $3$, \cref{thm:genericity} was proved by Giroux in \cite{Gi91} in a stronger form where $C^0$ is replaced by $C^\infty$.  \cref{thm:family genericity} can be inferred from Giroux's work on bifurcations \cite{Gi00} and the bypass-bifurcation correspondence. The technical heart of Giroux's work is based on the study of dynamical systems of vector fields on surfaces, a.k.a., Poincar\'e-Bendixson theory. In particular, one invokes a deep theorem of Peixoto \cite{Pei62} to prove the $C^\infty$-version of \cref{thm:genericity} and much more work to establish \cref{thm:family genericity}.

Our  proof  strategy is the following: First apply a $C^\infty$-small perturbation of $\Sigma \subset (M^3,\xi)$ such that the singularities of $\Sigma_{\xi}$ become Morse. There exists a finite collection of pairwise disjoint transverse arcs $\gamma_i, i \in I$, in $\Sigma$ such that any flow line of $\Sigma_{\xi}$ passes through some $\gamma_i$. In \cref{subsec:3d plug} we will construct a $3$-dimensional plug supported on a {small flow box $B_i=[0,s_i]\times [0,t_i]$, where $\gamma_i=\{\tfrac{s_i}{2}\}\times (\epsilon,t_i-\epsilon)$ and $\epsilon>0$ is small, such that no flow line that enters through $\{0\}\times (\epsilon,t_i-\epsilon)$ can leave through $\{s_0\}\times [0,t_i]$,} i.e., they all necessarily converge to singularities in the plug. Each plug consists of a large number of mushrooms constructed in \cref{sec:C-fold 3d}. This proves \cref{thm:genericity}. To prove \cref{thm:family genericity}, we slice $\Sigma \times [0,1]$ into thin layers using $\Sigma_i \coloneqq \Sigma \times \{\tfrac{i}{N}\}, 0 \leq i \leq N$, for large $N$ such that the difference between $(\Sigma_i)_{\xi}$ and $(\Sigma_{i+1})_{\xi}$ is small. (By ``small'' we mean the vector fields in question are $C^0$-close to each other. The global dynamics of $(\Sigma_i)_{\xi}$ may still drastically differ from that of $(\Sigma_{i+1})_{\xi}$.) Within each layer we insert plugs on $\Sigma_i$ as in the case of a single surface so that the isotopy from $\Sigma_i$ to $\Sigma_{i+1}$ is through a  $1$-Morse family of surfaces, i.e., $(\Sigma_t)_{\xi}$ is a $1$-Morse family  for all $\tfrac{i}{N} \leq t \leq \tfrac{i+1}{N}$. For technical reasons, it is desirable to eliminate the plugs created on $\Sigma_i$ when we reach $\Sigma_{i+1}$, replacing them by new plugs on $\Sigma_{i+1}$, so that one can inductively run from $i=0$ to $i=N$ and make all intermediate surfaces  $1$-Morse.   Then the only obstructions to convexity occur at finitely many instances where the surface is  $1$-Morse but not $1$-Morse$^+$,  corresponding to bypass attachments.

This section is organized as follows: In \cref{subsec:3d plug} we describe $3$-dimensional plugs and in \cref{subsection: installing uninstalling plugs} we explain how to ``install'' and ``uninstall'' plugs. The higher-dimen\-sional plugs will be described in \cref{sec:plug}. We then use this technology to prove \cref{thm:genericity} in \cref{subsec:genericity in 3d} and \cref{thm:family genericity} in \cref{subsec:family genericity in 3d}.

\subsection{\(3\)-dimensional plugs} \label{subsec:3d plug}

The construction of a plug is local. Consider $M = [-z_0,z_0] \times [0,s_0] \times [0,t_0]$ with coordinates $(z,s,t)$, contact form $\alpha = dz + e^s dt$, {and contact structure $\xi=\ker \alpha$.} Here $z_0,s_0,t_0>0$ are arbitrary, but for most of our applications, we should think of $z_0,s_0$ as being much smaller than $t_0$. In other words,  the condition $t_0 < (1-e^{-s_0}) z_0$ in \cref{prop:char foliation on 3d C-fold}  will not be satisfied.

Consider the surface $B = \{0\} \times [0,s_0] \times [0,t_0]$ with $B_{\xi} =\R\langle \p_s\rangle$.   We will refer to $(M,\alpha)$ as a {\em standard contact neighborhood of $B$ with parameters $s_0,t_0,z_0$.}  Let $\p_{-} B = \{0\} \times \{0\} \times [0,t_0]$ and $\p_{+} B = \{0\} \times \{s_0\} \times [0,t_0]$.  Pick a large integer $N$ such that $N\equiv 3 \op{mod } 4$. Let $\square_{k,l} \subset B$ be boxes defined by
\begin{equation*}
\square_{k,l} \coloneqq \left[ \frac{2l-1}{5} s_0, \frac{2l}{5} s_0 \right] \times \left[ \frac{4k+2l-1}{N} t_0, \frac{4k+2l+2}{N} t_0 \right],
\end{equation*}
where $0 \leq k < \lfloor N/4 \rfloor, l=1,2$. See \cref{fig:mushroom stack}.  In words, since the mushrooms can be packed ``tightly", it suffices to arrange two rows of mushrooms. 

\begin{figure}[ht]
	\begin{overpic}[scale=1.5]{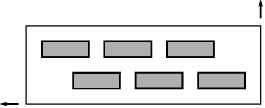}
		\put(3,3){$t$}
		\put(94,36){$s$}
	\end{overpic}
	\caption{Bases of the mushrooms on $B$. Here {$N=15$.}}
	\label{fig:mushroom stack}
\end{figure}

Applying \cref{prop:char foliation on 3d C-fold}, we install pairwise disjoint mushrooms $Z_{k,l}$ on $B$ such that the base of each $Z_{k,l}$ approximately equals $\square_{k,l}$.   The key property of the dynamics of the resulting $B^{\vee}_{\xi}$ is as follows: Suppose $\epsilon>0$ is small and we choose $N\gg 1/\epsilon$. Let $x =(z(x), s(x), t(x))\in \p_{-} B$. 
\be
\item If $t(x) \in (\epsilon, t_0-\epsilon)$, then {the maximal} possibly broken flow line of $B^{\vee}_{\xi}$ passing through $x$ converges to a sink in forward time.
\item  If a flow line passes through $x$ and does not converge to sink, then it exits along $y\in \bdry_+ B$ and $|t(y)-t(x)|<\epsilon$.
\ee

\subsection{Barricades} \label{subsection: barricades}

{\em In this subsection only, let $\Sigma$ be a manifold of dimension $m$ and $v$ a vector field on $\Sigma$ with only Morse singularities.}  We also fix a Riemannian metric on $\Sigma$. 

\begin{definition}
A {\em flow box} is an embedded cylinder $B=[0,s_0]\times D^{m-1}\subset \Sigma$ with coordinates $(s,x)$ over the disk $D^{m-1}$ such that $v|_B=\bdry_s$.
\end{definition}

Given $\epsilon>0$ small, let $B^\epsilon= [\epsilon,s_0-\epsilon]\times D^{m-1}_\epsilon$ be a slightly smaller flow box (called a {\em shrinkage of $B$}), where $D^{m-1}_\epsilon \subset D^{m-1}$ is a disk such that every point $x\in D^{m-1}\setminus D^{m-1}_\epsilon$ has metric distance $<\epsilon$ from $\bdry D^{m-1}_\epsilon$ and $\bdry D^{m-1}$.  

\begin{definition} [Barricade] 
A collection $B_I=\{B_i= [0,s_i] \times D^{m-1}\}_{i\in I}$ of pairwise disjoint flow boxes for $\Sigma$ is a \emph{barricade for $v$} if for each $B_i$ there exist a small constant $\epsilon_i>0$ and a shrinkage $B_i^{\epsilon_i}$ and the following hold:
\be
\item[(*)] each flow line of $v$ intersects some $B_i^\epsilon$ and for any $x \in \Sigma$ which neither is a singularity of $v$ nor is contained in any $B_i^{\epsilon}$, the flow line of $v$ passing through $x$ enters some $B_i^{\epsilon}$ or limits to some Morse singularity in forward time (resp.\ in backward time). 
\item[(**)] $B_I$ is locally finite, i.e., each compact subset of $\Sigma$ intersects only a finite number of $B_i$.
\ee
\end{definition}

The following theorem of Wilson~\cite[Theorem A]{Wi66}, slightly adapted to our situation, guarantees the existence of barricades:

\begin{theorem}[Wilson] \label{thm: barricade}
A vector field $v$ on a manifold $\Sigma$ with only Morse singularities has a barricade $B_I$.
\end{theorem}

The following lemma on taking refinements is immediate:

\begin{lemma}[Refinement] \label{lemma: refinement}
Given a flow box $B=[0,s_0]\times D^{m-1}$, a properly embedded submanifold $Z$ of $D^{m-1}$, and $\epsilon,\delta>0$ small, there exists a finite disjoint collection $\mathcal{B}$ of flow boxes $B_{1j}$, $j=1,\dots, k_1$, and $B_{2j}$, $j=1,\dots, k_2$, such that:
\be
\item $B_{1j}\subset (0,s_0/2)\times D^{m-1}$ and $B_{2j}\subset (s_0/2,s_0)\times D^{m-1}$;
\item $\pi_{D^{m-1}}(\cup_{j=1}^{k_1} B_{1j})$ does not intersect $Z$ and $\pi_{D^{m-1}}(\cup_{j=1}^{k_2} B_{2j})$ is contained in a $\delta$-neighborhood of $Z$; and
\item $\mathcal{B}$ is a barricade for the shrinkage $B^\epsilon$. 
\ee 
Here $\pi_{D^{m-1}}: B\to D^{m-1}$ is the projection onto the second factor. 
\end{lemma}

\subsection{Convex approximation in dimension \(3\)} \label{subsec:genericity in 3d}

Here we prove \cref{thm:genericity} in dimension $3$.

Given  any closed surface $\Sigma \subset (M,\xi)$, it is well-known (see e.g.\ \cite[Section 4.6]{Gei08}) that, after a $C^\infty$-small perturbation, we can assume that $\Sigma_{\xi}$  has only Morse type singularities.  By \cref{thm: barricade}, a barricade $B_I=\{B_i=[0,s_i]\times[0,t_i]\}_{i\in I}$ exists for $\Sigma_\xi$; moreover $I$ can be taken to be finite since $\Sigma$ is closed. Each $B_i$ has a standard contact neighborhood with parameters $s_i,t_i,z_i$, where $z_i>0$ is small and we  construct a $C^0$-small modification  $\Sigma^{\vee}$ of $\Sigma$ by replacing every $B_i$ by the plug $B_i^{\vee}$. The characteristic foliation $\Sigma^{\vee}_{\xi}$ satisfies Conditions (M1)--(M3) of \cref{prop:morse criterion}  and is Morse.

After a further $C^\infty$-small perturbation if necessary,  $\Sigma^{\vee}$ can be made  Morse$^+$:  If there exists a ``retrogradient" flow line $\ell$ from a negative index $1$ singularity to a positive index $1$ singularity, we take a flow box $B=\{0\}\times [0,s_0]\times[0,t_0]\subset \Sigma^\vee$ such that $\ell$ intersects $B^\epsilon$ exactly once and $B$ has a standard contact neighborhood with parameters $s_0,t_0,z_0$. The retrogradient flow line can be eliminated by taking a small {\em nonnegative} function $h:B\to \R_{\geq 0}$ with support on $B^{\epsilon/2}$ and replacing $B$ by $z=h(s,t)$; the modification has the effect of pushing the holonomy from $s=0$ to $s=s_0$ in the negative $t$-direction. 

Since $\Sigma^\vee$ is now Morse$^+$, it is Weinstein convex by \cref{prop:convex criterion}.

\subsection{Installing and uninstalling plugs} \label{subsection: installing uninstalling plugs}

The construction of a plug $B^\vee$ was sufficient to prove \cref{thm:genericity} in dimension $3$. In order to prove \cref{thm:family genericity}, we also need to interpolate between $B$ and $B^{\vee}$ with some control of the intermediate dynamics. We now explain this procedure.

Let $(M=[-z_0,z_0]\times[0,s_0]\times[0,t_0],\xi=\ker (dz+e^sdt))$ be as before and let $B_z \coloneqq \{z\} \times [0,s_0] \times [0,t_0]$. Replace $B_{z_0/2}$ by a plug $B^{\vee}_{z_0/2}$ such that its analogously defined $z$-height satisfies $\Zcal(B^{\vee}_{z_0/2}) \ll \tfrac{z_0}{2}$ and in particular $B^{\vee}_{z_0/2}$ is still contained in $M$.

\begin{figure}[ht]
	\begin{overpic}{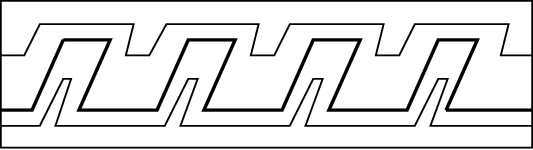}
		\put(101,-1){$B_0$}
		\put(101,6.5){$B^{\vee}_{z_0/2}$}
		\put(101,26){$B_{z_0}$}
	\end{overpic}
	\caption{The interpolation between $B_0, B^{\vee}_{z_0/2}$ and $B_{z_0}$; {the $z$-height of $B^{\vee}_{z_0/2}$ not drawn to scale.}}
	\label{fig:plug unplug}
\end{figure}

For the moment consider the PL model of the plug $B^{\vee}_{z_0/2}$, i.e., each mushroom $Z$ involved in the construction is replaced by the corresponding $Z_{\PL}$. It is fairly straightforward to foliate the regions bounded between $B_0$ and $B^{\vee}_{z_0/2}$ and between $B^{\vee}_{z_0/2}$ and $B_{z_0}$ by a family of PL surfaces; see \cref{fig:plug unplug} for a schematic picture. Then one can apply the smoothing scheme from \cref{subsec:smoothing PL fold} to smooth the corners of the leaves simultaneously and obtain the desired foliation $M = \cup_{0 \leq a \leq z_0} \widetilde{B}_a$, where $\widetilde{B}_0 = B_0$, $\widetilde{B}_{z_0} = B_{z_0}$, and $\widetilde{B}_{z_0/2}$ is the smoothed version of $B^{\vee}_{z_0/2}$.

To analyze the dynamics of $(\widetilde{B}_a)_{\xi}$ for each $a \in [0,z_0]$, we introduce the partially defined and possibly multiple-valued holonomy map $\rho_a: \p_{-} \widetilde{B}_a \dashrightarrow \p_{+} \widetilde{B}_a$, where  $\p_+ \widetilde{B}_a$ (resp.\ $\p_- \widetilde{B}_a$) is the side $s=s_0$ (resp.\ $s=0$)  of $\widetilde{B}_a$ as before: Given $x \in \p_{-} \widetilde{B}_a$,  if there exists a possibly broken partial flow line of $(\widetilde{B}_a)_{\xi}$ starting from $x$ and ending at $y \in \p_{+} \widetilde{B}_a$, then $y\in \rho_a(x)$. Note that such $y$ may not be unique. If there is no such flow line, then $\rho_a (x)$ is not defined.

Define the  {\em internal discrepancy} 
\begin{equation*}
\norm{\rho_a} \coloneqq \sup_{x \in \p_{-} \widetilde{B}_a} \abs{t(x) - t(\rho_a(x))},
\end{equation*}
where $t(x)$ refers to the $t$-coordinate of the point $x$; $\abs{t(x) - t(\rho_a(x))}=0$ if $\rho_a(x)$ is not defined; and the supremum is taken over all possible $\rho_a(x)$ if $\rho_a$ is not single-valued at $x$.

The following lemma will be important for our applications.

\begin{lemma} \label{lem:small holonomy 3d}
	The internal discrepancies $\sup_{0 \leq a \leq z_0} \norm{\rho_a} \to 0$ as $N \to \infty$.
\end{lemma}

\begin{proof}
 The lemma is not a statement about blocking and is rather a statement about the $t$-widths of the mushrooms $Z_{k,l}$: We will treat the case where $a\in[0,\tfrac{z_0}{2}]$. If a flow line enters a box $\square_{k,l}$ along the bottom and exits from the top, the maximum it is moved in the $t$-direction is the width $\tfrac{3t_0}{N}$ of the box.  Since a flow line or broken flow line passes through at most $2$ boxes, $\norm{\rho_a}\leq \tfrac{6t_0}{N}$.  
\end{proof}

We call the foliation from  $\widetilde B_0=B_0$ to $\widetilde B_{z_0/2}$  \emph{installing a plug} and the foliation from  $\widetilde B_{z_0/2}$ to $\widetilde B_{z_0}=B_{z_0}$  \emph{uninstalling a plug}. Then \cref{lem:small holonomy 3d} basically says that neither installing nor uninstalling a plug affects the local holonomy by much. For the rest of \cref{sec:CST revisited}, we assume that $N \gg 0$ without further mention.

 The following is based on the construction of mushrooms in \autoref{subsec:smoothing PL fold} and its slight generalization to $1$-parameter families:

\begin{lemma} \label{lemma: only 1-Morse singularities}
$(\widetilde B_a)_\xi$, $a\in[0,z_0]$, is gradient-like with respect to a $1$-Morse function $f_a: \widetilde B_a\to \R$ which agrees with $s$ on $\bdry \widetilde B_a$.
\end{lemma}

\subsection{Parametric convex approximation in dimension \(3\)} \label{subsec:family genericity in 3d}

Here we prove \cref{thm:family genericity} in dimension $3$.

Consider a contact structure $\xi$ on $\Sigma \times [0,1]$ such that $\Sigma \times \{0,1\}$ is Morse$^+$ in the sense of \cref{defn:morse-type hypersurface}. The goal is to show that up to an isotopy relative to the boundary,  $((\Sigma_t)_{\xi})_{t\in[0,1]}$ is a $1$-Morse family, where  $\Sigma_t \coloneqq \Sigma \times \{t\}$.

Define $L \coloneqq  \cup_{t\in[0,1]}\{x \in \Sigma_t  ~|~ \xi_x = T_x  \Sigma_t  \}$. Up to a $C^\infty$-small perturbation of $\xi$, we can assume that $L$ satisfies the following:
\begin{enumerate}
	\item[(S1)] $L$ is a properly embedded 1-submanifold; 
	\item[(S2)] the singularities of $(\Sigma_t)_{\xi}$ are $1$-Morse 
for all $t$; and
	\item[(S3)] the restricted coordinate function $t|_L: L \to [0,1]$ is Morse and all its critical points have distinct critical values.
\end{enumerate}

Suppose $0<a_1<\dots<a_m<1$ are the critical values of $t|_L$, which we assume to be irrational.  For each $t\in[0,1]$ there exists a barricade $B_{I^t}$ for $\Sigma_t$; moreover $B_{I^t}$ is a barricade for any vector field that is  sufficiently  close to $(\Sigma_t)_\xi$.  By the compactness of $[0,1]$, there exists  an integer $K \gg 0$ such that, for $i=0,1,\dots, K$, $B_{I_i}$, $I_i=I^{i/K}$, is a barricade for all $\Sigma_t$, $t\in[\tfrac{i-1}{K}, \tfrac{i+1}{K}]\cap [0,1]$.  Note that,  for each $a_j$, there exist unique $j_+$, $j_-$ such that $j_+ = j_- + 1$ and $\tfrac{j_-}{K} < a_j < \tfrac{j_+}{K}$.

Let $\pi: \Sigma\times[0,1] \to \Sigma$ be the projection onto the first factor. We claim that for each $i=0,\dots, K-1$ there are refinements of $B_{I_{i}}$ and $B_{I_{i+1}}$ (by abuse of notation we keep the same notation for the refinements) such that
\begin{equation} \label{eqn: disjoint}
\pi(B_{I_{i+1}})\cap \pi(B_{I_{i}})=\emptyset,
\end{equation}
and moreover we may choose the refinement so that the new $B_{I_{i+1}}$ remains a barricade for all $\Sigma_t$, $t\in[\tfrac{i}{K}, \tfrac{i+2}{K}]\cap [0,1]$.
The claim follows from viewing $B_{I_{i}}$ and $B_{I_{i+1}}$ as thin neighborhoods of collections $\gamma_i$ and $\gamma_{i+1}$ of arcs, taking their intersection $Z=\gamma_i\cap \gamma_{i+1}$ which we may take to be transverse, and applying \cref{lemma: refinement}.

We divide the proof into several steps.

\s\n
\textsc{Step 1.} \emph{From $\Sigma_0$ to $\Sigma_{1/(N'K)}^{\vee}$ where $N'>0$ is a large integer.}

\s
 Consider a flow box $B_i=[0,s_0] \times [0,t_0]$ of $B_{I_0}$. Let $ \p_{+} B_{i} = \{s_0\} \times [0,t_0]$ and $\p_-B_{i}= \{0\} \times [0,t_0]$.

For each positive integer $r$, define the \emph{external holonomy} $\widehat{\rho}_{i,r}: \p_{+} B_{i} \dashrightarrow \p_{-} B_{i}$ --- a multiple-valued, partially defined, $r$th return map from $\p_+ B_{i}$ to $\p_- B_{i}$ of $(\Sigma_0)_{\xi}$ ---  as follows:  For any $x \in \p_{+} B_{i}$, a point $y \in \p_{-} B_{i}$ is in the image $\widehat{\rho}_{i,r}(x)$ if there exists a possibly broken flow line $c:[0,1] \to \Sigma_0 $ such that $c(0)=x$, $c(1)=y$, and $c$ passes through $\op{int}(B_i)$ $(r-1)$ times. 
Of course $\widehat{\rho}_{i,r}$ is not necessarily defined on all of $\p_{+} B_{i}$ and when it is defined, it is not necessarily single-valued.

Since $(\Sigma_0)_{\xi}$ is Morse by assumption, (A) there exists $\delta>0$ such that
\begin{equation*}
\norm{\widehat{\rho}_{i,r}} \coloneqq  \inf_{x \in \p_{+} B_{i}}  \abs{t(x) - t(\widehat{\rho}_{i,r}(x))} > \delta,
\end{equation*}
 where we are taking $t(\widehat{\rho}_{i,r}(x))=\infty$ if $\widehat{\rho}_{i,r}(x)$ does not exist. 
Otherwise, there is a sequence of points $x_j\in \p_{+} B_{i}$ such that $\abs{t(x_j) - t(\widehat{\rho}_{i,r}(x_j))}\to 0$ and the compactness of the sequence of broken flow lines gives us $x_\infty\in  \p_{+} B_{i}$ such that $\abs{t(x_\infty) - t(\widehat{\rho}_{i,r}(x_\infty))}= 0$, which contradicts (M3) from \cref{prop:morse criterion}.   Moreover, the Morse condition implies that (B) there exists $r_0>0$ finite such that $\norm{\widehat{\rho}_{i,r_0}}=\infty$ for all $i\in I$. 

 We then install plugs $B_{I_0}^\vee$ on $B_{I_0}$ as described in \cref{subsec:3d plug} and \cref{subsection: installing uninstalling plugs} and obtain a foliation between $\Sigma_0$ and $\Sigma_0^{\vee}$.  As long as we take $N\gg 0$, i.e., the individual mushrooms are very small,  the internal discrepancies are $\ll \delta$ by \cref{lem:small holonomy 3d}. Together with (A) and (B), it follows that all the leaves of the foliation are Morse. 

{\em For convenience we assume that $\Sigma_0^{\vee}$ agrees with $\Sigma_0$ on the complement of $B_{I_0}$, and that the difference is contained in a small invariant neighborhood of $B_{I_0}$.  Also, below we construct $1$-parameter families of embedded surfaces that are disjoint away from a subset on which they all agree; the perturbation into a family of disjoint embedded surfaces is done by flowing in the transverse direction for a short time and will not be done explicitly.}

In order to interpolate between $\Sigma_0^{\vee}$ and $\Sigma_{1/(N'K)}^\vee$ for a large integer $N'>0$  which we take to be odd,  we use
$B_{I_1}$ satisfying \eqref{eqn: disjoint}. If $N'\gg 0$, then there is a $1$-parameter family of  embedded  surfaces $F_s \subset \Sigma \times [0,\tfrac{1}{N'K}]$, $s \in [0,1]$, such that $F_0 = \Sigma_0 \setminus N(B_{I_0})$, $F_1 \cap \Sigma_{1/(N'K)} \supset N(B_{I_1})\times\{\tfrac{1}{N'K}\}$, $\p F_s = \p N(B_{I_0})\times\{0\}$ for all $s \in [0,1]$,  the interiors of $F_s$ are disjoint,  and the $(F_s)_{\xi}$, $s\in[0,1]$, are $\epsilon$-close to one other so that $B_{I_0}$ is a barricade for all $(N(B_{I_0})\times\{0\})\cup F_s$; in particular, no new singularities are introduced in this process. The barricading condition can be guaranteed by having chosen $N'\gg 0$.
See the upper-left corner of \cref{fig:first bump} for an illustration of this procedure. By the barricading condition the surfaces  $(N(B_{I_0})\times\{0\})^\vee \cup F_s$ are Morse for all $s \in [0,1]$, where $(N(B_{I_0})\times\{0\})^\vee$ is $N(B_{I_0})\times\{0\}$ with $B_{I_0}^\vee$ installed.  

\begin{figure}[ht]
	\begin{overpic}[scale=.67]{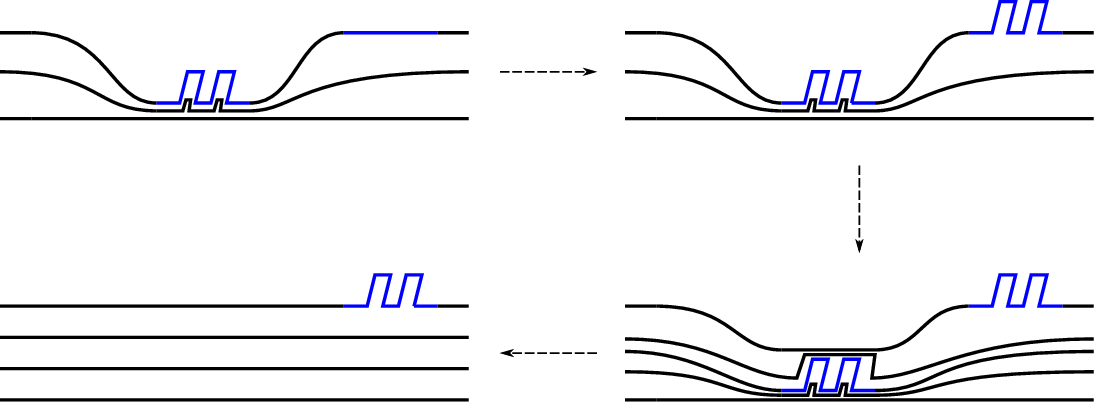}
		\put(-2,25){\small{$0$}}
		\put(-6.2,34){\small{$\frac{1}{N'K}$}}
	\end{overpic}
	\caption{Interpolation between $\Sigma_0$ and $\Sigma_{1/(N'K)}^{\vee}$ by Morse surfaces. The blue parts represent $B_{I_i}, i=0,1$.}
	\label{fig:first bump}
\end{figure}

Next we install a plug on $B_{I_1}\times \{\tfrac{1}{N'K}\} \subset  (N(B_{I_0})\times\{0\})^\vee \cup F_1 $, uninstall the plug on $B_{I_0}^{\vee}$, and lift the resulting surface up to $\Sigma_{1/(N'K)}$, as shown in the upper-right, lower-right, and lower-left corners of \cref{fig:first bump}, respectively. Moreover all the intermediate surfaces are Morse by analogous reasons. This finishes our construction of the foliation from $\Sigma_0$ to $\Sigma_{1/(N'K)}^{\vee}$.

\s\n
\textsc{Step 2.} \emph{From $\Sigma_{1/(N'K)}^{\vee}$ to $\Sigma_{1_{-}/K}^{\vee}$, where $1_{-}/K < a_1< 1_+/K$.}

\s
 Switching back and forth between  $B_{I_0}$ and $B_{I_1}$, we similarly construct the Morse foliation from $\Sigma_{1/(N'K)}^\vee$ to $\Sigma_{2/(N'K)}^\vee$, from $\Sigma_{2/(N'K)}^\vee$ to $\Sigma_{3/(N'K)}^\vee$, and so on as in Step 1, until we get to $\Sigma_{1/K}^\vee$.  Between $\Sigma_{1/K}^\vee$ and $\Sigma_{2/K}^\vee$ we use $B_{I_1}$ and $B_{I_2}$. Continuing in this manner we get to $\Sigma_{1_{-}/K}^{\vee}$. 

\s\n
\textsc{Step 3.} \emph{From $\Sigma_{1_{-}/K}^{\vee}$ to $\Sigma_{1_{+}/K}^{\vee}$.}

\s
The only modification needed in this step is due to the fact that the vector fields $(\Sigma_{1_{-}/K})_{\xi}$ and $(\Sigma_{1_{+}/K})_{\xi}$ are not $C^\infty$-close to each other in the usual sense. Rather, one observes either the birth or the death of a pair of nearby Morse singularities as we go from $(\Sigma_{1_{-}/K})_{\xi}$ to $(\Sigma_{1_{+}/K})_{\xi}$. In either case, we slightly modify the notion of barricades $B_{I_{1_{\pm}}}$ so that the unique (short) flow line connecting the pair of Morse singularities is the only flow line that does not pass through $B_{I_{1_{\pm}}}$. Similar remarks apply to all $a_i, 1 \leq i \leq m$.

\s\n
\textsc{Step 4.} \emph{From $\Sigma_{(K-1)/K}^{\vee}$ to $\Sigma_1$.}

\s
In this final step, the only new ingredient is to uninstall the plugs as we go from $\Sigma_1^{\vee}$ to $\Sigma_1$. By assumption $\Sigma_1$ is Morse and in fact convex. Hence by the same holonomy bound as in Step 1, all the intermediate surfaces are Morse.

\s
Finally we have foliated $\Sigma\times[0,1]$ by surfaces of the form $\Sigma_t$ which are all Morse. The only obstruction to convexity occurs when $(\Sigma_t)_\xi$ is Morse but not Morse$^+$ (and this corresponds to a bypass attachment; see \cref{prop:bb-correspondence3D}). This concludes the proof of \cref{thm:family genericity} in dimension $3$.

\subsection{Further remarks}

Compared to earlier groundbreaking works of Bennequin \cite{Ben83} and Eliashberg \cite{Eli92}, convex surface theory is  a more systematic framework for studying embedded surfaces in contact $3$-manifolds. It is sufficiently powerful that basically all known classification results of contact structures or Legendrian knots in this dimension follow from this theory.

The only ``drawback'' of convex surface theory, at least in its original form \cite{Gi91,Gi00}, is that the monster of dynamical systems on surfaces is always lurking behind the story. More precisely, if one just wants to classify contact structures or Legendrian knots up to isotopy, then the problem often reduces to a combinatorial one by combining Giroux's theory with, say, the bypass approach of \cite{Hon00}. However, if one wants to obtain higher homotopical information of the space of contact structures (say $\pi_n$ for $n\geq 1$), then some serious work on higher codimensional degenerations of Morse-Smale flows seems inevitable.

As an example, in \cite{Eli92} Eliashberg outlined the proof that the space of tight contact structures on $S^3$ is homotopy equivalent to $S^2$. This particular result is based on the study of characteristic foliations on $S^2 \subset S^3$, which is particularly simple since we never have periodic orbits. In more general contact manifolds such as $T^3$, one cannot necessarily rule out periodic orbits from characteristic foliations, and hence the bifurcation theory quickly becomes unwieldy (the work \cite{Ngo05} probably comes close to the limit of what one can do). However, in light of our reinterpretation/simplification of Giroux's theory, it suffices to understand the space of Morse gradient vector fields, instead of general Morse-Smale vector fields.

We hope our techniques can be applied to future studies of homotopy types of the space of contact structures. This topic however will not be pursued any further in this paper.

\section{Construction of mushrooms in dimension \(>3\)} \label{sec:C-fold hd}

The goal of this section is to generalize the construction of mushrooms in dimension $3$ in \cref{sec:C-fold 3d} to higher dimensions. Throughout this section, we will write $Z^3 \subset \R^3$ for the mushroom constructed in \cref{sec:C-fold 3d} and write $Z$ for the higher-dimensional mushroom to be constructed.

\subsection{Introduction} \label{subsec:intro to mushrooms in dimension greater than 3}
 We first introduce some notation which will be used throughout this paper.
 
\begin{definition}[Contact handlebodies and generalized contact handlebodies] $\mbox{}$
\be
\item	A {\em contact handlebody over a Weinstein domain $(X,\mu)$} is a contact manifold contactomorphic to 
	$$([0,C]_t\times X,\, \ker(dt+\mu)),$$ 
	where $C>0$ is the {\em thickness} of the handlebody.
\item A {\em generalized contact handlebody over a Weinstein domain $(X,\mu)$} is a contact manifold contactomorphic to
  \[
  \{(t,x)~|~ f_0(x) < t< f_1(x)\}\subset (\R_t\times X, \,\ker(dt+\mu)),
  \]
 where there exists a $1$-parameter family $f_t:X\to \R$, $t\in[0,1]$, of smooth functions such that $f_t(x)<f_{t'}(x)$ for all $t<t'$, $x\in X$ and the graphs $\{t=f_{t_0}(x)\}$ are Weinstein for all $t_0\in[0,1]$.
\ee
\end{definition}

A contact handlebody is a compact contact manifold with a contact form such that all its Reeb orbits are chords of the same length and a generalized contact handlebody is one such that all the Reeb orbits are chords but they need not have the same length.

\s
Let $(W,\lambda)$ be a complete Weinstein manifold of dimension $2n-2>0$ and $\R_t \times W$ be the contactization of $W$ with contact form $\beta = dt+\lambda$. Let $W^c \subset W$ be a compact subdomain such that $W = W^c \cup ([0,\infty)_{\tau} \times \Gamma)$, $\Gamma \coloneqq \p W^c$ is the contact boundary, and $[0,\infty)_{\tau} \times \Gamma$ is the positive half-symplectization of $\Gamma$. Let $\eta := \lambda|_{\Gamma}$ be the contact form on $\Gamma$; then $\lambda|_{[0,\infty) \times \Gamma}=e^{\tau} \eta$.   For $\tau'>0$ we also define
$$W^c_{\tau'}:= W^c \cup ([0,\tau']\times \Gamma).$$

The ambient contact manifold of a mushroom is
$$(M = \R^3_{z,s,t} \times W,\, \xi = \ker\alpha), \quad \alpha = dz + e^s\, \beta.$$
The hypersurface on which we construct the mushroom is $\Sigma = \{z=0\} \subset M$ with characteristic foliation $\Sigma_{\xi} = \p_s$. 
\begin{remark}
{For ease of} notation, we will not distinguish between the characteristic foliation, which is an oriented singular line field, and a trivializing vector field.
\end{remark}

A {\em mushroom $Z$ with contact handlebody profile $H=([0,t_0] \times W^c_{\tau_0},\, dt+\lambda)$, $\tau_0>0$,} is constructed by first taking the product hypersurface $Z^3 \times W^c$, where $Z^3_{PL}$ has base $[0,s_0]\times[0,t_0]$, and then damping out the $Z^3$-factor on $W^c_{\tau_0}-W^c$.  Roughly speaking, the goal is to fold $\Sigma$ using $Z$, so that the resulting characteristic foliation cannot pass through a region which approximates $H$. 

\begin{remark}
	One can think of the constructions in \cref{sec:C-fold 3d} as a special case where $W$ is a point and $H=[0,t_0]$ is equipped with the contact form $dt$.
\end{remark}

\subsection{Product hypersurface} \label{subsubsec:product hypersurface}

Recall that in \cref{sec:C-fold 3d} we constructed the mushroom
$$Z^3\subset (\R^3_{z,s,t},\, \ker(dz+e^s\, dt))$$
which agrees with $\R^2_{s,t}$ outside of a rectangle $\square = [0,s_0] \times [0,t_0]$.  Let $Z^3_{\xi}$ be the characteristic foliation on $Z^3$.

We will compute the characteristic foliation $Z'_{\xi}$ on the product hypersurface $Z' \coloneqq Z^3 \times W^c \subset M$. Choose vector fields $v$ on $Z^3$,  defined  away from the singularities of $Z^3_\xi$, such that  $\alpha|_{Z^3}(v)=1$  and $w$ on $W^c$, defined away from the zero set of $\lambda$, such that $\lambda(w)=1$.

\begin{lemma} \label{lemma: char foliation calculation}
Away from the zeros of $\alpha|_{Z^3}$ and $\lambda$, the characteristic foliation $Z'_\xi$ is given by
\begin{equation} \label{eqn:char foliation on product}
Z'_{\xi} = \R \langle Z^3_{\xi} \,+\, dz \wedge ds (Z^3_{\xi}, v)\, X_{\lambda} \rangle,
\end{equation}
where $X_\lambda$ is the Liouville vector field of $\lambda$.
\end{lemma}

\begin{proof}
One can easily check that
\begin{equation*}
T(Z^3 \times W^c) \cap \xi = \R \langle Z^3_{\xi}, w-e^s\, v, \ker\lambda \rangle.
\end{equation*}

Basically the calculation of $Z'_\xi$ is reduced to computing the kernel $K=aX+bY+cZ$ of the $3$-dimensional vector space $\R\langle X,Y,Z\rangle$ with a maximally nondegenerate alternating $2$-form $\langle\cdot,\cdot\rangle$.  One can easily verify that
$$K=\langle Y,Z\rangle X  +\langle Z,X\rangle Y + \langle X,Y\rangle Z$$
works. We have $e^{-s}\, d\alpha = ds \wedge dt + ds \wedge \lambda + d\lambda$, and  if we write $\langle \cdot,\cdot \rangle := e^{-s} \,d\alpha(\cdot,\cdot)$ and $Z'_\xi=aX+bY+cZ +dA$, where $X=Z^3_\xi$, $Y=X_\lambda$, $Z=w-e^s\,v$, and $A\in \ker \lambda$ and is not parallel to $X_\lambda$, then $d=0$ since otherwise there exists $B\in \ker \lambda$ such that $\langle A,B\rangle\not=0$ and $\langle w,B\rangle =0$. The remaining/relevant part  of the pairing is given as follows:
\begin{align*}
\langle Z^3_{\xi}, X_{\lambda} \rangle &= 0, \\
\langle Z^3_{\xi}, w-e^s \,v \rangle &= ds(Z^3_{\xi}) - e^s\, ds \wedge dt (Z^3_{\xi}, v), \\
\langle X_{\lambda}, w-e^s \, v \rangle &= 1.
\end{align*} 
Hence $Z'_\xi=K = Z^3_\xi -( ds(Z^3_{\xi}) - e^s\, ds \wedge dt (Z^3_{\xi}, v)) X_\lambda.$

Finally, since $\alpha-(dz+e^s \,dt)=0$ when evaluated on vectors on $Z^3$ and hence
$$(ds\wedge \alpha + dz\wedge ds -e^s\, ds\wedge dt)(Z^3_\xi, v)=0,$$
it follows that $Z'_\xi= Z^3_\xi + dz \wedge ds (Z^3_{\xi}, v) X_\lambda$.
\end{proof}

At the zeros of $\alpha|_{Z^3}$ and $\lambda$, \eqref{eqn:char foliation on product} can be interpreted as saying that $Z'_\xi$ contains the limit of the right-hand side as the points on $Z^3\times W^c$ approach the zero.

\begin{remark}
\cref{lemma: char foliation calculation} is rather general and  holds for  $Z^3$ replaced by any surface in $\R^3_{z,s,t}$.
\end{remark}

\subsection{Dynamics of the fold} \label{subsection: dynamics of z prime xi}

We now investigate the dynamics of $Z'_{\xi}$. Let us first consider the PL case $Z'_{PL}= Z^3_{PL}\times W^c$.

\begin{lemma} \label{lemma: Z prime PL}
The flow lines of $(Z'_{PL})_\xi$ passing through $\{-1\}_s\times(0,t_0)_t\times W^c$ eventually limit to a negative singularity of $(Z'_{PL})_\xi$ and in particular do not leave $Z'_{PL}$.
\end{lemma}

\begin{proof}
The lemma follows from two observations:
(i) Since $dz \wedge ds  ((Z^3_{PL})_{\xi}, v)$ is positive on $P_4$, negative on $P_2$, and vanishes on $P_0 \cup P_1 \cup P_3$, the term $dz \wedge ds  ((Z^3_{PL})_{\xi}, v) X_{\lambda}$ in \eqref{eqn:char foliation on product} is  a positive multiple of $X_\lambda$ on $P_4$, a negative multiple of $X_\lambda$ on $P_2$, and zero on $P_0\cup P_1\cup P_3$. [Sample sign calculation on $P_4$ (it is useful to refer to \cref{fig:linear_foliation}): $(Z^3_{PL})_\xi=-\bdry_s$ and $v$, which we take to be parallel to the $zt$-plane, has positive $\bdry_z$-component.  Hence $dz \wedge ds  ((Z^3_{PL})_{\xi}, v)>0$ on $P_4$.]  (ii) By \cref{lemma: char foliation on 3d PL-fold}, if a flow line of $(Z'_{PL})_\xi$ passes through $\{-1\}_s\times(0,t_0)_t\times W^c$, then its projection to $Z^3_{PL}$ only passes through  $P_0$, $P_1$, $P_2$, and $P_3$.
\end{proof}

Next we describe the smoothed version $Z'_\xi$. We identify the singular points of $Z'_{\xi}$: Recall from \cref{lemma: char foliation on 3d PL-fold} that $Z^3_{\xi}$ has four singular points $e_{\pm}, h_{\pm}$. By the sign calculations of $dz \wedge ds  ((Z^3_{PL})_{\xi}, v)$ from the proof of \cref{lemma: Z prime PL} and continuity, $dz \wedge ds (Z^3_{\xi}, v)>0$ on neighborhoods of $e_+, h_+$ and $<0$ on neighborhoods of $e_-,h_-$.  Hence for each singular point $x \in W^c$ of the Liouville vector field $X_{\lambda}$, there exist four singular points $e^x_{\pm}, h^x_{\pm}$ of $Z'_{\xi}$ whose Morse indices are given by:
\begin{align*}
\text{ind}(e^x_+) &= \text{ind}_W (x), &\text{ind}(h^x_+) &= \text{ind}_W (x)+1,\\
\text{ind}(e^x_-) &= 2n-\text{ind}_W (x), &\text{ind}(h^x_-) &= 2n-1-\text{ind}_W (x),
\end{align*}
where $\text{ind}_W (x)$ is the Morse index of $x \in W^c \subset W$ and  we recall that $\dim W=2n-2$. 

See the top left figure in \cref{fig:C-fold thermal} for $Z^3_\xi$ and the regions indicating the signs of $dz\wedge ds(Z^3_\xi,v)$. The red (resp.\ blue, white) region indicates where $dz\wedge ds(Z^3_\xi,v)$ or $dz\wedge ds(S_{\tau,\xi},v)$ is positive (resp.\ negative, zero). 

\begin{remark} \label{remark: birth-death}
In view of \cref{rmk: variant of prop:char foliation on 3d C-fold}, we may replace the nondegenerate singular points by birth-death singularities as in the top right of \cref{fig:C-fold thermal}. The advantage of the birth-death singularities is that \cref{lemma: function sigma 0} still holds but the singular points can be immediately eliminated; this will be useful for example when damping out in \cref{subsec: damping}.
\end{remark}

\begin{figure}[ht]
\s
	\begin{overpic}[scale=.3]{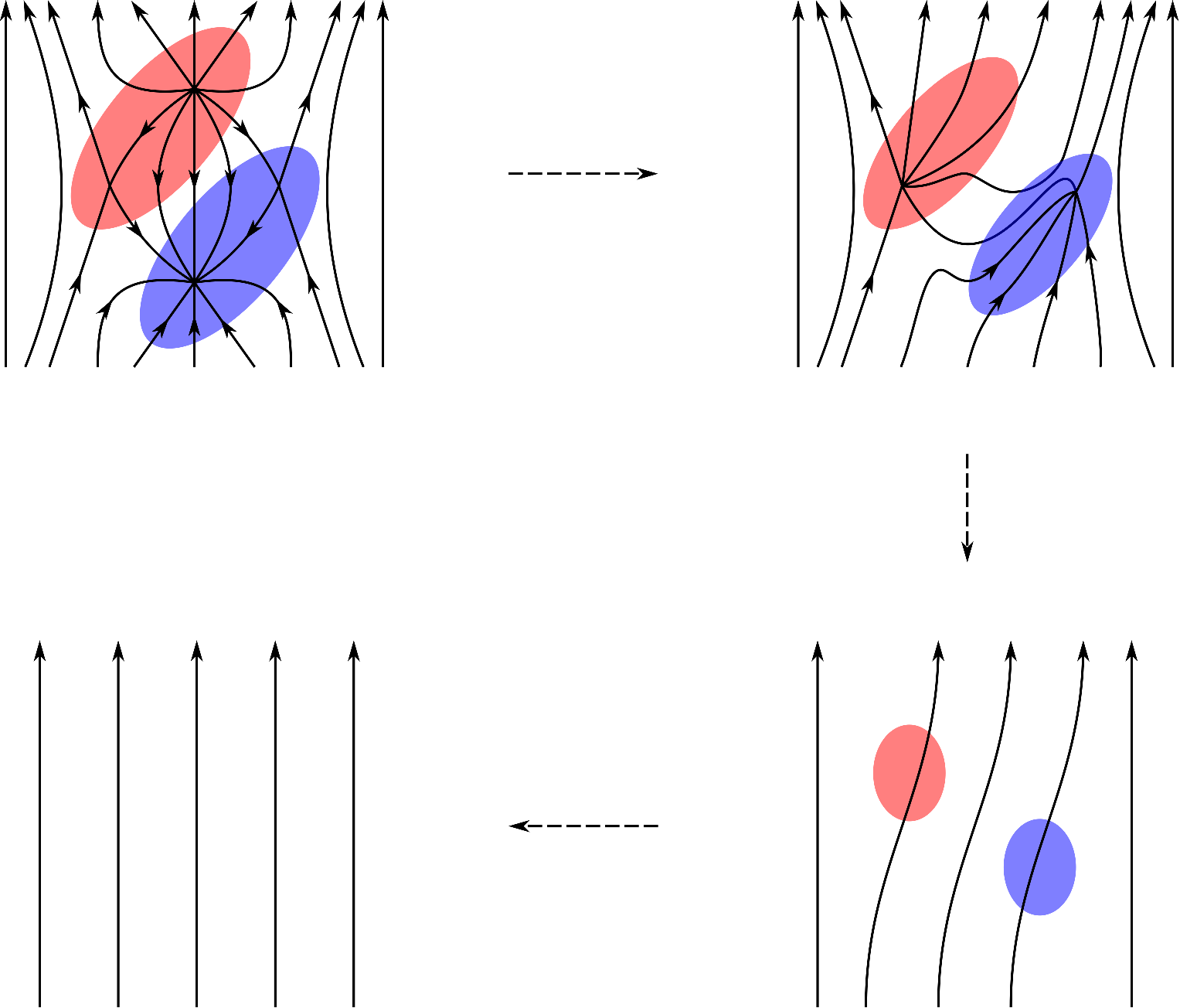}
	\end{overpic}
	\caption{The top left is $Z^3_\xi$ and the top right is an alternate perturbation of $Z^3_{PL,\xi}$ corresponding to $\tau=0$. The top right, bottom right, and bottom left are $S_{\tau,\xi}$ for some as $\tau$ goes from $0$ to $\tau_0$.}
	\label{fig:C-fold thermal}
\end{figure}

Let $\Sk(W)$ be the isotropic skeleton of $W^c$ with respect to $X_\lambda$.    Let $\kappa_1>\kappa_2>\kappa_3>0>\kappa_4>\kappa_5>\kappa_6$ with all $\kappa_i$ small as in  \cref{prop:char foliation on 3d C-fold}, and let $a\geq 0$ be small.  We define
\begin{gather*}
I_{a}^-:= \{-1\}\times [{\kappa_4},t_0+\kappa_2 -a]\subset \R^2_{s,t},\\
I_{a}^+:=\{s_0+1\}\times [{\kappa_5}-a,t_0+\kappa_3]\subset \R^2_{s,t},
\end{gather*}
so that $I_0^-$ (resp.\ $I_0^+$)  is the maximal interval with the  property that any flow line of $Z^3_{\xi}$ passing through the interval converges to a singularity of $Z^3_{\xi}$ in forward  (resp.\ backward) time. 

We now give a description of all the flow lines passing through $Z^3\times W^c$: 

\begin{lemma}[ Description of all flow lines passing through $Z^3\times W^c$] \label{lemma: function sigma 0}
There exist functions $\sigma_0^-,\sigma_0^+:  W^c \to \R_{\geq 0}$,  which vanish exactly on $\Sk(W)$ such that:
\be
\item each flow line of $Z'_{\xi}$ passing through $I_0^- \times \Sk(W)$  (resp.\ $I_0^+ \times \Sk(W)$)  converges to a singularity of $Z'_{\xi}$ in forward  (resp.\ backward)  time;
\item for $x \in W^c \setminus \Sk(W)$, each flow line passing through $I_{\sigma_0^- (x)}^- \times \{x\}$  (resp.\ $I_{\sigma_0^+ (x)}^+ \times \{x\}$)   converges to a singularity of $Z'_{\xi}$ in forward (resp.\ backward) time;
\item for $x \in W^c \setminus \Sk(W)$, each flow line passing through $(I_0^- \setminus I_{\sigma_0^- (x)}^-)\times \{x\}$  (resp.\ $(I_0^+\setminus I_{\sigma_0^+ (x)}^+) \times \{x\}$)  exits $Z^3 \times W^c$ along $Z^3 \times \p W^c$ in finite  forward (resp.\ backward)  time;
\item  for $x\in W^c$, each flow line passing through $\{-1\}\times(t_0+\kappa_2,t_0+\kappa_1]\times \{x\}$ (resp.\ $\{s_0+1\}\times [\kappa_6,\kappa_5)\times \{x\}$) exits from either $\{s_0+1\}\times (t_0+\kappa_3,t_0+\kappa_1]\times W^c$ (resp.\ $\{-1\}\times [\kappa_6,\kappa_4)\times W^c$) or $Z^3\times \bdry W^c$ in finite forward (resp.\ backward) time; in the former case, the $W^c$-coordinate $x'$ of the exit point of the flow line is on the time $\geq 0$ flow line of $X_\lambda$ starting at $x$;
\item  for $x\in W^c$, each flow line passing through $\{-1\}\times [\kappa_6,\kappa_4)\times \{x\}$ (resp.\ $\{s_0+1\}\times (t_0+\kappa_3,t_0+\kappa_1]\times \{x\}$) exits from $\{s_0+1\}\times [\kappa_6,\kappa_5)\times W^c$ (resp.\ $\{-1\}\times (t_0+\kappa_2,t_0+\kappa_1]\times W^c$) in finite forward (resp.\ backward) time; the $W^c$-coordinate $x'$ of the exit point of the flow line is on the time $\leq 0$ flow line of $X_\lambda$ starting at $x$;
\item each flow line outside of $\R_s\times[\kappa_6,t_0+\kappa_1]\times W^c$ has trivial holonomy;
\item all other flow lines are (i) flow lines between singularities, (ii) flow lines from a singularity to $Z^3\times \bdry W^c$, or (iii) flow lines from $Z^3\times \bdry W^c$ to a singularity.

\ee
Moreover, as $Z^3\to Z^3_{PL}$,  all $\kappa_i\to 0$ and $|\sigma_0^\pm|_{C^0}\to 0$. 
\end{lemma}

\begin{proof}
 This is an immediate consequence of \cref{lemma: char foliation calculation} and \cref{lemma: Z prime PL}, taking the limit $Z^3\to Z^3_{PL}$, and a case-by-case analysis of the various regions of the top left figure of \cref{fig:C-fold thermal}. 

Suppose the flow line passes through the red region times $W^c$.  Then either the flow line exits from $Z^3\times \bdry W^c$ or escapes to the white region times $W^c$. Once in the white region, the flow line either reaches $s=s_0+1$ or enters the blue region times $W^c$ and reaches a negative singularity. 

Suppose the flow line passes through the white region (e.g., passes through $s=-1$). Then the flow line reaches $s=s_0+1$, enters the blue region times $W^c$ (and hence reaches a negative singularity), or enters the red region times $W^c$ (already considered).

All $\kappa_i\to 0$ and $|\sigma_0^\pm|_{C^0}\to 0$ as $Z^3\to Z^3_{PL}$ by construction.
\end{proof}

Technically, the functions $\sigma_0^\pm$ account for the speed of convergence of flow lines of $(Z^3)_{\xi}$ towards its singularities and those of $X_{\lambda}$ in $W^c$ towards $\Sk(W)$.

\subsection{Damping} \label{subsec: damping} 

In order for the mushroom to be the image of a continuous map $\Sigma\to M$, one must damp out the  $Z^3$-fiber over $W^c_{\tau_0} \setminus \op{int}(W^c)=[0,\tau_0]\times \Gamma$ as $\tau$ grows. Recall the notation from \cref{subsec:intro to mushrooms in dimension greater than 3}. 

The damping procedure amounts to choosing an isotopy of surfaces $S_\tau$, $\tau\in[0,\tau_0]$, in $\R^3_{z,s,t}$ from $S_0=Z^3$ to the flat $S_{\tau_0}=\R^2_{s,t}$. We take $S_0=Z^3$ to have birth-death type singular points as in \cref{fig:C-fold thermal}; see \cref{remark: birth-death}.  In practice we also take $\tau_0>0$ to be arbitrarily small. We then set
$$\Ical_0 := \cup_{0 \leq \tau \leq \tau_0 } (S_{\tau} \times \{\tau\}) \subset \R^4_{z,s,t,\tau}$$
and the actual hypersurface in $M$ will be $\Ical_0 \times \Gamma$.

The PL model of $S_{\tau}$ is obtained by replacing $P_0$ by the rectangle $$[0,s_0]\times[-e^{-s_0/2}z, -e^{-s_0/2}z+t_0]$$ as in \cref{subsec:PL fold}, where the parameter $z$ ranges from $z_0$ to $0$ as $\tau$ goes from $0$ to $\tau_0$ (in other words, we are pushing the top face of the parallelepiped into the parallelepiped); its smoothing for $\tau>0$ will use a profile function $\phi$ such that $\phi'<0$ everywhere so that there are no singularities of the characteristic foliation. We have
\begin{equation*}
T\Ical_0 = \R\langle TS_{\tau}, \p_\tau+fw_0 \rangle,
\end{equation*}
where $w_0=\bdry_z-K_0\bdry_t$ is parallel to $P_2$ and $P_4$ and $f \leq 0$ is a $\tau$-dependent smooth function on $\R^3_{z,s,t}$ which vanishes when $\tau$ is close to $\{0,\tau_0\}$ or $z=0$.

We are now ready to compute the characteristic foliation $(\Ical_0 \times \Gamma)_{\xi}$. Let $S_{\tau,\xi}$ be the characteristic foliation on $S_\tau$, i.e., $\alpha|_{\R^3} (S_{\tau,\xi}) = 0$, and let $v$ be a vector field on $S_\tau$, defined away from the singularities of $S_{\tau,\xi}$, such that $\alpha|_{\R^3} (v)=1$.

\begin{lemma} \label{lemma: char fol calc 2}
The characteristic foliation $(\Ical_0 \times \Gamma)_\xi$ is given by
\begin{align} \label{eqn:char foliation damping 1}
(\Ical_0 \times \Gamma)_{\xi}
& = S_{\tau,\xi} + dz \wedge ds (S_{\tau,\xi}, v) (\p_\tau + f\,w_0)  + f \left( ds(S_{\tau,\xi})\, v - ds(v)\, S_{\tau,\xi} \right) \\
&\quad  +e^{-\tau}f(- ds \wedge dt (S_{\tau,\xi}, v)+ K_0\, dz\wedge ds(S_{\tau,\xi},v))\, R_{\eta}, \nonumber
\end{align}
 on the subset of $\Ical_0\times \Gamma$ where $v$ is defined. 
\end{lemma}

\begin{proof}
This is similar to the calculation of \cref{lemma: char foliation calculation}. We compute
\begin{equation*}
T(\Ical_0 \times \Gamma) \cap \xi = \R \langle S_{\tau,\xi}, e^{\tau+s} \,v - R_{\eta},\, \p_\tau + fw_0 +(-e^{-\tau-s} + K_0 e^{-\tau}) f\, R_\eta,\, \ker\eta \rangle,
\end{equation*}
where $R_\eta$ is the Reeb vector field of $\eta$.  Next we have
\begin{align*}
\alpha&=dz + e^s\,(dt+ e^\tau \eta),\\
e^{-s}\,d\alpha&=ds\wedge dt + e^\tau \,ds \wedge\eta + e^\tau \,(d\tau \wedge\eta + d\eta).
\end{align*}
Setting $X= S_{\tau,\xi}$, $Y=e^{\tau+s} v-R_\eta$, $Z=\p_\tau + f\,w_0 +(-e^{-\tau-s} + K_0 e^{-\tau}) f\, R_\eta$,
\begin{align*}
\langle X,Y\rangle &= e^{\tau+s}\, ds\wedge dt (S_{\tau,\xi},v) - e^\tau \, ds(S_{\tau,\xi})= e^\tau ( dz\wedge ds (S_{\tau,\xi},v)),\\
\langle X,Z\rangle &= e^\tau ds (S_{\tau,\xi}) (-e^{-\tau-s} f)= -e^{-s} f \,ds(S_{\tau,\xi}),\\
\langle Y,Z\rangle &= e^{\tau+s} ( e^\tau ds(v)(-e^{-\tau-s} f)) +e^\tau=e^\tau(1-f ds(v)).\\
(\Ical_1 \times \Gamma)_\xi& =  (1-f ds(v))\,S_{\tau,\xi} +  e^{-\tau-s} f \,ds(S_{\tau,\xi}) (e^{\tau+s} \,v-R_\eta)\\
& \quad + dz\wedge ds (S_{\tau,\xi},v)(\p_\tau + f\,w_0 +(-e^{-\tau-s} + K_0 e^{-\tau}) f\,R_\eta).
\end{align*}
A rearrangement of the terms gives the lemma.
\end{proof}

Note that \eqref{eqn:char foliation damping 1} agrees with \eqref{eqn:char foliation on product} at $\tau = 0$. The first two terms of $(\Ical_0 \times \Gamma)_{\xi}$ are analogous to those of $Z'_\xi$; see \cref{lemma: char foliation calculation}.  The third term $f(ds(S_{\tau,\xi}) v - ds(v) \,S_{\tau,\xi})$ lies in $\ker ds$ and, away from the corners,
\begin{itemize}
\item vanishes on $P_1 \cup P_3$,
\item has negative $\bdry_t$-component on $P_0 \cup P_4$, and
\item has positive $\bdry_t$-component on $P_2$.
\end{itemize}
See \cref{fig:flatten foliation}. In other words, the third term, when we project out the $s$- and $W$-directions, is a flow in the clockwise direction around $\bdry P_1$ as seen in the picture.  The last term of \eqref{eqn:char foliation damping 1} has a substantial contribution in the $R_\eta$-direction when the damping happens quickly, i.e., when $\tau_0$ is small and $f$ is large.  This is something we need to be careful about, but ultimately can be finessed away by stacking the mushrooms in a particular way in \cref{sec:plug}.

\begin{figure}[ht]
	\begin{overpic}[scale=.9]{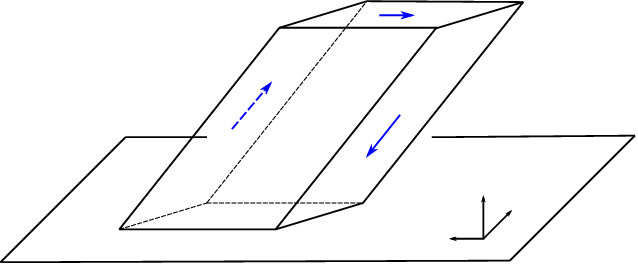}
		\put(68.5,3){\small{$t$}}
		\put(81,7){\small{$s$}}
		\put(73.2,10){\small{$z$}}
	\end{overpic}
	\caption{The vector field $f(ds(S_{\tau,\xi}) v - ds(v) S_{\tau,\xi})$ is depicted in blue.}
	\label{fig:flatten foliation}
\end{figure}

\subsection{Description of the characteristic foliation of the mushroom} \label{subsection: desc of char fol of mushroom}

In this subsection we summarize the dynamics of the characteristic foliation of the mushroom $Z_H$ of $\Sigma=\{z=0\}\subset M=\R^3_{z,s,t}\times W$ with profile $H=[0,t_0]\times W^c_{\tau_0}$. 

\begin{definition} \label{defn:Z fold along H}
	The \emph{mushroom of $\Sigma$ with profile $H$} is the hypersurface
	\begin{equation} \label{eqn:Z_H}
	Z_H := \left( \Sigma \setminus (\square \times W^c_{\tau_0}) \right)  \cup Z'_{PL} \cup (\Ical_0 \times \Gamma)_{PL},
	\end{equation}
	modulo  smoothing. (The smoothed versions do not have the subscripts PL.)  The region $\square \times W^c_{\tau_0} \subset \Sigma$,  where $\square=[0,s_0]\times[0,t_0]$, is the \emph{base} of $Z_H$, and the region $\Ical_0 \times \Gamma$ is the {\em damping region}.
\end{definition}

Let $\tau'_0 \in (0, \tau_0)$,  let $\kappa_1>\kappa_2>\kappa_3>0>\kappa_4>\kappa_5>\kappa_6$ with all $\kappa_i$ small as in \cref{prop:char foliation on 3d C-fold},  and let ${\sigma_1^\pm, \sigma_2^\pm:} W^c_{\tau'_0} \to \R_{\geq 0}$ be functions such that:
\begin{itemize}
	\item {$\sigma_1^\pm$} vanishes exactly on $\Sk(W)$ and {$\sigma_2^\pm$} vanishes on $W^c$;
    \item on $\{0 \leq \tau \leq \tau'_0 \}$, both {$\sigma_1^\pm=\sigma_1^\pm(\tau)$ and $\sigma_2^\pm = \sigma_2^\pm (\tau)$} are strictly increasing and reach their maximum at $\tau = \tau'_0$;
	\item  $\sigma_1^- (\tau'_0) + \sigma_2^- (\tau'_0) = t_0+\kappa_2-\kappa_4$ and $\sigma_1^+ (\tau'_0) + \sigma_2^+ (\tau'_0) = t_0+\kappa_3-\kappa_5$. 
\end{itemize}

As the smooth version of $Z_H$ limits to the PL version, all $\kappa_i\to 0$ and $|\sigma_i^\pm|_{C^0}\to 0$ on $W^c_{\tau_0}$.

We then define the compact submanifolds
\begin{gather*}
H_{\inward} := \{(t,x) ~|~ x \in W^c_{\tau'_0}, {\kappa_4} +\sigma_2^-(x) \leq t \leq t_0+{\kappa_2}-\sigma_1^-(x) \},\\
H_{\outward} := \{(t,x) ~|~ x \in W^c_{\tau'_0}, {\kappa_5} +\sigma_1^+(x) \leq t \leq t_0+{\kappa_3}-\sigma_2^+(x) \},
\end{gather*}
which approximate $H$ when all the smoothing parameters involved in the construction tend to $0$.  See \cref{fig:Z fold along H}.  We use the notation $X^{\circ}$ (and also $\op{int}(X)$) as in $H_{\inward}^\circ$ to denote the interior of a space $X$.

 Note that $\bdry H$ (which we assume has rounded corners) is convex; this follows from observing that $([-1,1]_t\times W,dt+\lambda)$ has contact vector field $t\bdry_t +X_\lambda$.

\begin{proposition} \label{prop:dynamics C-fold}
Assuming all the corner rounding parameters are sufficiently small, there exists a tubular neighborhood $[-\epsilon,\epsilon]_\ell\times \bdry H$ of $\bdry H=\{0\}\times \bdry H$ and $H_{\inward}$ and $H_{\outward}$ that approximate $H$ such that: 
\begin{itemize}
\item[(Z1)] $\bdry H_{\inward},\bdry H_{\outward}\subset [-\epsilon,\epsilon]\times \bdry H$ are graphical over $\bdry H$.
\item[(Z2)]  $Z_{H,\xi}$ is ``$1$-Morse" in the following sense: it satisfies (M1) and (M3) of \cref{prop:morse criterion} and
\be
\item[(M2')] every flow line passing through $x\in Z_H$ with $Z_{H,\xi}(x)\not=0$ converges to a singularity or goes to $\{s=+\infty\}$ in forward time and converges to a singularity or goes to $\{s=-\infty\}$ in backward time.
\ee 
\item[(Z3)] Any flow line of $Z_{H,\xi}$ that passes through $H^{\circ}_{\inward} \subset \{s=-1\}$ converges to a negative singularity of $Z_{H,\xi}$ in forward time. Similarly, any flow line of $Z_{H,\xi}$ that passes through $H^{\circ}_{\outward} \subset \{s=s_0+1\}$ converges to a positive singularity of $Z_{H,\xi}$ in backward time.
\item[(Z4)] Any flow line of $Z_{H,\xi}$ that does not pass through $H\cup ([-\epsilon,\epsilon]\times \bdry H)\subset \{s=-1\}$ has trivial holonomy.
\item[(Z5)] There exists a Morse function $F$ on $\bdry H$ such that $\bdry H_\xi$ is gradient-like for $F$ (and hence flows ``from $R_+(\bdry H)$ to $R_-(\bdry H)$'') and such that any flow line of $Z_{H,\xi}$ that passes through $(\ell,x)\in [-\epsilon,\epsilon]\times \bdry H\subset \{s=-1\}$ and does not converge to a singularity of $Z_{H,\xi}$: 
\be
\item passes through $(\ell', y)\in [-\epsilon,\epsilon]\times \bdry H\subset \{s=s_0+1\}$ with $F(y)\geq F(x)$; and
\item is parallel to $X_\lambda$ (resp.\ $-X_\lambda$) on $[-\epsilon,\epsilon]\times W^c_+$ (resp.\ $[-\epsilon,\epsilon]\times W^c_-$) when projected to $[-\epsilon,\epsilon]\times\bdry H$. 
\ee
Here $W^c_+$ is the portion of $\bdry H$ corresponding to $\{t_0\}\times W^c$ and $W^c_-$ is the portion corresponding to $\{0\}\times W^c$.
\end{itemize}
\end{proposition}

\begin{proof}
 (Z1) is by construction.  (Z4) is clear.  (Z2)--(Z5) follow from \cref{lemma: function sigma 0} and \cref{lemma: char fol calc 2}.  In (Z5) we take $\{\tau\}\times \Gamma$ to be level sets of $F$ so that the component of $Z_{H,\xi}$ in the direction of the Reeb vector field $R_\eta$ vanishes on $dF$.  
\end{proof}

\begin{figure}[ht]
	\begin{overpic}[scale=.8]{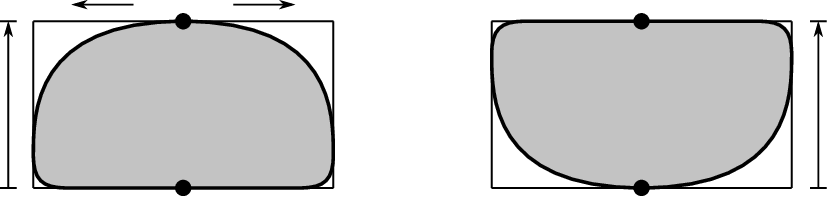}
		\put(21,10){$H_{\inward}$}
		\put(77.5,10){$H_{\outward}$}
		\put(-3,10){$t$}
		\put(-4.2,0){$ \kappa_4$}
		\put(-10,21){\small{$t_0+ \kappa_2$}}
		\put(13,25){$\tau$}
		\put(32,25){$\tau$}
		\put(19,24.5){\tiny{$\Sk(W)$}}
		\put(101,0){$ \kappa_5$}
		\put(101,21){\small{$t_0+ \kappa_3$}}
	\end{overpic}
	\caption{The shaded regions are $H_{\inward}$ and $H_{\outward}$, respectively. The area of the complements of $H_{\inward}$ and $H_{\outward}$ in the rectangles tend to 0 as all the parameters involved in the construction tend to 0.}
	\label{fig:Z fold along H}
\end{figure}

\section{Quantitative stabilization of open book decompositions} \label{sec:quantitative stab}

\subsection{Some definitions} \label{subsec: qstab basic definitions}

Let $M$ be a closed manifold. An \emph{open book decomposition} (abbreviated OBD) of $M$ is a pair $(B,\pi)$,
where $B \subset M$ is a closed codimension $2$ submanifold and 
$$\pi: M \setminus B \to S^1\subset \C$$
is a fibration which agrees with the angular coordinate $\theta$ on a neighborhood $B\times D^2$ of $B=B\times\{0\}$. We call $S_{\theta}:=\pi^{-1}(e^{i\theta}), e^{i\theta} \in S^1$, the \emph{pages} of the OBD, and call $B$ the \emph{binding}.

Let $\xi$ be a contact structure on $M$.

\begin{definition} $\mbox{}$
\be
\item An OBD $(B,\pi)$ is \emph{$\xi$-compatible} if there exists a contact form $\alpha$ for $\xi$, called an {\em adapted} contact form, such that the Reeb vector field $R_\alpha$ of $\alpha$ is transverse to all the pages and is tangent to $B$, and $\alpha|_B$ is a contact form on $B$. We also say ``$\alpha$-compatible" or simply ``compatible" if the contact structure is understood.
\item An $\alpha$-compatible $(B,\pi)$ is {\em strongly Weinstein} if all its pages $(S_\theta,\alpha|_{S_\theta})$ are Weinstein.
\ee
\end{definition}

Let $(B,\pi)$ be an $\alpha$-compatible OBD. Let $\op{arg}: S^1\to \R/2\pi\Z$ be the map $e^{i\theta}\mapsto \theta$. Then define $\rho:M \setminus B \to \R_{>0}$ by
\begin{equation*}
\rho(x) = d(\op{arg}\circ \pi)(R_\alpha(x)).
\end{equation*}
Roughly speaking, $\rho(x)$ measures infinitesimally how fast the orbit of $R_\alpha$ through $x$ traverses the pages.

\begin{definition} \label{defn:infinitesimal variation}
The \emph{infinitesimal variation} on the page $S_{\theta}$ is
\begin{equation*}
V_{\theta} \coloneqq \sup_{x \in S_{\theta}} \rho(x) / \inf_{x \in S_{\theta}} \rho(x) \in [1,\infty),
\end{equation*}
and the \emph{total infinitesimal variation} {is} $V \coloneqq \sup_{\theta \in [0,2\pi]} V_{\theta}$.
\end{definition}

The following is standard:

\begin{lemma}\label{lemma:criterion for contact handlebody}
If $V\equiv 1$, then $M\setminus \overline{S}_0$ the interior of a contact handlebody.
\end{lemma}

\begin{proof}
Let $t$ be the coordinate obtained by flowing in the direction of $R_\alpha$ starting from $S_0$.  Then $M\setminus \overline{S}_0\simeq (0,C)_t\times S_0$ and $\alpha=f_t\, dt + \beta_t$, where $f_t$ (resp.\ $\beta_t$) is a function (resp.\ $1$-form) on $S_0$ that depends on $t$. 

We claim that $R_\alpha=\bdry_t$ implies that $f_t=1$ and $\dot\beta_t=0$, where the dot denotes the derivative in the $t$-direction: Since $\alpha(R_\alpha)=1$, we have $f_t=1$. Then $d\alpha$ becomes $dt\wedge \dot\beta_t +d_{S_0}\beta_t$, where $d_{S_0}$ is the exterior derivative in the $S_0$-direction.  Finally, $i_{R_\alpha} d\alpha=0$ forces $\dot\beta_t=0$.
\end{proof}

\subsection{Quantitative stabilization for spheres} \label{subsection: qstab for sphere}

Before starting, we warn the reader that {\em the type of stabilization in this section is different from the notion of stabilization of an OBD in the Giroux correspondence which involves changing the topology of the page by a handle attachment and composing the monodromy with a suitable Dehn twist.}

Let $\xi_{std}$ be the standard contact structure on $S^{2n-1}=\{|z_1|^2+\dots +|z_n|^2=1\}\subset \C^n$ given by the restriction of $\alpha=\tfrac{1}{2}\sum_{i=1}^n(x_idy_i-y_idx_i)$ which we denote by $\alpha_{std}$. The standard OBD for $\xi_{std}$ can be constructed as follows: Starting with $z_1: S^{2n-1}\to \C$ which is a submersion away from $|z_1|=1$, we set
$B=z_1^{-1}(0)=S^{2n-3}=\{|z_2|^2+\dots + |z_n|^2=1\}$ and $$\pi=\tfrac{z_1}{|z_1|}: S^{2n-1}\setminus B \to S^1\subset \C.$$  
The pages $S_\theta$ are Weinstein $(2n-2)$-disks and the Reeb vector field is 
$$R_{\alpha_{std}}=\textstyle\sum_{i=1}^n (x_i\,\bdry_{y_i}-y_i\,\bdry_{x_i})=\textstyle\sum_{i=1}^n \bdry_{\theta_i},$$ where $\theta_i$ is the $i$th angular coordinate. Then $\rho(z_1,\dots,z_n)= d\theta_1 (\bdry_{\theta_1})=1$, $V\equiv 1$, and $S^{2n-1}\setminus \overline{S}_0$ is the interior of a genuine contact handlebody of thickness $2\pi$  by \cref{lemma:criterion for contact handlebody}.  Note that $B$ has an analogous OBD derived from $z_2: S^{2n-3}=\{|z_2|^2+\dots + |z_n|^2=1\}\to \C$. 

\begin{lemma}\label{lemma: qstab for S 2n-1}
For any positive integer $k>0$ and $\epsilon'>0$ small, there exists an $\alpha_{std}$-compatible, strongly Weinstein OBD $(B_k,\pi_k)$ of $S^{2n-1}$ such that each page is $C^\infty$-close to 
$$S_{\theta_0}\cup S_{\theta_0+2\pi/k}\cup \dots \cup S_{\theta_0+2\pi(k-1)/k}$$ 
(i.e., the union of $k$ evenly spaced pages for some $\theta_0$) outside an $\epsilon'$-small neighborhood of $B$ (with respect to the standard Euclidean metric on $\C^n$) and such that $S^{2n-1}$, cut open along a new page, is the interior of a contact handlebody of thickness $2\pi/k$.
\end{lemma}

\begin{proof}
We would like to ``stabilize" $(B,\pi)$ by replacing $z_1$ by $z_1^k$.  However, since $0$ is not a regular value of $z_1^k$, we use 
$$f_k:=z_1^k+\epsilon z_2^k+ \epsilon^2 z_3^k+\dots + \epsilon^{n-1}z_n^k,$$ 
where $\epsilon>0$ is small. We are thinking of $f_k$ as inductively defined as $z_1^k$ plus $\epsilon$ times $f_k$ corresponding to the binding $|z_2|^2+\dots+|z_n|^2=1$. We write $R=R_{\alpha_{std}}$.

\s\n
\textsc{Step 1.}{\em Verification that $0$ is a regular value of $f_k$.} Let $F_k$ be $f_k$ viewed as a map $\C^n\to \C$.  Then $dF_k(z_1,\dots,z_k)=k(z_1^{k-1},\epsilon z_2^{k-1},\dots,\epsilon^{n-1}z_n^{k-1})$. Next we precompose with the derivative of the inclusion map $i:S^{2n-1}\hookrightarrow \C^n$.  Let $z=(z_1,\dots,z_n)\in S^{2n-1}$.  Suppose there exists $z_j\not=0,1$.  Then there exists $v\in T_z S^{2n-1}$ with a nontrivial component in the $z_j$-direction and $df_k$ is surjective since $\epsilon^{j-1}z_j^{k-1}\not=0$.  Otherwise, some $z_j=1$ and $z_i=0$ for all $i\not=j$. Then $f_k(z)=\epsilon^{j-1}z_j^k\not=0$.   Hence $0$ is a regular value of $f_k$.  

We set $B_k=f_k^{-1}(0)$ $\pi_k=\tfrac{f_k}{|f_k|}: S^{2n-1}\setminus B_k\to S^1$, and $S_{k,\theta}= \pi_k^{-1}(\theta)$.

\s\n
\textsc{Step 2.}{\em Computation of $df_k(R)$.}  For $(z_1,\dots,z_k)\in S^{2n-1}\setminus B_k$, we use polar coordinates $(r_i,\theta_i)$ for $z_i$ and compute:
\begin{align} \nonumber
df_k(R)&= d(r_1^ke^{ik\theta_1}+\epsilon r_2^ke^{ik\theta_2}+\dots +\epsilon^{n-1}r_n^k e^{ik\theta_n})(\textstyle\sum_{j=1}^n \bdry_{\theta_j}) \\
\label{eqn: calc of dfkR}
&=ik (r_1^ke^{ik\theta_1}+\epsilon r_2^ke^{ik\theta_2}+\dots +\epsilon^{n-1}r_n^k e^{ik\theta_n})=ikf_k.
\end{align}
Observe that this equation is the version of the equation in \cite[p.411, last line]{Gi00} when $F_k$ is holomorphic.

\s\n
\textsc{Step 3.}{\em Verification of the properties.} The calculation \eqref{eqn: calc of dfkR} implies that $R$ is tangent to $B_k$ and transverse to $S_{k,\theta}$ and that $\pi_k$ is a fibration.  Moreover, for $(B_k,\pi_k)$, $\rho(z)=1$ and $V\equiv 1$, and $S^{2n-1}\setminus \overline{S}_{k,0}$ is the interior of a contact handlebody of thickness $2\pi/k$.  Since $R$ is tangent to $B_k$, for each $z\in B_k$, $df_k(\ker \alpha_{std}(z))=\C$. This, together with the invariance of $\ker\alpha_{std}$ under the standard almost complex structure on $\C^n$, implies that $B_k$ is a codimension two contact submanifold of $\ker \alpha_{std}$.  

Next we apply \cref{lemma: Weinstein}, proved in \cref{subsection: verification}, to show that $S_{k,\theta}$ is Weinstein after perturbing $f_k$ by adding $\sum_{i=1}^n c_i z_i^k$ for $c_i$ small.  Using the notation from \cref{subsection: verification}, we have $d\theta(R)=1$ and $d\phi(R)=0$; hence $d\phi(X_\beta)=0$ if and only if $d\phi=0$, where $\phi$ is viewed as a function on $S_{k,\theta=0}$.  One can compute that if $z\in S_{k,0}$ is a critical point of $\phi$, then all $|z_i|\not=0$; at such a point $d(z_1^k),\dots,d(z_n^k)$ are linearly independent.  
This provides enough perturbations to make $\phi$ Morse.

Finally, the $C^\infty$-closeness property is immediate from taking $\epsilon>0$ small.  
\end{proof}

\subsection{Verification of the Weinstein property} \label{subsection: verification}

Let $f: \C^n\to \C$ be a holomorphic function such that $f(0)=0$ with an isolated critical point at the origin.  For $\kappa>0$ sufficiently small, $f$ defines an OBD $(B,\theta)$ of the sphere $S_\kappa$ of radius $\kappa$, where $B= S_\kappa\cap f^{-1}(0)$ and $\theta=\op{arg} f: S_\kappa\setminus B\to \R/2\pi\Z$.  Also let $\phi:=- \log|f|$ on $S_\kappa\setminus B$. 

Let
\begin{itemize}
\item $\alpha=\tfrac{1}{2}\sum_{i=1}^n(x_i\,dy_i-y_i\,dx_i)=\tfrac{1}{2}\sum_{i=1}^n r_i^2\,d\theta_i$ be the standard Liouville form on $\C^n$ with Liouville vector field $X_\alpha= \tfrac{1}{2}\sum_{i=1}^n r_i \,\tfrac{\bdry}{\bdry r_i}$, where $(r_i,\theta_i)$ are polar coordinates corresponding to $(x_i,y_i)$;
\item $\alpha_{std}$ be the induced contact form on $S_\kappa$ for $\kappa>0$ small and $R=\tfrac{2}{r^2}\sum_i \tfrac{\bdry}{\bdry \theta_i}$ be the Reeb vector field on $S_\kappa$, where $r^2=\sum_i r_i^2$ (note that $R$ is defined on all of $\C^n$, not just on $S_\kappa$);
\item $\beta$ be the $1$-form induced by $\alpha_{std}$ on any page of the OBD with Liouville vector field $X_\beta$; we view $X_\beta$ as a vector field on $S_\kappa\setminus B$ that is tangent to the pages;
\item $X_\theta$ be the Hamiltonian vector field of $\theta$ (satisfying $i_{X_\theta}d\alpha= d\theta$), viewed as a function on $\C^n\setminus f^{-1}(0)$; and 
\item $J$ be the standard complex structure on $\C^n$.
\end{itemize}

The following is due to Emmanuel Giroux (presented here with his permission):

\begin{lemma}[Giroux]\label{lemma: Weinstein}
On $S_\kappa\setminus B$, if $d\phi(R)\equiv 0$, then $d\phi(X_\beta)= \tfrac{1}{d\theta(R)}|d\phi|^2$. 
\end{lemma}

\begin{proof}
We first claim that, at every point of $S_\kappa\setminus B$, the following identity holds:
\begin{equation} \label{eqn: Liouville restriction}
X_\alpha= X_\beta + a\, X_\theta + b\,R,
\end{equation}
where $a= \tfrac{1}{d\theta(R)}$ and $b= \tfrac{d\theta(X_\alpha)}{d\theta(R)}$. First note that the $d\alpha$-symplectic orthogonal complement of the tangent space $TS$ of a page $S$ is spanned by $R$ and $X_\theta$.  [Verification: $d\alpha(R,TS)= 0$ since $R, TS$ are tangent to $S_\kappa$; $d\alpha(X_\theta,TS)=d\theta(TS)=0$; and $d\alpha(X_\theta,R)=d\theta(R)>0$.] Hence we can write $X_\alpha= Y + aX_\theta + bR$ for some $Y\in TS$. Evaluating on $TS$ gives $Y=X_\beta$.  We can then determine $a$ and $b$ by applying $d\theta$ and $d( r^2)$ to \eqref{eqn: Liouville restriction}:
$d(r^2)(X_\alpha) = 2rdr (\tfrac{1}{2}\sum_i r_i \tfrac{\bdry}{\bdry r_i})= r^2$ and $d(r^2)(X_\alpha)=a d(r^2)(X_\theta)= -a d\theta(X_{r^2})=-a d\theta(-r^2R) $, so $a=  \tfrac{1}{d\theta(R)}$. Similarly, $d\theta(X_\alpha)= b d\theta(R)$, so $b= \tfrac{d\theta(X_\alpha)}{d\theta(R)}$.

It follows from \eqref{eqn: Liouville restriction} that 
$$ d\phi(X_\beta) = -\tfrac{d\phi(X_\theta)}{d\theta(R)} + d\phi(X_\alpha) - d\phi(R) \tfrac{d\theta(X_\alpha)}{d\theta(R)}.$$
Now, since the function $-\phi+i\theta$ is holomorphic when viewed as a function on $\C^n\setminus \{f=0\}$, we have $d\phi\circ J=d\theta$. Hence 
\begin{align*} 
d\phi(X_\alpha)&=- d\phi\circ J (JX_\alpha)=-d\theta (\tfrac{1}{2}\textstyle\sum_i \tfrac{\bdry}{\bdry \theta_i})=-\tfrac{r^2}{4}d\theta( R),\\
d\phi(X_\theta)&=- d\phi\circ J(JX_\theta) = -d\theta( JX_\theta)= -d\alpha(X_\theta,JX_\theta)\\
&= -|X_\theta|^2=-|d\theta|^2=-|d\phi|^2,\\
d\phi(R)&= -d\theta(J R)= d\theta(\tfrac{4}{r^2} X_\alpha),
\end{align*}
and 
$$ d\phi(X_\beta)= \tfrac{1}{d\theta(R)}( |d\phi|^2 - (d\phi(\tfrac{r}{2}R))^2 - (d\phi(\tfrac{2}{r}X_\alpha))^2).$$
Observe that $\tfrac{2}{r} X_\alpha$ and $\tfrac{r}{2} R$ are orthonormal unit vectors.  
Finally, since $d\phi(R)=0$, $d\phi(X_\beta)= \tfrac{1}{d\theta(R)} |d\phi|^2$ on each page. 
\end{proof}

\section{Construction of the plug} \label{sec:plug}

The goal of this section is to generalize the $3$-dimensional plug constructed in \cref{subsec:3d plug} to higher dimensions. This is the key construction that will allow us to prove \cref{thm:genericity} and \cref{thm:family genericity} (modulo the bypass claim, to be considered in \cref{part:BBC}) in \cref{sec:proof of genericity} in essentially the same way as in the $3$-dimensional case.

\subsection{Definition of a plug}\label{subsec: defn of plug}

Let us rephrase the $3$-dimensional case considered in \cref{subsec:3d plug} in a way that is amenable to higher-dimensional generalization. Consider the standard contact space $(\R^3,\ker(dz+e^s dt))$ and the surface $\Sigma= \{z=0\}\subset \R^3$. The plug is obtained by growing a mushroom along a box $U= [0,s_0] \times [0,t_0]\subset \Sigma$, where we are viewing $U$ as the truncated symplectization of the $1$-dimensional compact contact manifold $\p_{-} U =\{0\} \times [0,t_0]$ with contact form $dt$.

In higher dimensions, let $(Y,\ker\eta)$ be a compact contact manifold of dimension $2n-1$ with  convex  boundary. Let $$({N_{\epsilon_0}} (Y) \coloneqq Y \cup ([0,\epsilon_0] \times \p Y),\ker \eta)$$
be a small extension of $(Y,\ker\eta)$. Now we consider
$$(M^{2n+1} \coloneqq \R^2_{z,s} \times N_{\epsilon_0} (Y), \xi=\ker(dz+e^s \eta))$$
and the hypersurface $\Sigma \coloneqq \{z=0\}$. Let $U \coloneqq [0,s_0] \times N_{\epsilon_0} (Y)$ and let
$$\p_{-} U \coloneqq \{-1\} \times N_{\epsilon_0} (Y) \quad \mbox{and }\quad \p_{+} U \coloneqq \{s_0+1\} \times N_{\epsilon_0} (Y).$$

{\em From now on, we fix a Riemannian metric on $M$, which induces a metric on any submanifold  and such that $[0,\epsilon_0]\times \p Y$ has thickness $\epsilon_0$ with respect to this metric.}

\begin{definition} \label{defn: Y-shaped plug}
A {\em $Y$-shaped plug  with parameter $\epsilon>0$} is a $C^0$-small perturbation $\widetilde U$ of $U$ supported in the interior $U^{\circ}$ of $U$ such that:
\be
\item all the flow lines of $\widetilde U_\xi$ that pass through $\{-1\}\times Y^{\circ}$ flow to a negative singularity;
\item all the flow lines of $\widetilde U_\xi$ that pass through $\{s_0+1\}\times Y^{\circ}$ flow from a positive singularity;
\item for all  possibly broken  flow lines of $\widetilde U_\xi$ that go from $\p_- U$ to $\p_+ U$, the holonomy map is $\epsilon$-close to the identity when defined;
\item   $\widetilde U_\xi$ is gradient-like with respect to a Morse function $f: \widetilde U\to \R$ which agrees with $s$ on $\bdry \widetilde U$; in particular there are no possibly broken loops of $\widetilde U_\xi$.
\ee
\end{definition}

\begin{definition} \label{defn: Y-shaped pre-plug}
 A {\em $Y$-shaped pre-plug} $\widetilde U$ satisfies \cref{defn: Y-shaped plug} with (3) replaced by:
\be
\item[(3')] for each possibly broken flow line of $\widetilde U_\xi$ that goes from $\p_- U$ to $\p_+ U$, the holonomy map is obtained by following a small perturbation of $(\bdry Y)_{\ker \eta}$. 
\ee 
\end{definition}

\subsection{A Peter-Paul contactomorphism} \label{subsec:peter paul}

Let $(Y,\eta)$ be a contact manifold with a fixed choice of contact form $\eta$. Let $S$ be a hypersurface of $Y$ transverse to the Reeb vector field $R_\eta$.  Then $S$ has a neighborhood $S\times[-\epsilon,\epsilon]_\tau\subset Y$ on which $R_\eta=\bdry_\tau$.

The following is well-known:

\begin{lemma} \label{lemma: contactization}
If $R_\eta=\bdry_\tau$ on $S\times[a,b]_\tau\subset Y$, $a<b$, then $\eta= d\tau + \beta$, where $\beta$ is the pullback of a $1$-form on $S$. Moreover, $d\beta$ is symplectic on $S$.
\end{lemma}

In other words, $\eta$ is the {\em contactization of $(S,\beta)$}.  In particular, if $(S,\beta)$ is Weinstein then $S\times[a,b]$ is a contact handlebody.

\begin{proof}
We first write $\eta= fd\tau +\beta$, where $ f(\tau)  \in \Omega^0(S)$ and $\beta(\tau)\in \Omega^1(S)$.  Since $\eta(R_\eta)=1$, we have $f=1$.  Also, since $\mathcal{L}_{R_\eta}\eta=0$, $\beta(\tau)$ must be $\tau$-independent.  Finally, $d\beta$ is symplectic on $S$ due to the contact condition on $Y$.
\end{proof}

Given $(Y,\eta)$,  let $(M,\alpha)=(\R^2_{z,s} \times Y,dz + e^s \eta)$  and let $\phi_t: Y \stackrel\sim\to Y$ be the time-$t$ flow of $R_{\eta}$.

\begin{lemma}\label{lemma: peter paul 1}
The diffeomorphism
\begin{gather} \label{eqn:peter paul}
\Psi: M\stackrel\sim\to M,\quad  (z,s,y) \mapsto \left( e^{(-1+1/C)s} \cdot Cz, s/C, \phi_{(1-C) e^{-s}z} (y) \right),
\end{gather}
where $C>0$, is a contactomorphism.
\end{lemma}

\begin{proof}
We compute
\begin{align*}
\Psi^\ast (\alpha) &= d ( e^{(-1+1/C)s} \cdot Cz ) + e^{s/C} ( \eta + d( (1-C)e^{-s}z )) \\
&= e^{(-1+1/C)s} (1-C)z ds + e^{(-1+1/C)s} Cdz + e^{s/C} \eta \\
&\quad\quad  -e^{s/C} (1-C) e^{-s} zds + e^{s/C} (1-C) e^{-s} dz \\
&= e^{(-1+1/C) s} ( dz + e^s \eta ) = e^{(-1+1/C) s} \alpha.
\end{align*}
We explain the first line: By \cref{lemma: contactization}, $\eta$ can locally be written as $d\tau + \beta$, where $\beta$ is a $1$-form on a hypersurface $S\subset Y$ transverse to $R_\eta=\bdry_\tau$.  (Note that we can use the immersion $i:S\times\R_\tau\to Y$ with $i^*R_\eta=\bdry_\tau$ instead in \cref{lemma: contactization}.)  Then $\phi_{(1-C) e^{-s}z} (\tau,x)= (\tau+(1-C)e^{-s}z, x)$, where $x$ is the coordinate on $S$, and {$\phi_{(1-C) e^{-s}z}^*\eta= \eta + d((1-C)e^{-s}z)$.}
\end{proof}

 Let $M_{(z_0,s_0)} = [-z_0,z_0] \times [0,s_0] \times Y\subset (M,\alpha)$.  As an immediate corollary, by taking $C \gg 0$, we have:

\begin{lemma} \label{lem:peter paul}
For any $0< s'_0 \leq s_0$  and $z'_0>0$,  there exists $0 < z_0 \leq z'_0$, such that $M_{(z_0,s_0)}$ contactly embeds into $M_{(z'_0,s'_0)}$  and takes $\{z=0\}\cap M_{(z_0,s_0)}$ to $\{z=0\}\cap M_{(z'_0,s'_0)}$.
\end{lemma}

We call the contactomorphism $\Psi$ given by \eqref{eqn:peter paul} a \emph{Peter-Paul contactomorphism} for the following reason: In \cref{lem:peter paul}, $\Sigma = \{0\} \times [0,s] \times Y  \subset M_{(z_0,s_0)}$ is the hypersurface on which we want to create mushrooms. The length of the interval  $[-z_0,z_0]$  can be regarded as the given size of a neighborhood of $\Sigma$.  The map $\Psi$ then allows us to rob  the (already small) size of the neighborhood of $\Sigma$ to pay for a large size in the $s$-direction.

Observe that the Peter-Paul contactomorphism was not needed in \cref{sec:CST revisited} to make any $2$-dimensional surface convex.

\subsection{A pre-plug} \label{subsection: pre-plug}

Given {\em a standard Darboux ball $(Y^{2n-1},\eta')$ with convex boundary,} we explain how to construct a $Y$-shaped pre-plug $\widetilde U$  on 
$$\Sigma \coloneqq \{z=0\}\subset M=\R^2_{z,s}\times N_{\epsilon_0} (Y).$$ 
The $Y$-shaped pre-plug will be upgraded to a $Y$-shaped plug with parameter $\epsilon>0$ in the next two subsections. 

Modulo corner rounding, we may assume that $(Y,\ker \eta')$ is contactomorphic to $(Y_0\cup Y_1,\ker\eta)$ such that:
\be
\item $Y_0= (z_1^{-1}(\{|\zeta|\leq \epsilon'\}),\eta=\alpha_{std})$, where $z_1: S^{2n-1}\to \C_{\zeta}$ and $\alpha_{std}=\tfrac{1}{2}\sum_i (x_idy_i-y_idx_i)$ are as in \cref{subsection: qstab for sphere} and $\epsilon'>0$ is small;
\item $Y_1= ([0,2\pi]_t\times D^{2n-2},\eta=dt+\beta_{D^{2n-2}})$, where $\beta_{D^{2n-2}}$ is a standard Liouville form on $D^{2n-2}$ with one elliptic singular point; and
\item for each $t\in [0,2\pi]$, $\{t\}\times \bdry D^{2n-2}$ is glued to $z_1^{-1}(\epsilon'e^{it})$ so that the contact forms (and the Reeb vector fields) match.
\ee
(In particular, $Y_0\cup Y_1$ has a partial open book structure, where $Y_0$ is a neighborhood of the binding and $\{t\}\times D^{2n-2}$ is a retraction of a page; see \cref{defn: POBD}.)

We now apply the Peter-Paul contactomorphism to realize $(Y_0\cup Y_1,\eta)$ as a transverse slice (i.e., transverse to the characteristic foliation) on $\Sigma$; {\em we then blur the distinction between $Y$ and $Y_0\cup Y_1$ and write $Y=Y_0\cup Y_1$.} The price that we pay is that we lose control of the $z$-height. (Recall that the $z$-height restricts the thickness of the contact handlebody of the mushroom that we want to grow.)

To remedy this we apply quantitative stabilization (\cref{lemma: qstab for S 2n-1}) with $k\gg 0$ and $0<\epsilon\ll \epsilon'$ to $Y_0$ to obtain $\pi_k: Y_0\to S^1$. After a $C^\infty$-small perturbation of $Y_1$ whose size depends on $k$ and which we still denote by $Y_1$, there is a smooth extension $\pi_k: Y\to S^1$ which agrees with $(t,x)\mapsto kt$ on $Y_1$.  

Choose a finite number, say $N=5$, and a small constant $0<\epsilon''<\epsilon'$. Then, for $j=0,\dots, 4$, let $H'_j$ be the sector $\pi_k^{-1}([\tfrac{2\pi j}{5},\tfrac{2\pi (j+1)}{5}])$. Assuming we chose a suitable extension of $\pi_k$,  $H'_j$ is a (complete) contact handlebody of thickness $\tfrac{2\pi}{k}$. Next let $H_j$ be a slight modification of $H'_j$, $j=0,\dots,4$, obtained by removing an $\tfrac{\epsilon''}{2}$-neighborhood $\pi_k^{-1}(\{z\leq \tfrac{\epsilon''}{2}\})\cap Y_0$ of the binding $B_k$ and thickening the contact handlebody by flowing forward and backward by $\tfrac{\epsilon''}{k}$ in the Reeb direction.

Finally, we slightly (i.e., in a $C^\infty$-small manner) modify $\eta$ by ``shifting the binding" $B_k$ away from $H_0$ using a contactomorphism $\phi$ and construct the analogous contact handlebody $\phi(H_0)$ of thickness $\tfrac{2\pi+2\epsilon''}{k}$ such that $\phi(H_0)$ contains the $\epsilon''$-neighborhood of $B_k$, as follows:

\s\n
{\em Brief explanation of ``shifting the binding".}
By flowing along the characteristic foliation of each page, one can normalize $\alpha_{std}$ on a small neighborhood $B_k\times D^2$ of the binding $B_k$ as 
$$\alpha_{std}=f\beta + g (xdy-ydx),$$ 
where $D^2$ is a disk with Euclidean/polar coordinates $(x,y)$/$(r,\theta)$ and small radius; $\beta=\alpha_{std}|_{B_k}$; $f=f(r)$ and $f(0)=1$; $g=g_0 r^2+O(r^3)$ and $g_0$ is a nonvanishing function on $B_k$; and {\em the pages are still $\theta=\mbox{const}$.} Letting 
$$\alpha:= f^{-1}\alpha_{std}= \beta+ h(xdy-ydx),\quad \mbox{where} \quad h=f^{-1}g,$$ and $X=\tfrac{\bdry}{\bdry x}- h y R_\beta$, where $R_\beta$ is the Reeb vector field of $\beta$, we compute: 
\begin{gather*}
d\alpha= d\beta + dh\wedge(xdy-ydx) +2h dxdy,\\
i_X d\alpha= 2h dy + O(r), \quad d (\alpha(X))= -2h dy + O(r),
\end{gather*}
and hence $\mathcal{L}_X \alpha = O(r)$. Therefore, the translation by $aX$, $a<0$ small, on $B_k\times D^2$ (where defined), is close to a contactomorphism and the modification needed to make it into a contactomorphism $\phi$ is on the order of $a\cdot O(r)$, which is an order of magnitude smaller. Hence $\phi(H_0)$ has the property of being close to the $x\mapsto x+a$-translate of $H_0$ with error much smaller than $a$; in particular $\phi(H_0)$ contains the $\epsilon''$-neighborhood of $B_k$.  Finally, one may adjust the contact form on $\phi(H_0)$ by multiplying by a function that is close to $1$ so that $\phi(H_0)$ becomes a contact handlebody.

\s
In what follows we abuse notation and refer to $\phi(H_0)$ by $H_0$.

The pre-plug $\widetilde U$ consists of mushrooms $Z_{j}$, $0\leq j \leq 4$, with bases
$$B_{j} \coloneqq [\tfrac{8-2j}{10},\tfrac{9-2j}{10}]_s \times H_{j};$$
see \cref{fig:open book dynamics}. (Strictly speaking, the base $B_0$ is obtained by shifting $B'_0=[\tfrac{8}{10}, \tfrac{9}{10}]\times H_0$ by the map $(s,x)\mapsto (s+f_0(x),x)$, where $f_0$ is a $C^0$-small smooth function on $H_0$. We will assume that we started with a slightly smaller $B'_0= [\tfrac{8}{10}+\delta, \tfrac{9}{10}-\delta]\times H_0$ and that the base for the contact handlebody $H_0$ {\em contains} $B_0$.) We also require the damping of $Z_1,\dots, Z_4$ to occur inside the $\tfrac{3\epsilon''}{4}$-neighborhood of $B_k$ so that the damping regions of $H_1,\dots,H_4$ are contained in $H_0$ and the damping region of $H_0$ to occur inside $H_2$.

\begin{figure}[ht]
	\begin{overpic}[scale=.35]{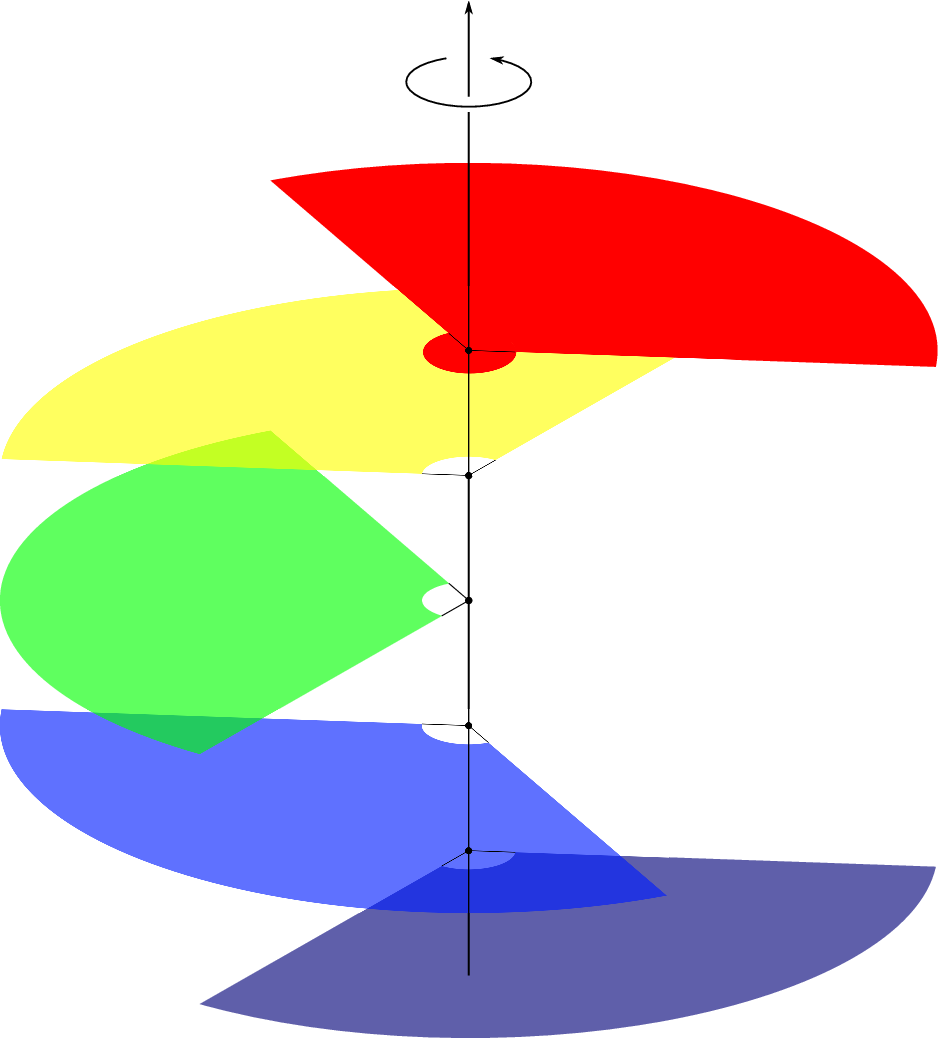}
		\put(41.5,99){$s$}
		\put(52.5,90.5){$\theta$}
	\end{overpic}
\caption{$H_0, \dots, H_4$, from top to bottom.}
\label{fig:open book dynamics}
\end{figure}

\n
{\em Verification of (3') in \cref{defn: Y-shaped pre-plug}.} This is a direct consequence of \cref{prop:dynamics C-fold} and our choice of the ordering of the mushrooms in the $s$-direction.

We first give names to regions of $\bdry_- U$: Viewing $H_j$ as a subset of $\bdry_- U$, let $\widetilde H_j$ be the closure of the union of $H_j$ and the set of points such that the holonomy from $s=\tfrac{8-2j}{10}-\delta$ to $\tfrac{9-2j}{10}+\delta$ (for $\delta>0$ small) is not trivial or does not exist; note that $\widetilde H_j$ is contained in a small neighborhood of $H_j$.  We denote the portion of $\widetilde H_j$ that closely approximates it and acts as a sink by $H_{j,\op{in}}$ and {$\widetilde H_j \setminus H_{j,\op{in}}$} by $H_{j, \bdry}$, and the corresponding products with $[\tfrac{8-2j}{10},\tfrac{9-2j}{10}]$ by $B_{j,\op{in}}$ and $B_{j,\bdry}$. 

The dynamics of $\widetilde U$ in forward time is described as follows: Let $x\in \bdry_-U$ and let $\ell_x$ be the flow line of $\widetilde U$ passing through $x$. 
\s\n
(A) If $\ell_x$ enters $B_{j,\op{in}}$, $j=0,\dots,4$, then $\ell_x$ converges to a singularity in $Z_j$.

\s\n
(B) If $\ell_x$ enters $B_{j,\bdry}$, then $\ell_x$ exits $Z_j$ at a point near $\bdry H_j$ --- {\em recall that as $\ell_x$ passes through $Z_j$ it flows ``from $R_+(\bdry H_j)$ to $R_-(\bdry H_j)$" in the sense of (Z5) of \cref{prop:dynamics C-fold} with possibly large nontrivial components in the Reeb direction of $\Gamma_{\bdry H_j}$ on the damping region} --- and one of the following will happen:
\begin{enumerate}
	\item $\ell_x$ follows $\p_s$ until $\{s=1\}$ and exits the plug; the holonomy map is obtained by following a small perturbation of $(\bdry Y)_{\ker \eta}$;
	\item $\ell_x$ follows $\p_s$ until $B_{i,\op{in}}$ for $i<j$; we then apply (A) after letting the new $j$ equal $i$; 
	\item $\ell_x$ follows $\p_s$ until $B_{i,\bdry}$ for $i<j$; we then apply (B) after letting the new $j$ equal $i$.
\end{enumerate}
Note that $H_0$ was chosen so that $Z_0$ captures all the trajectories that enter $\bdry_-U$, ``survives to'' $s= \tfrac{8}{10}$, and is not close to $\bdry Y$. A similar analysis can be applied to the dynamics of $\widetilde U$ in backward time and (3') of \cref{defn: Y-shaped pre-plug} then holds.

In light of \cref{prop:morse criterion} and \cref{lem:C^infty perturb}, we conclude that $\widetilde U$ is Morse.

\subsection{\texorpdfstring{$\epsilon$}{\textepsilon}-short hypersurfaces} \label{subsec:epsilon convexity}

In this subsection we strengthen \cref{thm:genericity} in a quantitative way. Namely, in addition to requiring that $\Sigma_{\xi}$ be Morse$^+$ or  $1$-Morse$^+$,  we also require all the smooth flow lines of $\Sigma_{\xi}$ to be short. 

\begin{definition} \label{defn:epsilon convexity}
A closed hypersurface $\Sigma \subset (M,\xi)$ is {{\em $\epsilon$-short}} if the length of any smooth flow line of $\Sigma_{\xi}$ is shorter than $\epsilon$ with respect to the induced metric on $\Sigma$.
\end{definition}

{Observe that any closed Morse$^+$ hypersurface} $\Sigma$ is {$\epsilon$-short for a sufficiently} large $\epsilon$ which depends on $\Sigma$.  For our purposes we take $\epsilon>0$ to be a small number which is independent of the choice of convex hypersurface. \cref{thm:genericity} can be strengthened as follows:

\begin{theorem} \label{thm:epsilon genericity}
Given $\epsilon>0$, any closed hypersurface $\Sigma$ in a contact manifold can be $C^0$-approximated by  an $\epsilon$-short and Morse$^+$ one. Moreover, if $\Sigma$ is Weinstein convex, then there exists a $t$-invariant neighborhood $[-\delta,\delta]_t\times \Sigma$ of $\{0\}\times \Sigma$ and a $1$-parameter family of pairwise disjoint embeddings $\phi_t: \Sigma\to [-\delta,\delta]\times \Sigma$, $t\in [-\delta,0]$, such that: 
\be
\item $\phi_{-\delta}(\Sigma)=\{-\delta\}\times \Sigma$, 
\item $\phi_t(\Sigma)$ is $C^0$-close to $\{t\}\times \Sigma$ for all $t\in[-\delta,0]$, 
\item $\phi_{0}(\Sigma)$ is $1$-Morse$^+$ and $\epsilon$-short,  
and 
\item $\phi_t(\Sigma)$ have the same number and type of singular points for all $t\in[-\delta,0)$ and the whole family $\phi_t(\Sigma)$, $t\in[-\delta,0]$ is Weinstein convex.
\ee
\end{theorem}

 In words, what (3) and (4) are saying is that the singular points of the characteristic foliation of $\phi_t(\Sigma)$ remain the same for $t<0$ and all the singular points --- necessarily of birth-death type --- are created at the same time when $t$ reaches $0$. 

\cref{thm:epsilon genericity} holds in dimension $3$ by \cref{subsec:genericity in 3d},  together with a slightly more careful analysis of the characteristic foliation when installing a mushroom.  This will be the base case of our inductive argument.

\subsection{Construction of the plug} \label{subsec:hd plug}

The goal of this subsection is to prove the following: 

\begin{theorem} \label{thm: existence of Y-shaped plug}
Let $Y$ be the standard Darboux ball of dimension $2n-1$ with convex boundary. Then  for any $\epsilon>0$ small there exists a $Y$-shaped plug with parameter $\epsilon$,  provided \autoref{thm:epsilon genericity} holds for contact manifolds of any dimension $\leq 2n-1$.
\end{theorem}

{\em We adopt a simplification due to Eliashberg, Fauteux-Chapleau, and Pancholi, used with their permission, while retaining certain elements of the original version of the paper.} 

\begin{proof}
 Given $\epsilon>0$ small, choose $0<\epsilon_0\ll \epsilon$. Let $\Sigma$ be the pushoff of $\bdry Y$ towards $\bdry N_{\epsilon_0}(Y)$ such that $\Sigma$ is the {\em standard Weinstein convex sphere}, which we take to satisfy:
\be
\item[(Sph1)] $\Sigma\simeq S^{2n-2}$ and $\Gamma(\Sigma)\simeq S^{2n-3}$, where $\simeq$ means ``diffeomorphic";
\item[(Sph2)] $R_\pm(\Sigma)\simeq D^{2n-2}$ is a Weinstein domain with precisely one singular point, and the singular point on $R_+(\Sigma)$ (resp.\ $R_-(\Sigma)$) has index $0$ (resp.\ index $2n-2$). 
\ee

Applying \cref{thm:epsilon genericity}, there exists a foliation by $\Sigma_t':=\phi_t(\Sigma)$ for $t\in[-\delta,0]$ satisfying (1)--(4). We are starting with $\Sigma_{-\delta}'$ which is the standard Weinstein convex sphere by (1); for $t\in[-\delta,0)$, $\Sigma_t'$ remains the standard Weinstein convex sphere by (4); and for $t=0$, $\Sigma_0'$ is $\epsilon$-short and Weinstein convex with birth-death singular points by (3). 

\s\n
{\em Setup for $\widetilde U$.}
The $Y$-shaped plug $\widetilde U$ consists of two mushrooms $Z_0$, $Z_2$ with contact handlebody profiles $H_0$, $H_2$; a pre-plug $Z_1$ with profile $H_1$;   and bases $[\tfrac{2j}{6}s_0,  \tfrac{2j+1}{6}s_0]\times H_j$ (i.e., $Z_0$ has the smallest $s$-value, followed by $Z_1$, and then by $Z_2$), such that the following hold:
\be
\item[(P1)] $N_{\epsilon_0}(Y)\supset H_0\cup H_1\cup H_2 \supset Y$ and $H_0\cap H_2=\emptyset$;
\item[(P2)] the damping region of $H_0$ is a subset of $H_{1,\op{in}}$ and the damping region of $H_2$ is a subset of $H_{1,\op{out}}$.  
\ee
Note that $H_1$ does not need a damping region.
See \cref{fig: plug-new-version}.  Recall that on the damping regions the characteristic foliation may have large components in the Reeb direction of the dividing set. Since these can potentially create trouble with the $\epsilon$-shortness, we require (P2).  
\begin{figure}[ht]
	\begin{overpic}[scale=.8]{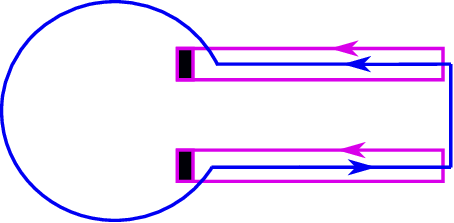}
	\put(80,40){$H_0$}
	\put(5,30){$H_1$}
	\put(80,3){$H_2$}
	\end{overpic}
	\caption{A schematic diagram of the plug, with the $s$-direction projected out (the boundaries on the right-hand side do not actually exist). The arrows indicate the direction of the flow along the characteristic foliations of the handlebodies (i.e., from the positive side to the negative side).  The circular region represents $N(\Gamma)$ and the shaded regions are the damping regions.}
	\label{fig: plug-new-version}
\end{figure}

\s\n
{\em Analysis of the dynamics of $\widetilde U_\xi$.}  We will specify $H_0,H_1,H_2$ later, but for the moment we analyze the effect of the plug $\widetilde U$ when we stack the mushrooms as above. 
\be
\item  Given a flow line of $\widetilde U_\xi$ that enters $Z_0$, either it flows to a negative singularity of $Z_0$ or passes through $Z_0$ after flowing parallel to the characteristic foliation of $\bdry H_0$ as given by (Z5) of \cref{prop:dynamics C-fold}.
\item  Given a flow line of $\widetilde U_\xi$ that enters $Z_1$ (including one that exits $Z_0$ in the previous step), either it flows to a negative singularity of $Z_1$ or passes through $Z_1$ after flowing parallel to the characteristic foliation of $\bdry H_1$.  Notice that all the potentially problematic orbits that had large components in the Reeb direction in $Z_0$ (i.e., those that flowed out of the damping region of $H_0$ or near $R_-(\bdry H_0)$) flow to a negative singularity of $Z_1$ by (P2).  
\item A similar consideration holds for a flow line of $\widetilde U_\xi$ that enters $Z_2$ (including one that exits $Z_1$ in the previous step).  
\item Backwards flow lines can be analyzed similarly. 
\ee
The orbits that pass through $\widetilde U$ flow along the characteristic foliations of $\bdry H_0,\bdry H_1$, and $\bdry H_2$ without ever entering damping regions; this means they flow closely along $\bdry (H_0\cup H_1\cup H_2)$. 

\s\n
{\em Description of $H_0,H_1,H_2$.}  Let us write $\Gamma=\Gamma(\Sigma_0')$ and let $N(\Gamma)=\Gamma\times D^2_{\rho,\theta}$ be a tubular neighborhood of $\Gamma$ with contact form $\beta_\Gamma + C (\tfrac{\rho^2}{2}d\theta)$ for $C>0$ small and such that the thickness in the $D^2$-direction is $\leq \epsilon''$ for $\epsilon''>0$ small.  We take $H_0$ (resp.\ $H_2$) to be a thin contact handlebody over an $\epsilon'''$-retraction of $R_+(\Sigma_0')$ (resp.\ $R_-(\Sigma_0')$), where $0< \epsilon'''\ll \epsilon''$, and take the damping region to be small. There exists $\delta_0'>0$ small such that 
$$R_+(\Sigma'_{-\delta_0})\setminus N(\Gamma)\subset H_{0,\op{out}} \quad \mbox{and}\quad  R_-(\Sigma'_{-\delta_0})\setminus N(\Gamma)\subset H_{2,\op{in}}.$$
We then take $H_1$ to be union of $N(\Gamma)$ and the region of $N_{\epsilon_0}(Y)$ bounded by $\Sigma'_{-\delta_0}$. There is some flexibility in choosing the dividing set and we take $\Gamma_{\bdry H_1}$ and also the damping region to be a subset of $H_{2,\op{in}}$. (P1) and (P2) hold by construction and also $\bdry (H_0\cup H_1\cup H_2)$ is $\epsilon$-short.  

\begin{lemma} \label{lemma: contactmorphic to standard Darboux ball}
The contact manifold $H_1$, after a $C^\infty$-small perturbation, is contactomorphic to a standard Darboux ball with convex boundary. 
\end{lemma}

\begin{proof}[Proof of \cref{lemma: contactmorphic to standard Darboux ball}]
Let $H_1':=[0,1]_t\times D$, where $D=D^{2n-1}$, and let $D_t:=\{t\}\times D$. 
We sketch the proof that if $\alpha$ is a contact form on $H_1'$ such that:
\be
\item on a neighborhood of $[0,1]\times \bdry D_t$, $\alpha=dt+\beta$, where $\beta$ is independent of $t$ and is the symplectization of a standard contact form on $\bdry D$.
\item the characteristic foliation on each $D_t$ consists of a positive elliptic singularity $e_t$ and all its trajectories go from $e_t$ to $\bdry D_t$;
\ee
then $(H_1',\ker\alpha)$ is contactomorphic to $(H_1',\ker (dt+\beta))$ for some extension of $\beta$ to $D$. 
The assumptions imply that $\alpha=f_t dt+ \beta_t$, where $f_t$ is a function on $D_t$ and $\beta_t$ is a $1$-form on $D_t$. We slightly perturb $\alpha$ so that the $(N(e_t),\beta_t|_{N(e_t)})$ are all diffeomorphic, where $N(e_t)\subset D_t$ is a neighborhood of $e_t$, and apply a $1$-parameter family of diffeomorphisms to straighten the characteristic foliation and make $\beta_t$ independent of $t$, which we now write as $\beta'$.   The contact condition implies that $f_t>0$; applying the Reeb flow gives the normalization $fdt+\beta'$; and we divide by $f$. 

The lemma then follows modulo adjusting/rounding corners.
\end{proof}

In view of \cref{lemma: contactmorphic to standard Darboux ball} and the construction of a pre-plug from \cref{subsection: pre-plug}, $\widetilde U$ is a $Y$-shaped plug with parameter $\epsilon$. 
\end{proof}

\section{Approximation by convex hypersurfaces} \label{sec:proof of genericity}

In this section we complete the proofs of \cref{thm:epsilon genericity} and \cref{thm:family genericity}. The main technical ingredient is the higher-dimensional plug constructed in \cref{sec:plug}.  In fact, our proofs are basically the same as those for the $3$-dimensional case discussed in \cref{sec:CST revisited}.

\begin{proof}[Proof of \cref{thm:epsilon genericity}]
 The proof is by induction on the dimension of $M$. 

Fix a Riemannian metric on $M$. Given a closed hypersurface $\Sigma \subset (M,\xi)$, we may assume that the singularities of $\Sigma_{\xi}$ are isolated and Morse after a $C^{\infty}$-small perturbation.  By \cref{thm: barricade} a finite barricade $B_I=\{B_i=[0,s_i]\times Y_i\}_{i\in I}$ exists for $\Sigma_\xi$, where each $Y_i$ is a standard $(2n-1)$-dimensional Darboux ball. 
	
 Choose $\epsilon'>0$ much smaller than the sizes of the $Y_i$. We then replace each $B_i$ with a $Y_i$-shaped plug with parameter $\epsilon'$ constructed in \cref{sec:plug};  note that the construction of the $Y_i$-shaped plug uses the inductive step of \cref{thm:epsilon genericity} for $\op{dim}(M)-2$.   Let $\Sigma^{\vee}$ be the resulting hypersurface.   A trajectory that passes near $Y_i$ is either trapped by $Y_i$ or has holonomy at most $\epsilon'$.  Since $\epsilon'$ is small, the trajectory that continues is close to the original trajectory on $\Sigma_\xi$ and will be trapped by some other $Y_j$ by the positioning of the barricade.  Hence  $\Sigma^{\vee}_{\xi}$ satisfies Conditions (M1)--(M3) of \cref{prop:morse criterion}.  A  further $C^{\infty}$-small perturbation of $\Sigma^{\vee}$ will make it convex by \cref{prop:convex criterion}. 

 The $\epsilon$-Morse$^+$ property is guaranteed if all the $B_i, i \in I$, have diameters $<\tfrac{\epsilon}{3}$ and all the trajectories of $\Sigma_\xi\setminus (\cup_{i\in I} B_i)$ have lengths $< \tfrac{\epsilon}{3}$. 

 Finally, the second statement (for $\Sigma$ already Weinstein convex) follows from observing that all singularities of all the $Y_i$-shaped plugs can be turned on simultaneously; see \cref{rmk: variant of prop:char foliation on 3d C-fold} which also holds for higher dimensions.  
\end{proof}

\begin{proof}[Proof of \cref{thm:family genericity}]
Let $(\Sigma \times [0,1],\xi)$ be a contact manifold such that the hypersurfaces $\Sigma_i$, $i=0,1$, are {Weinstein convex.} The proof from \cref{subsec:family genericity in 3d} carries over almost verbatim with ``surface" replaced by ``hypersurface."  We use \cref{lemma: refinement} to construct refinements of barricades so that we can install/uninstall two sets of barricades.  

Suppose for simplicity that the barricade consists of one  flow box $B = [0,s_0] \times D^{2n-1}$.  Even though the folded hypersurface $\Sigma^{\vee}$ can be made convex, the intermediate hypersurfaces appearing in the procedure of installing and uninstalling  a $D^{2n-1}$-shaped plug  need \emph{not} be Morse.
To remedy this defect, we take a cover  $D^{2n-1} = \cup_{1 \leq i \leq K} U_i$ by a finite number of balls of small diameter $\tfrac{1}{N}$  for which there exists a partition
$${\mathcal{P}:}\{1,\dots,  K\}\to\{1,\dots,2n\}$$
such that $U_i\cap U_j=\emptyset$ if $\mathcal{P}(i)=\mathcal{P}(j)$. Now we choose $2n$ pairwise distinct values in $(0,s_0)$ and position $U_i$ along $[0,s_0]$ such that all the $U_i$ with the same $\mathcal{P}$-value have the same $s$-value.    When we install/uninstall all the $U_i$-shaped plugs, the interior discrepancy for $B$ goes to zero as $N\to\infty$ by the analog of \cref{lem:small holonomy 3d} and the fact that the $\epsilon$ that appears in \cref{thm: existence of Y-shaped plug} is one order of magnitude smaller than the sizes of $U_i$. 

Finally, proceeding as in \cref{subsec:family genericity in 3d}, we can foliate $\Sigma\times[0,1]$ by hypersurfaces of the form $\Sigma_t$,  which we may assume are $1$-Morse by the smallness of the interior discrepancy and Step 1 of \cref{subsec:family genericity in 3d}.   The only obstruction to convexity occurs when $(\Sigma_t)_\xi$ fails to be Morse$^+$, which occurs at isolated moments.
\end{proof}

\newpage

\part{Bypass-bifurcation correspondence}\label{part:BBC}

\section{Folded Weinstein hypersurfaces}\label{sec:folded-weinstein}

In this section, we repackage the Morse theory on Morse hypersurfaces from \cref{sec:convexity criterion} in terms of \emph{folded Weinstein hypersurfaces}, a generalization of convexity.

\subsection{Definitions}

\begin{definition} \label{defn:folded Weinstein hypersurface}
An oriented hypersurface $\Sigma \subset (M,\xi)$ is a \emph{folded Weinstein hypersurface} if the characteristic foliation $\Sigma_{\xi}$ satisfies the following properties:
\begin{itemize}
\item[(FW1)] There exist pairwise disjoint closed codimension-$1$ submanifolds $K_i \subset \Sigma$ cutting $\Sigma$ into $2m$ pieces for some $m\geq 1$, i.e.,
\[
\Sigma = W_1 \cup_{K_1} \dots \cup_{K_{2m-1}} W_{2m},
\]
where $W_i$ are compact cobordisms with boundary $\partial W_i=K_i\cup K_{i-1}$. For notational convenience, we set $K_0 = K_{2m} = \emptyset$. We call $K_i$ the \emph{folding loci} of $\Sigma$. 

\item[(FW2)] The singular points of $\Sigma_{\xi}$ in each $W_i$ have the same sign, and the sign changes when crossing $K_i$. We assume the singular points in $W_1$ are positive and that each $W_i$ has at least one singular point. 

\item[(FW3)] There exists a Morse function $f_i$ on each $W_i$ such that $K_{i-1}$ and $K_i$ are regular level sets and $(W_i)_{\xi}$ is gradient-like with respect to $f_i$. In particular, $\Sigma_{\xi}$ is transverse to each $K_i$.
\end{itemize}
\end{definition}

\begin{figure}[ht]
	\centering
	\begin{overpic}[scale=.38]{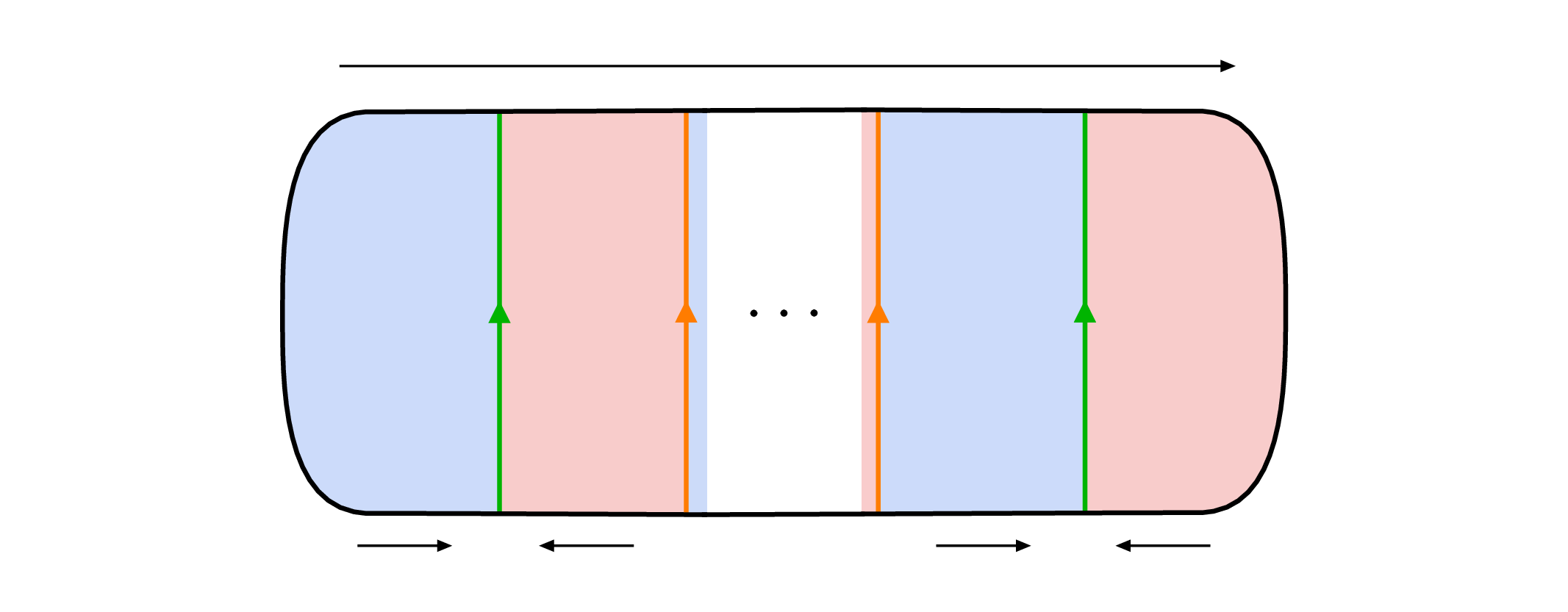}
        \put(23,18){$W_1$}
        \put(36,18){$W_2$}
        \put(59,18){$W_{2m-1}$}
        \put(73,18){$W_{2m}$}
	\end{overpic}
	\caption{A schematic depiction of a folded Weinstein hypersurface. Positive (resp. negative) regions are blue (resp. red), the oriented characteristic foliation goes from left to right, and the arrows below indicate the direction of the Liouville vector fields. Dividing sets are green, anti-dividing sets are orange.}
	\label{fig:fw1}
\end{figure}

By definition, any folded Weinstein hypersurface is Morse, and conversely any Morse hypersurface can be equipped with the structure of a folded Weinstein hypersurface. The folded Weinstein framework is useful because the argument in \cref{prop:convex criterion} furnishes a contact form $\alpha$ for $\xi$ such that $(-1)^{i+1}W_i$ is a Weinstein cobordism with Weinstein structure $(\lambda = \alpha\mid_{W_i}, f_i)$. In particular, the orientation on $W_i$ given by the Weinstein structure agrees with (resp.\ is opposite to) the orientation inherited from $\Sigma$ if the singular points of $(W_i)_{\xi}$ are positive (resp.\ negative); folded Weinstein structures are thus a specification of folded symplectic structures \cite{cannasdasilva2010folded,breen2024folded}.

As we will see below, if $K_i$ is a folding locus with $i$ odd (hence $W_i$ is a \emph{positive region} and $W_{i+1}$ is a \emph{negative region}), the contact germ near $K_i$ is locally indistinguishable from that of a dividing set of a convex hypersurface. There are differences in the case where $i$ is even (i.e., when $W_i$ is negative and $W_{i+1}$ is positive). For this reason, we refer to an odd-indexed folding locus as a \emph{dividing set} and an even-indexed folding locus as an \emph{anti-dividing set}. Before investigating these differences, we discuss some examples.  

\begin{example}
If $\Sigma$ is a convex hypersurface such that $R_{\pm} (\Sigma)$ are Weinstein manifolds, then  $\Sigma$ can be given the structure of a folded Weinstein hypersurface where the folding locus coincides with the dividing set $\Gamma_{\Sigma}$. There are no anti-dividing set components. On the other hand, a folded Weinstein hypersurface is not always convex because there may exist trajectories of $\Sigma_{\xi}$ from a negative singularity to a positive one. Nevertheless, since a $C^\infty$-small perturbation of a Morse hypersurface is Morse$^+$ by \cref{lem:C^infty perturb}, any folded Weinstein hypersurface is $C^{\infty}$-generically convex by \cref{prop:convex criterion}.
\end{example}

\begin{example}
Let $\Sigma = T^2$ be a torus with the characteristic foliation as depicted in \cref{fig:fw_ex2}. The characteristic foliation is Morse, but thanks to the retrogradient trajectory it is not Morse$^+$ (and $\Sigma$ is not convex). A folded Weinstein structure is indicated in the figure. 
\end{example}

\begin{figure}[ht]
	\centering
	\begin{overpic}[scale=.43]{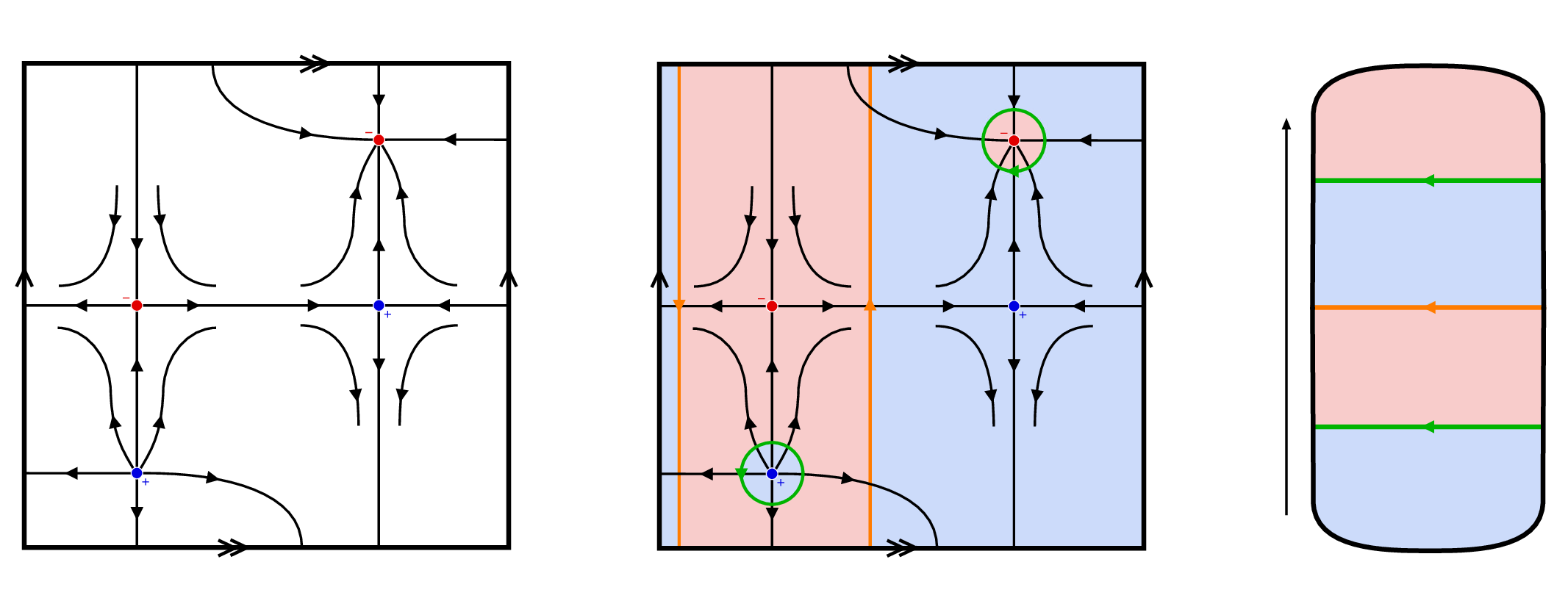}
        \put(89.5,7.5){$W_1$}
        \put(89.5,15){$W_2$}
        \put(89.5,22.5){$W_3$}
        \put(89.5,30){$W_4$}
	\end{overpic}
	\caption{A folded Weinstein $T^2$. The folding loci $K_1$ and $K_3$ in green are dividing sets, while $K_2$ in orange is an anti-dividing set.}
	\label{fig:fw_ex2}
\end{figure}

We introduce the following terminology:

\begin{terminology}
Given a folded Weinstein hypersurface $\Sigma\subset(M,\xi)$ and a critical point $p$ of $\Sigma_\xi$:
    \be
    \item The \emph{characteristic-stable manifold} $\op{stab}_\Sigma(p)$ of $p$ is the stable manifold of $p$ in $\Sigma$ with respect to $\Sigma_\xi$ with the usual orientation.
    \item The {\em Weinstein-stable manifold} $\op{stab}_{W_i}(p)$ of $p\in W_i$ is the stable manifold of $p$ in $\overline{W}_i$ with respect to the Liouville vector field on $W_i$.
\ee
The characteristic-unstable and Weinstein-unstable manifolds are defined similarly and are denoted $\op{unstab}_\Sigma(p)$ and $\op{unstab}_{W_i}(p)$, respectively.
\end{terminology}

Observe that $\op{stab}_\Sigma(p)$ properly contains $\op{stab}_{W_i}(p)$ if $p\in W_i$ with $i$ odd, while $\op{unstab}_\Sigma(p)$ properly contains $\op{stab}_{W_i}(p)$ if $p\in W_i$ with $i$ even.

\subsection{Normalization near a folded Weinstein hypersurface} \label{subsec:contact form on folded Weinstein}

Recall that if $\Sigma \subset (M,\xi)$ is a convex hypersurface, then there exists a standard tubular neighborhood $U(\Sigma)\simeq \R_t\times \Sigma$ of $\Sigma$ such that $\xi|_{U(\Sigma)} = \ker (fdt+\beta)$, where $f \in C^{\infty} (\Sigma)$ and $\beta \in \Omega^1(\Sigma)$. In particular, $\xi$ is $t$-invariant and $U(\Sigma)$ has infinite size. 

The goal of this subsection is to generalize this to folded Weinstein hypersurfaces. To a folded Weinstein surface $\Sigma \subset (M,\xi=\ker \alpha)$ we associate a \emph{standard tubular neighborhood} $U(\Sigma)$ of $\Sigma$ and a contactomorphism $\phi: U(\Sigma)\stackrel\sim\to ((-\epsilon,\epsilon)_t \times \Sigma, \ker\alpha_{\Sigma})$, where $\alpha_\Sigma$ is a normalized contact form to be specified below, $\epsilon>0$ is sufficiently small, and $\phi(\Sigma)=\{0\}\times \Sigma =: \Sigma_0$. Importantly, $\xi$ will not be $t$-invariant, nor will the neighborhood have infinite size. Both of these features arise because of the contact germ near an anti-dividing set.

Given a folded Weinstein hypersurface $\Sigma = W_1 \cup_{K_1} \cup \dots \cup_{K_{2m-1}} W_{2m}$, we will construct the normalized contact form $\alpha_{\Sigma}$ on $\R\times\Sigma$ such that $\alpha_{\Sigma} |_{\Sigma} = \alpha|_{\Sigma_0}$, up to rescaling by a positive function. \cref{lem:char foliation} then gives the desired contactomorphism.

\s\n
\textsc{Step 1.} \emph{Construct the contact form away from the folding loci.}
\s

For each $i$, let $U(K_i)\subset \Sigma$ be a collar neighborhood of $K_i$ with coordinates $U(K_i) = [-1,1]_{\tau} \times K_i$, such that the oriented characteristic foliation $\Sigma_{\xi}$ is directed by $\partial_{\tau}$ on $U(K_i)$. Note that this implies $\{-1\} \times K_i \subset W_i$ and $\{1\} \times K_i \subset W_{i+1}$. 

Let $W_i^{\circ} := W_i \setminus (U(K_{i-1}) \cup U(K_i))$. Because $\Sigma$ is folded Weinstein, we may rescale the ambient contact form by a positive function if necessary and assume that $\beta_i := \alpha\mid_{W_i^{\circ}}$ is a Liouville form on $W_i^{\circ}$ for all $i$. By extending $W_i^{\circ}$ slightly into each neighborhood of the folding loci, we can also make the following assumption about the Liouville forms $\beta_i$ (which are purely technical, facilitating the gluing of contact forms later on). 

\begin{itemize}
    \item If $i$ is odd (so that $W_i$ is a positive region), near $\partial U(K_i)$ the Liouville vector field is $X_{\beta_i} = -\frac{1}{2\tau}\, \partial_{\tau}$.  
    
    \item If $i$ is even (so that $W_i$ is a negative region, hence is an oppositely oriented Weinstein cobordism), near $\partial U(K_i)$ the Liouville vector field is $X_{\beta_i} = \frac{1}{2\tau}\, \partial_{\tau}$.  
\end{itemize}
We define
\begin{equation} \label{eqn:contact form in middle}
\alpha_{\Sigma} := (-1)^{i+1} dt + \beta_i
\end{equation}
on $\R \times (\Sigma \setminus \cup_{i=1}^{2m-1} U(K_i))$. Everything in this step is consistent with the approach taken in normalizing the contact structure near a convex hypersurface, where the contact form rescales to $\pm dt + \beta$ on $R_{\pm}$. 

\s\n
\textsc{Step 2.} \emph{Construct the contact form near dividing sets.}
\s

For $i$ odd, the construction of the contact form on $\R \times U(K_i)$ is analogous to the convex case. First, assume after rescaling that $\alpha\mid_{U(K_i)} = e^{-\tau^2}\, \eta$, where $\eta$ is a contact form on $K_i$. We have $d(e^{-\tau^2}\, \eta) = -2\tau e^{-\tau^2}\, d\tau\, \eta + e^{-\tau^2}\, d\eta$, so near $\partial U(K_i)$ this defines a Liouville form with Liouville vector field $-\frac{1}{2\tau}\, \partial_{\tau}$. Note that 
\[
d(e^{-\tau^2}\, \eta)^n = (-2\tau e^{-\tau^2}\, d\tau\, \eta + e^{-\tau^2}\, d\eta)^n = -2n\tau e^{-n\tau^2}  \, d\tau\, \eta\, (d\eta)^{n-1}
\]
hence the orientation of $U(K_i)$ induced by the orientation of $\Sigma$, which agrees with the Liouville orientation for $\tau < 0$, is given by $d\tau\, \eta\, (d\eta)^{n-1}$. In $\R_t \times U(K_i)$, define
\begin{equation} \label{eqn:contact form near max}
\alpha_{\Sigma} := f(\tau)\, dt + e^{-\tau^2}\, \eta
\end{equation}
where $f:[-1,1]_{\tau} \to \R$ is a function which is decreasing, odd, and equals $\pm 1$ when $\tau$ is close to $\mp 1$. The reader may verify that the contact condition is $f'(\tau) + 2\tau f(\tau) < 0$, which holds as $f'(\tau) < 0$ and $2\tau f(\tau) \leq 0$. 

Note that $\alpha$ is $t$-invariant in $\R_t \times U(K_i)$, and in particular $U(K_i)$ is convex with respect to $\partial_t$. This is in contrast to the behavior near an anti-dividing set, as considered in the final step. 

\s\n
\textsc{Step 3.} \emph{Construct the contact form near anti-dividing sets.}
\s

For $i$ even, assume after rescaling that $\alpha\mid_{U(K_i)} = e^{\tau^2}\, \eta$, where $\eta$ is a contact form on $K_i$. Note that 
\[
d(e^{\tau^2}\, \eta)^n = (2\tau e^{\tau^2}\, d\tau\, \eta + e^{\tau^2}\, d\eta)^n = 2n\tau e^{n\tau^2}\, d\tau\, \eta\, (d\eta)^{n-1}
\]
hence the orientation of $U(K_i)$ induced by the orientation of $\Sigma$, which agrees with the Liouville orientation for $\tau > 0$, is again given by $d\tau\, \eta\, (d\eta)^{n-1}$. Let $f:[-1,1]_{\tau}\to \R$ be the same function as before, and let $g:[-1,1]_{\tau}\to \R$ be a smooth nonpositive function to be determined later (according to the contact condition) which is $0$ near $\tau = \pm 1$; see \cref{fig:function graph near min}. 

\begin{figure}[ht]
	\begin{overpic}[scale=.4]{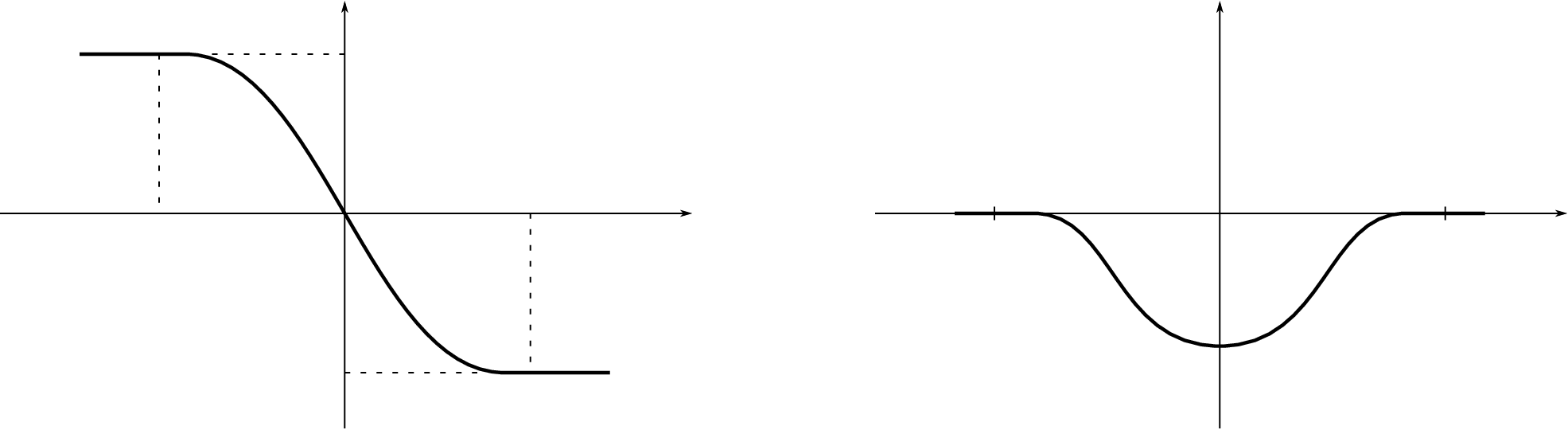}
		\put(7.5,11){$-1$}
		\put(31.7,11){$1$}
		\put(17.5,2.5){$-1$}
		\put(19.5,25){$1$}
		\put(61.5,11){$-1$}
		\put(91.5,11){$1$}
		\put(45,12){$\tau$}
		\put(101,12){$\tau$}
		\put(6,26){$f(\tau)$}
		\put(83,4.5){$g(\tau)$}
	\end{overpic}
	\caption{The graph of functions used in the contact form given by \eqref{eqn:contact form near min}.}
	\label{fig:function graph near min}
\end{figure}

On $\R_t\times U(K_i)$, define
\begin{equation} \label{eqn:contact form near min}
\alpha_{\Sigma} = -f(\tau)\, dt - tg(\tau)\, d\tau + e^{\tau^2} \, \eta.
\end{equation}
As preliminary observations, note that $\alpha_{\Sigma}$ is not $t$-invariant, that $dt$ is scaled by $-f(\tau)$ instead of $f(\tau)$ to correctly glue the contact form with that of \textsc{Step 1}, and that $\alpha_{\Sigma}\mid_{U(K_i)} = e^{\tau^2}\, \eta = \alpha\mid_{U(K_i)}$. This time, the reader may verify the contact condition $f'(\tau) - 2\tau f(\tau) - g(\tau) > 0.$ Given $f$, one can choose the function $g$ as above to satisfy this condition.

For later use, note that $\alpha_{\Sigma}$ restricts to the Liouville form $\beta_{i,t}=- tg(\tau) d\tau + e^{\tau^2} \eta$ on $\{t\} \times (U(K_i) \setminus K_i)$ for any $t \in \R$.  We compute the Liouville vector fields
\begin{equation*}
X_{\beta_{i,t}} = 1/(2\tau) (\partial_{\tau} + te^{-\tau^2} g(\tau) R_{\eta}),
\end{equation*}
where $R_{\eta}$ denotes the Reeb vector field on $(K_i, \eta)$. It follows that
\begin{equation} \label{eqn:char foliation near min}
U(K_i)_{t,\xi} := \left( \{t\} \times U(K_i) \right)_{\xi} = \partial_{\tau} + te^{-\tau^2} g(\tau) R_{\eta}.
\end{equation}

To illustrate the difference between a dividing set and an anti-dividing set, consider the case $2n+1 = 3$. On the folding locus $K_i$, choose a coordinate $\theta$ so that the induced contact form is $\eta = d\theta$. The folded Weinstein surface is locally $\Sigma = \{0\}_t \times [-1,1]_{\tau} \times S^1_{\theta}\subset \R_t \times [-1,1]_{\tau} \times S^1_{\theta}$, oriented as $d\tau \wedge d\theta$, and the characteristic foliation is directed by $\partial_{\tau}$.

If $K_i$ is a dividing set, the contact form in \eqref{eqn:contact form near max} is, to first order, 
\[
\alpha = f(\tau)\, dt + e^{-\tau^2}\, \eta \approx -\tau\, dt + d\theta.
\]
The ($t$-invariant) contact structure is drawn along a flow line of the characteristic foliation on the left of \cref{fig:fw-3d}. In contrast, if $K_i$ is an anti-dividing set, the contact form in \eqref{eqn:contact form near min} has the first order approximation
\[
\alpha = -f(\tau)\, dt - tg(\tau)\, d\tau + e^{\tau^2}\, \eta \approx \tau\, dt + a_0t\, d\tau + d\theta
\]
where $a_0 := |g(0)|$ is a constant. The contact condition $f'(\tau) - 2\tau f(\tau) - g(\tau) > 0$ in general implies $|g(\tau)|> 1 + 2\tau f(\tau)$, so assume that $a_0 = 1 + \ve$ for some $\ve > 0$; indeed, one may check that (to first order) $\alpha \wedge d\alpha = \ve\, dt\, d\tau\, d\theta$. Along a flow line of the characteristic foliation, the contact structure rotates in a right-handed manner relative to $\Sigma$; along a vertical Legendrian $\R_t \times \{0\}_{\tau}\times \{\theta_0\}_{\theta}$ sitting over the anti-dividing set, the contact structure rotates in a left-handed manner. These phenomena are drawn on the right side of \cref{fig:fw-3d}.

\begin{figure}[ht]
	\centering
    \vskip-0.5cm
	\begin{overpic}[scale=.43]{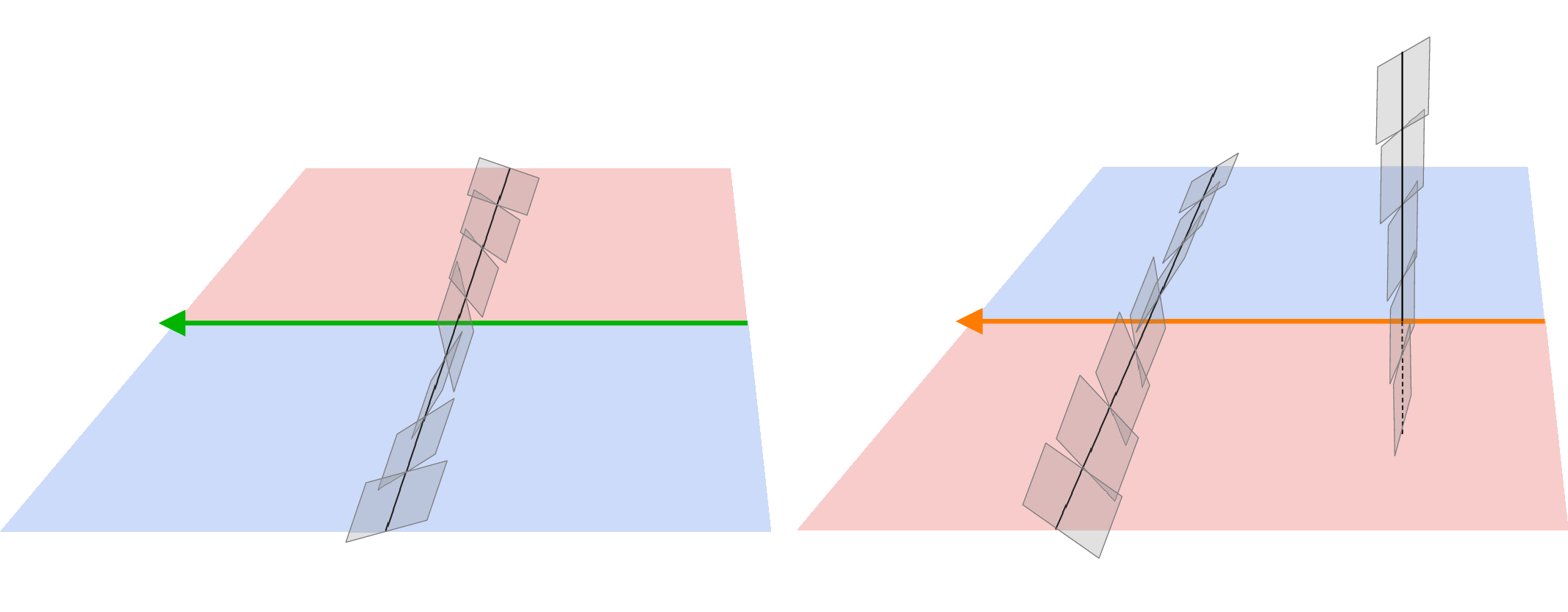}
       
	\end{overpic}
    \vskip-0.2cm
	\caption{The contact structure near a dividing set on the left and an anti-dividing set on the right. In both figures, $\tau$ points into the page, $\theta$ points from right to left, and $t$ points vertically. }
	\label{fig:fw-3d}
\end{figure}

\subsection{The Creation Lemma}\label{sec:creation-lemma}

It will be advantageous in \cref{sec:BBC_high} to systematically perturb the characteristic foliation of a folded Weinstein hypersurface by creating additional critical points. Rather than use the non-graphical mushrooms of \cref{sec:C-fold hd}, a simpler graphical model, the \emph{box fold}, is sufficient. The PL construction and smoothing of box folds was developed in detail in \cite{breen2025torus}; here we briefly describe the PL model in \cref{subsubsec:box-folds} and provide the precise statement from \cite{breen2025torus} needed for later application in \cref{subsubsec:statement-creation-lemma}.

\subsubsection{Box folds}\label{subsubsec:box-folds}

In dimension $3$, the construction of a box fold begins with the same initial model as in \cref{sec:C-fold 3d}: a flow box $[0,s_0]\times [0,\tilde{t}_0]$ of the characteristic foliation, thought of as the symplectization of the contact manifold $([0,\tilde{t}_0], \, d\tilde{t})$. In the setting of a folded Weinstein hypersurface, this should be thought of as a small flow box contained inside a positive subcobordism region, so that its neighborhood is a further contactization $([0,z_0]\times [0,s_0]\times [0,\tilde{t}_0], \, dz + e^s\, d\tilde{t})$.

\begin{figure}[ht]
	\begin{overpic}[scale=.5]{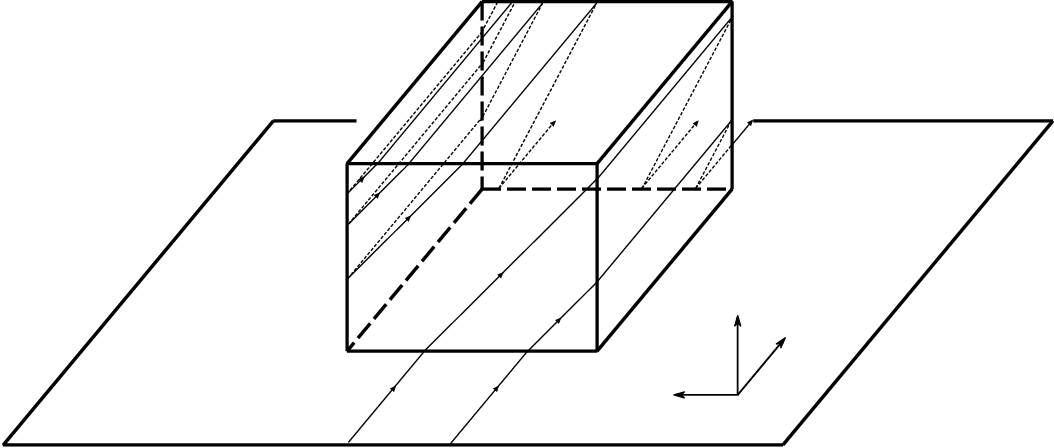}
		\put(61.3,4){$\tilde t$}
		\put(67,12){$z$}
		\put(75.5,10){$s$}
	\end{overpic}
\caption{The PL model of a box fold in dimension $3$.}
\label{fig:box bump}
\end{figure}

We then consider a PL box with vertical sides and a height $z_0$; see \cref{fig:box bump}. By performing the same characteristic foliation analysis as in \cref{sec:C-fold 3d} --- which is done carefully in \cite{breen2025torus} --- one can determine that the upper edge of the box fold at $\tilde{t} = \tilde{t}_0$ attracts the characteristic foliation in backward time, exactly as in a mushroom. Upon smoothing, this edge corresponds to a pair of positive index $0$ and index $1$ singular points of the characteristic foliation; see \cref{fig:box char foliation}. Consequent of the lack of a non-vertical side is the lack of negative singular points, and a box fold may informally be thought of as ``only the positive part'' of a mushroom.

\begin{figure}[ht]
	\begin{overpic}[scale=.5]{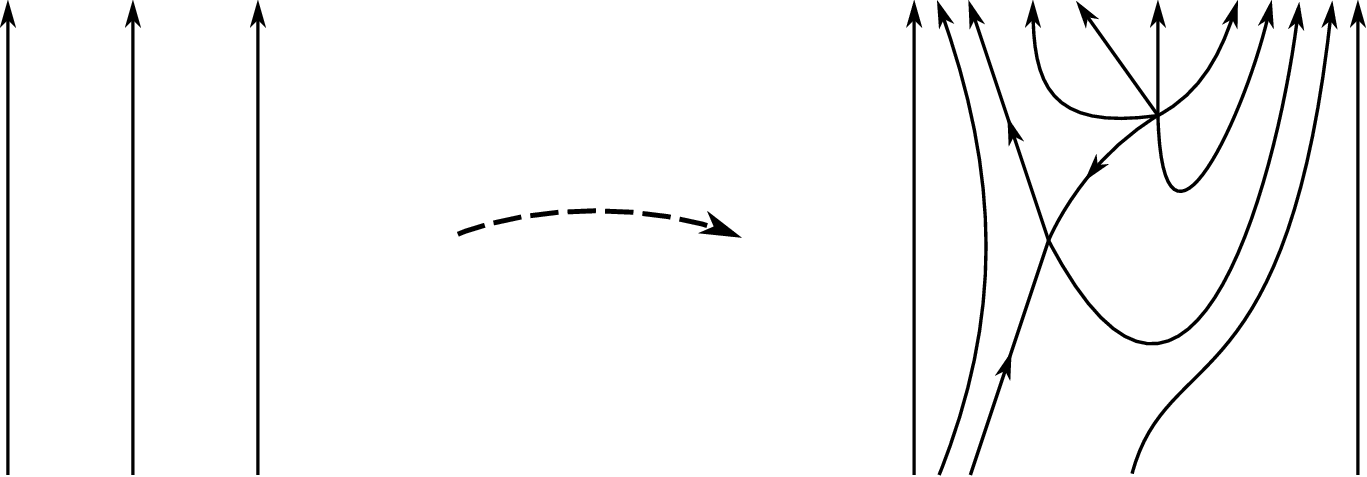}
		\put(85.6,25){\tiny{$e_+$}}
		\put(78,16){\tiny{$h_+$}}
	\end{overpic}
\caption{The characteristic foliation before and after installing a smoothed box fold; c.f.\ \cref{fig:Z char foliation}.}
\label{fig:box char foliation}
\end{figure}

The higher-dimensional model replaces the contact manifold $([0,\tilde{t}_0], \, d\tilde{t})$ by a contact handlebody $([0,\tilde{t}_0] \times W, \, d\tilde{t} + \lambda)$ based over a Weinstein domain $(W, \lambda)$. Upon box folding, each critical point of the underlying Weinstein structure gives birth to a pair of positive singular points in the characteristic foliation. The result is then stated precisely in the following:

\subsubsection{Statement of the Creation Lemma}\label{subsubsec:statement-creation-lemma}

Let $(W, \lambda)$ be a Weinstein domain of dimension $2n-2$ and set $V := [0,s_0]\times [0,\tilde{t}_0] \times W$, so that $(V, \, e^s\, (d\tilde{t} + \lambda))$ is compact symplectization of a contact handlebody. For a smooth function $F:V \to \R$, let $\lambda_{F} := dF + e^s\, (d\tilde{t} + \lambda)$ denote a Liouville form on $V$ and $X_{\lambda_F} = \partial_s + X_{F}$ its Liouville vector field. The following is a restatement of \cite[Theorem 4.1(1)]{breen2025torus}:

\begin{lemma}[Creation Lemma]\label{lem:creation-lemma}
For $V$ as above, there is a smooth function $F:V \to [0,\infty)$, compactly supported in the interior of $V$, such that 
\begin{enumerate}
    \item the Liouville vector field $X_{\lambda_{F}}$ is Morse, and 
    \item for each critical point of index $k$ in the underlying Weinstein domain $(W_, \lambda)$, $X_{\lambda_{F}}$ has critical points of index $k$ and index $k+1$.
\end{enumerate}
\end{lemma}

In contrast to previous instances of creation and elimination of critical points, in which two points along a single trajectory are created or eliminated, \cref{lem:creation-lemma} creates (and in reverse, eliminates) many critical points at once; c.f.\ \cite[Proposition 12.21]{CE12} for Weinstein homotopies, or \cite[Lemma 3.3]{Gi91}, \cite[\S 3.3]{Eli92} for characteristic foliations in dimension $3$.

\begin{remark} \label{rmk: parametric box fold}
    Compared to the creation of mushrooms, it is much more straightforward to apply the Creation Lemma (i.e., install box folds) parametrically given the graphical nature of the fold. In particular, if $\Sigma \times [0,1]_t$ is a neighborhood of a folded Weinstein hypersurface on which we wish to apply a box fold, then with $F$ denoting the function furnished by \cref{lem:creation-lemma}, the family of functions $t\, F$ induces a parametric installation from $t=0$ to $t=1$ which preserves the folded Weinstein decomposition at each time. 
\end{remark}

\section{Bypass-bifurcation correspondence in dimension \(3\)}

Informally, the \textit{bypass-bifurcation correspondence} says that the $1$-parametric failure of convexity due to the existence of a retrogradient on a hypersurface is witnessed by a bypass attachment. In this section we focus on dimension $2n+1=3$ (see \cref{prop:bb-correspondence3D} below for a precise statement). 

There is a different proof that we briefly outline, previously known to experts and with details recently recorded in Scharitzer's master’s thesis \cite{scharitzer2022thesis}. A bypass attachment is a local operation, modifying a convex surface only in a neighborhood of the attaching arc of the bypass half-disk. Letting $B\subset \Sigma_0$ denote a tubular neighborhood of the assumed retrogradient, the question is then reduced to understanding the contact structure on $B\times [-1,1]$.

At this point, there are two $3$-dimensional ``miracles." The first is that one can take $\partial B$ to be Legendrian using the \emph{Legendrian realization principle} (see \cite[Theorem 3.7]{Hon00}). This gives good control over the contact structure near $\partial B\times [-1,1]$. The second and more significant miracle is Eliashberg's theorem (see \cite[Theorem 2.1.3]{Eli92}) on the uniqueness of tight contact structures on the $3$-ball. Using these two facts, one can prove \cref{prop:bb-correspondence3D} in dimension $3$ by arguing that both the bifurcation and the bypass attachment produce tight contact structures on the $3$-ball $B\times [-1,1]$ with the same boundary conditions, hence must coincide.

Unfortunately, both of the above-mentioned miracles fail in dimension $>3$: the first for dimensional reasons and the second by results of \cite{Eli91,Ust99}. Nevertheless, our proof of the higher-dimensional bypass-bifurcation correspondence will follow the same general outline as in dimension $3$ by replacing the Legendrian boundary condition on $\partial B$ with a transverse boundary condition and Eliashberg's theorem with a direct proof  (using ideas that are present in the proof of Eliashberg's theorem) that both the bifurcation and the (trivial) bypass attachment produce the standard ball in a Darboux chart. 

The resulting proof is fairly involved. In this section we prove the $3$-dimensional statement, using the same strategy employed in higher dimensions for clarity and also for the independent interest of readers with a low-dimensional preference. We will prove the correspondence in all dimensions in \cref{sec:BBC_high}.

\begin{proposition}\label{prop:bb-correspondence3D}
Let $(M,\xi)$ be a contact $3$-manifold and let $(M_\Sigma,\alpha_\Sigma):=(\Sigma\times[-\delta,\delta],\alpha_\Sigma)$ be a standard neighborhood of a folded Weinstein surface
\begin{equation}\label{eq:folded-decomposition3D}
\Sigma_0 = W_1\cup_{K_1} W_2\cup_{K_2} W_3\cup_{K_3} W_4
\end{equation}
in $(M,\xi)$, where:
\begin{enumerate}[label={(\arabic*)}]
    \item $W_2$ and $W_3$ are Weinstein cobordisms associated with index 1 critical points $q_-$ and $q_+$, respectively; and\label{part:bb-3d-cobordisms}
    \item the closed intervals $\mathrm{stab}_{W_2}(q_-)$ and $\mathrm{stab}_{W_3}(q_+)$ meet at a single point in $K_2$. \label{part:bb-3d-intersection}
\end{enumerate}
Then $(M_\Sigma,\alpha_\Sigma)$ is contactomorphic, relative boundary, to the bypass cobordism associated to some bypass data in $\Sigma\times \{-\delta\}$.
\end{proposition}

We first introduce notation, with respect to which we can identify the specific bypass attachment data in $\Sigma\times \{-\delta\}$:
\begin{itemize}
    \item For each $t\in[-\delta,\delta]$, we denote by $\Sigma_t$ the surface $\Sigma\times\{t\}$.  As seen in \cref{sec:folded-weinstein}, $\Sigma_t$ admits a folded Weinstein structure compatible with \eqref{eq:folded-decomposition3D} for all $t\in [-\delta, \delta]$, not just $t=0$.  When the slice $\Sigma_t$ under consideration is clear, we will refer to the Weinstein cobordisms $W_j$, $1\leq j\leq 4$, without specifying the value of $t$.

    \item On $\Sigma_t$, the Weinstein-stable manifold of the unique critical point of $W_2$ is denoted $D_{-,t}^\dagger$, while the analogous stable manifold in $W_3$ is denoted $D_{+,t}^\dagger$.  The corresponding Weinstein-unstable manifolds are denoted $D_{-,t}$ and $D_{+,t}$.

    \item The disk $(D_{+,-\delta})^\epsilon\subset \Sigma_{-\delta}$, obtained as a small positive Reeb pushoff of $D_{+,-\delta}$, is denoted $D_+$, while $(D_{-,-\delta})^{-\epsilon}\subset \Sigma_{-\delta}$ is denoted $D_-$, and we write $\Lambda_{\pm}$ for $\partial D_{\pm}$. 
\end{itemize}
There is a unique intersection point between the Weinstein-stable manifolds $D_{-,0}^\dagger$ and $D_{+,0}^\dagger$ by \hyperref[part:bb-3d-intersection]{\cref*{prop:bb-correspondence3D}\ref*{part:bb-3d-intersection}} and $(\Lambda_{\pm};D_{\pm})$ provides the data along which the bypass is attached to $\Sigma_{-\delta}$.

\begin{figure}[ht]
    \centering
    \begin{overpic}[width=0.8\textwidth]{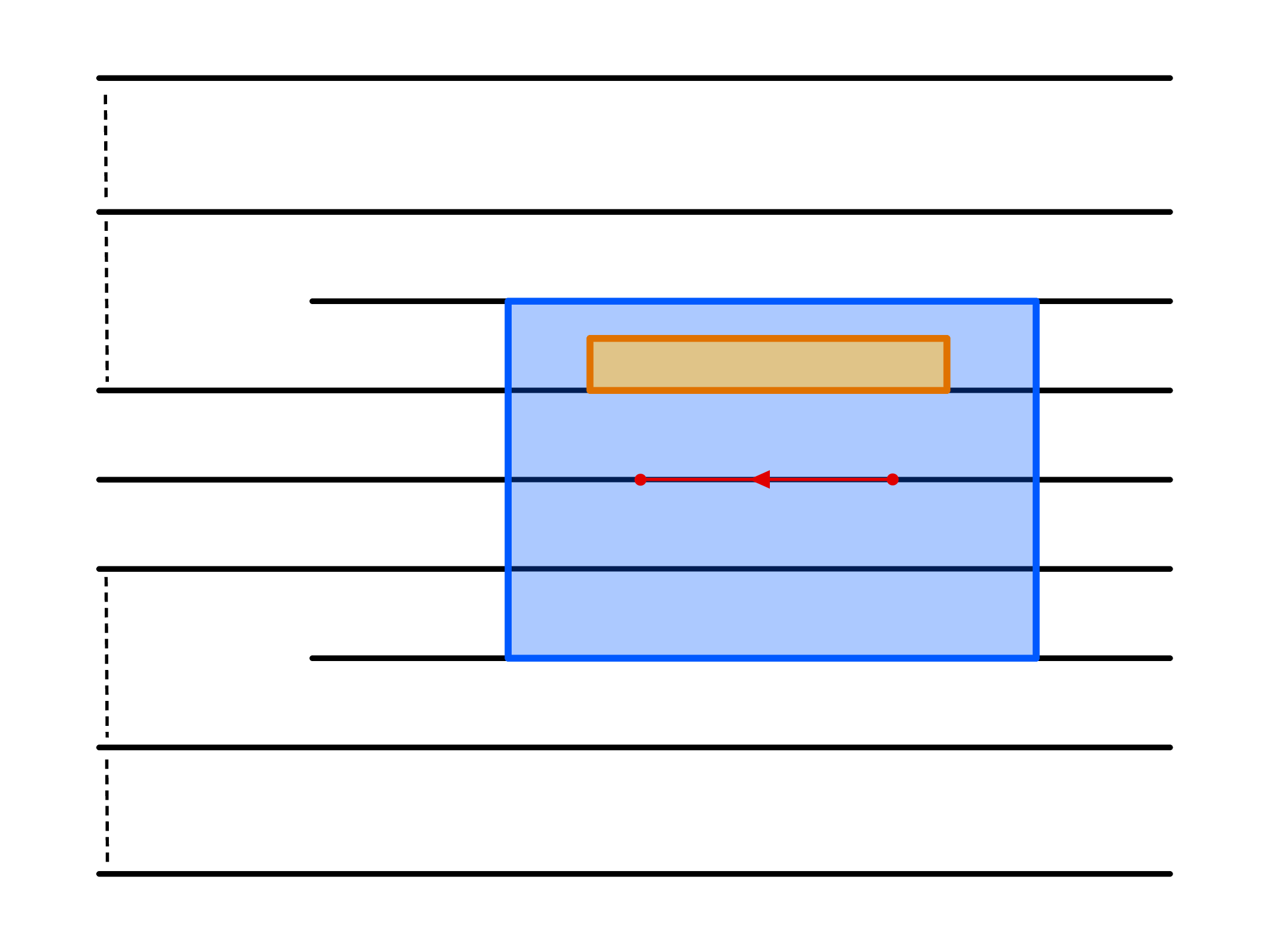}
        \put(94,5){\footnotesize $\Sigma_{-\delta}$}
        \put(94,15){\footnotesize $\Sigma_{-\delta/2}$}
        \put(94,22.25){\footnotesize $\Sigma_{-t_0}$}
        \put(94,29.5){\footnotesize $\Sigma_{-t_0/2}$}
        \put(94,36.5){\footnotesize $\Sigma_{0}$}
        \put(94,43.75){\footnotesize $\Sigma_{t_0/2}$}
        \put(94,50.75){\footnotesize $\Sigma_{t_0}$}
        \put(94,58){\footnotesize $\Sigma_{\delta/2}$}
        \put(94,68){\footnotesize $\Sigma_{\delta}$}

        \put(43,25.75){\textcolor{darkblue}{$T$}}
        \put(58,19.5){\footnotesize \textcolor{darkblue}{$\widehat{B}_{-t_0}$}}
        \put(58,53){\footnotesize \textcolor{darkblue}{$\widehat{B}_{t_0}$}}
        \put(59,45.5){\footnotesize \textcolor{Orange}{$G$}}
        \put(45,38.5){\footnotesize \textcolor{darkred}{appearance of retrogradient}}
        \put(49.5,35){\footnotesize \textcolor{darkred}{$q_+$}}
        \put(69.5,35){\footnotesize \textcolor{darkred}{$q_-$}}

        \put(10, 10){\footnotesize box fold installation}
        \put(10, 63){\footnotesize box fold removal}

        \put(10, 22.5){\footnotesize $t$-invariant}
        \put(10, 50.75){\footnotesize $t$-invariant}

    \end{overpic}
    \caption{A schematic for our proof of the bypass-bifurcation correspondence.}
    \label{fig:schematic-3d}
\end{figure}

Our argument roughly proceeds as follows (see \cref{fig:schematic-3d}):
\begin{enumerate}
    \item We identify a ball $\widehat{B}\subset\Sigma_0$ containing the retrogradient in \cref{prop:bb-correspondence3D}, such that $(\Sigma_0)_\xi$ points transversely out of $\partial \widehat{B}$, $\xi$ is $t$-invariant on 
    $$(\Sigma\times[-\delta/2,\delta/2])\setminus (\widehat{B}\times[-t_0,t_0])$$ 
    for some $t_0<\delta/2$,  and $\Sigma_t$ is convex for $t_0\leq |t|\leq \delta$.
    This requires the use of ``buffer regions" near $\Sigma\times\{\pm\delta\}$, on which we modify the characteristic foliation of $\Sigma_t$. (In  \cref{fig:schematic-3d} these regions are labeled ``box fold installation" and ``box fold removal".)
    \item We decompose $\widehat{B}\times[-t_0,t_0]$ into a Darboux ball $G$ on which $\xi$ is $t$-invariant and its complement $T=(\widehat{B}\times[-t_0,t_0])\setminus G$, which is foliated by $2$-spheres $S_\mu$, $\mu\in [0,1]$, such that $S_{\mu=0}$ is the ``outer boundary'' of $T$ and each $R_+(S_\mu)$, $R_-(S_\mu)$, $\mu\in[0,1]$, is deformation equivalent to the standard Weinstein ball with a unique critical point (\cref{lemma:standard-sphere-3d} and \cref{lemma:family-of-spheres3D}).
    \item We show that $T$ is contactomorphic to the cobordism obtained by attaching the bypass to $S_{\mu=0}$ along $(\Lambda_{\pm};D_{\pm})$ (which has been moved up from $\Sigma_{-\delta}$ to $\Sigma_{-t_0/2}$).
\end{enumerate}

\s\n
\textbf{Construction of the ball $\mathbf{\widehat{B}\subset\Sigma_0}$.}
\s

The ball $\widehat{B}\subset\Sigma_0$ we build will have a boundary which is transverse to $(\Sigma_0)_\xi$.  As a first approximation to $\widehat{B}$, let us consider a small neighborhood $B$ of the union of $D^\dagger_{-,0}=\mathrm{stab}_{W_2}(q_-)$ and $D^\dagger_{+,0}=\mathrm{stab}_{W_3}(q_+)$. By \hyperref[part:bb-3d-intersection]{\cref*{prop:bb-correspondence3D}\ref*{part:bb-3d-intersection}}, $D^\dagger_{-,0}\cap D^\dagger_{+,0}$ is a single point and $D^\dagger_{-,0}\cup D^\dagger_{+,0}$ is a closed arc with endpoints lying in $K_2$ and intersecting $K_2$ in exactly one other, interior point.  
More explicitly, we construct $B$ as follows: Letting $\Lambda^\dagger_\pm:=\partial D^\dagger_{\pm,0}$, the union $\Lambda^\dagger_-\cup\Lambda^\dagger_+$ consists of three points in the contact $1$-manifold $K_2$, and we denote by $C_2\subset K_2$ a standard contact neighborhood of $\Lambda^\dagger_-\cup\Lambda^\dagger_+$, as depicted in \cref{fig:folded-retrogradient-3d}.  
\begin{figure}[ht]
	\centering
    \begin{subfigure}[t]{0.48\textwidth}
    \centering
    \begin{overpic}[width=0.8\textwidth]{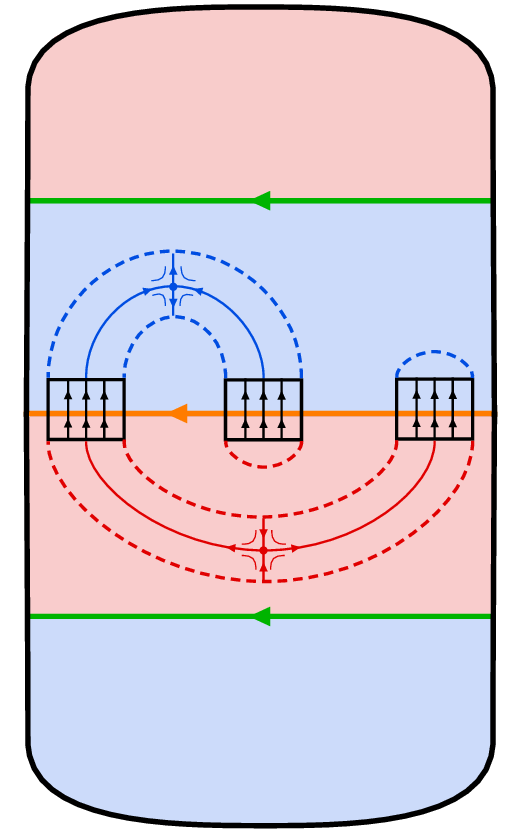}
        \put(-4,15){$W_1$}
        \put(-4,35){$W_2$}
        \put(-3.5,49){\small \textcolor{Orange}{$K_2$}}

        \put(-3.5,25){\small \textcolor{divgreen}{$K_1$}}
        \put(-3.5,75){\small \textcolor{divgreen}{$K_3$}}
        \put(-4,62){$W_3$}
        \put(-4,84){$W_4$}
        \put(33.25,32){\tiny \textcolor{darkred}{$q_-$}}
        \put(23,67.35){\tiny \textcolor{darkblue}{$q_+$}}
    \end{overpic}
	\caption{The folded Weinstein structure on $\Sigma_0$ hypothesized by \cref{prop:bb-correspondence3D}. In this schematic, $C_1$ is dashed in red, $C_2$ is the intersection of $K_2$ with the boxes, $C_3$ is dashed in blue, and $C_h$ is the union of the vertical sides of the boxes.}
	\label{fig:folded-with-cps-3d}
    \end{subfigure}\hfill
    \begin{subfigure}[t]{0.48\textwidth}
    \centering
	\begin{overpic}[width=\textwidth]{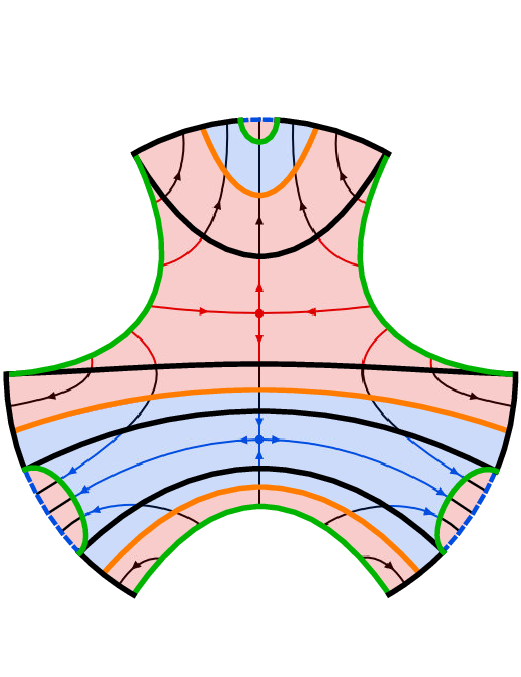}
        \put(33.25,32.5){\tiny \textcolor{darkblue}{$q_+$}}
        \put(33.75,51.5){\tiny \textcolor{darkred}{$q_-$}}
        \put(2,20.5){\textcolor{darkblue}{$C_3$}}
        \put(-1.5,28){\small \textcolor{divgreen}{$K_3$}}
        \put(11,11){$C_h$}
        \put(36,20){\textcolor{divgreen}{$C_1$}}

        \put(57,56){\textcolor{divgreen}{$C_1$}}
        \put(16,56){\textcolor{divgreen}{$C_1$}}
        \put(74.5,28){\small \textcolor{divgreen}{$K_3$}}
        \put(71,20.5){\textcolor{darkblue}{$C_3$}}
        \put(32,85){\textcolor{divgreen}{$K_3$}}
        \put(39,85){\textcolor{darkblue}{$C_3$}}
    \end{overpic}
	\caption{A schematic of the neighborhood $B$ after additional modifications (tilting and flowing), depicting the characteristic foliation $B_\xi$.}
	\label{fig:b-3d}
    \end{subfigure}
    \caption{Constructing a neighborhood in $\Sigma_0$ of the stable/unstable disks along which the retrogradient occurs.}
    \label{fig:folded-retrogradient-3d}
\end{figure}
The neighborhood $B$ of $D^\dagger_{-,0}\cup D^\dagger_{+,0}$ may be chosen so that its intersection with a neighborhood of $K_2\subset\Sigma_0$ has the form $[-1,1]_\tau\times C_2$, where $\tau$ is the coordinate in the hypothesized folded Weinstein standard neighborhood. Consequently, $\alpha_\Sigma$ restricts to $[-1,1]_\tau\times C_2$ as $e^{\tau^2}\eta$, where $\eta$ is the induced contact form on $K_2$, the characteristic foliation is directed by $\partial_{\tau}$, and
\[
([-1,0]_\tau\times C_2,e^{\tau^2}\eta)
\quad\text{and}\quad
([0,1]_\tau\times C_2,e^{\tau^2}\eta)
\]
are Weinstein cobordisms (with corners), the former corresponding to the intersection of $[-1,1]_\tau\times C_2$ with $W_2$ and the latter to the intersection with $W_3$.  We obtain the ball $B$ from $[-1,1]_\tau\times C_2$ by attaching:
\be
\item cornered Weinstein cobordisms to $\{-1\}\times\Lambda^\dagger_-$ consisting of a Weinstein $1$-handle associated to $q_-$ and a small thickening of the remaining component of $\{-1\}\times C_2$; and 
\item cornered Weinstein cobordisms to  $\{1\}\times\Lambda^\dagger_+$ consisting of a Weinstein $1$-handle associated to $q_+$ and a small thickening of the remaining component of $\{1\}\times C_2$;
\ee 
see \cref{fig:folded-with-cps-3d}.

Now the boundary of $B$ naturally decomposes as
\[
\partial B = C_1 \cup C_h \cup C_3,
\]
where:
\begin{itemize}
    \item $C_h=[-1,1]\times\partial C_2$;
    \item $C_1$ is the portion of $(\bdry B)\setminus C_h$ in $W_2$; and
    \item $C_3$ is the portion of $(\bdry B)\setminus C_h$ in $W_3$.
\end{itemize}

By flowing $C_1$ backward along $(\Sigma_0)_\xi$, we may assume that $C_1$ is contained in $K_1$. Likewise, by flowing $C_3$ forward along $(\Sigma_0)_\xi$, we may assume that $C_3\subset \overline{W}_4$ and transversely intersects $K_3$ along $\bdry C_3$. Notice that $(\Sigma_0)_\xi$:
\begin{itemize}
    \item points into $B$ along $C_1$;
    \item is parallel to $\partial B$ along $C_h$; and
    \item points out of $B$ along $C_3$.
\end{itemize}
By tilting $\partial B$ along $C_h$, we may assume that $(\Sigma_0)_\xi$ points out of $B$ along $C_h\cup C_3$.  The resulting Weinstein cobordism is depicted schematically in \cref{fig:b-3d}, with $\partial_{\mathrm{in}}B=C_1$ and $\partial_{\mathrm{out}}B=C_h\cup C_3$.

\begin{figure}[ht]
	\centering
    \begin{subfigure}[t]{0.48\textwidth}
    \centering
    \begin{overpic}[width=\textwidth]{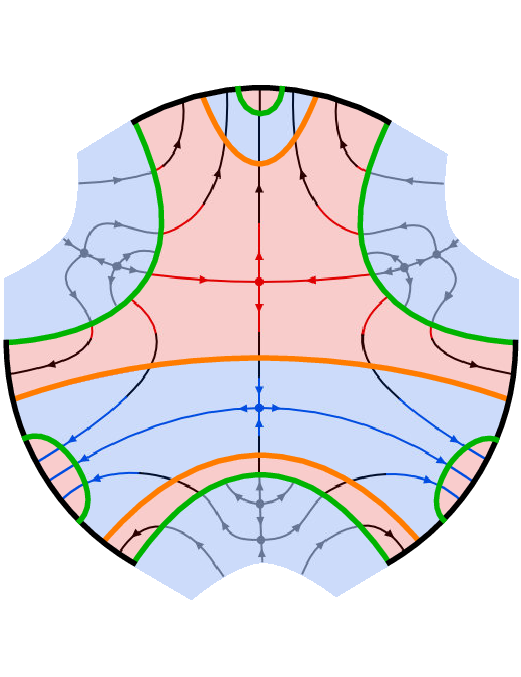}
        \put(34.25,37.25){\tiny \textcolor{darkblue}{$q_+$}}
        \put(34.5,56.5){\tiny \textcolor{darkred}{$q_-$}}

        \put(39.7,23.5){\tiny \textcolor{gray}{$e_+$}}
        \put(16,62.5){\tiny \textcolor{gray}{$e$}}
        \put(59,62.5){\tiny \textcolor{gray}{$e_-$}}

    \end{overpic}
	\caption{The ball with the additional collar neighborhood $[-s_0,0]\times \partial_{\mathrm{in}} B$, in which the Creation Lemma has been applied.}
	\label{fig:nontransverse}
    \end{subfigure}\hfill
    \begin{subfigure}[t]{0.48\textwidth}
    \centering
	\begin{overpic}[width=\textwidth]{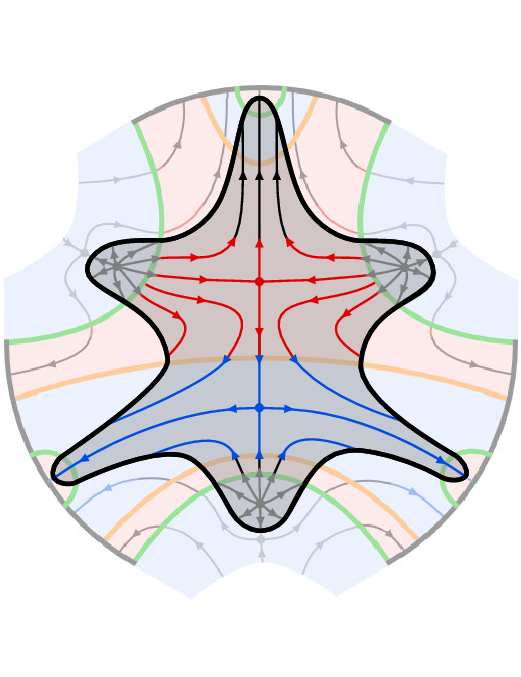}
        \put(34.25,37.25){\tiny \textcolor{darkblue}{$q_+$}}
        \put(34.5,56.5){\tiny \textcolor{darkred}{$q_-$}}
    \end{overpic}
	\caption{The ball $\widehat{B}$ with boundary transverse to the characteristic foliation.}
	\label{fig:transverse}
    \end{subfigure}
    \caption{}
    \label{fig:non-transverse-ball-3d}
\end{figure}

Next we apply a perturbation of $\Sigma_0$ which is supported on $W_1$ and reverses the characteristic foliation along $\partial_{\mathrm{in}}B$. Recall that $\partial_{\mathrm{in}}B$ consists of three disjoint intervals in $K_1$ transverse to $(\Sigma_0)_\xi$, and there exists a (one-sided) neighborhood $N(\partial_{\mathrm{in}}B)=[-s_0,0]\times\partial_{\mathrm{in}}B$ of $\partial_{\mathrm{in}}B$ in $W_1$ to which the contact form restricts as $e^s\eta$.  Applying the Creation Lemma (\cref{lem:creation-lemma}) performs a $C^0$-small isotopy of $\Sigma_0$, supported in $N(\partial_{\mathrm{in}}B)$, such that the resulting characteristic foliation includes one elliptic and one hyperbolic critical point, connected by a flow line of $(\Sigma_{0})_\xi$, for each of the three components of $\partial_{\mathrm{in}}B$.  We denote the resulting elliptic critical points by $e$, $e_-$, and $e_+$. Specifically, referring to \cref{fig:nontransverse},  the bottom elliptic point is $e_+$, and the top left (resp.\ right) elliptic point is $e$ (resp.\ $e_-$). The result of the perturbation will be denoted $\Sigma_0$ from now on.

\begin{remark}
All three of the critical points $e$, $e_-$, and $e_+$ are positive, but the notation foreshadows the roles that these points will play below: flow lines connecting (analogs of) $e_\pm$ to (analogs of) $q_\pm$ will allow us to apply the Elimination Lemma to these pairs, leaving the index 0 critical point $e$.  In higher dimensions, we will consider a pair of index $n-1$ critical points $e_\pm$ along with an index 0 critical point $e$.
\end{remark}

Finally, $\widehat{B}$ is obtained from $B$ by adjoining the Weinstein-unstable manifolds of the newly created elliptic points $\{e,e_-,e_+\}$ and slightly retracting so that $(\Sigma_0)_\xi$ is transverse to $\bdry \widehat{B}.$


\s\n
\textbf{Decomposing the region $\mathbf{\widehat{B}\times[-t_0,t_0]}$.}
\s

Since the box folds used in the Creation Lemma are graphical, the construction of $\widehat{B}$ can be carried out $t$-parametrically. Hence there exists a $C^0$-small isotopy of $\Sigma\times[-\delta,\delta]$ which is supported on $W_1\times(-\delta,\delta)$,\footnote{In particular this means $(\Sigma_{\delta})_\xi$ and $(\Sigma_{-\delta})_\xi$ are unperturbed.} such that, after the isotopy: 
\be
\item[(B1)] the resulting characteristic foliation $(W_1)_\xi$ on $W_1$ is $t$-invariant for $t\in[-\delta/2,\delta/2]$;
\item[(B2)] $W_1$ is Weinstein for all $t\in[-\delta,\delta]$;
\item[(B3)] for each $t\in[-\delta/2,\delta/2]$, $(W_1)_\xi$ includes elliptic critical points $e^t, e_-^t$, and $e_+^t$, as well as a canceling set of three hyperbolic critical points; 
\item[(B4)] $(\Sigma_t\setminus\widehat{B})_\xi$ is $t$-invariant for $-\delta/2\leq t\leq \delta/2$; and
\item[(B5)] $(\Sigma_t)_\xi$ is $t$-invariant for $t_0/2\leq|t|\leq\delta/2$, where $0<t_0< \delta/2$ is small.
\ee
Indeed, in view of \cref{rmk: parametric box fold}, we can use the interval $-\delta\leq t\leq -\delta/2$ to ``install the box fold," interpolating from the original characteristic foliation on $W_1$ to the foliation resulting from the Creation Lemma, and then use the interval $\delta/2\leq t \leq \delta$ to ``uninstall the box fold."
See the schematic in \cref{fig:schematic-3d}.

Since $(\Sigma_t)_\xi$ is not perturbed outside of $W_1$ during the Creation Lemma isotopy, it continues to be the case that:
\be
\item[(B6)] $\Sigma_t$ is convex for all $t\neq 0$.
\ee  
Indeed, $(\Sigma_t)_\xi$ contains critical points $q_+^t$ and $q_-^t$ projecting to $q_+$ and $q_-$ in $\widehat{B}_0$, respectively, and 
$\op{stab}_{W_3}(q_+^t)\cap \op{stab}_{W_2}(q_-^t)\not=\emptyset$ precisely when $t=0$; see \cref{fig:b-hat-foliation-3d}.
\begin{figure}[ht]
	\centering
    \begin{subfigure}{0.33\textwidth}
    \centering
    \begin{overpic}[width=\textwidth]{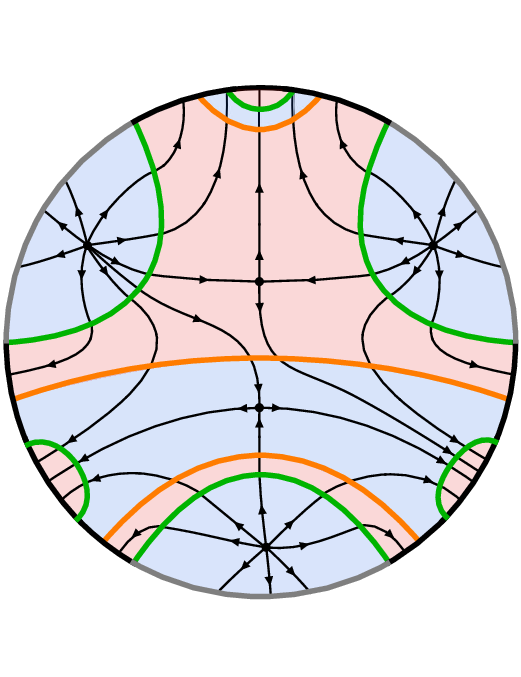}
        \put(7,18.5){\footnotesize \textcolor{divgreen}{$K_3$}}
        \put(55,11.5){\footnotesize \textcolor{divgreen}{$K_1$}}
        \put(15.5,48){\footnotesize \textcolor{Orange}{$C_2$}}
        \put(39.7,25){\tiny $e_+^t$}
        \put(12,72.5){\tiny $e^t$}
        \put(60,72.5){\tiny $e_-^t$}
        \put(33,35.5){\tiny $q_+^t$}
        \put(33,54){\tiny $q_-^t$}
    \end{overpic}
	\caption{$t<0$}
	\label{fig:b-hat-neg-3d}
    \end{subfigure}\hfill
    \begin{subfigure}{0.33\textwidth}
    \centering
	\begin{overpic}[width=\textwidth]{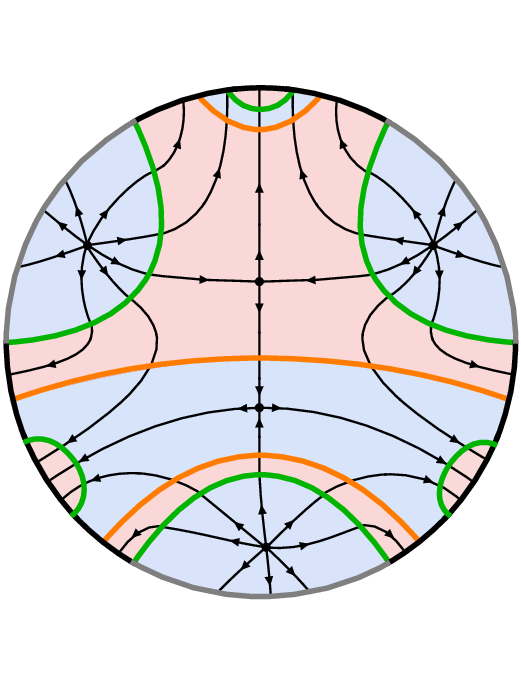}
        \put(7,18.5){\footnotesize \textcolor{divgreen}{$K_3$}}
        \put(55,11.5){\footnotesize \textcolor{divgreen}{$K_1$}}
        \put(15.5,48){\footnotesize \textcolor{Orange}{$C_2$}}
        \put(39.7,25){\tiny $e_+^t$}
        \put(12,72.5){\tiny $e^t$}
        \put(60,72.5){\tiny $e_-^t$}
        \put(33,35.5){\tiny $q_+^t$}
        \put(33,54){\tiny $q_-^t$}
    \end{overpic}
	\caption{$t=0$}
	\label{fig:b-hat-0-3d}
    \end{subfigure}
    \hfill
    \begin{subfigure}{0.33\textwidth}
    \centering
	\begin{overpic}[width=\textwidth]{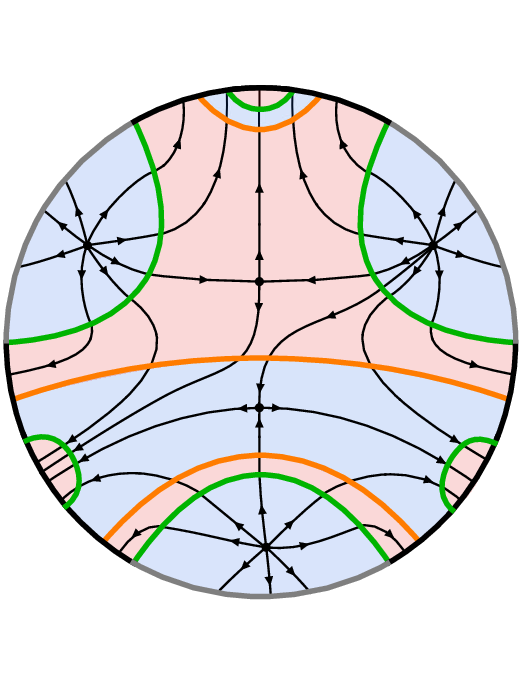}
        \put(7,18.5){\footnotesize \textcolor{divgreen}{$K_3$}}
        \put(55,11.5){\footnotesize \textcolor{divgreen}{$K_1$}}
        \put(15.5,48){\footnotesize \textcolor{Orange}{$C_2$}}
       \put(39.7,25){\tiny $e_+^t$}
        \put(12,72.5){\tiny $e^t$}
        \put(60,72.5){\tiny $e_-^t$}
        \put(33,35.5){\tiny $q_+^t$}
        \put(33,54){\tiny $q_-^t$}
    \end{overpic}
	\caption{$t>0$}
	\label{fig:b-hat-pos-3d}
    \end{subfigure}
    \caption{The characteristic foliation on $\widehat{B}_t$ for various values of $t$.} 
    \label{fig:b-hat-foliation-3d}
\end{figure}

We now restrict attention to $\Sigma\times I_{t_0}$, where $I_{t_0}$ denotes the interval $[-t_0,t_0]$. On this neighborhood of $\Sigma_0$, the characteristic foliation $(W_1)_\xi$ is $t$-invariant by (B1). 
For all $t\in I_{t_0/2}$ we may assume that
\[
\op{stab}_{W_3}(q_+^t)\cap C_2 = \bdry D^\dagger_{+,t} = (\Lambda^\dagger_+)^{t}
\quad\text{and}\quad
\op{stab}_{W_2}(q_-^t)\cap C_2 = \bdry D^\dagger_{-,t} = (\Lambda^\dagger_-)^{-t},
\]
where $(\Lambda^\dagger_+)^{t}$ is the $t$-Reeb pushoff of $\Lambda^\dagger_+$ and $(\Lambda^\dagger_-)^{-t}$ is the $(-t)$-Reeb pushoff of $\Lambda^\dagger_-$.

In preparation for decomposing $\widehat{B}\times I_{t_0}$, we want to show that $S:=-\bdry (\widehat{B}\times I_{t_0})$, after rounding, is a convex surface whose positive and negative regions $R_+(S)$ and $R_-(S)$ are standard.  Here we orient $S$ so that the orientations of $S$ and $\Sigma_{-t_0}$ agree on their overlap.

\begin{lemma}\label{lemma:standard-sphere-3d}
The surfaces $R_+(S)$ and $R_-(S)$ are each deformation equivalent to the standard Weinstein ball with a unique critical point.
\end{lemma}

Let us write $\widehat{B}_t:=\widehat{B}\times\{t\}$, where $t\in I_{t_0}$. 

\begin{proof}
We first identify the critical points of $S_\xi$.  The characteristic foliation of the piecewise-smooth surface
\[
-\partial(\widehat{B}\times I_{t_0}) = -\widehat{B}_{t_0} \cup ((\partial\widehat{B})\times I_{t_0}) \cup \widehat{B}_{-t_0}
\]
agrees with that of $\widehat{B}_{-t_0}$, and differs from that of $\widehat{B}_{t_0}$ by an orientation change.  Because the characteristic foliation of $\widehat{B}$ is transverse to $\partial \widehat{B}$, the characteristic foliation of the vertical boundary $(\partial\widehat{B})\times I_{t_0}$ is nonsingular. From top to bottom, the holonomy $\partial \widehat{B}_{t_0} \cong S^1 \to S^1 \cong \widehat{B}_{-t_0}$ of the vertical foliation is constant over the intersection points with the dividing and anti-dividing sets and clockwise along $\bdry \widehat{B}\cap W_4$.
This (bottom-to-top) holonomy is depicted in the green highlighted segments of \cref{fig:family-of-spheres-3d}. By rounding in a sufficiently small neighborhood of the corners, we may assume that the critical points of $S_\xi$ are contained in the disks $\widehat{B}_{\pm t_0}$.  Then $R_+(S)$ contains five critical points:
\begin{itemize}
    \item hyperbolic critical points $q_+^{-t_0}\in\widehat{B}_{-t_0}$ and $q_-^{t_0}\in -\widehat{B}_{t_0}$;
    \item elliptic critical points $e^{-t_0},e_-^{-t_0},e_+^{-t_0}\in\widehat{B}_{-t_0}$.
\end{itemize}
Analogously, the critical points of $R_-(S)$ are:
\begin{itemize}
    \item hyperbolic critical points $q_-^{-t_0}\in\widehat{B}_{-t_0}$ and $q_+^{t_0}\in -\widehat{B}_{t_0}$;
    \item elliptic critical points $e^{t_0},e_-^{t_0},e_+^{t_0}\in -\widehat{B}_{t_0}$.
\end{itemize}
Note that the signs of these critical points are determined using the orientation of $S$.

Now the Weinstein homotopy type of $R_+(S)$ is determined by a neighborhood of the union of the closures of the characteristic-stable manifolds of the positive critical points of $S_\xi$.  The points $e^{-t_0},e_-^{-t_0}$, and $e_+^{-t_0}$ have index $0$, and thus comprise their own characteristic-stable manifolds. Next, $\op{stab}_S(q_+^{-t_0})$ consists of $q_+^{-t_0}$ along with two flow lines: one which converges in backward time to $e_+^{-t_0}$ and another which converges in backward time to $e^{-t_0}$.  Finally, $\op{stab}_S(q_-^{t_0})$ is made up of $q_-^{t_0}$ along with a pair of flow lines converging in backward time to $e_-^{-t_0}$ and $e^{-t_0}$, respectively.  See the blue highlighted flow lines in the leftmost column of \cref{fig:family-of-spheres-3d}.  By applying the Elimination Lemma to the pairs $\{q_+^{-t_0},e_+^{-t_0}\}$ and $\{q_-^{t_0},e_-^{-t_0}\}$, we see that $R_+(S)$ is deformation equivalent to a neighborhood of the index 0 critical point $e^{-t_0}$, which is a standard Weinstein ball.  

Similarly, the pairs $\{q_+^{t_0},e_+^{t_0}\}$ and $\{q_-^{t_0},e_-^{t_0}\}$ in $R_-(S)$ cancel via the Elimination Lemma, rendering $R_-(S)$ deformation equivalent to a standard ball; see the red highlighted flow lines in the leftmost column of \cref{fig:family-of-spheres-3d}.
\end{proof}

\begin{figure}[ht]
	\centering
    \begin{overpic}[width=\textwidth]{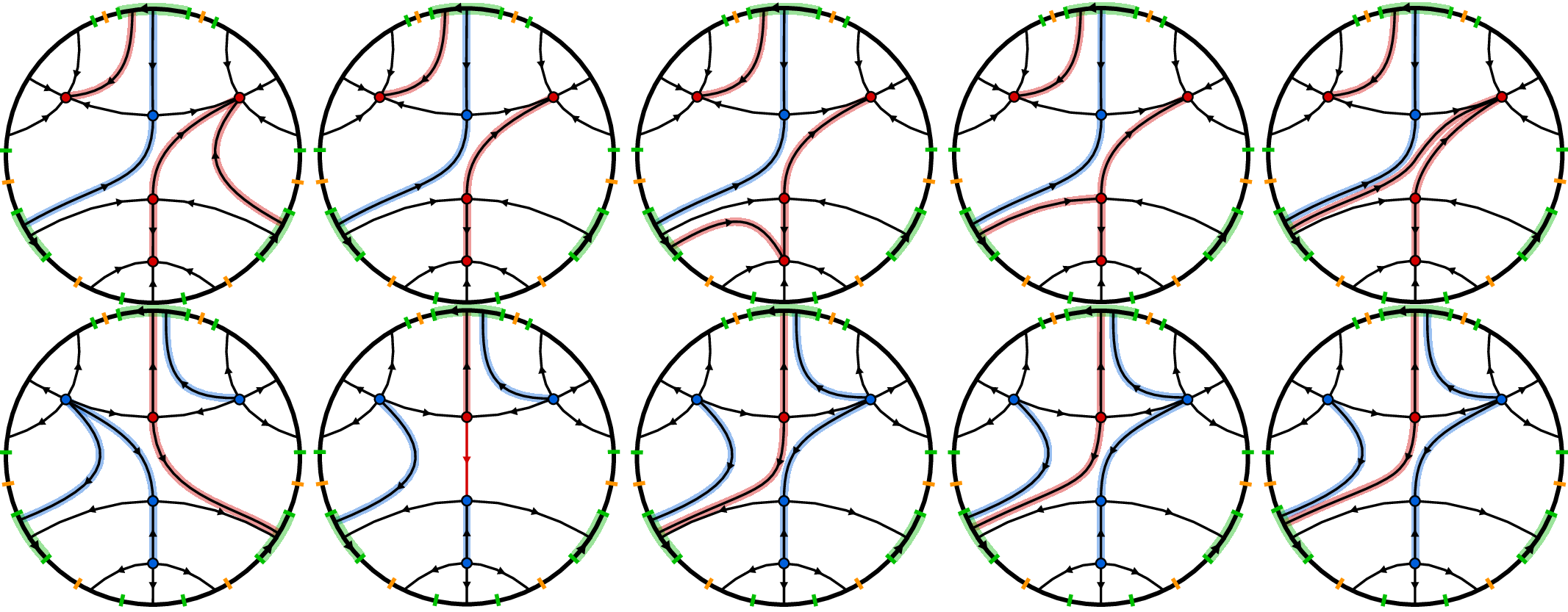}
        \put(11,3.5){\tiny \textcolor{darkblue}{$e_+^{-t_0}$}}
        \put(6,5.25){\tiny \textcolor{darkblue}{$q_+^{-t_0}$}}
        \put(5.5,14.25){\tiny \textcolor{darkblue}{$e^{-t_0}$}}
        \put(11.1,15.55){\tiny \textcolor{darkblue}{$e_-^{-t_0}$}}
        \put(10.75,10.55){\tiny \textcolor{darkred}{$q_-^{-t_0}$}}

        \put(10.75,23){\tiny \textcolor{darkred}{$e_+^{t_0}$}}
        \put(7,24.25){\tiny \textcolor{darkred}{$q_+^{t_0}$}}
        \put(3.5,30.5){\tiny \textcolor{darkred}{$e^{t_0}$}}
        \put(11.75,33.4){\tiny \textcolor{darkred}{$e_-^{t_0}$}}
        \put(7,29.5){\tiny \textcolor{darkblue}{$q_-^{t_0}$}}

        \put(27.5,-1.5){\tiny $\mu =\tfrac{2}{3}$}
        \put(68,-1.5){\tiny $\mu =\mu^*$}
	\end{overpic}
	\caption{The characteristic foliation on $S_\mu$, for various values of $0\leq \mu\leq 1$. The top row gives the upper hemispheres, while the bottom row gives the lower hemispheres; in each column, the boundary circles are identified via a bottom-to-top holonomy map, which is explicitly counterclockwise in the regions highlighted in green.}
	\label{fig:family-of-spheres-3d}
\end{figure}

We now construct the region $T\subset\widehat{B}\times I_{t_0}$ identified in \cref{fig:schematic-3d} by considering a 1-parameter family of pairwise disjoint spheres $S_\mu$, $0\leq \mu\leq 1$, where $S_\mu$ is obtained by slightly retracting $\widehat{B}$ by an amount that is proportional to $\mu$ and rounding the corners of $-\partial(\widehat{B}\times[t_-(\mu),t_+(\mu)])$. Here the monotonic functions $t_\pm\colon[0,1]\to[-t_0,t_0]$, 
\[
t_-(\mu):= -t_0+\tfrac{3}{2}\mu t_0
\qquad\text{and}\qquad
t_+(\mu):= t_0-\tfrac{1}{4}\mu t_0
\]
simultaneously raise the lower hemisphere and lower the upper hemisphere of $S_0=S$.  By defining
\[
T := S\times[0,1]_\mu := \bigcup_{0\leq \mu\leq 1}S_\mu,
\]
we ensure that the complementary region $G:=(\widehat{B}\times I_{t_0})\setminus T$ is $t$-invariant, being a subset of $\Sigma\times[t_0/2,t_0]$.

At last, we verify that $T$ is contactomorphic to a trivial bypass attached to the sphere $S$.

\begin{lemma}\label{lemma:family-of-spheres3D}
The spheres $S_\mu$ can be made simultaneously convex after a small perturbation of $\xi$ supported in the interior of $T$.
\end{lemma}

\begin{proof}
First, we will show that for all but one or two values of $0\leq\mu\leq 1$, $S_\mu$ is Morse${}^+$, and therefore convex. One value is $\mu=\tfrac{2}{3}$, where the lower hemisphere of $S_\mu$ is approximated by $\widehat{B}_0$ and its characteristic foliation includes the retrogradient from the negative critical point $q_-^0$ to the positive critical point $q_+^0$; this is the convexity-breaking solid red flow line in the second column of \cref{fig:family-of-spheres-3d}. Another bifurcation among negative critical points may occur at some value $\tfrac{2}{3}<\mu^*<1$, as shown in the fourth column of the same figure. This does not affect convexity. We will argue that the first convexity-breaking bifurcation can be perturbed away.

We start by analyzing the characteristic-stable manifolds of the positive critical points. As in the proof of \cref{lemma:standard-sphere-3d}, the critical points of $(S_\mu)_\xi$ can be assumed to lie on the disks $\widehat{B}_{t_+(\mu)}$ and $\widehat{B}_{t_-(\mu)}$. The lower hemisphere $\widehat{B}_{t_-(\mu)}$ contains the positive critical points $q_+^{t_-(\mu)}$, $e^{t_-(\mu)}$, $e_-^{t_-(\mu)}$, and $e_+^{t_-(\mu)}$, as well as the stable manifold of $q_+^{t_-(\mu)}$, for all $0\leq\mu\leq 1$.  The remaining positive critical point of $(S_\mu)_\xi$ is $q_-^{t_+(\mu)}$, which lies in the upper hemisphere $\widehat{B}_{t_+(\mu)}$ and has stable manifold including two flow lines converging in backward time to $e_-^{t_-(\mu)}$ and $e^{t_-(\mu)}$ for all $0\leq\mu\leq 1$.  The stable manifold of the positive critical point $q_+^{t_-(\mu)}$, on the other hand, includes flow lines converging in backward time to
\begin{itemize}
    \item $e^{t_-(\mu)}$ and $e_+^{t_-(\mu)}$ for $0\leq\mu<\tfrac{2}{3}$;
    \item $q_-^{t_-(\mu)}$ and $e_+^{t_-(\mu)}$ for $\mu=\tfrac{2}{3}$;
    \item $e_-^{t_-(\mu)}$ and $e_+^{t_-(\mu)}$ for $\tfrac{2}{3}<\mu\leq 1$.
\end{itemize}
While a retrogradient appears when $\mu=\tfrac{2}{3}$, the salient feature for our purposes is the existence of a flow line from $e_+^{t_-(\mu)}$ to $q_+^{t_-(\mu)}$, for all $0\leq \mu\leq 1$.  In particular, the flow lines connecting $e_+^{t_-(\mu)}$ to $q_+^{t_-(\mu)}$ and $e_-^{t_-(\mu)}$ to $q_-^{t_+(\mu)}$ in $R_+(S_\mu)$ vary continuously with respect to $\mu$, so the Elimination Lemma produces a $C^0$-small perturbation of $T$ which simultaneously eliminates the pairs of critical points connected by these flow lines.

Next we explain the potential second bifurcation, affecting only the dynamics of $R_-(S)$, which could occur at some value $\tfrac{2}{3}<\mu^*<1$.  In this bifurcation, the characteristic-unstable manifold of the negative critical point $q_-^{t_-(\mu)}$ includes flow lines converging to
\begin{itemize}
    \item $e^{t_+(\mu)}$ and $e_+^{t_+(\mu)}$ for $\tfrac{2}{3}<\mu<\mu^*$;
    \item $e^{t_+(\mu)}$ and $q_+^{t_+(\mu)}$ for $\mu=\mu^*$;
    \item $e^{t_+(\mu)}$ and $e_-^{t_+(\mu)}$ for $\mu^*<\mu\leq 1$.
\end{itemize}
Such a bifurcation would occur as the lower hemisphere $\widehat{B}_{t_-(\mu)}$ comes very close to the upper hemisphere $\widehat{B}_{t_+(\mu)}$, and whether this occurs depends on the particular choice of $t_0$ above. This is of no concern, because after continuously eliminating the trajectory from $q_+^{t_+(\mu)}$ to $e_+^{t_+(\mu)}$ for all $0\leq \mu \leq 1$, this bifurcation no longer exists and the characteristic-unstable flow line of $q_-^{t_-(\mu)}$ in question converges to $e_-^{t_+(\mu)}$ for all $\tfrac{2}{3}< \mu\leq 1$. In fact, after the continuous elimination of the flow line from $e_+^{t_-(\mu)}$ to $q_+^{t_-(\mu)}$ in $R_+(S_{\mu})$, the same characteristic-unstable flow line of $q_-^{t_-(\mu)}$ then converges to $e_-^{t_+(\mu)}$ for all $0\leq \mu \leq 1$, allowing for a further continuous cancellation of these points.

The result is that, after applying the Elimination Lemma to two flow lines --- first in $R_+(S_\mu)$ and then in $R_-(S_\mu)$ --- each characteristic foliation $(S_\mu)_\xi$ has precisely two critical points, elliptic and of opposite sign, meaning that $S_\mu$ is Morse${}^+$, and thus convex.
\end{proof}

Following the perturbation identified in \cref{lemma:family-of-spheres3D}, $\xi$ is a $\mu$-invariant structure on $T\simeq S\times[0,1]_\mu$, and it follows from \cref{prop:bypass_attachment} that $T$ is contactomorphic to a trivial bypass attachment.  This fact will be used below to identify $(M_\Sigma,\alpha_\Sigma)$ with a bypass attachment to $\Sigma_{-\delta}$.

\s\n
\textbf{Comparing to the bypass attachment along $\mathbf{(\Lambda_{\pm};D_{\pm})}$.}
\s

We have now normalized and decomposed $(M_\Sigma=\Sigma\times[-\delta,\delta],\xi=\ker\alpha_\Sigma)$ into three regions:
\begin{itemize}
    \item[(R1)] $G$, on which $\xi$ is $t$-invariant;
    \item[(R2)] $T=S\times [0,1]_\mu$, on which, by \cref{lemma:family-of-spheres3D}, $\xi$ is standard in the sense that each $S_\mu$ is convex and $R_\pm(S_\mu)$ is deformation equivalent to the standard Weinstein ball with a unique critical point;
    \item[(R3)] $M_\Sigma\setminus(G\cup T)$, on which $\xi$ is standard in the sense that each surface $\Sigma_t\setminus(G\cup T)$ is convex.
\end{itemize}
The final step in our proof of \cref{prop:bb-correspondence3D} is to demonstrate that the bypass attachment to $\Sigma_{-\delta}$ along $(\Lambda_{\pm};D_{\pm})$ is compatible with this decomposition.

The abstract bypass cobordism $(M_\Sigma,\xi'=\ker\alpha'_\Sigma)$ associated to the data $(\Lambda_{\pm};D_{\pm})\subset \Sigma_{-\delta}$ consists of four steps:
\begin{itemize}
    \item a $t$-invariant thickening of $\Sigma_{-\delta}$;
    \item the attachment of a contact $1$-handle along a sphere determined by $(\Lambda_{\pm};D_{\pm})$;
    \item the attachment of a smoothly canceling contact $2$-handle along a sphere determined by $(\Lambda_{\pm};D_{\pm})$; and
    \item a $t$-invariant thickening of the resulting ``top surface," which is diffeomorphic to $\Sigma_{-\delta}$.
\end{itemize}
Without loss of generality we may assume the following:
\begin{enumerate}[label=(\alph*)]
    \item $\xi=\xi'$ on $M_\Sigma\setminus T$;
    \item the bypass is attached to $\widehat{B}_{-t_0/2}\subset \Sigma_{-t_0/2}$ with data $(\Lambda_{\pm};D_{\pm})$; and
    \item the contact $1$- and $2$-handles determined by $(\Lambda_{\pm};D_{\pm})$ are contained in $T$.
\end{enumerate} 

It remains to show that $\xi|_T$ and $\xi'|_T$ are contact isotopic relative to the boundary. The key observation is that, {\em as bypass data on $S$,} $(\Lambda_{\pm};D_{\pm})$ is trivial.  This can be seen by noting that $\Lambda_+$ is below $\Lambda_-$ in the sense of \cref{def:bypass} or, for those more familiar with the $3$-dimensional theory, by noting that $D_+\cup D_-$ produces a trivial attaching arc.  See \cref{fig:bypass-data-on-s-3d}. A contact structure $\xi'$ on $S\times[0,1]$ obtained by a trivial bypass attachment also satisfies (R2) above by the last sentence of \cref{prop:bypass_attachment}. Hence $\xi|_T$ and $\xi'_T$ are contact isotopic rel boundary.

\begin{figure}[ht]
	\centering
    \begin{overpic}[width=0.75\textwidth]{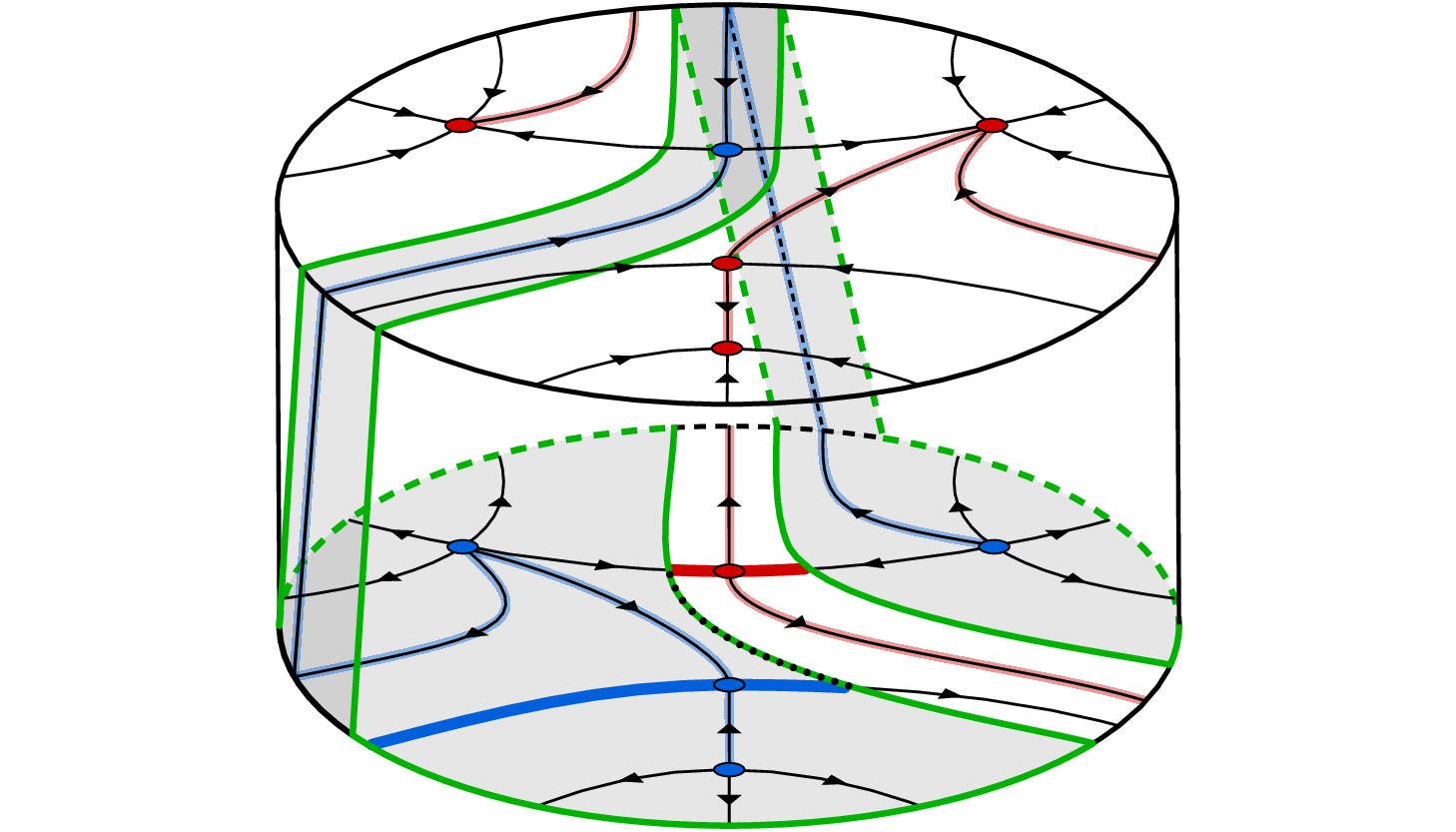}
        \put(20,1){\small $\hat{B}_{-t_0}$}
        \put(77,50){\small $-\hat{B}_{t_0}$}
        \put(40.5,11.75){\footnotesize \textcolor{darkblue}{$D_+^{-\epsilon}$}}
        \put(46,20.25){\footnotesize \textcolor{darkred}{$D_-^{\epsilon}$}}
	\end{overpic}
	\caption{The bypass data $(\Lambda_{\pm};D_{\pm})$ depicted on $S$, which is oriented as the boundary of $M_\Sigma\setminus(G\cup T)$.  The shaded region is $R_+(S)$, and the dividing set $\Gamma_S$ is depicted in green. The disks $D_{\pm}^{\mp \epsilon}$ are in solid blue and red, respectively, on the lower hemisphere $\hat{B}_{-t_0}$. In dotted black we see a short Reeb chord indicating that $\Lambda_+$ is below $\Lambda_-$. The classical (trivial) bypass attaching arc is a smoothing of the union of $D_+^{-\epsilon}$, $D_-^{\epsilon}$, and the dotted black Reeb chord.}
	\label{fig:bypass-data-on-s-3d}
\end{figure}

This concludes the proof of \cref{prop:bb-correspondence3D}.

\section{Bypass-bifurcation correspondence in dimension \(>3\)}\label{sec:BBC_high}

In this section we prove the bypass-bifurcation correspondence in all dimensions, stated precisely as follows (c.f. \cref{prop:bb-correspondence3D} and its subsequent discussion).

\begin{proposition}\label{prop:bb-correspondence}
Let $(M,\xi)$ be a contact $(2n+1)$-manifold and let $(M_\Sigma,\alpha_\Sigma):=(\Sigma\times[-\delta,\delta],\alpha_\Sigma)$ be a standard neighborhood of a folded Weinstein hypersurface
\begin{equation}\label{eq:folded-decomposition}
\Sigma_0 = W_1\cup_{K_1} W_2\cup_{K_2} W_3\cup_{K_3} W_4
\end{equation}
in $(M,\xi)$, where:
\be
    \item $W_2$ and $W_3$ are Weinstein cobordisms associated with index $n$ critical points $q_-$ and $q_+$, respectively;
    \item the Legendrian spheres $\mathrm{stab}_{W_2}(q_-)\cap K_2$ and $\mathrm{stab}_{W_3}(q_+)\cap K_2$ intersect $\xi_{K_2}$-transversely at a single point in $K_2$.
\ee
Then $(M_\Sigma,\alpha_\Sigma)$ is contactomorphic, relative boundary, to the bypass cobordism associated to the bypass data $(\Lambda_{\pm}, D_{\pm})\subset \Sigma_{-\delta}$ where, letting $D_{\pm,-\delta}$ denote the Weinstein-unstable disks of $q_{\pm}$, we have $D_{\pm} := (D_{\pm,-\delta})^{\pm \epsilon}$.
\end{proposition}

The proof of the three-dimensional case, given in the previous section, translates seamlessly to the higher-dimensional case. The strategy is identical, and many of the details generalize naturally when interpreted in the correct dimension. We proceed by retracing the steps above, being brief when there are no additional complications or observations worth addressing. 

\s\n
\textbf{Construction of the ball $\mathbf{\widehat{B}\subset\Sigma_0}$.}
\s

Let $D_{\pm,0}^{\dagger}$ denote the Weinstein-stable manifold of $q_{\pm}$ and let $\Lambda_{\pm}^{\dagger}:= \partial D_{\pm,0}^{\dagger}$.  There exists a single a retrogradient from $q_-$ to $q_+$ determined by the $\xi_{K_2}$-transverse intersection of the Legendrian $(n-1)$-spheres $\Lambda_+^{\dagger} \cap \Lambda_-^{\dagger}\subset K_2$. We define a small tubular neighborhood $B$ of $D_{+,0}^{\dagger} \cup D_{-,0}^{\dagger}$ as follows:

Since $\Lambda_+^{\dagger} \cup \Lambda_-^{\dagger}$ is a $\xi_{K_2}$-transverse intersection of Legendrian $(n-1)$-spheres, a standard tubular neighborhood of $\Lambda_+^{\dagger} \cup \Lambda_-^{\dagger} \subset K_2$ is given by a contact handlebody $$(C_2 := [-c,c]_t\times A_2,dt+\beta)$$ over the Weinstein domain $(A_2,\beta)$, obtained by plumbing sufficiently small disk cotangent bundles $\D^*\Lambda_+^{\dagger}$ and $\D^*\Lambda_-^{\dagger}$.  Using the standard Gray/Moser technique, we can arrange so that $\alpha_{\Sigma}\mid_{C_2}=C (dt+\beta)$, $C>0$, after a slight perturbation of $K_2$ transverse to $(\Sigma_0)_\xi$. From our description of the standard neighborhood of a folded Weinstein hypersurface, a tubular neighborhood of $K_2$ in $\Sigma_0$ is of the form $([-1,1]_{\tau}\times K_2, e^{\tau^2} \eta)$, where $\eta :=\alpha_{\Sigma}\mid_{K_2}$. Thus, we take as the starting point for $B$ the thickening $[-1,1]_{\tau}\times C_2$. We obtain $B$ by attaching neighborhoods of $D_{+,0}^{\dagger}$ and $D_{-,0}^{\dagger}$ along $\{1\}\times C_2$ and $\{-1\}\times C_2$, respectively. Such neighborhoods are symplectomorphic to Weinstein $n$-handles with cores $D_{+,0}^{\dagger}\setminus([-1,1]_\tau\times C_2)$ and $D_{-,0}^{\dagger}\setminus([-1,1]_\tau\times C_2)$ and attaching spheres $\{1\}\times \Lambda_+^{\dagger}$ and $\{-1\}\times \Lambda_-^{\dagger}$, respectively. See \cref{fig:BhighD}, and compare with \cref{fig:folded-with-cps-3d} in dimension $2n+1 = 3$.

\begin{figure}[ht]
	\centering
    \begin{overpic}[width=0.75\textwidth]{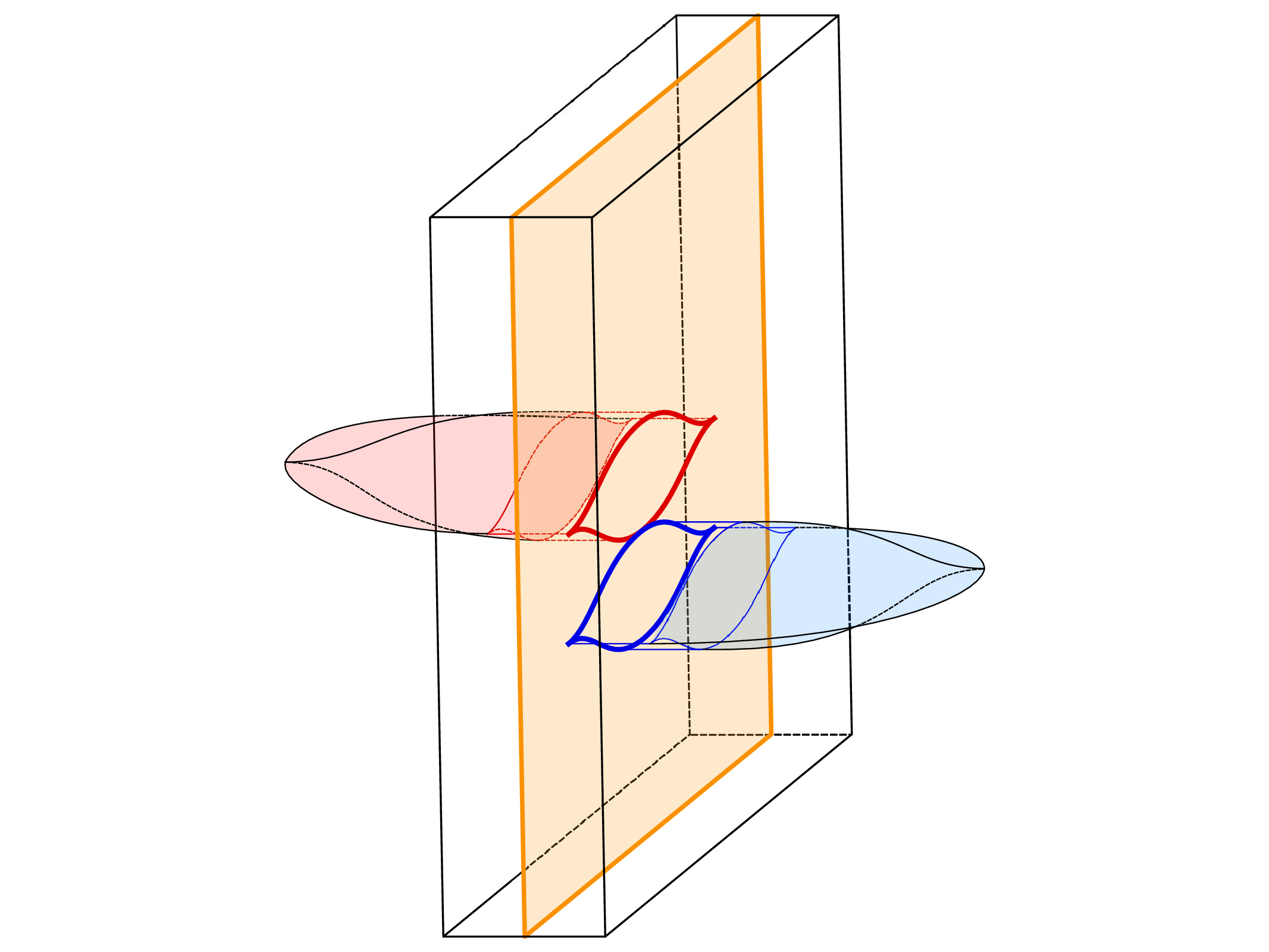}
       \put(36,10){\textcolor{Orange}{$K_2$}}
       \put(42,20){\small \textcolor{darkblue}{$\Lambda_{+}^{\dagger}$}}
       \put(55,45){\small \textcolor{darkred}{$\Lambda_{-}^{\dagger}$}}
	\end{overpic}
	\caption{The construction of $B$, depicted schematically for $2n+1 = 5$. The $\tau$-coordinate is left to right, the black box is $[-1,1]_{\tau}\times K_2$, and the shaded disks are slight retractions of $D_{\pm,0}^{\dagger}$. The orange plane is a front projection of $K_2$, and the entire figure may be taken as a front projection of $M$.}
	\label{fig:BhighD}
\end{figure}

Note that $B$ is diffeomorphic to a $2n$-dimensional ball: $[-1,1]\times C_2$ is smoothly built from a $0$-handle (a neighborhood of $\Lambda_+^{\dagger}\cap \Lambda_-^{\dagger}$) and two $(n-1)$-handles (with punctured spheres $\Lambda_{\pm}^{\dagger}\setminus \{pt\}$ as cores). Then $B$ is obtained by attaching two $n$-handles which by construction smoothly cancel the $(n-1)$-handles. 

As in the three-dimensional case, the piecewise-smooth boundary of $B$ decomposes as $\partial B = C_1 \cup C_h \cup C_3$ where:
\begin{itemize}
    \item $C_h=[-1,1]\times\partial C_2$;
    \item $C_1$ is the result of contact $(-1)$-surgery on $\{-1\}\times C_2$ along $\{-1\}\times\Lambda^\dagger_-$;
    \item $C_3$ is the result of contact $(-1)$-surgery on $\{1\}\times C_2$ along $\{1\}\times\Lambda^\dagger_+$.
\end{itemize}
By an identical modification as in dimension $3$ --- namely, flowing $C_1$ backward and $C_3$ forward along $(\Sigma_0)_\xi$ and tilting $C_h$, we may assume that $C_1\subset K_1$, $C_3\subset \overline{W}_4$ and transversely intersects $K_3$ along $\bdry C_3$, and $(\Sigma_0)_\xi$ points out of $\partial_{\mathrm{out}}B:= C_h \cup C_3$ and into $\partial_{\mathrm{in}} B := C_1$. 

To construct the ball $\widehat{B}$ whose characteristic foliation $\widehat{B}_\xi$ is positively transverse to $\bdry \widehat{B}$, we install a box fold near $\partial_{\mathrm{in}} B$, supported on a one-sided neighborhood $N(\partial_{\mathrm{in}} B) = [-s_0,0]\times \partial_{\mathrm{in}} B$ of $\partial_{\mathrm{in}} B$ in $W_1$. By construction, $\partial_{\mathrm{in}} B$ is obtained from the contact handlebody $C_2$ by a contact $(-1)$-surgery along $\Lambda^\dagger_-$, viewed as a subset of $\{0\}\times A_2\subset [-c,c]\times A_2$.

\begin{claim} \label{claim: generalized contact handlebody}
    The incoming region $\bdry_{\op{in}}B$ is a generalized contact handlebody 
    $$\{(t,x)~|~  -c \leq t\leq  f(x)\}\subset (\R\times A_2, \,\ker(dt+\beta)),$$ 
    where $f(x)$ is close to $c$.
\end{claim}

\begin{proof}
    We consider the slightly simpler model of applying a contact $(-1)$-surgery to $([-1,1]\times \D^* S^{n-1},dt+\beta)$ along $\{0\}$ times the zero section $S^{n-1}$.
    Let $\tau$ be a symplectic Dehn twist along $S^{n-1}$ (in particular $\tau^* d\beta=d\beta$), supported on a small neighborhood $N(S^{n-1})$ of $S^{n-1}\subset \D^* S^{n-1}$.  In the process of removing $[-1/2,1/2]\times N(S^{n-1})$ and regluing, we must replace $dt+\beta$ by $dt+ \beta_t$, where $\beta_t$ interpolates between $\tau_*\beta$ and $\beta$. The claim then follows by observing that the Reeb vector field of $dt+\beta_t$ is close to $\bdry_t$ since $\beta$ and $\tau_*\beta$ are close to zero on $N(S^{n-1})$.
\end{proof} 

Since $\partial_{\mathrm{in}} B$ is close to being a contact handlebody, we can apply the Creation Lemma and install a box fold in $N(\partial_{\mathrm{in}} B) = [-s_0,0]\times \partial_{\mathrm{in}} B$.  (Note that, since $\bdry_{\op{in}} B\subset K_1$, which is a dividing set, there are no contact width issues and we may assume that $\alpha_\Sigma|_{\bdry_{\op{in}} B}$ is some positive multiple of the contact form from \cref{claim: generalized contact handlebody}.)  The Weinstein domain $(A_2,\beta)$ has three critical points: a critical point of index $0$ and two critical points of index $n-1$. 
Since each index $k$ critical point of $A_2$ yields critical points of index $k$ and $k+1$ in the box fold, the box fold has six critical points, with indices $0,1$, $n-1$, $n$, $n-1$, and $n$. Let $e$ denote the index $0$ critical point and $e_{\pm}$ denote the index $n-1$ critical point associated to $e_{\pm}'$.

Finally, $\widehat{B}$ is obtained from $B$ as in dimension $3$ by adjoining the Weinstein-unstable manifolds of $e$, $e_+$, and $e_-$ and slightly retracting so that $(\Sigma_0)_\xi$ is transverse to $\bdry \widehat{B}$.

\s\n
\textbf{Decomposing the region $\mathbf{\widehat{B}\times[-t_0,t_0]}$.}
\s

The discussion preceding \cref{lemma:standard-sphere-3d} works for $\op{dim} M$ greater than $3$ with minimal change.  Summarizing that discussion:
\begin{itemize}
    \item We repeat the construction of $\widehat{B}$ in a $t$-parametric manner, performing a $C^0$-small isotopy of $\Sigma\times[-\delta,\delta]$ which is supported on $W_1\times(-\delta,\delta)$, such that (B1)--(B6) hold after isotopy, with (B3) replaced by:
    \be
    \item[(B3')] for each $t\in[-\delta/2,\delta/2]$, there are three canceling pairs of critical points on $W_1$; one pair will have indices $0$ and $1$ while the other two pairs have indices $n-1$ and $n$.  
    \ee
    The index $0$ critical point will be denoted $e^t$ and the index $n-1$ critical points will be denoted $e^t_+$ and $e^t_-$. See \cref{fig:b-hat-foliation-3d} for the $n=1$ case.

    \item We then restrict attention to the neighborhood $\Sigma\times I_{t_0}$ of $\Sigma_0$. Each $(\Sigma_t)_\xi$, $t\in I_{t_0}$, contains index $n$ critical points $q^t_+$ and $q^t_-$, corresponding to $q_+$ and $q_-$ in $\widehat{B}_0$, respectively.  Moreover, $\op{stab}_{W_3}(q_+^t)\cap \op{stab}_{W_2}(q_-^t)\not=\emptyset$ precisely when $t=0$ and 
    \begin{equation}\label{bbc:eq:b-hat-dynamics}
    \op{stab}_{W_3}(q_+^t)\cap C_2 = \partial D^\dagger_{+,t} = (\Lambda^\dagger_+)^{t}
    \quad\text{and}\quad
    \op{stab}_{W_2}(q_-^t)\cap C_2 = \partial D^\dagger_{-,t} = (\Lambda^\dagger_-)^{-t},
    \end{equation}
    for all $t\in I_{t_0/2}$, where $\Lambda^\dagger_\pm\subset C_2\subset K_2$ is the Legendrian $(n-1)$-sphere identified above.
\end{itemize}

Let $S$ be the hypersurface obtained from $-\bdry (\widehat{B}\times I_{t_0})$ by rounding corners.
The following is the higher-dimensional analog of \cref{lemma:standard-sphere-3d}:

\begin{lemma}\label{lemma:standard-sphere-all-dim}
The hypersurface $S$ can be taken to be convex, with the regions $R_\pm(S)$ each deformation equivalent to the standard Weinstein $2n$-ball with a unique critical point.
\end{lemma}

\begin{proof}
The hypersurface $S$ is obtained by rounding the piecewise-smooth hypersurface
\[
-\partial(\widehat{B}\times I_{t_0}) = -\widehat{B}_{t_0} \cup ((\partial\widehat{B})\times I_{t_0}) \cup \widehat{B}_{-t_0}.
\]
Since the rounding occurs on a small neighborhood of the corners, our analysis of $S_\xi$ will follow from an analysis of the characteristic foliation of this piecewise-smooth hypersurface.  

The critical points of $S_\xi$ are all contained in $\widehat{B}_{\pm t_0}$, five of which are positive:
\begin{itemize}
    \item critical points $q_+^{-t_0}\in\widehat{B}_{-t_0}$ and $q_-^{t_0}\in\widehat{B}_{t_0}$ of index $n$, projecting to $q_+$ and $q_-$ in $\widehat{B}_0$, respectively;
    \item critical points $e_+^{-t_0},e_-^{-t_0}\in\widehat{B}_{-t_0}$ of index $n-1$, projecting to  $e_+$ and $e_-$ in $\widehat{B}_0$, respectively;
    \item a critical point $e^{-t_0}\in\widehat{B}_{-t_0}$ of index $0$, projecting to $e$ in $\widehat{B}_0$;
\end{itemize}
and five of which are negative:
\begin{itemize}
    \item critical points $q_-^{-t_0}\in\widehat{B}_{-t_0}$ and $q_+^{t_0}\in\widehat{B}_{t_0}$ of index $n$, projecting to $q_-$ and $q_+$ in $\widehat{B}_0$, respectively;
    \item critical points $e_+^{t_0},e_-^{t_0}\in\widehat{B}_{t_0}$ of index $n-1$, projecting to  $e_+$ and $e_-$ in $\widehat{B}_0$, respectively;
    \item a critical point $e^{t_0}\in\widehat{B}_{t_0}$ of index $0$, projecting to $e$ in $\widehat{B}_0$.
\end{itemize}

Next we analyze the characteristic-stable manifolds $\op{stab}_S(p)$ of the positive critical points $p$ of $S$ and show that $R_+(S)$ is deformation equivalent to the standard Weinstein $2n$-ball; the $R_-(S)$ case is similar.  Observe that the Weinstein- and character\-istic-stable manifolds coincide for positive critical points.  

(1) $\op{stab}_S(e^{-t_0})=\{e^{-t_0}\}$, since $\op{ind}(e^{-t_0})=0$.  

(2) The characteristic-stable manifolds $\op{stab}_S(e_-^{-t_0})$ and $\op{stab}_S(e_+^{-t_0})$ are punctured copies of the spheres $\Lambda_-\subset K_1$ and $(\Lambda_+^\dagger)'\subset K_1$, which is the image of $(\Lambda_+^\dagger)^{-t_0/2}\subset K_2$ under the backwards flow of $S_\xi$ to $K_1$, respectively. We write
\[
\tilde{\Lambda}_- = \op{stab}_S(e_-^{-t_0})\cup\{e^{-t_0}\}
\quad\text{and}\quad
\tilde{\Lambda}_+^\dagger = \op{stab}_S(e_+^{-t_0})\cup\{e^{-t_0}\},
\]
where the notation $\tilde{\Lambda}_-$ and $\tilde{\Lambda}_+^\dagger$ indicates that these spheres are contained in the interior of $W_1\subset \widehat{B}_{-t_0}$.  Moreover, the characteristic-\emph{un}stable manifolds of $e_-^{-t_0}$ and $e^{-t_0}$ together form a neighborhood of $\tilde{\Lambda}_-$ and, similarly, the characteristic-unstable manifolds of $e_+^{-t_0}$ and $e^{-t_0}$ together form a neighborhood of $\tilde{\Lambda}_+^\dagger$.  See \cref{fig:b-hat-neg-3d} for the $n=1$ case, where these spheres are given by $\tilde{\Lambda}_-=\{e^{-t_0},e_-^{-t_0}\}$ and $\tilde{\Lambda}_+^\dagger=\{e^{-t_0},e_+^{-t_0}\}$.

(3)  The characteristic-stable manifolds $\op{stab}_S(q_+^{-t_0})$ and $\op{stab}_S(q_-^{t_0})$ of $q_+^{-t_0}\in\widehat{B}_{t_0}$ and $q_-^{t_0}\in-\widehat{B}_{t_0}$ are $n$-disks which, according to \eqref{bbc:eq:b-hat-dynamics}, satisfy 
\begin{gather} \label{eqn: stab S}
\op{stab}_S(q_+^{-t_0})\cap C_2 = (\Lambda_+^\dagger)^{-t_0/2}\subset \widehat{B}_{-t_0},\\
\label{eqn: stab S2}\op{stab}_S(q_-^{t_0})\cap C_2  = (\Lambda_-^\dagger)^{-t_0/2}\subset -\widehat{B}_{t_0}.
\end{gather}
The exponent $-t_0/2$ on the right-hand side comes from observing that \eqref{bbc:eq:b-hat-dynamics} holds for $|t|<t_0/2$ and that the characteristic foliation is $t$-invariant for $t_0/2\leq |t|\leq \delta$. When identifying the Reeb shifts of the Legendrian spheres in $C_2$, we use the induced contact form on $C_2$ inherited from the global contact form; see \cref{fig:analysis}.

(3A) Flowing $\op{stab}_S(q_+^{-t_0})\cap C_2$ backwards across $W_2$ and into $W_1$,  $\bdry(\op{stab}_S(q_+^{-t_0}))$ is precisely the $(n-1)$-sphere $\tilde{\Lambda}_+^\dagger$ identified above.  Moreover, because $e_+^{-t_0}$ is the index-$(n-1)$ critical point of $\tilde{\Lambda}_+^\dagger$, $\op{stab}_S(q_+^{-t_0})$ contains a unique flow line from $e_+^{-t_0}$ to $q_+^{-t_0}$.

(3B) It remains to consider the most complicated case $\op{stab}_S(q_-^{t_0})$.
\begin{figure}[ht]
	\centering
    \begin{overpic}[width=\textwidth]{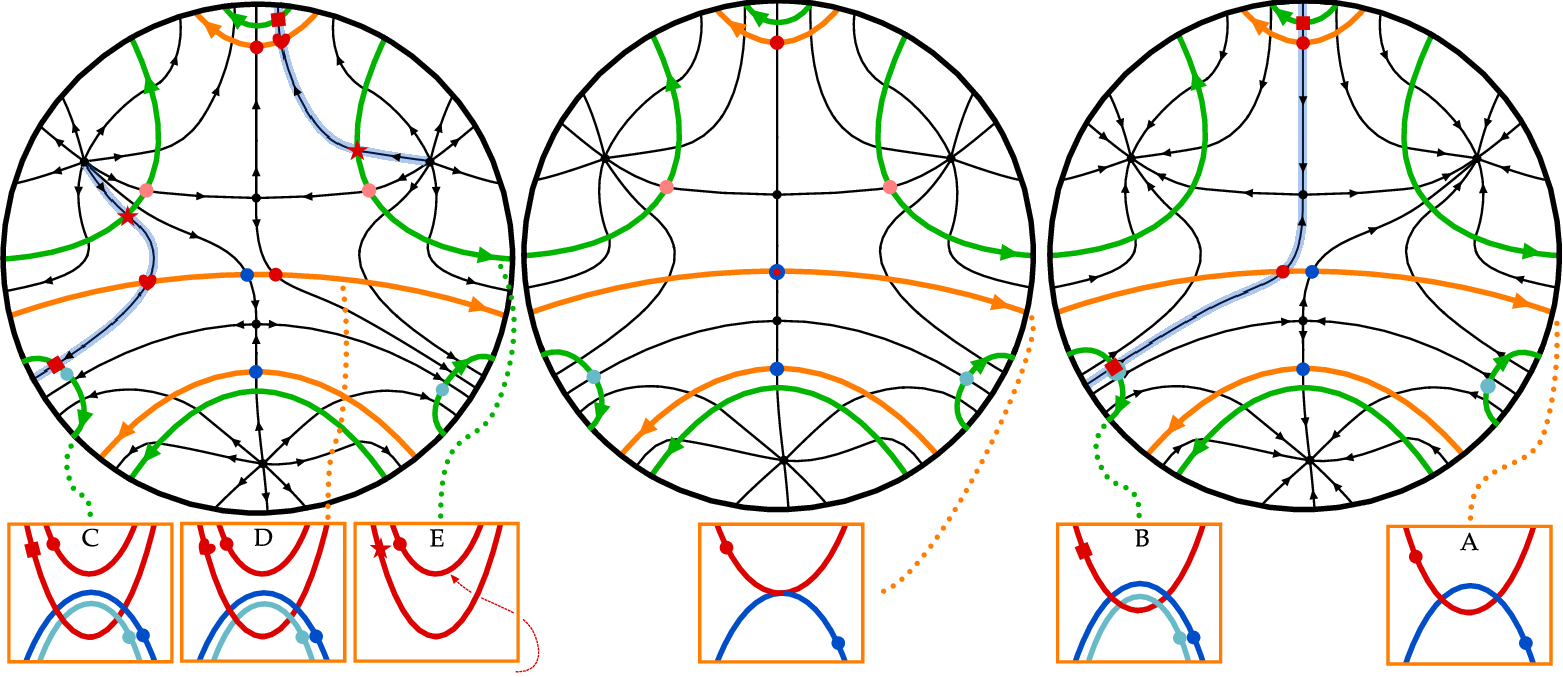}
        \put(0,42){\tiny $t=-t_0$}
        \put(34,42){\tiny $t=0$}
        \put(67,42){\tiny $t=t_0$}
        
        \put(44,27){\footnotesize \textcolor{Orange}{$C_2$}}
        \put(49,-0.5){\tiny \textcolor{Orange}{$C_2$}}
        \put(50.5,20.5){\tiny \textcolor{darkblue}{$\Lambda_+^{\dagger}$}}
        \put(47.5,24){\tiny $q_+$}
        \put(50.5,31.75){\tiny $q_-$}
        \put(47.5,38.5){\tiny \textcolor{darkred}{$\Lambda_-^{\dagger}$}}
        \put(44,31.75){\tiny \textcolor{Salmon}{$\Lambda_-$}}
        \put(61,22){\tiny \textcolor{Cyan}{$\Lambda_+$}}
        \put(50.5,2){\tiny \textcolor{darkblue}{$\Lambda_+^{\dagger}$}}
        \put(47.5,7.5){\tiny \textcolor{darkred}{$\Lambda_-^{\dagger}$}}

        \put(93,-0.5){\tiny \textcolor{Orange}{$C_2$}}
        \put(72,-0.5){\tiny \textcolor{divgreen}{$K_3$}}
        \put(75,-0.5){\tiny \textcolor{darkblue}{$(-1)$}}

        \put(4.5,-0.5){\tiny \textcolor{divgreen}{$K_3$}}
        \put(7.5,-0.5){\tiny \textcolor{darkblue}{$(-1)$}}

        \put(15.5,-0.5){\tiny \textcolor{Orange}{$C_2$}}
        \put(18.5,-0.5){\tiny \textcolor{darkblue}{$(-1)$}}
        \put(11.5,-0.5){\tiny \textcolor{Cyan}{$(+1)$}}

        \put(29.25,-0.5){\tiny \textcolor{darkred}{$(-1)$}}
        \put(26.5,-0.5){\tiny \textcolor{divgreen}{$K_1$}}
	\end{overpic}
	\caption{The analysis of $\op{stab}_S(q_-^{t_0})$, highlighted in blue; c.f.\ the first column of \cref{fig:family-of-spheres-3d}.}
	\label{fig:analysis}
\end{figure}
The analysis is summarized in \cref{fig:analysis}, which requires some explanation: First, three low-dimensional schematic copies of $\widehat{B}_t$ are drawn (c.f.\ \cref{fig:b-hat-foliation-3d}) with the various contact level sets corresponding to the folded Weinstein structure of $\Sigma_0$ in green and orange. The arrows indicate the Reeb direction of each level set induced by the global contact form. The characteristic foliation on $t=\pm t_0$ is oriented  according to the orientation of the sphere $S$, so that the left panel of the figure agrees with \cref{fig:b-hat-neg-3d}, while the right panel of the figure has a foliation opposite to that of \cref{fig:b-hat-pos-3d}. For reference, the unoriented foliation on $t=0$ is drawn, along with the main Legendrian spheres of interest: $\Lambda_{\pm}$ and $\Lambda_{\pm}^{\dagger}$. While the low-dimensional pictures are schematic for the purpose of keeping track of the relevant shifts, the local spin-symmetric front projections at the bottom of the figure are accurate. Each of these front projections depicts a collection of Legendrian spheres in $C_2$; after performing the indicated surgeries, we are left with Legendrian spheres in the contact manifold indicated below the front projection.

The backward passage of $\op{stab}_S(q_-^{t_0})$ on $S$ through various contact level sets is tracked sequentially through the surgery diagrams labeled A, B, C, D, E. We begin with A, where $\op{stab}_S(q_-^{t_0})\cap C_2 = (\Lambda_-^{\dagger})^{-t_0/2}$. The level set $K_3$ is obtained from $C_2$ by a contact $(-1)$-surgery along $(\Lambda_+^{\dagger})^{t_0}$, corresponding to attachment of the Weinstein handle associated to $W_3$; this produces panel B, and $\op{stab}_S(q_-^{t_0})$ intersects along the red Legendrian sphere in $B$, indicated with a small red box. For reference, we have also drawn a copy of $\Lambda_+^{-t_0}$, a negative Reeb shift of the cocore, in light blue. Continued backward flow along $\op{stab}_S(q_-^{t_0})$ pushes across a small portion of $W_4$ before meeting $\partial\widehat{B}_{t_0}$, where the vertical portion of $S$ induces a small negative Reeb shift.  Upon reaching $\partial\widehat{B}_{-t_0}$, the flow lines are reversed and $\op{stab}_S(q_-^{t_0})$ again meets $K_3$ in panel C as the box-labeled Legendrian.  Panel C also includes a red dot-labeled Legendrian, which is a copy of $(\Lambda^\dagger_-)^{t_0}$, for reference. Backwards flow across $W_3$ induces a contact $(+1)$-surgery on $\Lambda_+^{-t_0}$, producing the surgery diagram in panel D, where $\op{stab}_S(q_-^{t_0})$ meets $C_2$ as the heart-labeled Legendrian. We pass through $W_2$ to reach $K_1$ by performing a contact $(-1)$-surgery along $(\Lambda_-^{\dagger})^{t_0}$, which is the red dot-labeled Legendrian in panels C, D, and E. Canceling the blue $(\pm 1)$ surgeries gives the final surgery diagram presentation of $K_1$ in panel E, along which $\op{stab}_S(q_-^{t_0})$ intersects as the star-labeled negative Reeb shift of $\Lambda_-$, i.e., the belt sphere of the handle attached along $(\Lambda_-^{\dagger})^{t_0}$. Pushing into $W_1$, we find that the boundary of $\op{stab}_{S}(q_-^{t_0})$ is the $(n-1)$-sphere $\tilde{\Lambda}_-$, and that $\op{stab}_{S}(q_-^{t_0})$ contains a unique flow line from $e_-^{-t_0}$ to $q_-^{t_0}$.  

Finally, we apply the Elimination Lemma to the pairs $\{q_+^{-t_0},e_+^{-t_0}\}$ and $\{q_-^{t_0},e_-^{-t_0}\}$, as in the proof of \cref{lemma:standard-sphere-3d}, to conclude that $R_+(S)$ is deformation equivalent to a standard Weinstein ball.
\end{proof}

Next, we construct the region $T\subset\widehat{B}\times I_{t_0}$ as $S\times[0,1]_\mu$ as in the $3$-dimensional case.  Specifically, $S_\mu$ is obtained by  slightly retracting $\widehat{B}$ by an amount that is proportional to $\mu$ and rounding the corners of $-\partial(\widehat{B}\times[t_-(\mu),t_+(\mu)])$, where the monotonic functions $t_\pm:[0,1]\to [-t_0,t_0]$,
\begin{equation} \label{eqn: t sub pm second version}
t_-(\mu) = -t_0 + \tfrac{3}{2}\mu t_0
\quad\text{and}\quad
t_+(\mu) = t_0 - \tfrac{1}{4}\mu t_0
\end{equation}
are defined as before.  The region $G:=(\widehat{B}\times I_{t_0})\setminus T$ is then contained in $\Sigma\times[t_0/2,t_0]$ and thus carries a $t$-invariant contact structure.  It remains, then, to verify that $T$ is contactomorphic to a trivial bypass attached to $S$ --- that is, to generalize \cref{lemma:family-of-spheres3D} to the $\op{dim} M>3$ case.

\begin{lemma}\label{lemma:family-of-spheres-all-dim}
The spheres $S_\mu$ can be made simultaneously convex after a small perturbation of $\xi$ supported in the interior of $T$.
\end{lemma}

\begin{proof}
As with the proof of \cref{lemma:standard-sphere-all-dim}, the $3$-dimensional argument readily generalizes to all dimensions, provided we carefully analyze the characteristic-stable and characteristic-unstable manifolds.  

The critical points of $(S_\mu)_\xi$ lie on the disks $\widehat{B}_{t_+(\mu)}$ and $\widehat{B}_{t_-(\mu)}$.  There are ten critical points for each $0\leq \mu\leq 1$, five of which are positive:
\begin{itemize}
    \item critical points $q_+^{t_-(\mu)}\in\widehat{B}_{t_-(\mu)}$ and $q_-^{t_+(\mu)}\in\widehat{B}_{t_+(\mu)}$ of index $n$, projecting to $q_+$ and $q_-$ in $\widehat{B}_0$, respectively;
    \item critical points $e_+^{t_-(\mu)},e_-^{t_-(\mu)}\in\widehat{B}_{t_-(\mu)}$ of index $n-1$, projecting to  $e_+$ and $e_-$ in $\widehat{B}_0$, respectively;
    \item a critical point $e^{t_-(\mu)}\in\widehat{B}_{t_-(\mu)}$ of index $0$, projecting to $e$ in $\widehat{B}_0$.
\end{itemize}
As in the $3$-dimensional case, the characteristic foliation $(S_\mu)_\xi$ will feature a retrogradient when $\mu=\tfrac{2}{3}$, and another bifurcation may occur for some value $\tfrac{2}{3}<\mu^*<1$.

\s\n
{\bf The $R_+(S_\mu)$ case.} We will show that, for each $0\leq \mu\leq 1$, there exists a unique flow line from $e_+^{t_-(\mu)}$ to $q_+^{t_-(\mu)}$ and from $e_-^{t_-(\mu)}$ to $q_-^{t_+(\mu)}$. The Elimination Lemma then implies that $R_+(S_\mu)$ is deformation equivalent to the standard Weinstein $2n$-ball.

As in the proof of \cref{lemma:standard-sphere-all-dim}, the characteristic-stable manifolds of $e_{\pm}^{t_-(\mu)}$ and $e^{t_-(\mu)}$ together form spheres in $\widehat{B}_{t_-(\mu)}$:
\[
\tilde{\Lambda}_- = \op{stab}_{S_\mu}(e_-^{t_-(\mu)})\cup \op{stab}_{S_\mu}(e^{t_-(\mu)})
\quad\text{and}\quad
\tilde{\Lambda}_+^\dagger = \op{stab}_{S_\mu}(e_+^{t_-(\mu)})\cup \op{stab}_{S_\mu}(e^{t_-(\mu)}).
\]
Again, these are spheres contained in the interior of $W_1$ and we obtain a neighborhood of $\tilde{\Lambda}_-$ (resp.\ $\tilde{\Lambda}_+^\dagger$) from the characteristic-unstable manifolds of $e^{t_-(\mu)}$ and $e_-^{t_-(\mu)}$ (resp.\ $e^{t_-(\mu)}$ and $e_+^{t_-(\mu)}$).

Next we consider $\op{stab}_{S_\mu}(q_+^{t_-(\mu)})$, which varies with $\mu$: 
Although $\op{stab}_{S_\mu}(q_+^{t_-(\mu)})\cap C_2$ is a possibly trivial Reeb pushoff of $\Lambda_+^\dagger$, regardless of the value of $0\leq\mu\leq 1$, the result of flowing this intersection backwards across $W_2$ to $K_1=\partial W_1\subset\widehat{B}_{t_-(\mu)}$ is:
\begin{itemize}
    \item Legendrian isotopic to $(\Lambda_+^\dagger)'$ for $0\leq\mu < \tfrac{2}{3}$;
    \item the punctured sphere $(\Lambda_+^\dagger)'\setminus\Lambda_-$ for $\mu=\tfrac{2}{3}$; and
    \item Legendrian isotopic to the Legendrian sum $(\Lambda_+^\dagger)'\uplus\Lambda_-$ for $\tfrac{2}{3}\leq\mu\leq 1$.
\end{itemize}
Recall that $(\Lambda_+^\dagger)'$ is the image of $(\Lambda_+^\dagger)^{-t_0/2}\subset K_2$ under the backwards flow of $(S_0)_\xi$ to $K_1$. In particular, $(S_{2/3})_\xi$ features a retrogradient from $q_-^{t_-(\mu)}$ to $q_+^{t_-(\mu)}$, $K_1$ is obtained from $C_2$ by a contact $(-1)$-surgery along $\Lambda_-^{\dagger}$ as in panel E of \cref{fig:analysis}, and passing through the retrogradient corresponds to handlesliding $(\Lambda_+^{\dagger})'$ up across the surgery; see \cref{fig:analysis2}.  

\begin{figure}[ht]
	\centering
    \begin{overpic}[width=0.75\textwidth]{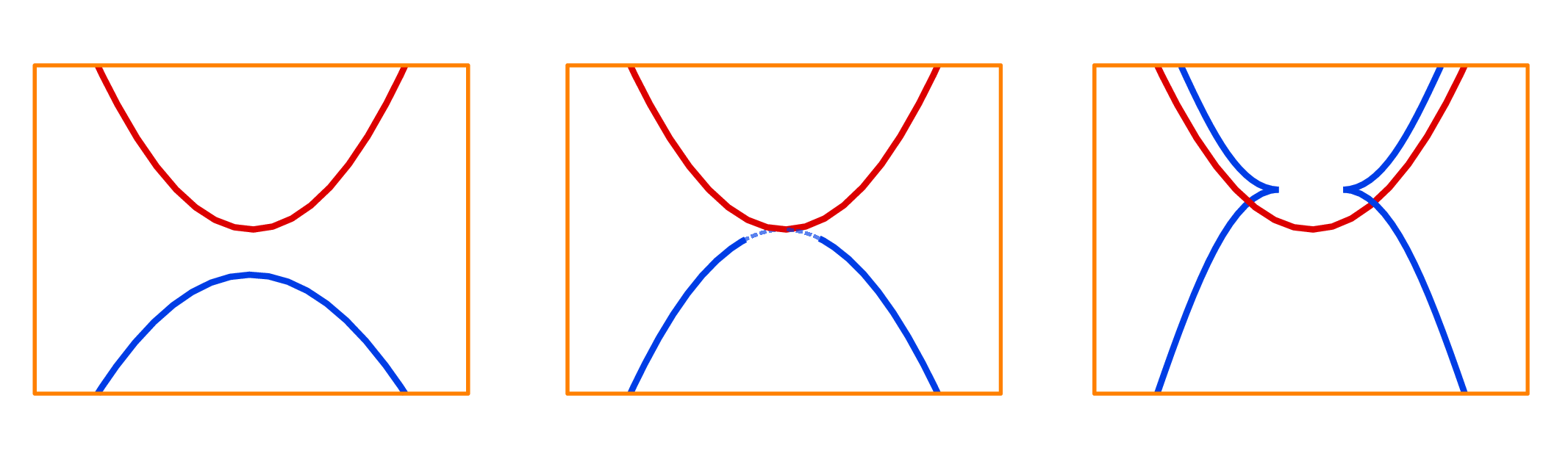}
       \put(49,1){\small \textcolor{divgreen}{$K_1$}}
       \put(14,1){\small \textcolor{divgreen}{$K_1$}}
       \put(84,1){\small \textcolor{divgreen}{$K_1$}}

       \put(10,26.5){\footnotesize $0\leq \mu < \tfrac{2}{3}$}
       \put(48,26.5){\footnotesize $\mu = \tfrac{2}{3}$}
       \put(78,26.5){\footnotesize $\tfrac{2}{3}< \mu \leq 1$}

       \put(60,5){\small \textcolor{Orange}{$C_2$}}
       \put(26,5){\small \textcolor{Orange}{$C_2$}}
       \put(93.75,5){\small \textcolor{Orange}{$C_2$}}

       \put(58,20){\tiny \textcolor{darkred}{$(-1)$}}
       \put(24,20){\tiny \textcolor{darkred}{$(-1)$}}
       \put(91.75,20){\tiny \textcolor{darkred}{$(-1)$}}
	\end{overpic}
	\caption{Analyzing $\op{stab}_{S_\mu}(q_+^{t_-(\mu)})$. Each copy of $K_1$ seen here lies in the lower hemisphere $\widehat{B}_{t_-(\mu)}$ of $S_\mu$, and the red and blue Legendrians are $(\Lambda_-^\dagger)^{-t_-(\mu)}$ and $\op{stab}_{S_\mu}(q_+^{t_-(\mu)})\cap K_1$, respectively.}
	\label{fig:analysis2}
\end{figure}

Continuing to flow backwards into $W_1$, the flow lines passing through $\op{stab}_{S_\mu}(q_+^{t_-(\mu)})\cap K_1$ will all limit in backward time to $e^{t_-(\mu)}$, $e_-^{t_-(\mu)}$, or $e_+^{t_-(\mu)}$; relevant for us is the existence of a unique flow line limiting in backward time to $e_+^{t_-(\mu)}$ --- corresponding to the critical point of $\tilde{\Lambda}_+^\dagger$ --- for all values of $0\leq \mu\leq 1$.

The treatment of $\op{stab}_{S_\mu}(q_-^{t_+(\mu)})$, $0\leq \mu \leq 1$, is analogous: For each $\mu$, the intersection $\op{stab}_{S_\mu}(q_-^{t_+(\mu)})\cap K_1$ is a Reeb pushoff of $\Lambda_-$ and flows backwards to the sphere $\tilde{\Lambda}_-$, meaning that there is a unique flow line from $e_-^{t_-(\mu)}$ to $q_-^{t_+(\mu)}$.  Applying the Elimination Lemma to the pairs $\{q_+^{t_-(\mu)},e_+^{t_-(\mu)}\}$ and $\{q_-^{t_+(\mu)},e_-^{t_-(\mu)}\}$ then allows us to see that $R_+(S_\mu)$ is Weinstein homotopic to the standard Weinstein $2n$-ball with unique critical point.

\s\n
{\bf The $R_-(S_\mu)$ case.} The negative region $R_-(S_\mu)$ is addressed similarly, though, as in dimension $3$, a second bifurcation is possible at some parameter value $\tfrac{2}{3}<\mu^*<1$, and we assume that the Elimination Lemma is applied to $R_+(S)$ before we address $R_-(S)$.  

We explain the second bifurcation, which in dimension $3$ occurs in the fourth column of \cref{fig:family-of-spheres-3d}. Here, $\op{unstab}_{S_\mu}(q_-^{t_-(\mu)})$ intersects $C_2\subset K_2$ along a Legendrian isotopic to $\Lambda_-^\dagger$ for all $\tfrac{2}{3}<\mu\leq 1$.  Pushing across $W_3$ corresponds to a surgery, and $\op{unstab}_{S_\mu}(q_-^{t_-(\mu)})$ meets $K_3$ along a sphere Legendrian isotopic to $\Lambda_+\uplus\Lambda_-^\dagger$, as shown in \cref{fig:analysis3}, which depicts a surgery diagram for $K_3$ set in $C_2$.

\begin{figure}[ht]
	\centering
    \begin{overpic}[width=0.75\textwidth]{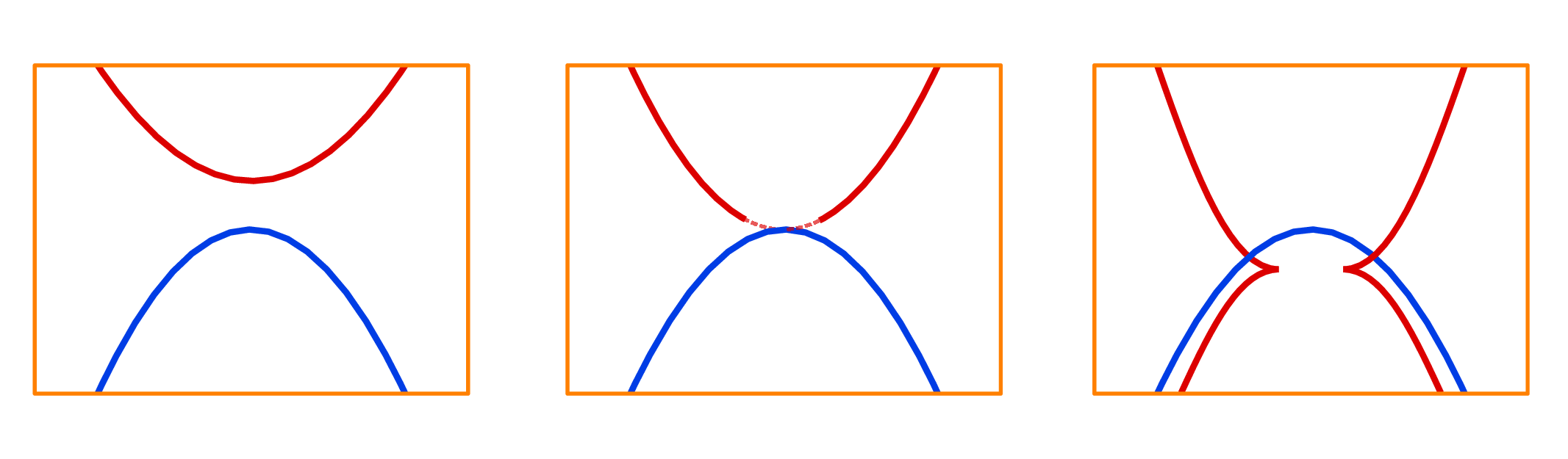}
       \put(49,1){\small \textcolor{divgreen}{$K_3$}}
       \put(14,1){\small \textcolor{divgreen}{$K_3$}}
       \put(84,1){\small \textcolor{divgreen}{$K_3$}}

       \put(10,26.5){\footnotesize $0\leq \mu < \tfrac{2}{3}$}
       \put(48,26.5){\footnotesize $\mu = \tfrac{2}{3}$}
       \put(78,26.5){\footnotesize $\tfrac{2}{3}< \mu \leq 1$}

       \put(60,5){\small \textcolor{Orange}{$C_2$}}
       \put(26,5){\small \textcolor{Orange}{$C_2$}}
       \put(93.75,5){\small \textcolor{Orange}{$C_2$}}

       \put(58,10){\tiny \textcolor{darkblue}{$(-1)$}}
       \put(24,10){\tiny \textcolor{darkblue}{$(-1)$}}
       \put(91.75,10){\tiny \textcolor{darkblue}{$(-1)$}}
	\end{overpic}
	\caption{Analyzing $\op{unstab}_{S_\mu}(q_-^{t_-(\mu)})$. Each copy of $K_3$ seen here lies in the lower hemisphere $\widehat{B}_{t_-(\mu)}$ of $S_\mu$, and the blue and red Legendrians are $(\Lambda_+^\dagger)^{t_-(\mu)}$ and $\op{unstab}_{S_\mu}(q_-^{t_-(\mu)})\cap K_3$, respectively.}
	\label{fig:analysis3}
\end{figure}

This intersection occurs within $\widehat{B}_{t_-(\mu)}$; continued flow takes us into $W_4$, across the vertical portion of $\partial S$, and back into $K_3$, this time within $\widehat{B}_{t_+(\mu)}$.  From here, the behavior of $\op{unstab}_{S_\mu}(q_-^{t_-(\mu)})$ depends on whether its intersection with $K_3$ is a positive, negative, or trivial Reeb pushoff of $\Lambda_+\uplus\Lambda_-^\dagger$.  If the intersection is precisely $\Lambda_+\uplus\Lambda_-^\dagger$, then we will find a unique flow line between $q_+^{t_+(\mu)}$ and $q_-^{t_-(\mu)}$, while an intersection which is a positive (respectively, negative) Reeb pushoff of $\Lambda_+\uplus\Lambda_-^\dagger$ results in a unique flow line between $e_+^{t_+(\mu)}$ (respectively, $e_-^{t_+(\mu)}$) and $q_-^{t_-(\mu)}$.  We define $\mu^*$ so that the case $\op{unstab}_{S_\mu}(q_-^{t_-(\mu)})\cap K_3=\Lambda_+\uplus\Lambda_-^\dagger$ corresponds to $\mu=\mu^*$, and this is our second bifurcation.

As in the $3$-dimensional case, however, this second bifurcation is of no concern.  The continuous elimination of the pairs $\{e_+^{t_+(\mu)},q_+^{t_+(\mu)}\}$ and $\{e_+^{t_-(\mu)},q_+^{t_-(\mu)}\}$ precludes this bifurcation, and the flow line of $\op{unstab}_{S_\mu}(q_-^{t_-(\mu)})$ emanating from $q_-^{t_-(\mu)}$ which variously converges to $e_-^{t_+(\mu)}$, $e_+^{t_+(\mu)}$, or $q_+^{t_+(\mu)}$ becomes a unique flow line to $e_-^{t_+(\mu)}$, for all $0\leq\mu\leq 1$.

Finally, we apply the Elimination Lemma as in the previous cases to conclude that $S_\mu$ is convex for all $0\leq\mu\leq 1$. 
\end{proof}

Hence $\xi$ is $\mu$-invariant on $T\simeq S\times[0,1]_\mu$ and $(T,\xi)$ is contactomorphic to a trivial bypass attachment.

\s\n
\textbf{Comparing to the bypass attachment along $\mathbf{(\Lambda_{\pm};D_{\pm})}$.}
\s
As in the $3$-dimensional case, we have normalized and decomposed $(M_\Sigma=\Sigma\times[-\delta,\delta],\xi=\ker\alpha_\Sigma)$ into three regions:
\begin{itemize}
    \item[(R1)] $G$, on which $\xi$ is $t$-invariant;
    \item[(R2)] $T=S\times [0,1]_\mu$, on which, by \cref{lemma:family-of-spheres-all-dim}, $\xi$ is standard in the sense that each $S_\mu$ is convex and $R_\pm(S_\mu)$ is deformation equivalent to the standard Weinstein ball with a unique critical point;
    \item[(R3)] $M_\Sigma\setminus(G\cup T)$, on which $\xi$ is standard in the sense that each $\Sigma_t\setminus(G\cup T)$ is convex.
\end{itemize}

On the other hand, we may assume the following holds for the abstract bypass cobordism $(M_\Sigma, \xi')$:
\begin{enumerate}[label=(\alph*)]
    \item $\xi=\xi'$ on $M_\Sigma\setminus T$;
    \item the bypass is attached to $\widehat{B}_{-t_0/2}\subset \Sigma_{-t_0/2}$ with data $(\Lambda_{\pm};D_{\pm})$ as given in the statement of \cref{prop:bb-correspondence}; and
    \item the contact $1$- and $2$-handles determined by $(\Lambda_{\pm};D_{\pm})$ are contained in $T$.
\end{enumerate} 

It remains to show that $\xi|_T$ and $\xi'|_T$ are contact isotopic relative to the boundary.

\begin{claim} \label{claim: trivial}
    When viewed as a bypass on $S$, the bypass data $(\Lambda_{\pm}; D_{\pm})$ is trivial of type (TB1). 
\end{claim}

See \cref{def:bypass} for the definition of a trivial bypass of type (TB1). Strictly speaking, in view of (b) above and \cref{eqn: t sub pm second version}, we take $S$ to be $S_{\mu=1/3}$.

\begin{proof}
We argue that $\Lambda_+$ is a standard Legendrian unknot filled by a standard Lagrangian disk $D_+$, and that $\Lambda_+$ is below $\Lambda_-$ in the sense of \cref{def:bypass}. While this is possible to confirm visually in dimension $3$ (c.f.\ \cref{fig:bypass-data-on-s-3d}), some more care must be taken in higher dimensions. 

First consider the folded Weinstein hypersurface $\widehat{B}_0\subset \Sigma_0$, after having applied the Creation Lemma in $W_1$.  Recall that the construction of $\widehat{B}_0$ creates a critical point $e_+$ in $W_1$ which is in canceling position with $q_+$, and that $\widehat{B}_0$ contains a neighborhood of the characteristic-stable manifolds of $q_+, q_-$, and $e_+$. Let $W_1'\subset W_1$ be the Weinstein subdomain such that $W_1$ is obtained from $W_1'$ by attaching an $(n-1)$-handle corresponding to $e_+$, and let $M^{2n-1}=\bdry W_1'$.  Note that we considered no such $W_1'$ in the 3-dimensional case, where $n=1$.

Next let $\mathcal{U} \subset M$ be a neighborhood of $(\op{stab}_{\Sigma_0}(e_+) \cup \op{stab}_{\Sigma_0}(q_+)) \cap M$. Since the handles associated to $q_+$ and $e_+$ are in geometrically canceling position by assumption, $\mathcal{U}$ is contactomorphic to a Darboux ball such that the attaching sphere of $q_+$ and the attaching region of $e_+$ are modeled on the spin-symmetric front projection on the right-hand side of \cref{fig:compare1}. A schematic depiction is on the left side of the figure, where the gray sheet (tipped vertically to more easily draw the subsequent handle attachments) represents the contact level set $M$, and the region between $M$ and $K_1$ is the Weinstein $(n-1)$-handle associated to $e_+$. Moreover, since the retrogradient from $q_-$ to $q_+$ meets $K_2$ along the $\xi_{K_2}$-transverse intersection $\Lambda_-^{\dagger} \pitchfork \Lambda_+$, the characteristic-stable manifold $\op{stab}_{\Sigma_0}(q_-)$, which by definition intersects $K_1$ along $\Lambda_-$, descends further and intersects $\mathcal{U}$ along a $\xi_{M}$-transverse intersection with the attaching sphere of $q_+$. (In the schematic left side of \cref{fig:compare1}, the negative $n$-handle with critical point $q_-$ is represented by the red stable and unstable disks. The positive $n$-handle with critical point $q_+$ is attached above $C_2$ so that its stable disk is in geometrically canceling position with the $(n-1)$-handle and simultaneously induces a retrogradient from $q_-$.) By shrinking $\mathcal{U}$ and performing a small isotopy if necessary, we may assume $\op{stab}_{\Sigma_0}(q_-)$ intersects the front projection on the right-hand side of \cref{fig:compare1}. 

\begin{figure}[ht]
	\centering
    \begin{overpic}[width=\textwidth]{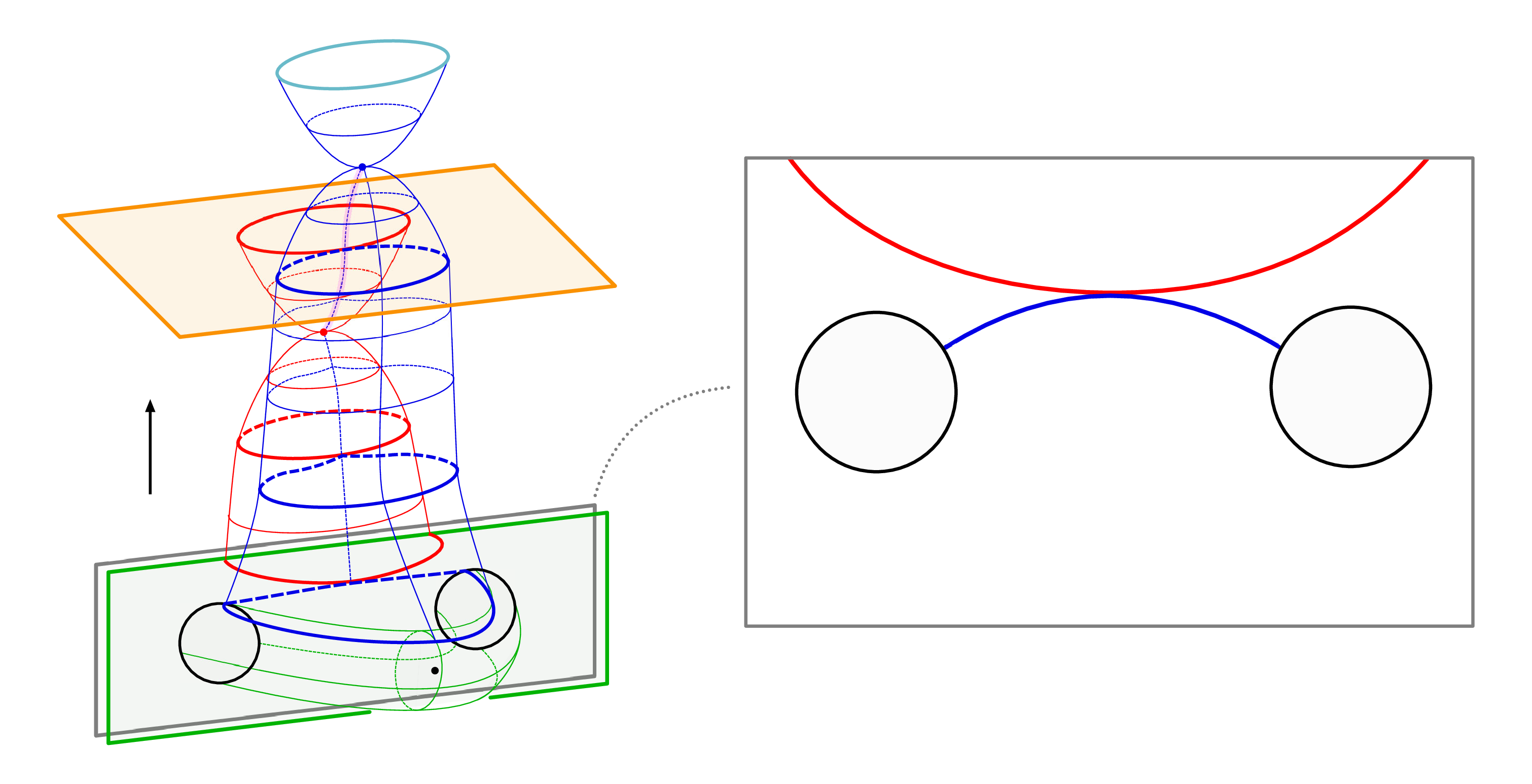}
       \put(4, 15){\footnotesize \textcolor{gray}{$M$}}
       \put(14, 1.25){\footnotesize \textcolor{divgreen}{$K_1$}}
       \put(4, 38.5){\footnotesize \textcolor{Orange}{$C_2$}}
       \put(30, 48){\footnotesize \textcolor{Cyan}{$\Lambda_+$}}
       \put(7.5,26){\footnotesize $(\Sigma_0)_{\xi}$}

       \put(19.75, 40){\tiny \textcolor{darkblue}{$q_+$}}
       \put(15, 28){\tiny \textcolor{darkred}{$q_-$}}
       \put(29.5, 7){\tiny \textcolor{black}{$e_+$}}

       \put(69.5, 7){\textcolor{gray}{$\mathcal{U}\subset M$}}
       \put(66.5, 28){\footnotesize \textcolor{darkblue}{$\op{stab}_{\Sigma_0}(q_+) \cap M$}}
       \put(66.5, 36){\footnotesize \textcolor{darkred}{$\op{stab}_{\Sigma_0}(q_-) \cap \mathcal{U}$}}
       
	\end{overpic}
	\caption{On the left is a schematic depiction of the folded Weinstein structure on $\Sigma_0$, after the Creation Lemma has been applied. The vertical arrow indicates the general direction of $(\Sigma_0)_\xi$ and the retrogradient from $q_-$ to $q_+$ is highlighted in pink. On the right is the neighborhood $\mathcal{U}\subset M$.}
	\label{fig:compare1}
\end{figure}

Next, we ``shift down'' and consider the characteristic foliation on $\widehat{B}_{-t_0/2}\subset \Sigma_{-t_0/2}$.  As described by \eqref{bbc:eq:b-hat-dynamics}, this shift perturbs the characteristic foliation of $\widehat{B}_{0}$ in a neighborhood of $K_2$, inducing a small negative Reeb shift to $\op{stab}_{\Sigma_{0}}(q_+)$.  Note that $\op{stab}_{\Sigma_{-t_0/2}}(q_-)$ avoids $K_2$, and is thus unaffected by the perturbation.  The resulting intersections $\op{stab}_{\Sigma_{-t_0/2}}(q_-)\cap\mathcal{U}$ and $\op{stab}_{\Sigma_{-t_0/2}}(q_+)\cap \mathcal{U}$ are seen in the first panel of \cref{fig:compare2}. Next, we (locally) obtain the dividing set of $S$ by attaching to $M$ the $(n-1)$-handle associated to $e_+$ to obtain $K_1$, attaching the Weinstein handle associated to $q_+$, and removing the (upside down) Weinstein handle associated to $q_-$. This gives the surgery diagram presentation of the dividing set $\Gamma$ in the middle frame. 

\begin{figure}[ht]
	\centering
    \begin{overpic}[width=\textwidth]{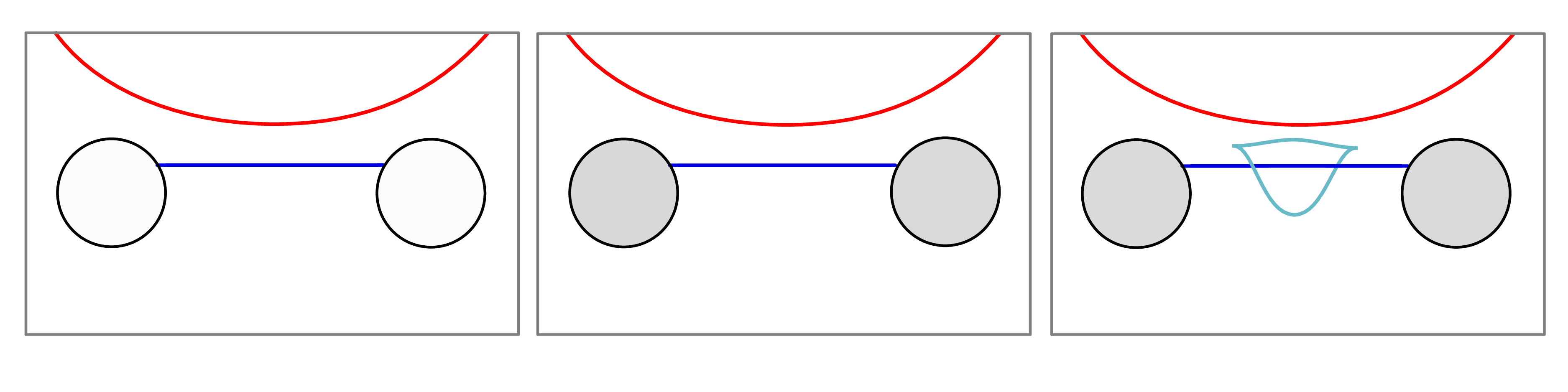}
        \put(16,1){\small \textcolor{gray}{$M$}}  
        \put(29,5){\small \textcolor{gray}{$M$}}  
        \put(49,1){\small \textcolor{divgreen}{$\Gamma$}} 
        \put(52,12.5){\footnotesize \textcolor{darkblue}{$(-1)$}} 
        \put(52,19){\footnotesize \textcolor{darkred}{$(-1)$}} 
        \put(61.75,5){\small \textcolor{gray}{$M$}} 
        \put(82.5,1){\small \textcolor{divgreen}{$\Gamma$}}
        \put(94.5,5){\small \textcolor{gray}{$M$}} 
        \put(85.25,12.5){\footnotesize \textcolor{darkblue}{$(-1)$}} 
        \put(85.25,19){\footnotesize \textcolor{darkred}{$(-1)$}} 
        \put(82,9){\footnotesize \textcolor{Cyan}{$\Lambda_+$}} 
	\end{overpic}
	\caption{In the middle panel we attach the handle associated to $e_+$ (denoted by shading the attaching region gray), attach the Weinstein handle associated to $q_+$, and remove the negative Weinstein handle associated to $q_-$. The resulting surgeries present the dividing set $\Gamma$ of the convex sphere.}
	\label{fig:compare2}
\end{figure}

In the front projection, the positive part of the bypass data $(\Lambda_{\pm}; D_{\pm})$ admits the following description: The disk $D_+$ is the cocore of the handle associated to $q_+$ and $\Lambda_+$ is its belt sphere; a slight down-shift of this data is given by the cocore and belt sphere of the blue surgery in \cref{fig:compare2}. By \cite{casals2019fronts}, this data is Legendrian isotopic to the standard disk filling of the arbitrarily small meridional standard unknot drawn in the third frame. In particular, since the negative bypass data is contained in the red surgery region, we see that (after canceling the handles which are in canceling position) $\Lambda_+$ is a standard unknot with a standard disk filling which is below $\Lambda_-$.
\end{proof}

By \cref{claim: trivial}, $(\Lambda_{\pm}; D_{\pm})$ defines a trivial bypass and hence $(T,\xi'|_T)$ is also standard (see (R2)).  This completes the proof of \cref{prop:bb-correspondence}.

\part{Applications}\label{part:OBD}

\section{The existence of (partial) open book decompositions} \label{sec:giroux correspondence}

The goal of this section is to prove \cref{cor:OBD existence}, a stronger and more precise version of \cref{cor:POBD existence}, {and \cref{cor: Legendrians on a page}.} The proofs are, again, essentially the same as the proofs in the $3$-dimensional case; see \cite{Gi02} for the absolute case and \cite{HKM09} for the relative case.

\begin{proof}[Proof of \cref{cor:OBD existence}]
Let $(M,\xi)$ be a {closed connected} contact manifold of dimension $2n+1$. Choose a generic self-indexing Morse function $f: M \to \R$. Then the regular level set $\Sigma \coloneqq f^{-1} (n+\tfrac{1}{2})$ is a smooth hypersurface which divides $M$ into two connected components {$M \setminus \Sigma = Y_0 \cup Y_1$.} It follows that $Y_i, i=0,1$, deformation retracts (along $\pm \nabla f$) to the skeleton $\Sk(Y_i)$, which is a finite $n$-dimensional CW-complex.

Writing $Y$ for either of $Y_0$ or $Y_1$, we now construct $N(\Sk(Y))$ as a compact contact handlebody.  There exists a neighborhood of the $0$-cells of $Y$ that can be written as a contact handlebody $H_0=[-1,1]\times W_0$, where $W_0$ is Weinstein.  Arguing by induction, assume that  a neighborhood of  the $k$-skeleton of $\Sk(Y)$ can be realized as a contact handlebody $H_k=[-1,1]\times W_k$, where $W_k$ is Weinstein and $\Gamma_k=\{0\}\times \bdry W_k$ is the dividing set of $\bdry H_k$.  We explain how to attach the $(k+1)$-handles to $\bdry H_k$, where $k+1\leq n$,  in a contact manner.  Write $K$ for the {core disk} of a $(k+1)$-handle. Then $\dim \bdry K=k$ and by dimension reasons $\bdry K\subset \bdry H_k$, after possible perturbation, can be isotoped into $\Gamma_k$ using the Liouville flow on $W_k$. Let $v$ (resp.\ $w$) be a nonvanishing ($TM$-valued) vector field along $\Gamma_k$ that is tangent to $\xi$, transverse to $\bdry H_k$ (resp.\ transverse to $\Gamma_k$ and tangent to $\bdry H_k$), and is symplectically orthogonal to $\xi|_{\Gamma_k}$. 

Next we would like to isotop $\bdry K$ to an isotropic submanifold $\bdry K'$ in $\Gamma_k$ (it may be Legendrian if $k+1=n$) and then isotop $K$ to an isotropic submanifold {$K'\subset Y\setminus \op{int}H_k$} with boundary $\bdry K'$, using Gromov's $h$-principle \cite[p.\ 339]{Gro86} for isotropic submanifolds in a contact manifold. To this end we show that:

\begin{claim} \label{claim: formal}
There is a formal isotopy $\phi_t:TK\to TM$, $t\in[0,1]$, such that:
\be
\item $\phi_0$ is the derivative of the inclusion map;
\item $\phi_t$ is an injective bundle map for all $t\in[0,1]$;
\item the fibers of $\phi_1(TK)$ are $\xi$-isotropic; 
\item the fibers of $\phi_t(TK)$ for all $t\in [0,1]$ along $\bdry K$ have the form $\R\langle v\rangle$ times a plane of $\xi':=\xi\cap T\Gamma_k$; and
\item when $t=1$, the planes in (4) are $\xi'$-isotropic.
\ee
\end{claim}

\begin{proof}[Proof of \cref{claim: formal}]
We will explain the Legendrian (i.e., $k+1=n$) case, which is the hardest case.  Since $K$ is a disk, it is clearly formally Legendrian inside its disk neighborhood $N(K)$.  The key point is to make $\bdry K$ formally isotropic as well.

Let $\tau$ be a trivialization of $\xi|_{N(K)}$. Projecting out the Reeb direction and using the trivialization $\tau$, the embedding $K\hookrightarrow N(K)$ can be converted into the map $\phi^\flat_0: K\to G(n,2n)$, where $G(n,2n)$ is the Grassmannian of $n$-planes in $\R^{2n}$. Since $K$ is a disk, $\phi^\flat_0$ is homotopic to $\phi^\flat_{1/2}: K\to \mathcal{L}_n$, where $\mathcal{L}_n\subset G(n,2n)$ is the Lagrangian Grassmannian.  Next, we would like to further homotop $\phi^\flat_{1/2}$ to $\phi^\flat_1:K\to\mathcal{L}_n$ such that $v(x)\in \phi^\flat_1(x)$ for all $x\in \bdry K$. At this point we note that over $\bdry K$ the trivial complex bundle $\xi$ satisfies
$$\xi\simeq \R\langle v,w\rangle \oplus \xi'$$ 
and that the classification of complex vector bundles over $\bdry K \simeq S^{n-1}$ is given by $\pi_{n-2}U\simeq 0$ or $\Z$. In the former case, $\xi'$ is a trivial complex vector bundle; in the latter case, $\xi'$ is classified by its $(n-1)/2$nd Chern class, but then $\C \oplus \xi'$ is trivial, so the Chern class must vanish and $\xi'$ must be trivial.  Hence we can view the desired map $\phi^\flat_1$ as restricting to $\bdry K\to \mathcal{L}_{n-1}\subset \mathcal{L}_n$, corresponding to a standard inclusion $\R^{2n-2}\hookrightarrow \R^{2n}$.

We now claim that $\pi_{n-1}\mathcal{L}_{n-1}\to \pi_{n-1}\mathcal{L}_n$ is surjective.  Using the fact that $\mathcal{L}_n= U(n)/O(n)$, we have:
$$\begin{CD}
 \pi_{n-1} U(n) @>>> \pi_{n-1} \mathcal{L}_n @>>> \pi_{n-2} O(n) @>>> \pi_{n-2} U(n) \\
  @ AAaA @AA b A @ AAcA @AAdA\\
 \pi_{n-1} U(n-1) @>>> \pi_{n-1} \mathcal{L}_{n-1} @>>> \pi_{n-2} O(n-1) @>>> \pi_{n-2} U(n-1)
\end{CD}$$
Using the homotopy exact sequences for $U(n)/U(n-1)= S^{2n-1}$ and $O(n)/O(n-1)=S^{n-1}$, it follows that $a,c$ are surjective and $d$ is injective.  The claim then follows from the five lemma.

The claim then allows us to homotop $\phi^\flat_{1/2}$ to $\phi^\flat_1$ such that, viewed as family of maps $\phi_t: TK\to TM$, $t\in[0,1]$, Conditions (1)--(3) and (5) hold and (4) holds for $t=0,1$. It remains to modify $\phi_t$ so that (4) holds for all $t\in[0,1]$: Observe that $Z:=\sqcup_{t\in[0,1]}\phi_t(TK|_{\bdry K})$ has dimension $2n$. Since $\phi_t$ may be taken to be generic, $Z\pitchfork \R\langle w\rangle$ and projecting to $\R\langle v\rangle \oplus \xi'$ along $\R\langle w\rangle$ is an isomorphism for each $\phi_t(T_xK)$, $t\in[0,1]$, $x\in \bdry K$. Hence (4) is also satisfied and \cref{claim: formal} follows.
\end{proof}

Hence  a neighborhood of the $(k+1)$-skeleton of $\Sk(Y)$ can be realized as a contact handlebody $H_{k+1}=[-1,1]\times W_{k+1}$  and $N (\Sk(Y_0)) \cup N (\Sk(Y_1))$ can be realized as compact contact handlebodies with sutured convex boundary and its complement in $M$ has sutured concave boundary.

Now identify $M \setminus (N (\Sk(Y_0)) \cup N (\Sk(Y_1)))$ with $\Sigma \times [0,1]$ such that if we write $\Sigma_t \coloneqq \Sigma \times \{t\}$, then $\Sigma_i = \p N (\Sk(Y_i)), i=0,1$, are convex with dividing sets corresponding to the sutures. By \cref{thm:family genericity}, $\xi|_{\Sigma \times [0,1]}$ is given by a finite sequence of bypass attachments, which can be further turned into a sequence of modifications of the trivial Weinstein POBD of
\begin{equation*}
N (\Sk(Y_0)) = [-1,1] \times W_{0,n},
\end{equation*}
according to \cref{lem:bypass to POB} (here $W_{0,n}$ is $W_n$ for $Y_0$). In this way we obtain a  $1$-Weinstein POBD of $M \setminus N (\Sk(Y_1))$, viewed as a contact manifold with sutured concave boundary. 
It remains to fill in $N (\Sk(Y_1))$ in the obvious manner to get a compatible  $1$-Weinstein  OBD.  (Alternatively, we can attach all the contact $n$-handles of the bypass attachments to $N(\Sk(Y_0))$, which gives a generalized contact handlebody, and all the contact $(n+1)$-handles of the bypass attachments to $N(\Sk(Y_1))$, which also gives a generalized contact handlebody.  Hence we can view the OBD as consisting of two halves and each half is a generalized contact handlebody foliated by $1$-Weinstein domains.) 
\end{proof}

Next we turn to the relative case, i.e., to contact manifolds with boundary. The following theorem is a more precise version of \cref{cor:POBD existence} (there, the boundary condition was vaguely stated). 

\begin{theorem} \label{thm:strong POBD existence}
If $(M,\xi)$ is a compact contact manifold with sutured concave boundary  and $R_\pm(\bdry M)$ are Weinstein, there exists a compatible $1$-Weinstein POBD where each of $R_+(\bdry M)$ and $R_-(\bdry M)$ extends to a page.
\end{theorem}

\begin{proof}
Choose a generic self-indexing Morse function $f:M \to [\tfrac{1}{2},\infty)$ such that $f \equiv \tfrac{1}{2}$ on $\p M$. In other words, $f$ has no index $0$ critical points. As in the absolute case, consider the hypersurface $\Sigma \coloneqq f^{-1} (n+\tfrac{1}{2})$ which divides $M$ into two components $Y_i, i=0,1$, such that $Y_0$ contains all the critical points of index at most $n$ and $Y_1$ contains all the critical points of index at least $n+1$. By the handle attachment discussion in the proof of \cref{cor:OBD existence}, we can turn the critical points in $Y_0$ into isotropic handles attached to $\p M$ along the suture which we still denote by $N(\Sk(Y_0))$ although it is no longer a contact handlebody, and the critical points in $Y_1$ into the handle decomposition of a contact handlebody $N (\Sk(Y_1))$ with sutured convex boundary, as in the closed case. The rest of the proof proceeds as in the closed case.
\end{proof}

\begin{proof}[Proof of \cref{cor: Legendrians on a page}]
{We modify the proof of \cref{cor:OBD existence} as follows: In the decomposition of $M$ into contact handlebodies $Y_0$ and $Y_1$, we first take a standard contact neighborhood $N(\Lambda)$ of the Legendrian $\Lambda$ and a corresponding contact Morse function $f$ on $N(\Lambda)$ that is constant on $\bdry N(\Lambda)$. We then extend $f$ arbitrarily to a Morse function which is self-indexing on $M\setminus N(\Lambda)$, realize the $k$-handles with $k\leq n$ as contact handles, and attach them to $N(\Lambda)$ to obtain $Y_0$.  The rest of the proof is the same as that of \cref{cor:OBD existence}.}
\end{proof}

\clearpage 
\appendix
\addtocontents{toc}{\protect\vskip12pt}
\crefalias{section}{appendix}

\section{Bypass attachments in higher dimensions}\label{appendix}

As defined in \cite{Hon00}, a \emph{bypass half-disk} for a convex surface $\Sigma$ in a contact $3$-manifold $(M, \xi)$ is a convex $2$-disk with piecewise-smooth boundary $a \cup b$, where 
\begin{enumerate}
    \item $a$ is an arc embedded on $\Sigma$, transversally intersecting the dividing set of the latter thrice (at $\partial a$, and once in the interior of $a$), 
    \item $b$ is properly embedded in the complement of $\Sigma$, and 
    \item the dividing set of the half-disk consists of a single arc with endpoints connecting the two components of $a \setminus \Gamma_{\Sigma}$. 
\end{enumerate}
The third author showed in \cite{Hon00} that if one isotopes $\Sigma$ past the half-disk, the result is a new convex surface $\Sigma'$ whose dividing set differs from that of $\Sigma$ via a combinatorial modification relative to the arc $a$; see \cref{fig:bypass-half-disk}. 

\begin{figure}[ht]
	\centering
	\begin{overpic}[scale=.42]{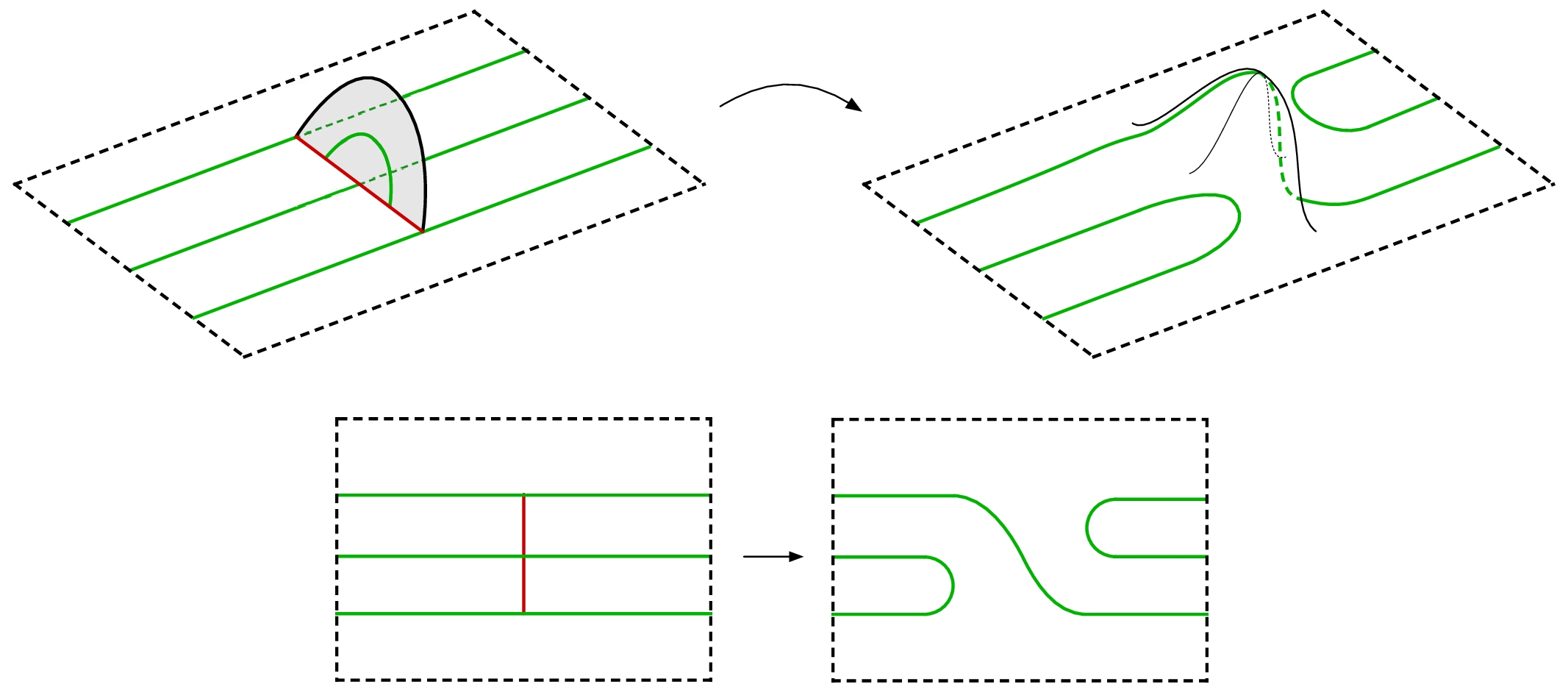}
        \put(32.75,14){\textcolor{darkred}{$a$}}
        \put(4,24){$\Sigma$}
        \put(89,24){$\Sigma'$}
	\end{overpic}
	\caption{Isotoping past a bypass half-disk.}
	\label{fig:bypass-half-disk}
\end{figure}

The trace of the isotopy from $\Sigma$ to $\Sigma'$, after a perturbation, can be assumed to be diffeomorphic to an embedding $\Sigma \times [0,1]_t \hookrightarrow M$ where $\Sigma \times \{0\} = \Sigma$ and $\Sigma \times \{1\} = \Sigma'$; we will refer to such a submanifold as a \emph{bypass cobordism} associated to a bypass half-disk. As the dividing sets of $\Sigma \times \{0\}, \Sigma \times \{1\}$ may differ, $\Sigma \times \{t\}$ is not necessarily convex for all $t$ (otherwise the dividing sets of $\Sigma$ and $\Sigma'$ would be the same). In fact, one can arrange for convexity to fail at precisely $t=\frac{1}{2}$, and nowhere else. A bypass cobordism from $\Sigma$ to $\Sigma'$ therefore describes a momentary failure of convexity in a $1$-parameter family of surfaces; see the \textit{isotopy discretization} of Colin \cite{colin1997discretization,honda2002gluing}. \cref{part:BBC} shows that, in fact, every instance of failure of convexity, in any dimension, is described by a higher-dimensional version of a bypass cobordism (\cref{thm:BBC}, \cref{prop:bb-correspondence}). 

The perspective of bypasses most appropriate for us --- specifically, the perspective which motivates the higher-dimensional definition of a bypass --- is the decomposition of a bypass cobordism into a sequence of smoothly canceling contact handles. In dimension $3$, this was elucidated by Ozbagci \cite{ozbagci2011contact} and is a reformulation of Giroux's contact cell decompositions \cite{Gi02}. Briefly, a $3$-dimensional contact $1$-handle is a smooth $1$-handle endowed with a contact structure $\xi$ and contact vector field $v$ so that $v$ is inwardly transverse to the attaching region, outwardly transverse to the co-attaching region, has an index $1$ critical point in its interior, and whose dividing curves with respect to $v$ are as depicted in the top middle frame of \cref{fig:bypass-contact-handles}. Contact topologically, a contact $1$-handle is a standard jet-bundle neighborhood of its core Legendrian arc. The model for a contact $2$-handle is obtained by considering the contact vector field $-v$ and reversing the roles of the attaching and co-attaching region. Contact handle attachment is performed in a way which respects the dividing sets, positive, and negative regions. The attaching sphere of a contact $2$-handle is $D_+ \cup D_-$, where $D_{\pm}$ is a properly embedded arc in $R_{\pm}$ of the attaching region, and the core may be taken to be foliated by Legendrian arcs interpolating from $D_+$ to $D_-$; see \cref{fig:n+1}. 

\begin{figure}[ht]
	\centering
	\begin{overpic}[scale=.42]{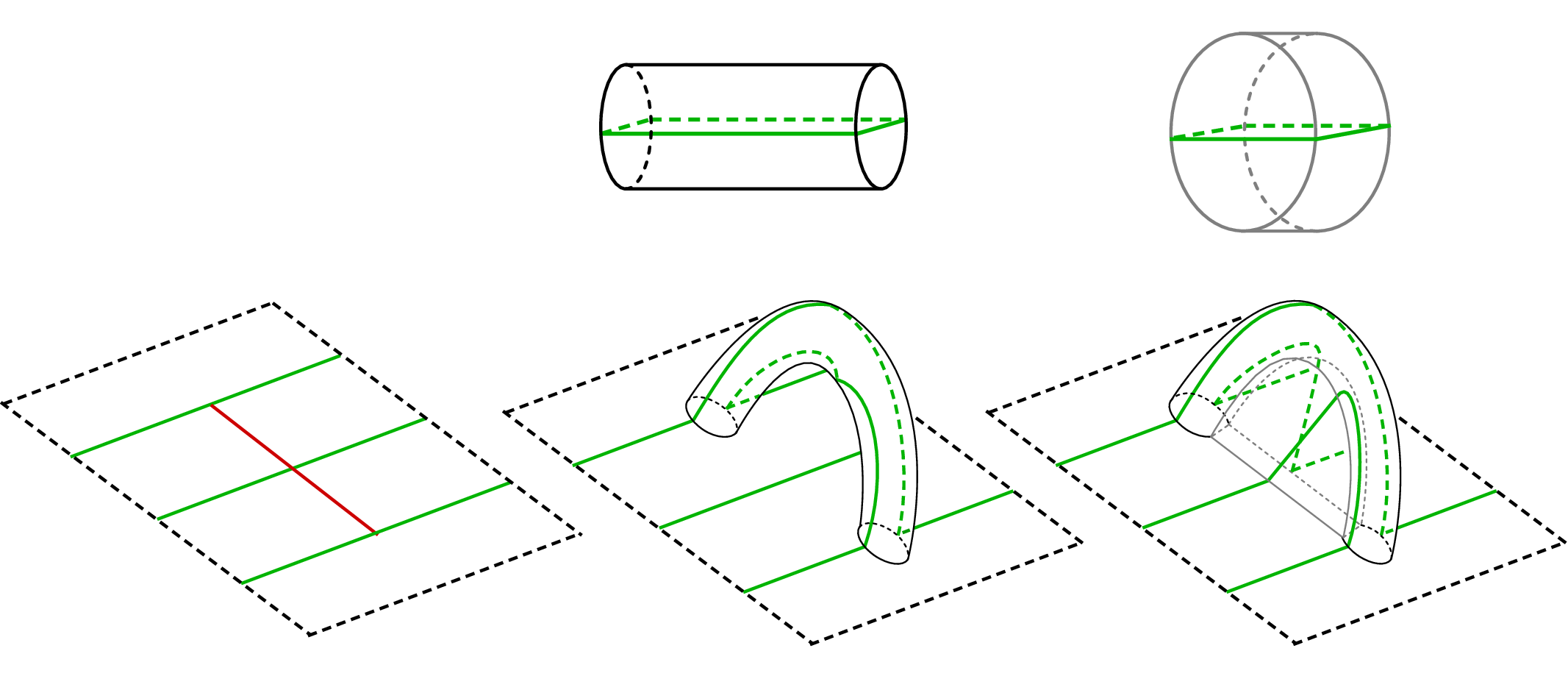}
       
	\end{overpic}
	\caption{A $3$-dimensional bypass cobordism described as a contact $1$-handle and contact $2$-handle attachment.}
	\label{fig:bypass-contact-handles}
\end{figure}

The bypass cobordism associated to a bypass-half disk is given by first attaching a contact $1$-handle along $\partial a$ with core $b$. Along the boundary of the $1$-handle, the dividing curves make a left-handed half-rotation relative to $\Sigma$, as depicted in the lower middle frame of \cref{fig:bypass-contact-handles}. A smoothly-canceling contact $2$-handle, whose core is approximately the bypass half-disk itself, completes the bypass attachment as depicted on the right side of the same figure. 

The reason we adopt the contact handle formulation of a bypass attachment is because one can translate naturally between the language of contact handle decompositions and supporting open book decompositions: in dimension $3$, a contact $1$-handle adds a $1$-handle to the page of an associated (partial) open book decomposition, and a contact $2$-handle builds in partially-defined monodromy from a neighborhood of $D_+$ to a neighborhood of $D_-$; again, see \cref{fig:n+1}. From this perspective, a \textit{trivial bypass} --- one which does not change the isotopy type of the dividing curves --- corresponds to a positive (partial) open book stabilization.

\begin{figure}[ht]
	\centering
	\begin{overpic}[scale=.2]{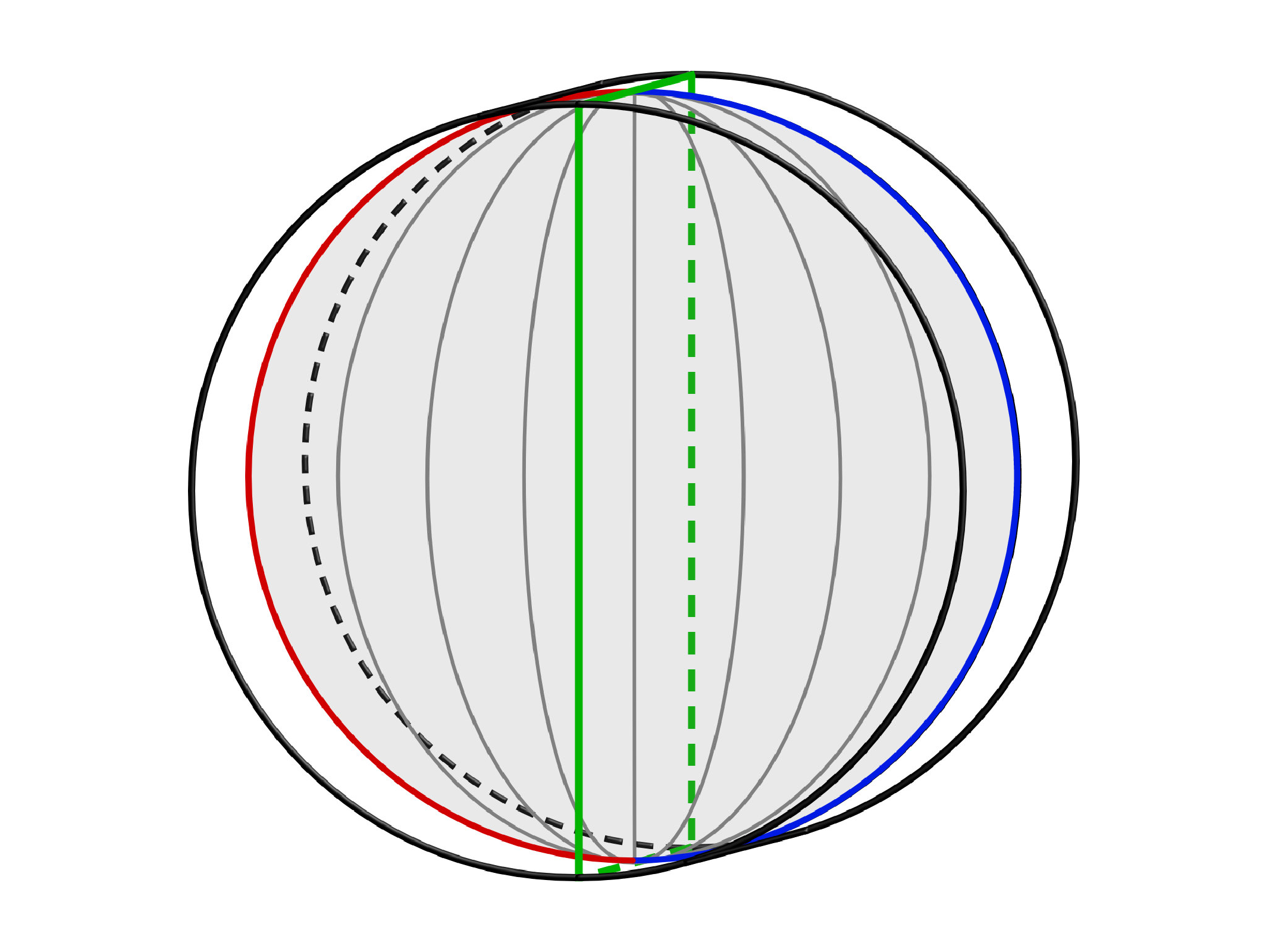}
        \put(10,24){\footnotesize \textcolor{darkred}{$D_-$}}
        \put(85,50){\footnotesize \textcolor{darkblue}{$D_+$}}
	\end{overpic}
	\caption{The structure of a contact $2$-handle, with core disk in gray foliated by Legendrian arcs interpolating from $D_+$ to $D_-$. }
	\label{fig:n+1}
\end{figure}

The goal of the present appendix is to develop this framework in higher dimensions. Much of this material originally appeared in the preprint \cite{HH18} and was reformulated in the preliminary material of sequel papers such as \cite{BHH23,breen2027bypass}. Here we collect and further reformulate all definitions, statements, and proofs which are necessary to independently support the results in \cref{part:BBC} and the applications in \cref{part:OBD}.

\s\n
{\em Organization.} In \cref{subsec:contact-handles} we detail the construction of contact $n$-handles and contact $(n+1)$-handles, with the main results being \cref{prop:n_handle} and \cref{prop:(n+1)_handle}. \cref{subsec:POBD} then discusses partial open book decompositions, in particular their interaction with contact handles (\cref{lemma:handle-to-POBD}) and the notion of positive stabilization (\cref{lem:positive_stabilization}). Finally, in \cref{subsec:bypass-appendix} we define the higher-dimensional bypass attachment. In the intermediate it is necessary to recall some notions of Legendrian Kirby calculus largely developed by Casals-Murphy \cite{casals2019fronts}; these results are collected in \cref{subsubsec:sums-slides}. The main result describing a bypass attachment is then \cref{prop:bypass_attachment}.

\begin{remark}
In the present appendix, we will typically not distinguish between a Liouville domain and its (ideal) completion (c.f.\ \cite{giroux2020ideal}). 
\end{remark}

\subsection{Contact handles}\label{subsec:contact-handles}

A bypass cobordism will eventually be defined as a contact $n$-handle followed by a smoothly canceling contact $(n+1)$-handle. Here we describe these handles in detail; see also \cite{Sac19}.

\subsubsection{Contact $n$-handles}\label{subsubsec:n_handle}

One ``sutured" model for a contact $n$-handle is a contactization of a critical Weinstein handle: if $(h_n^{2n}, \lambda_n)$ is a Weinstein $n$-handle with Liouville vector field $X_{\lambda_n}$, then one may consider the contact manifold with contact vector field given by
\[
([-1,1]_z\times h_n^{2n},\, \alpha_n = dz + \lambda_n,\, v_n = z\, \partial_z + X_{\lambda_n}).
\]
To be explicit in a manner conducive to the requisite technical details, we instead define a \emph{contact $n$-handle} as a triple $(H_n, \alpha_n, v_n)$, where 
\begin{equation} \label{eqn: n-handle}
    H_n = \left\{ |x| \leq 1 \right\} \times \left\{ z^2 + |y|^2 \leq 1 \right\} \subset \R_x^n \times \R_y^n \times \R_z
\end{equation}
is a smooth $n$-handle, $\alpha_n = dz - 2y \cdot dx - x \cdot dy$ is a contact form, and $v_n = -x \cdot \partial_x + 2y \cdot \partial_y + z \,\partial_z$ is a contact vector field which is gradient-like for the Morse function $z^2+|y|^2-|x|^2$ with a unique critical point of index $n$ at the origin.  The core disk $D = \{z = |y|= 0\}$ of $H_n$ is Legendrian and there is a contact embedding $H_n\hookrightarrow J^1(D^n)$, where $J^1(D^n) \cong \R \times T^*D^n$ is the standard $1$-jet space of $D^n$ and $D$ is taken to the $0$-section.

We have the following description of a contact $n$-handle attachment to a convex hypersurface:

\begin{proposition} \label{prop:n_handle}
Let $(M,\xi)$ be a $(2n+1)$-dimensional contact manifold with convex boundary $\Sigma = \partial M = R_+ \cup \Gamma \cup R_-$, where $\Gamma=\bdry R_\pm$. Let $\Lambda$ be an $(n-1)$-dimensional Legendrian sphere in $(\Gamma, \xi_{\Gamma})$. Then a contact $n$-handle can be attached to $(M, \xi)$ along $\Lambda$ to yield a new contact manifold $(M', \xi')$ with convex boundary $\Sigma' = \partial M' = R'_+ \cup  \Gamma'\cup R'_-$ such that:
		\begin{enumerate}
			\item $R'_+$ is obtained by attaching a Weinstein $n$-handle to $R_+$ along $\Lambda\subset \Gamma=\partial R_+$;
		
			\item $R'_-$ is obtained by attaching a Weinstein $n$-handle to $R_-$ along $\Lambda \subset \Gamma= \partial R_-$;
		
			\item as a contact manifold, $\Gamma'$ is obtained from $\Gamma$ by performing a contact $(-1)$-surgery along $\Lambda \subset \Gamma$; and

			\item as a smooth manifold, $\Sigma'$ is diffeomorphic to the union of $R'_+$ and $R'_-$, glued along their boundaries by the identity map.
		\end{enumerate}
\end{proposition}

Before proving the proposition, we proceed with some preliminary work. We decompose the boundary of a contact $n$-handle as $\partial H_n = \partial_1 H_n \cup \partial_2 H_n$, where
\begin{align*}
    \partial_1 H_n &= \left\{ |x| = 1 \right\} \times \left\{ z^2 + |y|^2 \leq 1 \right\}, \\
    \partial_2 H_n &= \left\{ |x| \leq 1 \right\} \times \left\{ z^2 + |y|^2 = 1 \right\}.
\end{align*}
Both $\partial_1 H_n$ and $\partial_2 H_n$ are convex with respect to $v_n$, with $v_n$ pointing into $H_n$ along $\partial_1 H_n$ and out of $H_n$ along $\partial_2 H_n$.

We analyze $\partial_1 H_n$ in some detail. Observe that $f(x,y,z)=z$ is the contact Hamiltonian for the contact vector field $v_n$. Since $\alpha_n(v_n)=0$ if and only if $z=0$, the $v_n$-dividing set is
\begin{equation*}
	\Gamma_{\partial_1 H_n} = \left\{ z=0 \right\} \times \left\{ |x| =1 \right\} \times \left\{ |y| \leq 1 \right\},
\end{equation*}
equipped with the induced contact form $\alpha_n|_{\Gamma_{\partial_1 H_n}} = -2y \cdot dx - x \cdot dy$. Consider $J^1(S^{n-1})$ with the standard contact form $du - p \cdot dq$, where $u \in \R, q,p \in \R^n$ are coordinates such that $|q|=1$ and $p \cdot q = 0$. Then there exists a contact embedding
\begin{gather}\label{eqn:attaching_region_n}
	\phi: (\Gamma_{\partial_1 H_n}, \alpha_n|_{\Gamma_{\partial_1 H_n}}) \hookrightarrow J^1(S^{n-1}),\\
\nonumber (x,y)\mapsto (u,q,p)=(-x\cdot y,\,x, y-(x\cdot y)\, x),
\end{gather}
which sends the boundary $\partial D$ of the core disk to the $0$-section of $J^1(S^{n-1})$. Moreover, if we identify $J^1(S^{n-1})$ with the ideal boundary of a positive half-symplectization $[c,\infty)\times J^1(S^{n-1})$, then $R_+ (\partial_1 H_n)$ is identified with a collar neighborhood of 
$$S^{n-1} \subset  \left\{ \infty \right\}\times J^1(S^{n-1}) \subset [c,\infty]\times J^1(S^{n-1}),$$ 
where the first inclusion is given by the $0$-section.  A similar identification holds for $R_- (\partial_1 H_n)$.

\begin{claim} \label{claim: def equiv}
The symplectic manifolds $(R_+ (\partial_2 H_n), \alpha_n|_{R_+ (\partial_2 H_n)})$ and $(R_- (\partial_2 H_n), \alpha_n|_{R_- (\partial_2 H_n)})$ are exact deformation equivalent (relative to the boundary) to the critical Weinstein handle $(T^\ast D^n, -2p \cdot dq - q \cdot dp)$.
\end{claim}

\begin{proof}
Again, since $f(x,y,z)=z$ is the contact Hamiltonian for $v_n$, 
$$R_+(\partial_2 H_n)= \bdry_2 H_n \cap \{z>0\}.$$ Hence we can use coordinates $(x,y)$ on $R_+(\partial_2 H_n)$ and view $z$ as a function $z(x,y)$. Then
\[
\alpha_n|_{R_+ (\partial_2 H_n)}=dz-2y\cdot dx - x\cdot dy.
\]
We can then interpolate from $dz-2y\cdot dx - x\cdot dy$ to $-2y\cdot dx - x\cdot dy$ by taking $tdz -2y\cdot dx-x\cdot dy$ for $t\in[0,1]$. The $R_-$ case is analogous.
\end{proof}

\begin{proof}[Proof of \cref{prop:n_handle}.]
Fix an identification $i:\partial D\stackrel\sim\to\Lambda$. By the above discussion, $i$ can be extended to a contactomorphism $i:\Gamma_{\partial_1 H_n}\stackrel\sim\to N_\Gamma(\Lambda)$, where $N_\Gamma(\Lambda)$ is a standard tubular neighborhood of $\Lambda$ in $\Gamma$. Next, $i$ extends to an identification of contact germs on $\partial_1 H_n$ and a tubular neighborhood of $N_\Gamma(\Lambda)$ in $\Sigma$. Finally, using $i$ we attach $H_n$ to $(M,\xi)$ along $\Lambda \subset \Sigma$ and ``round the corners'' as in the symplectic case to obtain $(M',\xi')$ with convex boundary.  More precisely, we slightly enlarge $H_n$ to $H_n'$ near $\bdry (\bdry_1 H_n)$ so that attaching $H_n'$ yields $(M',\xi')$ with smooth and convex boundary. Letting $\bdry_1 H_n'$ be the portion of $\bdry H_n'$ that is attached to $(M,\xi)$ and $\bdry_2 H_n'$ its complement, \cref{claim: def equiv} also applies to $R_\pm(\bdry_2 H_n')$. (We may need to slightly adjust the contact vector field $v_n$ on $H_n$ so that it agrees with the assumed contact vector field witnessing convexity of $\partial M$ on the overlap.)

Claims (1) and (2) follow from \cref{claim: def equiv}, which identifies $R_+ (\partial_2 H_n)$ with (the completion of) a Weinstein handle. Since $H_n$ is attached along $\Gamma$, $R'_+ = R_+ (\Sigma')$ is obtained by attaching a Weinstein handle to $R_+$ along $\Lambda$. Claims (3) and (4) are immediate.
\end{proof}

\subsubsection{Contact $(n+1)$-handles}\label{subsec:(n+1)_handle}

One can quickly define a model for a contact $(n+1)$-handle by taking a contact $n$-handle $(H_n, \alpha_n, v_n)$ and considering the triple $(H_n, \alpha_n, -v_n)$. Observe that $-v_n$ is a gradient-like contact vector field for the Morse function $|x|^2-|y|^2-z^2$, which has a unique critical point of index $n+1$ at the origin. This point of view, however, has the drawback that the contact germ on the attaching region of the $(n+1)$-handle is harder to characterize. Instead, we consider the standard contact neighborhood of an $(n+1)$-dimensional disk foliated by $n$-dimensional Legendrian disks. This approach is more complicated but will be useful for constructing bypass attachments.

\begin{definition}
An {\em $(n+1)$-dimensional $\Theta$-disk} is an embedded $(n+1)$-dimensional disk $D$ in a contact $(2n+1)$-manifold $(M,\xi)$ such that $\xi\cap TD$ is diffeomorphic to the singular foliation $\ker \beta$ of the model disk $\left\{ z^2 + |x|^2 \leq 1 \right\} \subset \R^{n+1}_{z,x}$, where
\begin{equation} \label{eqn:unknot_foliation}
	\beta = (z^2 - |x|^2 +1) \, dz + 2zx \cdot dx= (z^2-|x|^2+1)\, dz + z\, d(|x|^2).
\end{equation}
\end{definition}

The model foliation is a singular, radially (in the $|x|$ variable) symmetric foliation on $\Theta$ such that $\partial \Theta$ is a closed leaf, all other leaves are disks, and the singular locus is precisely $\left\{ z=0, |x|=1 \right\} \subset \partial \Theta$; see \cref{fig:std_foliation}. 
According to \cite[Lemma~3.8]{Hua15}, there exists a unique contact germ on $\Theta \subset \Theta \times \R^{n}_{y}$, embedded as the $0$-section, such that all leaves are Legendrian.

\begin{figure}[ht]
	\begin{overpic}[scale=.18]{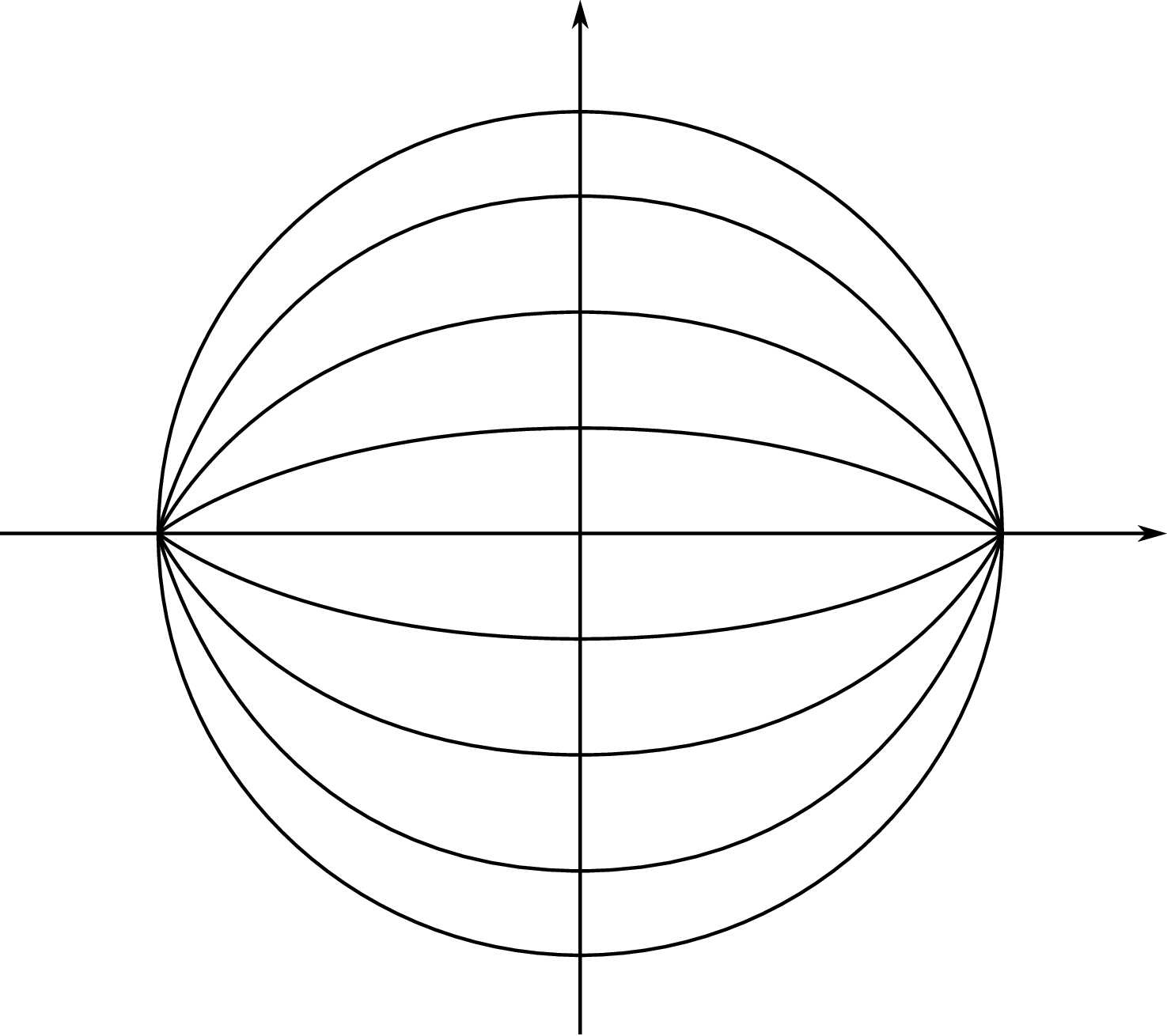}
		\put(45,90){$z$}
		\put(100,39){$x$}
	\end{overpic}
	\caption{A $\Theta$-disk in the $n=1$ case. The higher-dimensional case can be obtained by applying an $S^{n-1}$-spin symmetry about the $z$-axis.}
	\label{fig:std_foliation}
\end{figure}

We define the contact form 
	\begin{equation*}
		\alpha_{n+1} = (z^2 - |x|^2 +1) \, dz + 2zx \cdot dx + x \cdot dy
	\end{equation*}
on $\Theta \times \R^n$, whose contact condition the reader can readily verify using 
$$d\alpha_{n+1}=-4x\cdot dxdz +dx\cdot dy.$$
Observe that $\alpha_{n+1}|_{\Theta} = \beta$, which implies that all leaves of the induced foliation are Legendrian as desired. Moreover, since each $\partial_{y_i}$ is a contact vector field, we can assume that $\Theta \subset \Theta \times \R^n$ has an arbitrarily large neighborhood with respect to the contact form $\alpha_{n+1}$.

Next we construct a contact vector field $v_{n+1}$ on $(\Theta \times \R^n, \alpha_{n+1})$ which is gradient-like for some Morse function on $\Theta \times \R^n$ with a unique critical point of index $n+1$ at the origin.  Let $v_{n+1}$ be the contact vector field associated to the contact Hamiltonian function $f(x,y,z) = -2z + x \cdot y$ and contact form $\alpha_{n+1}$ (i.e., $\alpha_{n+1}(v_{n+1})=f$). One can compute that:
\begin{align*}
	(z^2 + 3|x|^2 +1)\, v_{n+1}  = & -(2 - 2|x|^2) z\,  \partial_z - (z^2 - |x|^2 +3) \, x \cdot \partial_x\\
 & \qquad + \left( (z^2 + 3|x|^2 +1)\, y - 4zx \right) \cdot \partial_y.
\end{align*}

Observe that $v_{n+1}$ has a unique zero at the origin. Moreover, in a small neighborhood of the origin, $v_{n+1}$ is approximated, up to first order and rescaling, by
	\begin{equation*}
		\widetilde{v}_{n+1} = -2z \,\partial_z -3x \cdot \partial_x + y \cdot \partial_y.
	\end{equation*}
Since $\widetilde{v}_{n+1}$ has a nondegenerate zero at the origin of index $n+1$, so does $v_{n+1}$.

We define a \emph{contact $(n+1)$-handle} as a triple $(H_{n+1}, \alpha_{n+1}, v_{n+1})$, where $\alpha_{n+1}, v_{n+1}$ are as above and
$$H_{n+1} = \left\{ z^2 + |x|^2 \leq 1 \right\} \times \left\{ |y| \leq 2 \right\} \subset \Theta \times \R^n.$$
The rest of the present subsection is then devoted to proving the following analog of \cref{prop:n_handle} for $(n+1)$-handles:

\begin{proposition} \label{prop:(n+1)_handle}
Let $(M,\xi)$ be a $(2n+1)$-dimensional contact manifold with convex boundary $\Sigma = \partial M = R_+ \cup \Gamma \cup R_-$. Let $D_\pm \subset R_\pm$ be Lagrangian disks with cylindrical ends that limit to the same Legendrian sphere $\Lambda \subset \Gamma = \partial R_\pm$. Then a contact $(n+1)$-handle can be attached to $(M, \xi)$ along $D_+ \cup D_-$ to yield a new contact manifold $(M', \xi')$ with convex boundary $\Sigma' = \partial M' = R'_+ \cup \Gamma' \cup R'_-$, where:
		\begin{enumerate}
            \item the disks $D_+$ and $D_-$ have standard neighborhoods $N(D_+)\subset R_+$ and $N(D_-)\subset R_-$ such that $N(\Lambda)=\op{int}(\overline{N(D_+)}\cap \overline{N(D_-)})$ is a standard neighborhood of $\Lambda\subset\Gamma$;

            \item $R'_\pm$ is obtained by removing $\overline{N(D_\pm)}$ from $R_\pm$, so that 
            $$\partial \overline{R}'_\pm=(\partial \overline{R}_\pm\cup\bdry \overline{N(D_\pm)})\setminus N(\Lambda),$$ 
            and there is a contactomorphism $\psi_\pm\colon \D J^1(S^{n-1})\to \bdry \overline{N(D_\pm)}\setminus N(\Lambda)$, where $\D J^1(S^{n-1})$ is the disk 1-jet bundle of $S^{n-1}$;
		
			\item as a contact manifold, $\Gamma'$ is obtained from $\Gamma$ by performing a contact $(+1)$-surgery along $\Lambda \subset \Gamma$; and
		
			\item as a smooth manifold, $\Sigma'$ is diffeomorphic to $\overline{R}'_+\cup_{\psi}\overline{R}'_-$, where $\psi\colon\partial \overline{R}_-'\to\partial \overline{R}_+'$ satisfies
            \[
            \psi\mid_{\partial \overline{R}_-'\cap\partial \overline{R}_-} = \mathrm{id}_{\Gamma\setminus N(\Lambda)}
            \quad\text{and}\quad
            \psi\mid_{\partial \overline{R}_-'\cap\partial \overline{N(D_-)}} =  \psi_+\circ\psi_-^{-1}.
            \]
		\end{enumerate}
\end{proposition}

We proceed by analyzing the boundary of a contact $(n+1)$-handle. As in \cref{subsubsec:n_handle}, let us decompose the boundary as $\partial H_{n+1} = \partial_1 H_{n+1} \cup \partial_2 H_{n+1}$, where
\begin{gather*}
\partial_1 H_{n+1} = \left\{ z^2 + |x|^2 = 1 \right\} \times \left\{ |y| \leq 2 \right\},\\
\partial_2 H_{n+1} = \left\{ z^2 + |x|^2 \leq 1 \right\} \times \left\{ |y| = 2 \right\}.
\end{gather*}

\begin{claim} \label{claim: transverse}
$\partial_1 H_{n+1}$ and $\partial_2 H_{n+1}$ are $v_{n+1}$-convex, with $v_{n+1}$ pointing into $H_{n+1}$ along $\partial_1 H_{n+1}$ and out of $H_{n+1}$ along $\partial_2 H_{n+1}$.
\end{claim}

\begin{proof}
The fact that $v_{n+1}$ is negatively transverse to $\partial_1 H_{n+1}$ follows from
	\begin{equation*}
		\left( z^2 + 3|x|^2 +1 \right) v_{n+1} \left( z^2 + |x|^2 \right) = -4 \left( 1 - |x|^2 \right) z^2 - 2 \left( z^2 - |x|^2 +3 \right) |x|^2 < 0
	\end{equation*}
near $\left\{ z^2 + |x|^2 =1 \right\}$. Similarly, the fact that $v_{n+1}$ is positively transverse to $\partial_2 H_{n+1}$ follows from
	\begin{align*}
		\left( z^2 + 3|x|^2 +1 \right) v_{n+1} (|y|^2) &= 2\left(\left( z^2 + 3|x|^2 +1 \right) |y|^2 - 4zx \cdot y \right)\\
										  &\geq 2\left( (z^2 + 3|x|^2 +1) |y|^2 - 4 |z| |x| |y|\right) \\
										  &\geq 2\left( (z^2 + 3|x|^2 +1) |y|^2 - 2\left( z^2 + |x|^2 \right) |y|\right) >0
	\end{align*}
near $\left\{ |y|=2 \right\}$.
\end{proof}

\s
\n
\textbf{Description of $\partial_1 H_{n+1}$.}
\s

Since $\alpha_{n+1}(v_{n+1})=f$, the $v_{n+1}$-dividing set of $\partial_1 H_{n+1}$ is $\Gamma_{\partial_1 H_{n+1}} = \left\{ 2z = x \cdot y \right\} \subset \partial_1 H_{n+1}.$

\begin{claim} \label{claim: Legendrian ruling}
$R_+ (\partial_1 H_{n+1})$ and $R_-(\partial_1 H_{n+1})$ are exact deformation equivalent (relative to the boundary) to $T^\ast D^n$ with the standard Liouville form $p \cdot dq$ and Liouville vector field $p\cdot \partial_p$.  Moreover, $\overline{{R}_{\pm} (\partial_1 H_{n+1})}$ is the ideal compactification of $R_{\pm} (\partial_1 H_{n+1})$, obtained by compactifying each fiber to a closed disk by adding the sphere at infinity.
\end{claim}

\begin{proof}
By construction $\partial_1 H_{n+1}$ is foliated by Legendrian spheres
$$S^{n}_{y=b}=\{z^2+|x|^2=1\}\times\{y=b\}$$
such that each $S^{n}_{y=b}$ intersects $\Gamma_{\partial_1 H_{n+1}}$ in an equatorial sphere $S^{n-1}_{y=b}$ which is Legendrian in $\Gamma_{\partial_1 H_{n+1}}$. The claim is a consequence of the following general lemma.
\end{proof}

\begin{lemma}\label{lemma: normalization of foliation}
Let $\alpha$ be a Liouville form on $\R^n\times D^n$ with coordinates $(x,y)$ such that the pullback of $\alpha$ to each $y=b$, $b\in D^n$, is zero and the Liouville vector field is positively transverse to each $S^{n-1}_{|x|=r}\times D^n$ for $r\geq 1$. Then after applying a fiberwise diffeomorphism of $\R^n\times D^n$ (i.e., taking each $\R^n\times\{y=b\}$ to itself), $\alpha$ is exact deformation equivalent to the Liouville form $x \cdot dy$.
\end{lemma}

\begin{proof}
This is the relative version of the Arnold-Liouville theorem \cite[\S 49]{arnold1989mathematical}.  Observe that, since the pullback of $\alpha$ to each $y=b$ vanishes, we have $\alpha=\sum_i f_i(x,y)\, dy_i$.  Since
\begin{equation}\label{d alpha}
d\alpha=\sum_i d_x f_i \wedge dy_i + \sum_i d_y f_i\wedge dy_i
\end{equation}
is symplectic, it follows that $d_xf_1\wedge \dots \wedge d_xf_n$ is nowhere vanishing, where $d_x$ and $d_y$ refer to the exterior derivatives in the $x$- and $y$-directions.  Hence $F: \R^n \to \R^n$, $(x,y)\mapsto (f_1(x,y),\dots,f_n(x,y),y)$, is a local diffeomorphism.

Next we normalize the Liouville vector field.  Writing it as $X+Y$, where $X$ and $Y$ are components in the $\partial_x$- and $\partial_y$-directions, we claim that $Y=0$.  Indeed, 
\[
i_Xd\alpha=\sum_i X(f_i) \, dy_i \quad \text{and} \quad  i_Yd\alpha= -\sum_i Y(x_i) \, d_x f_i+\, \cdots,
\]
where the dots refer to terms in $dy_i$.  Then $i_{X+Y}d\alpha=\alpha$ implies $Y=0$, keeping in mind $d_x f_1\wedge\dots\wedge d_x f_n\not=0$. Hence, after applying a fiberwise diffeomorphism of $\R^n\times D^n$, we may assume that $X=\sum_i x_i\, \partial_{x_i}$. Integrating along $X$ we obtain a radial coordinate $s$ and write $\alpha= e^s (g_1\, dy_1+\dots +g_n \, dy_n)$, where $(g_1,\dots,g_n)$ are functions of coordinates on $S^{n-1}$.

Finally, the contact condition for $\alpha$ implies that the map $G=(g_1,\dots,g_n): S^{n-1}\to \R^n$ avoids the origin and descends to a local diffeomorphism $\overline G: S^{n-1} \to S^{n-1}=(\R^n-\{0\})/ \R^+$.  If $n>2$, then this implies that $\deg\overline G=\pm 1$, which in turn implies that $F$ is a diffeomorphism.  Outside of a compact region, we can therefore arrange $\alpha$ to be of the form $x\cdot dy$.  Inside the compact region, we can interpolate between $\alpha$ and $x\cdot dy$, which gives the exact deformation of Liouville forms.
\end{proof}

\begin{remark} \label{rmk: Legendrian ruling}
The foliation $\{S^n_{y=b}\}_{b\in D^n}$ is the higher-dimensional analog of the {\em Legendrian ruling} in dimension three~\cite[Proof of Prop. 3.1]{Hon00}.
\end{remark}

\s
\n
\textbf{Description of $\partial_2 H_{n+1}$.}
\s
Next let us turn to $\partial_2 H_{n+1}$, which is an $S^{n-1}$-family of $(n+1)$-dimensional $\Theta$-disks $\Theta_y$, $y\in S^{n-1}$. The $v_{n+1}$-dividing set is
$$\Gamma_{\partial_2 H_{n+1}} = \left\{ 2z = x \cdot y \right\} \subset \partial_2 H_{n+1},$$
and the restriction $\Theta'_y:=\Gamma_{\partial_2 H_{n+1}}\cap \Theta_y$ is an $n$-dimensional $\Theta$-disk.
There exists a neighborhood $N$ of the $0$-section of the $1$-jet space $J^1(S^{n-1})$ which is an $S^{n-1}$-family of $n$-dimensional $\Theta$-disks and which is contactomorphic to $\Gamma_{\partial_2 H_{n+1}}$. (Proof left to the reader.)
 
\begin{remark} $\mbox{}$
    \be
    \item A key observation is that we are using the {\em same} contact vector field $v_{n+1}$ to define {\em both} $\Gamma_{\bdry_1 H_{n+1}}$ and $\Gamma_{\bdry_2 H_{n+1}}$, and hence $\Gamma_{\bdry_1 H_{n+1}}=\Gamma_{\bdry_2 H_{n+1}}$ along the corner $\bdry(\bdry_1 H_{n+1})=\bdry(\bdry_2 H_{n+1})$, and no contact-geometric ``edge-rounding'' is necessary.
    \item Alternatively, we can apply an edge-rounding procedure which generalizes that of \cite[Lemma 3.11]{Hon00}, using the higher-dimensional analog of Legendrian rulings (see \cref{rmk: Legendrian ruling}).
    \ee
\end{remark}

\begin{claim}\label{claim: calc of char fol}
The characteristic foliation $(\bdry_2 H_{n+1})_\xi$ is parallel to the vector field:
$$X=(4x\cdot y) x\cdot \bdry_x - (z^2+3|x|^2+1) y\cdot \bdry_x +(2x\cdot y) z \bdry_z.$$
\end{claim}

\begin{proof}
We compute 
    \begin{align*}
        X\lrcorner d\alpha_{n+1} &= ((4x\cdot y)x\cdot \bdry_x - (z^2+3|x|^2+1) y\cdot \bdry_x +(2x\cdot y)z \bdry_z) \lrcorner (-4x\cdot dxdz +dx\cdot dy) \\
        &=-4(4x\cdot y)|x|^2 dz+ (4x\cdot y) x\cdot dy - (z^2+3|x|^2+1) (-4x\cdot y)dz +4 (2x\cdot y) zx\cdot dx\\
        &=(4x\cdot y)( x\cdot dy + (z^2-|x|^2 +1) dz + 2zx\cdot dx) =(4x\cdot y) \alpha_{n+1},
    \end{align*}
    since $y\cdot dy=0$ on $\bdry_2 H_{n+1}=\Theta\times S^{n-1}$.
\end{proof}

Consider the singular foliation $\ker(\eta)$ on $\bdry_2 H_{n+1}$ determined by
\[
\eta=-f d(|x|^2+z^2) - (1-|x|^2-z^2)df,
\]
where $f(x,y,z)=-2z+x\cdot y$ is the contact Hamiltonian function identified above, so that $\alpha_{n+1}(v_{n+1})=f$.
By construction, the interior of the dividing set $\Gamma_{\bdry_2 H_{n+1}}=\{f=0\}$ forms a leaf of $\ker(\eta)$, while the singular locus of $\ker(\eta)$ coincides with $\bdry \Gamma_{\bdry_2 H_{n+1}}$. Indeed, $\ker(\eta)$ provides a genuine (non-singular) foliation of $\op{int}(\Theta)\times S^{n-1}$ and taking the union of the singular locus of $\ker(\eta)$ with any one of these non-compact leaves produces a manifold diffeomorphic to $\Gamma_{\bdry_2 H_{n+1}}$.

\begin{claim}\label{claim:char-fol-transverse-to-leaves}
The vector field $X$ is transverse to the leaves of $\ker(\eta)$ on $\op{int}(\Theta)\times S^{n-1}$. In particular, $\eta(X)>0$ on $\op{int}(\Theta)\times S^{n-1}$.
\end{claim}

\begin{proof}
Notice that
\begin{align*}
X\lrcorner d(|x|^2) &= X\lrcorner(2x\cdot dx)\\
    &= 8(x\cdot y)|x|^2-2(z^2+3|x|^2+1)(x\cdot y)\\
    &= 2(x\cdot y)(|x|^2-z^2-1)
\end{align*}
and $X\lrcorner d(z^2) = 4(x\cdot y)z^2$, so
\[
X\lrcorner d(|x|^2+z^2) = -2(x\cdot y)(1-|x|^2-z^2).
\]
At the same time,
\begin{align*}
X\lrcorner df &= X\lrcorner(x\cdot dy + y\cdot dx - 2dz)\\
    &=X\lrcorner (y\cdot dx) - 2X\lrcorner dz\\
    &= 4(x\cdot y)^2 - (z^2+3|x|^2+1)|y|^2 - 4(x\cdot y)z.
\end{align*}
Recalling that $|y|=2$, we have
\begin{align*}
X\lrcorner\eta &= -(1-|x|^2-z^2)(-2(x\cdot y) f + X\lrcorner df)\\
    &= -(1-|x|^2-z^2)(4(x\cdot y)z-2(x\cdot y)^2 + 4(x\cdot y)^2-4z^2-12|x|^2-4-4(x\cdot y)z)\\
    &= (1-|x|^2-z^2)(4z^2+12|x|^2+4-2(x\cdot y)^2).
\end{align*}
Now $1-|x|^2-z^2>0$ throughout $\op{int}(\Theta)\times S^{n-1}$ and the Cauchy-Schwarz inequality tells us that $(x\cdot y)^2\leq 4|x|^2$, so that $4z^2+12|x|^2+4-2(x\cdot y)^2\geq 4(z^2+|x|^2+1)>0$.
\end{proof}

\begin{figure}[ht]

\begin{tikzpicture}[
  scale=0.6,
  line cap=round,
  line join=round,
  >={Stealth[length=4mm,width=2.4mm]}
]
\tikzset{->-/.style={decoration={
  markings,
  mark=at position #1 with {\arrow{>}}},postaction={decorate}}}
\tikzset{-<-/.style={decoration={
  markings,
  mark=at position #1 with {\arrow{<}}},postaction={decorate}}}

\definecolor{FigureGreen}{RGB}{93,184,70}
\definecolor{FigureGray}{RGB}{150,150,150}

\def\A{6.2}              
\def\B{2.25}             
\def\TopY{4.65}          
\def\BotY{-1.55}         

\def\BlackLW{1pt}
\def\GrayLW{0.6pt}
\def\GreenDashLW{0.84pt}
\def\GreenChordLW{2.4pt}
\def\PointRadius{1.32pt}

\def\TopLeftAngle{215}
\def\TopRightAngle{35}
\def\BotLeftAngle{145}
\def\BotRightAngle{325}

\newcommand{\EPoint}[3]{%
  \coordinate (#3) at ($(#1)+({\A*cos(#2)},{\B*sin(#2)})$);
}

\coordinate (TopC) at (0,\TopY);
\coordinate (BotC) at (0,\BotY);

\coordinate (TopL) at ($ (TopC) + (-\A,0) $);
\coordinate (TopR) at ($ (TopC) + (\A,0) $);
\coordinate (BotL) at ($ (BotC) + (-\A,0) $);
\coordinate (BotR) at ($ (BotC) + (\A,0) $);

\EPoint{TopC}{\TopLeftAngle}{TL}
\EPoint{TopC}{\TopRightAngle}{TR}
\EPoint{BotC}{\BotLeftAngle}{BL}
\EPoint{BotC}{\BotRightAngle}{BR}

\draw[line width=\BlackLW]
  ($ (TopC)+(\A,0) $) arc[start angle=0,end angle=-360,x radius=\A,y radius=\B];

\draw[line width=\BlackLW,dashed]
  ($ (BotC)+(-\A,0) $) arc[start angle=180,end angle=0,x radius=\A,y radius=\B];
\draw[line width=\BlackLW]
  ($ (BotC)+(\A,0) $) arc[start angle=0,end angle=-180,x radius=\A,y radius=\B];

\draw[line width=\BlackLW] (TopL) -- (BotL);
\draw[line width=\BlackLW] (TopR) -- (BotR);

\draw[-<-=.5,FigureGray,line width=\GrayLW] (TopL) -- (TopR);
\draw[->-=.5,FigureGray,line width=\GrayLW] (BotL) -- (BotR);


\foreach \yy in {
0.75, 0.5, 0.25
}{
\draw[-<-=.5,FigureGray,line width=\GrayLW]
  ($ (TopC)+(-\A,0) $) arc[start angle=180,end angle=0,x radius=\A,y radius=\yy*\B];
\draw[-<-=.5,FigureGray,line width=\GrayLW]
  ($ (TopC)+(-\A,0) $) arc[start angle=180,end angle=0,x radius=\A,y radius=-\yy*\B];
\draw[->-=.5,FigureGray,line width=\GrayLW]
  ($ (BotC)+(-\A,0) $) arc[start angle=180,end angle=0,x radius=\A,y radius=\yy*\B];
\draw[->-=.5,FigureGray,line width=\GrayLW]
  ($ (BotC)+(-\A,0) $) arc[start angle=180,end angle=0,x radius=\A,y radius=-\yy*\B];
}

\newcommand{\GreenChord}[2]{%
  \draw[FigureGreen,line width=\GreenChordLW] (#1) -- (#2);
  \fill[FigureGreen] (#1) circle (\PointRadius);
  \fill[FigureGreen] (#2) circle (\PointRadius);
}


\draw[FigureGreen,line width=\GreenDashLW,dashed]
  (TL) arc[start angle=\TopLeftAngle-360+10,end angle=\TopRightAngle-10,x radius=1.02*\A,y radius=1.02*\B];
\draw[FigureGreen,line width=\GreenDashLW,dashed]
  (TL) arc[start angle=\TopLeftAngle-360+22.5,end angle=\TopRightAngle-22.5,x radius=1.08*\A,y radius=1.08*\B];
\draw[FigureGreen,line width=\GreenDashLW,dashed]
  (TL) arc[start angle=\TopLeftAngle-360+45,end angle=\TopRightAngle-45,x radius=1.41*\A,y radius=1.41*\B];
\draw[FigureGreen,line width=\GreenDashLW,dashed]
  (TL) arc[start angle=\TopLeftAngle-360+45+22.5,end angle=\TopRightAngle-45-22.5,x radius=2.63*\A,y radius=2.63*\B];

\draw[FigureGreen,line width=\GreenDashLW,dashed]
  (TR) arc[start angle=\TopRightAngle+10,end angle=\TopLeftAngle-10,x radius=1.02*\A,y radius=1.02*\B];
\draw[FigureGreen,line width=\GreenDashLW,dashed]
  (TR) arc[start angle=\TopRightAngle+22.5,end angle=\TopLeftAngle-22.5,x radius=1.08*\A,y radius=1.08*\B];
\draw[FigureGreen,line width=\GreenDashLW,dashed]
  (TR) arc[start angle=\TopRightAngle+45,end angle=\TopLeftAngle-45,x radius=1.41*\A,y radius=1.41*\B];
\draw[FigureGreen,line width=\GreenDashLW,dashed]
  (TR) arc[start angle=\TopRightAngle+45+22.5,end angle=\TopLeftAngle-45-22.5,x radius=2.63*\A,y radius=2.63*\B];
  
\GreenChord{TL}{TR}


\draw[FigureGreen,line width=\GreenDashLW,dashed]
  (BL) arc[start angle=\BotLeftAngle+10,end angle=\BotRightAngle-10,x radius=1.02*\A,y radius=1.02*\B];
\draw[FigureGreen,line width=\GreenDashLW,dashed]
  (BL) arc[start angle=\BotLeftAngle+22.5,end angle=\BotRightAngle-22.5,x radius=1.08*\A,y radius=1.08*\B];
\draw[FigureGreen,line width=\GreenDashLW,dashed]
  (BL) arc[start angle=\BotLeftAngle+45,end angle=\BotRightAngle-45,x radius=1.41*\A,y radius=1.41*\B];
\draw[FigureGreen,line width=\GreenDashLW,dashed]
  (BL) arc[start angle=\BotLeftAngle+45+22.5,end angle=\BotRightAngle-45-22.5,x radius=2.63*\A,y radius=2.63*\B];

\draw[FigureGreen,line width=\GreenDashLW,dashed]
  (BR) arc[start angle=\BotRightAngle-360+10,end angle=\BotLeftAngle-10,x radius=1.02*\A,y radius=1.02*\B];
\draw[FigureGreen,line width=\GreenDashLW,dashed]
  (BR) arc[start angle=\BotRightAngle-360+22.5,end angle=\BotLeftAngle-22.5,x radius=1.08*\A,y radius=1.08*\B];
\draw[FigureGreen,line width=\GreenDashLW,dashed]
  (BR) arc[start angle=\BotRightAngle-360+45,end angle=\BotLeftAngle-45,x radius=1.41*\A,y radius=1.41*\B];
\draw[FigureGreen,line width=\GreenDashLW,dashed]
  (BR) arc[start angle=\BotRightAngle-360+45+22.5,end angle=\BotLeftAngle-45-22.5,x radius=2.63*\A,y radius=2.63*\B];
  
\GreenChord{BL}{BR}

\node[above,FigureGreen] at (7.5,6) {$\Gamma_{\partial_2 H_{n+1}}\cap\{y=2\}$};
\node[above,FigureGreen] at (7.5,-4) {$\Gamma_{\partial_2 H_{n+1}}\cap\{y=-2\}$};

\end{tikzpicture}
	\caption{The foliation $\ker(\eta)$ on $\bdry_2 H_{n+1}$, in the case $n=1$. The green dashed curves represent $\ker(\eta)$, while the gray flow lines indicate the oriented characteristic foliation. The characteristic foliation on $\bdry_1 H_{n+1}$ is omitted. According to \cref{claim:char-fol-transverse-to-leaves}, the leaves of $\ker(\eta)$ form a 1-parameter family through which $\Gamma_{\bdry_2 H_{n+1}}$ may be pushed to either $\bdry(\bdry_2 H_{n+1})\cap\{f\geq 0\}$ or $\bdry(\bdry_2 H_{n+1})\cap\{f\leq 0\}$.}
	\label{fig:outgoing-boundary}
\end{figure}

In light of \cref{claim:char-fol-transverse-to-leaves}, we may use $X$ to flow $\Gamma_{\bdry_2 H_{n+1}}$ backwards along the characteristic foliation to the portion of $\bdry(\bdry_2 H_{n+1})$ where $f=-2z+x\cdot y$ is nonnegative; notice that this is an isotopy rel boundary. Alternatively, we may use $X$ to flow $\Gamma_{\bdry_2 H_{n+1}}$ forward across $R_-(\bdry_2 H_{n+1})$ to the portion of $\bdry(\bdry_2 H_{n+1})$ where $f$ is nonpositive. See \cref{fig:outgoing-boundary}.

\begin{proof}[Proof of \cref{prop:(n+1)_handle}]
We may assume that $D_+\cup D_-$ has a Legendrian ruled neighborhood $N(D_+\cup D_-)$ in $\Sigma$.  By \cref{claim: Legendrian ruling} the contact germ on $\partial_1 H_{n+1}$ can be identified with the restriction of the contact germ on $\Sigma$ to $N(D_+\cup D_-)$. The analysis of the contact germ on $\partial_2 H_{n+1}$ from above shows that $R'_+ = R_+ (\Sigma')$ is obtained from $R_+$ by first removing a standard neighborhood of the Lagrangian disk $D_+ \subset R_+$ and then (partially) completing the Liouville domain. This is equivalent to modifying the canonical $1$-form $\lambda=-pdq$ near the vertical boundary $\pi^{-1}(\bdry N)$ of the disk cotangent bundle $\pi:\D^*N\to N$ of a manifold $N$ with nonempty boundary $\bdry N$ so that the Liouville vector field becomes positively transverse to $\bdry \D^*N$; see \cite[Lemma 12.8]{CE12}. The $R'_-$ case is analogous.
\end{proof}

\subsection{Partial open book decompositions}\label{subsec:POBD}

The relative notion of an open book decomposition for a contact manifold with boundary was first considered in dimension $3$ by Honda-Kazez-Matić \cite{HKM09}. Here we reformulate the higher-dimensional version first described in \cite{HH18}, drawing from the presentation in \cite{BHH23}.

\begin{definition}
Given a Weinstein domain $W$, a \emph{cornered Weinstein subdomain} $S \subset W$ is a (possibly empty) codimension-$0$ submanifold with corners which satisfies the following properties:
\begin{enumerate} 
\item There exists a decomposition $\partial S = \partial_{\inward} S \cup \partial_{\outward} S$ such that
\begin{enumerate}
    \item $\partial_{\inward} S$ and $\partial_{\outward} S$ are compact manifolds with smooth boundary that intersect along their boundaries;
\item $\partial(\partial_{\inward} S) = \partial(\partial_{\outward} S)$ is the codimension-$1$ corner of $\partial S$; and
\item $\partial_{\outward} S = S\cap \partial W$ and is a proper subset of each component of $\partial W$.
\end{enumerate}
\item The Liouville vector field $X_{\lambda}$ on $W$ is inward-pointing along $\partial_{\inward} S$ and outward-pointing along $\partial_{\outward} S$.
\item $W\setminus S$ is a Weinstein domain after smoothing.
\end{enumerate}
The region $S$ (without reference to the space $W$) is called a {\em cornered Weinstein cobordism}.
\end{definition}

The main example of a cornered Weinstein subdomain is the standard neighborhood of a regular \cite{EGL18} Lagrangian disk.

\begin{definition} \label{defn: POBD}
A {\em strongly Weinstein partial open book decomposition (POBD)} for a compact cooriented contact manifold $(M^{2n+1},\xi)$ with Weinstein convex boundary is a smoothing of the manifold with corners
\[
\left(W\times[0,\tfrac{1}{2}]\right) \, \cup\, \left(S\times[\tfrac{1}{2},1]\right)  \, \cup\, \left(D^2 \times \partial W\right)\, / \sim'_{\phi},
\]
where:
\begin{enumerate}
\item $W_t:=W\times \{t\}$ is a $1$-Weinstein domain for each $t\in [0,\tfrac{1}{2}]$ and $S_t:=S\times\{t\}$ is a $1$-Weinstein cornered Weinstein cobordism for each $t\in[\tfrac{1}{2},1]$;
\item viewing $S$ as a cornered subdomain of $W$, $\sim'_\phi$ identifies $S_{1/2}$ with a cornered Weinstein subdomain of $W_{1/2}$ via the identity and $S_1$ with a cornered Weinstein subdomain of $W_0$ via an embedding $\phi: S\hookrightarrow W$ that is the identity on $\bdry_{\op{out}}S$;
\item there exists a contact form $\alpha$ on the glued manifold which restricts to the $1$-Weinstein Liouville forms on each $W_t$ and $S_t$ and whose Reeb vector field $R_\alpha$ is given by $\partial_t$ on $(\partial W)\times [0,\tfrac{1}{2}]$ and $(\partial S)\times [\tfrac{1}{2},1]$; and
\item $\sim'_\phi$ identifies $\bdry D^2\times \{p\}$ with $([0,1]/0\sim 1)\times \{p\}$ for each $p\in \bdry W$ and the contact structure on $D^2 \times \partial W$ is  a ``standard" contact structure on a neighborhood of the contact submanifold $\partial W$.
\end{enumerate}
We denote a POBD in abstract form by $(W, \phi: S\to W)$ and refer to $(M, \xi)$ as its \emph{total space}.
\end{definition}  

The identifications from (2) will be denoted $\sim_\phi$, to distinguish from $\sim'_\phi$ which consists of identifications from (2) and (4).

Note that the boundary $\partial M$of the total space of a POBD is convex with decomposition $R_+ \cup \Gamma \cup R_-$, where $R_+ = W \setminus S$ and $R_- = W \setminus \phi(S)$.

\begin{example}\label{ex:POBD_ex}
Let $W = \D^*S^n$ be the standard Weinstein structure on the disk cotangent bundle and $S\subset W$ be a neighborhood of a fiber $D_+:= D_pS^n$. Let $\Phi: W \to W$ be the symplectic Dehn twist (supported) around the $0$-section, and define $\phi: S \to W$ by $\phi:= \Phi\mid_S$. We claim that the total space $(M, \xi)$ of $(W, \phi:S \to W)$ is contactomorphic to a ball with standard convex boundary. Indeed, note that after smoothing corners, the piece $W \times [0, \tfrac{1}{2}] \cong J^1(S^n)$ is given by a contact $0$-handle (a Darboux ball), followed by a contact $n$-handle, whose core is obtained by puncturing the $0$-section of $W$. The piece $S\times [\tfrac{1}{2},1]$ may then be interpreted as a contact $(n+1)$-handle with attaching sphere $D_+ \cup D_-$, where $D_- := \phi(D_+)$. The contact $(n+1)$-handle smoothly cancels the $n$-handle, and by \cref{prop:(n+1)_handle}, $R_+$ of the resulting boundary is obtained by removing $S$ from $W$, hence is a standard symplectic $2n$-ball. See \cref{fig:POBD_ex}.

In fact, the total space $(M, \xi)$ is contactomorphic to a Darboux ball. Observe that we can glue on another Darboux ball (i.e., a contact $(2n+1)$-handle) to $(M, \xi)$ to obtain a closed contact manifold. Moreover, the codimension-$1$ foliation of the Darboux ball by symplectic disks upgrades the POBD $(W, \phi:S \to W)$ to a full open book decomposition with page $\D^*S^n$ and monodromy given by a symplectic Dehn twist. The resulting total space is well-known to be contactomorphic to $(S^n, \xi_{\mathrm{st}})$. This restricts to a contactomorphism of $(M, \xi)$, the total space of the POBD $(W, \phi: S\to W)$, with a Darboux ball.
\end{example}
\begin{figure}[ht]
	\begin{overpic}[scale=.43]{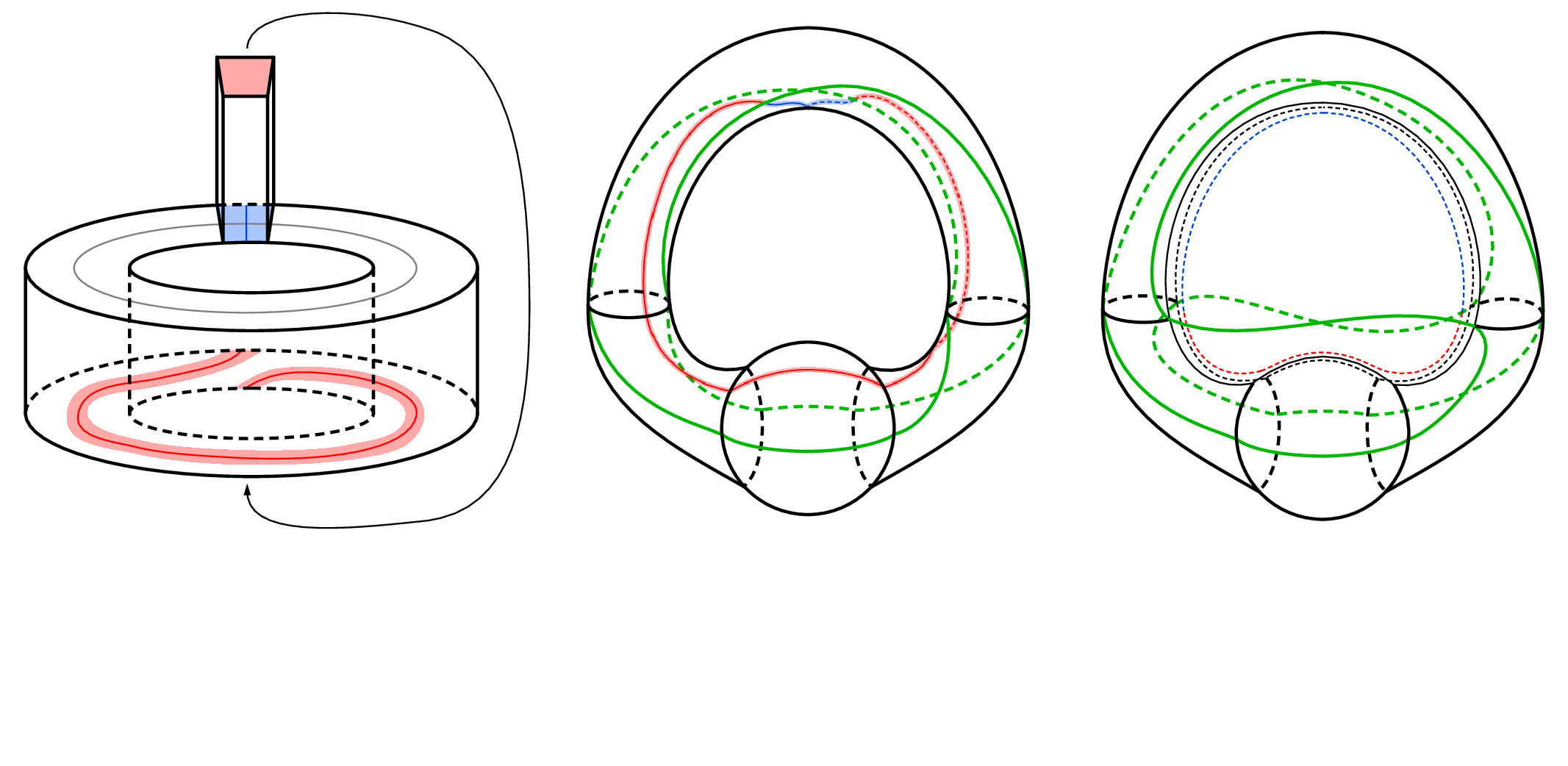}
		\put(23,13){\small $\phi$}
        \put(18,39){\tiny $S \times [\tfrac{1}{2},1]$}
	\end{overpic}
    \vskip-1.75cm
	\caption{The abstract POBD $(W, \phi:S \to W)$ on the left, and resulting contact manifold with boundary from \cref{ex:POBD_ex}. In the first panel, $S_{1/2}$ is shaded in blue. The middle panel depicts a contact $0$-handle and a(n embedded) contact $n$-handle giving $J^1(S^n)$; the attaching sphere of the contact $(n+1)$ is indicated in blue and red. The third panel depicts the result of the contact $(n+1)$-handle attachment, after some perturbation.}
	\label{fig:POBD_ex}
\end{figure}

\cref{ex:POBD_ex} hints at a more general relationship between contact handles and POBDs, given by the following lemma.

\begin{lemma}\label{lemma:handle-to-POBD}
Let $(M, \xi)$ admit a POBD $(W, \phi:S \to W)$ with Weinstein convex boundary $\partial M = R_+ \cup \Gamma \cup R_-$. 
\begin{enumerate}
    \item If $\Lambda \subset \Gamma$ is a Legendrian sphere in the dividing set and $(M',\xi')$ denotes the contact manifold obtained by a contact $n$-handle $H_n$ attached along $\Lambda$, then $(M',\xi')$ is supported by the POBD 
    \[
    (W\cup h_n, \,\phi:S \to W \cup h_n),
    \]
    where $h_n$ is the underlying Weinstein handle attachment to $W$ with attaching sphere $\Lambda$. 

    \item If $D_{\pm} \subset R_{\pm}$ are regular Lagrangian disks with $\partial D_+ = \partial D_- \subset \Gamma$ and $(M',\xi')$ denotes the contact manifold obtained by a contact $(n+1)$-handle $H_{n+1}$ attached along $D_+ \cup D_-$, then $D_+$ can be taken to be disjoint from $S$ and $(M',\xi')$ is supported by the POBD 
    \[ (W, \,  \phi\cup \psi: S \cup N(D_+) \to W),\]
    where $N(D_{\pm})$ is a standard Weinstein neighborhood of $D_{\pm}$ and $\psi: N(D_+) \to N(D_-)$ is an exact symplectomorphism that takes $D_+$ to $D_-$. 
\end{enumerate}
\end{lemma}

\begin{proof}
(1) We perturb the corners of the contact $n$-handle model in \cref{subsubsec:n_handle} (see Equation \eqref{eqn: n-handle}) and view it as 
$$H_n = [-1,1]_z \times \{|x| \leq 1\} \times \{|y| \leq 1\} \subset \R_z \times \R^n \times \R^n$$ with contact form $\alpha_n = dz - 2y\cdot dx - x\cdot dy$. Note that each $\{z\} \times \{|x| \leq 1\} \times \{|y| \leq 1\}$ has the structure of an index-$n$ Weinstein handle $h_n$ whose attaching sphere is (a parallel shift of) $\Lambda$. Attachment of the contact handle $H_n$ to $M$ then results in a POBD with an extra Weinstein handle $h_n$ attached to the page with no additional monodromy. This gives (1). 

(2) Since $D_+\subset R_+$, $\op{int}(D_+)$ is disjoint from $S$.  By dimension reasons and considerations of descending submanifolds of critical points on $S$, we may assume that $\bdry D_+$ and $\bdry_{\op{in}}S$ are disjoint in $\Gamma$. Hence $D_+$ and $S$ can be taken to be disjoint.

The contact $(n+1)$-handle $H_{n+1}$ is obtained by reversing the contact vector field in the above contact $n$-handle model $H_n$, and is foliated by Weinstein $n$-handles $\{z\} \times \{|x| \leq 1\} \times \{|y| \leq 1\}$. The attaching region of the $(n+1)$-handle contains the union of $\{\pm 1\}\times \{|x| \leq 1\} \times \{|y| \leq 1\}$, and the Weinstein $n$-handles have core $|x|=0$ (instead of $|y|=0$ as in the $n$-handle case). Identifying the cores $|x|=0$ with (slight retractions of) $D_{\pm}$, we see that a contact $(n+1)$-handle attachment along $D_+ \cup D_-$ then induces the monodromy from $N(D_+)$ to $N(D_-)$, yielding (2).
\end{proof}

As with open book decompositions of closed manifolds, there is a positive stabilization operation preserving the contactomorphism type of the total space of a POBD, in fact modeled on \cref{ex:POBD_ex}. 

\begin{definition}\label{def:POBD_stab}
Let $L\subset W$ be a regular Lagrangian disk with Legendrian boundary $\subset \partial W$.  Then the POBD
\begin{align*}
    (W\cup h, \,  \tau\circ (\phi\cup\mathrm{id}_h) : S\cup h \to W\cup h)
\end{align*}
is called a {\em (positive) stabilization of $(W, \phi:S\to W)$ along $L$}, where: 
\begin{enumerate}
\item $h$ is a Weinstein $n$-handle with core Lagrangian disk $L'$ such that $\partial L=\partial L'$, and
\item $\tau: W\cup h \to W\cup h$ is a (positive symplectic) Dehn twist about the Lagrangian sphere $L\cup L'$. 
\end{enumerate}
\end{definition}

\begin{remark}
A second and equivalent model for POBD stabilization is given by precomposing the monodromy rather than postcomposing (c.f.\ \cite[Definition 1.2.3]{BHH23}): 
\[
(W\cup h,\, (\phi\cup \op{id}_h) \circ \tau: \tau^{-1}(S) \cup N(L)  \to W\cup h),
\]
where $N(L)\subset W$ is the Weinstein $n$-handle neighborhood of $L$.
For the sake of simplicity, we will focus on the model in \cref{def:POBD_stab}. 
\end{remark}

\begin{lemma}\label{lem:positive_stabilization}
If $(M, \xi)$ admits a POBD $(W, \phi:S \to W)$ and $(M', \xi')$ is obtained by performing a positive POBD stabilization along $L$, then $(M', \xi')$ is contactomorphic to $(M, \xi)$. 
\end{lemma}

\begin{proof}
The proof is analogous to, but more complicated than, the closed case (c.f.\ \cite[Theorem 2.17]{Et06}). 
	
Consider the POBD from \cref{ex:POBD_ex} given (in different notation) by $(W_0,\phi_0:S_0 \to W_0)$, where $W_0 = \D^\ast S^n$ is the disk cotangent bundle with the standard Weinstein structure, $S_0$ is the standard neighborhood of a Lagrangian fiber $L_0=D^*_p S^n$, and $\phi_0$ is the partially-defined positive Dehn twist along the $0$-section. As observed in \cref{ex:POBD_ex}, the total space is contactomorphic to a Darboux ball.
    
We now describe the ``plumbing" 
$$(W', \phi':S'\to W')=(W, \phi: S\to W)*(W_0,\phi_0:S_0 \to W_0).$$  
Let $R_0=W_0\setminus S_0$ and let $R\subset W$ be a standard cornered neighborhood of the disk $L$ along which we are stabilizing. Write $\partial R= \partial_{\mathrm{in}} R\cup \partial_{\mathrm{out}} R$, where $\partial_{\mathrm{out}} R=\partial R\cap \partial W$ and $\partial_{\mathrm{in}} R=\partial R\setminus \op{int} \bdry_{\op{out}}R$. 
We then set $W'=W_0\cup_{R_0=R} W$ so that each cotangent disk fiber of $R_0$ intersects $L$ exactly once and $S'= S_0\cup S$. After plumbing, $S_0$ becomes the $n$-handle $h$ in \cref{def:POBD_stab} and $L_0$ its cocore disk. Let $L'$ denote the core disk of the $n$-handle $h$ (i.e., the intersection of $S_0$ with the $0$-section of $\D^*S^n$).

We then construct $M'$ from
\[
\left(W'\times[0,\tfrac{1}{2}]\right) \, \cup\, \left(S'\times[\tfrac{1}{2},1]\right) \,\cup\, (D^2\times \bdry W')\, / \sim'_{\phi \, \cup\, \mathrm{id}_h}, 
\]
by applying contact $(-1)$-surgery to $(L \cup L') \times \{\tfrac{1}{4}\}$. Viewing $W\subset W'=W\cup h$ and $S\subset S' = S \cup h$, we can decompose $M'=H_0\cup H_1$, where
\[
H_0 = \left(W\times [0,\tfrac{1}{2}]\right) \,\cup\, \left(S\times[\tfrac{1}{2},1]\right) \, \cup \,  (D^2\times (\partial W'\setminus h)) \setminus \left(R\times[\tfrac{1}{8},\tfrac{3}{8}]\right) \, / \sim'_{\phi \, \cup\, \mathrm{id}_h},
\]
and $H_1$ is obtained from 
\[
\left(h\times [0,1]\right) \, \cup \,\left(R\times[\tfrac{1}{8},\tfrac{3}{8}]\right) \,\cup\,  (D^2\times ((\partial W')\cap h))  \, / \sim'_{\phi \, \cup\, \mathrm{id}_h}
\]
by performing a contact $(-1)$-surgery along $(L \cup L') \times \{\tfrac{1}{4}\}$.

We now consider several cases: (i) Assume that $\bdry L\subset  \bdry_{\op{out}} S$.  Then
\[
\partial H_0 \cap \partial H_1 =(\partial_{\mathrm{in}} R \times [\tfrac{1}{8},\tfrac{3}{8}])\, \cup\, (R\times \{\tfrac{1}{8},\tfrac{3}{8}\}) \, \cup\, (\partial_{\mathrm{out}}R \times  ([0,\tfrac{1}{8}]\cup[\tfrac{3}{8},1]))
\]
is a standard convex $S^{2n}$, modulo corner rounding. Hence $H_0$ is contactomorphic to $(M,\xi)$ with a Darboux ball removed, in view of the condition $\bdry_{\op{out}}S\supset \bdry_{\op{out}} R$, which follows from $\bdry L\subset  \bdry_{\op{out}} S$.  On the other hand, $H_1$ is supported by the POBD in \cref{ex:POBD_ex} and is contactomorphic to a Darboux ball.
It follows (see \cite[\S 5]{BEM15}) that $(M',\xi')$ is contactomorphic to the connected sum of $(M,\xi)$ and the standard contact $S^{2n+1}$, which in turn is contactomorphic to $(M,\xi)$. 

(ii) Assume that $\bdry L \cap \bdry_{\op{out}} S=\emptyset$. Then we replace $(W, \phi:S \to W)$ by $(W, \phi\cup \op{id}_{S_{\bdry L}}: S\cup S_{\bdry L}\to W)$, where $S_{\bdry L}$ is given as follows: Let $N(\bdry L)$ be a standard neighborhood of the Legendrian $\bdry L$ in $\bdry W$.  Then let $S_{\bdry L}$ be a collar neighborhood of $N(\bdry L)$ in $W$ with a subset of corners smoothed so that $S_{\bdry L}$ is a cornered Weinstein subdomain of $W$ with $\bdry_{\op{out}}S_{\bdry L}= N(\bdry L)$. One can verify that $(M,\xi)$ and the contact manifold corresponding to $(W, \phi\cup \op{id}_{S_{\bdry L}}: S\cup S_{\bdry L}\to W)$ are contactomorphic. We are now in Case (i).

(iii) In the general case, we may arrange so that $\bdry L\pitchfork \bdry (\bdry_{\op{out}} S)$.  Assuming $\phi=\op{id}$ near $\bdry W$, we replace  $(W, \phi:S \to W)$ by $(W, \phi\cup \op{id}_{S_{\bdry L}}: S\cup S_{\bdry L}\to W)$ as before, where $S\cup S_{\bdry L}$ is given the appropriate corner smoothing.  Again one can verify that $(M,\xi)$ and the contact manifold corresponding to $(W, \phi\cup \op{id}_{S_{\bdry L}}: S\cup S_{\bdry L}\to W)$ are contactomorphic and we are in Case (i).
\end{proof}

\subsection{Bypass attachments}\label{subsec:bypass-appendix}

Let $\Sigma$ be a convex hypersurface with convex decomposition $\Sigma = R_+ \cup \Gamma \cup R_-$. Recall that in dimension $3$, the classical notion of a bypass attaching arc is one which intersects the dividing curve thrice, two times at its endpoints. In other words, a bypass attaching arc is the union of two arcs $D_+ \cup D_-$, where $D_{\pm}\subset R_{\pm}$ share one endpoint; see \cref{fig:bypass-new-persp}. Additionally, a \emph{trivial} bypass attaching arc \cite{Hon00} is one equivalent to either of the two models on the right side of \cref{fig:bypass-new-persp}. This perspective is what generalizes to all dimensions in the following definition. 

 \begin{figure}[ht]
	\begin{overpic}[scale=0.43]{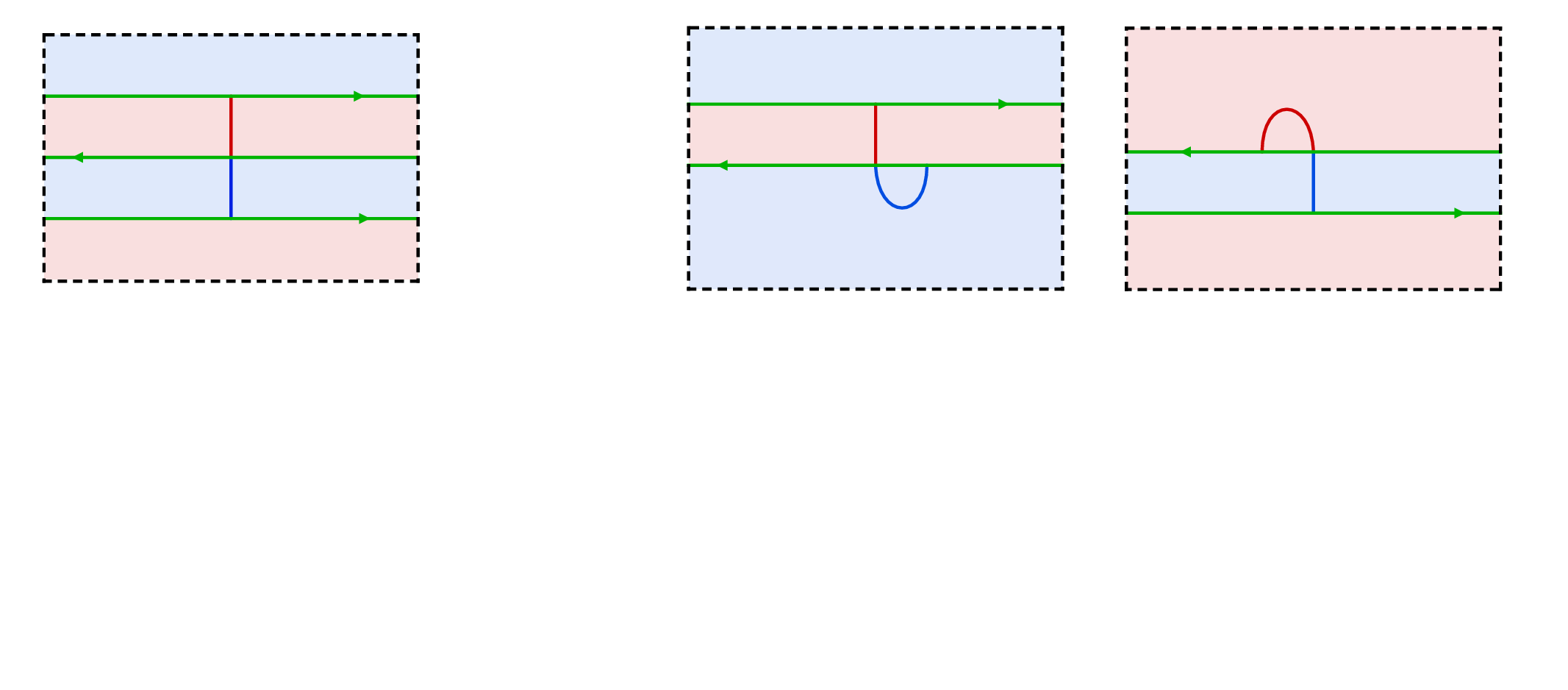}
	\put(15.5,31.5){\small \textcolor{darkblue}{$D_+$}}
    \put(15.5,35.25){\small \textcolor{darkred}{$D_-$}}
	\end{overpic}
    \vskip-3.5cm
	\caption{A new perspective on the classical bypass attaching arc; compare with \cref{fig:bypass-half-disk}. On the right are the two formulations of a trivial bypass arc. The arrows indicate the induced Reeb orientation of the dividing set.}
	\label{fig:bypass-new-persp}
\end{figure}

\begin{definition}\label{def:bypass}
A \emph{bypass attachment data} is a tuple $(\Lambda_{-}, \Lambda_+; D_{-}, D_+)$, where $D_{\pm}\subset R_{\pm}$ is a properly embedded Lagrangian disk filling of $\Lambda_{\pm}\subset\Gamma$ such that $\Lambda_-$ and $\Lambda_+$ intersect $\xi_{\Gamma}$-transversely (by this we mean $T\Lambda_-$ and $T\Lambda_+$ span $\xi_\Gamma$ at each point of intersection) at one point; see the middle panel of \cref{fig:bypass_attachment}. 

The bypass attachment data is \emph{trivial} if either of the following holds: 
\begin{itemize}
    \item[(TB1)] $(\Lambda_+; D_+)$ is a pair consisting of a standard Legendrian unknot and standard Lagrangian disk such that $\Lambda_+$ is below $\Lambda_-$; when $n=1$, this means there exists a disk in $R_+$ with boundary $\gamma\cup D_+$, where $\gamma\subset \Gamma$ is an arc with upper endpoint $\Lambda_+\cap \Lambda_-$ with respect to the Reeb direction.
    \item[(TB2)] $(\Lambda_-; D_-)$ is a pair consisting of a standard Legendrian unknot and standard Lagrangian disk such that $\Lambda_-$ is above $\Lambda_+$.
\end{itemize}
See \cref{fig:trivial_data}.   
\end{definition}

\begin{figure}[ht]
	\begin{overpic}[scale=0.4]{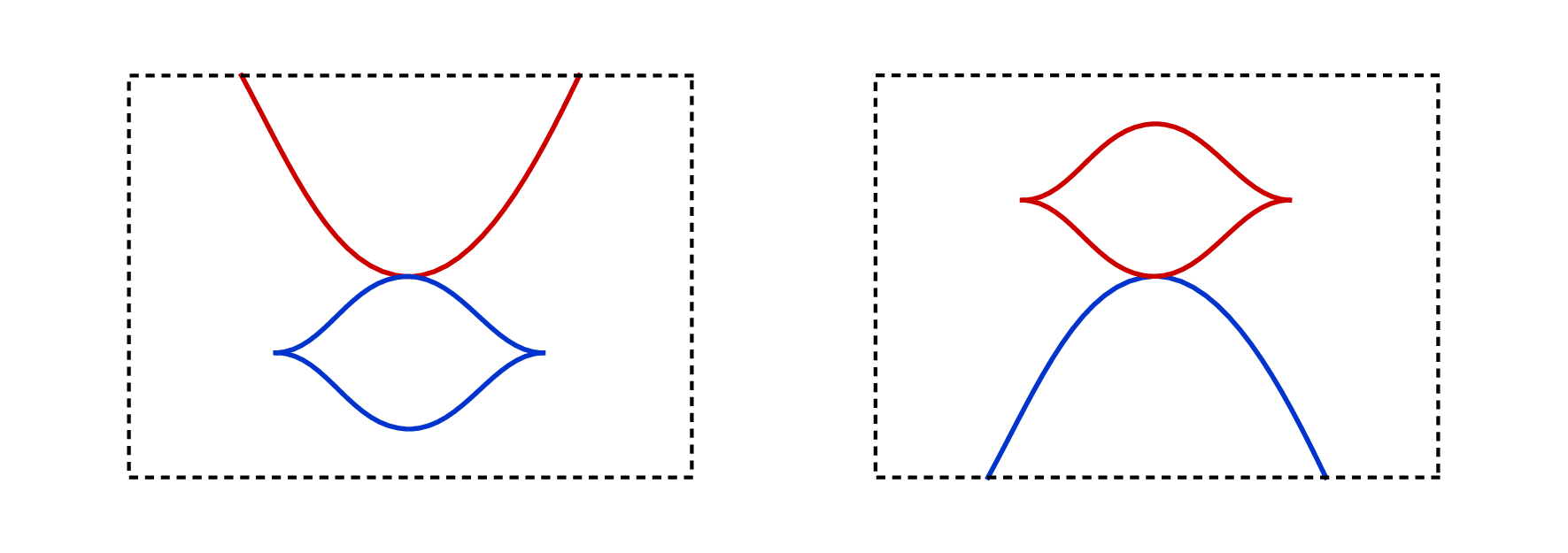}
	    \put(23,31.5){\small (TB1)}
	    \put(32,8){\small \textcolor{darkblue}{$\Lambda_+$}}
	    \put(64,25){\small \textcolor{darkred}{$\Lambda_-$}}
	    \put(71,31.5){\small (TB2)}
	\end{overpic}
	\caption{Trivial bypass data when viewed in the spin-symmetric front projection of the dividing set (see \cref{subsubsec:sums-slides}). On the left, $\Lambda_+$ bounds the standard disk filling in $R_+$, and on the right, $\Lambda_-$ bounds the standard disk filling in $R_-$.}
	\label{fig:trivial_data}
\end{figure}

In dimension $3$, a bypass attachment begins with a contact $(n=1)$-handle attached along the endpoints of the attaching arc. The perspective which generalizes to higher dimensions is that the endpoints of the attaching arc are the connected sum of the Legendrian $0$-spheres $\partial D_+$ and $\partial D_-$ in the dividing set. In the forthcoming description of a bypass attachment (\cref{prop:bypass_attachment}), a contact $n$-handle will be attached along a certain Legendrian that we denote by $\partial D_- \uplus \partial D_+$. While a direct definition of the $\uplus$ operation was described in \cite{HH18}, it is sufficient for our purposes to use a diagrammatic definition by appealing to the results of Casals-Murphy \cite{casals2019fronts}, described in the following.

\subsubsection{Legendrian sums and handleslides}\label{subsubsec:sums-slides} We adopt the convention, described in more generality and detail in \cite{casals2019fronts}, of drawing $n$-dimensional front projections of a $(2n-1)$-dimensional contact manifold $(\Gamma,\xi)$ with implicit $S^{n-1}$-spin symmetry. We will use boxes to identify regions of such spin symmetry; for example, see \cref{fig:spin1}. With a single exception (\cref{fig:bypass-slide}), all of our front projections exhibit global spin symmetry.

\begin{figure}[ht]
	\begin{overpic}[scale=0.35]{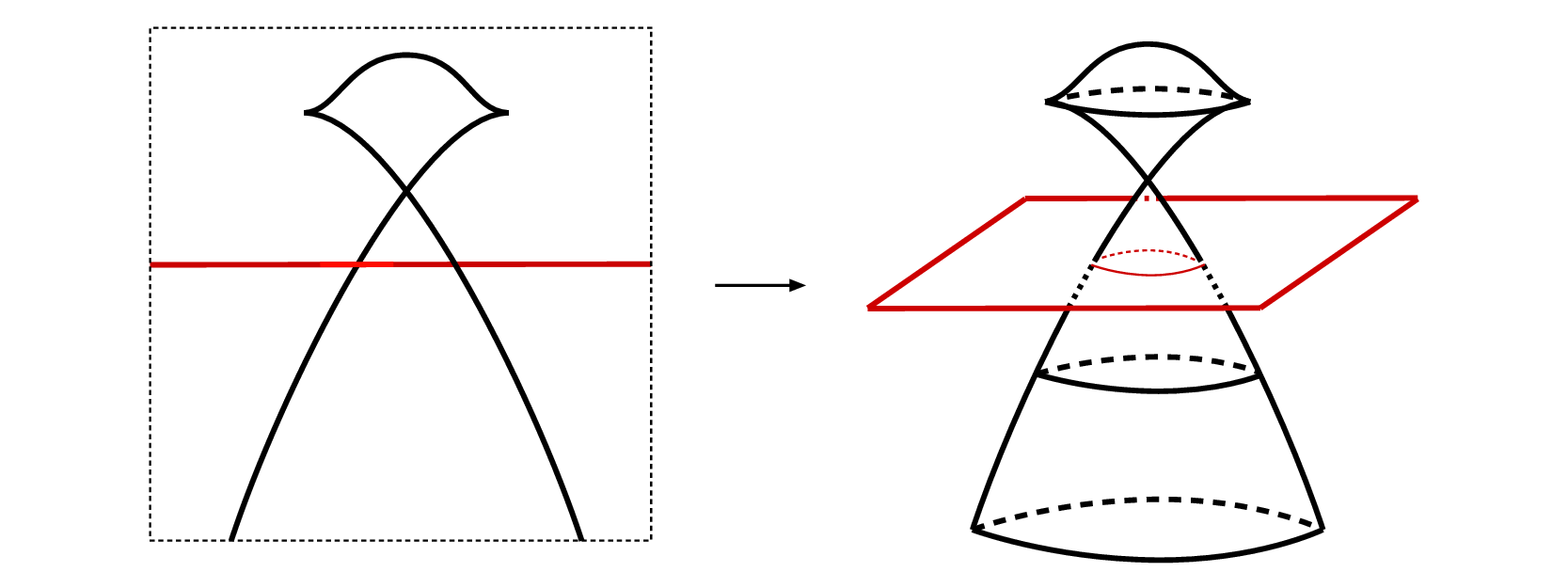}
	\end{overpic}
	\caption{A planar front projection implicitly $S^1$-spun around the central vertical axis to produce a Legendrian surface front projection.}
	\label{fig:spin1}
\end{figure}

\begin{definition}\label{def:uplus}
Let $\Lambda_1, \Lambda_2\subset \Gamma$ be Legendrians intersecting $\xi$-transversely at a point. Define the operations $\Lambda_1 \uplus\Lambda_2$ and $\Lambda_2 \uplus\Lambda_1$ diagrammatically in a spin-symmetric local front projection according to \cref{fig:uplus}.    
\end{definition}

\begin{figure}[ht]
	\begin{overpic}[scale=0.35]{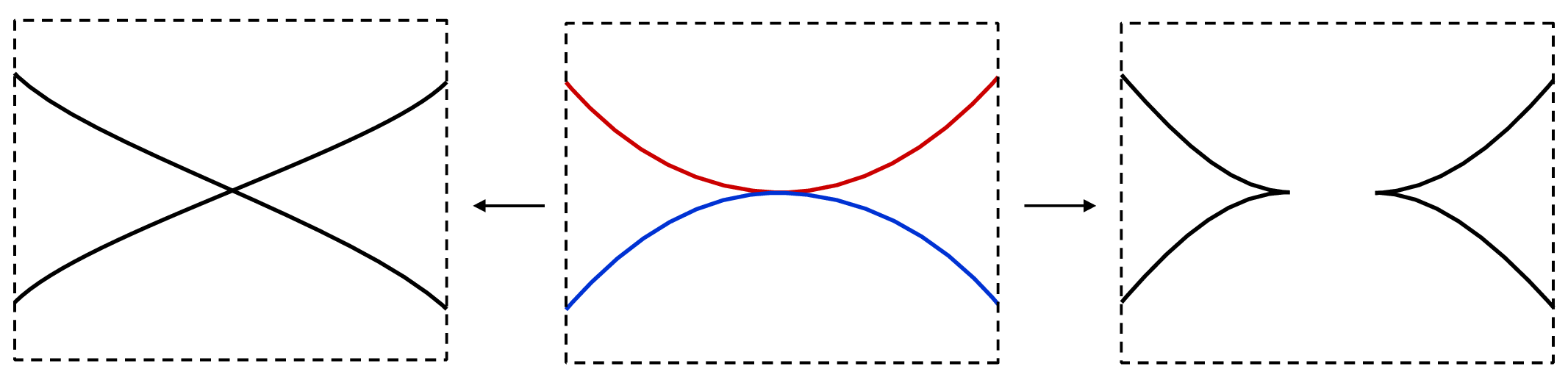}
	    	\put(55,5.5){\small \textcolor{darkblue}{$\Lambda_1$}}
  \put(55,17){\small \textcolor{darkred}{$\Lambda_2$}}
  \put(80.25,17){\small $\Lambda_1 \uplus \Lambda_2$}
  \put(10.25,17){\small $\Lambda_2 \uplus \Lambda_1$}
	\end{overpic}
	\caption{Diagrammatically defining the $\uplus$ operation.}
	\label{fig:uplus}
\end{figure}

\cref{def:uplus} is the contact analog of Polterovich's \cite{Pol91} Lagrangian connected sum for transverse Lagrangian intersections in symplectic manifolds. We write $\Lambda_1 \uplus \Lambda_2$ instead of $\Lambda_1 \# \Lambda_2$ since the latter is commonly used to denote the Legendrian connected sum, {\em which is a different operation,} although smoothly $\Lambda_1 \uplus \Lambda_2$ is diffeomorphic to the connected sum of $\Lambda_1$ and $\Lambda_2$. The definition is justified by \Cref{prop:handleslide} below, which is a restatement of \cite[Proposition 2.17]{casals2019fronts} and gives local front projections for slides of Legendrians across contact $(\pm 1)$ surgeries, generalizing work of Ding-Geiges \cite{Ding2009HandleMI} in dimension $3$. 

\begin{proposition}\label{prop:handleslide}
Let $\Lambda_0, \Lambda$ be Legendrian spheres intersecting $\xi$-transversely at a single point. After performing a contact $(-1)$-surgery along $\Lambda_0$, the Legendrian $\Lambda^{-\epsilon}$ is Legendrian isotopic in the surgered manifold to $(\Lambda \uplus \Lambda_0)^{\epsilon}$; see \cref{fig:slide-sum}.
\end{proposition}

Here we are using the following ``Reeb shift'' notation:

\begin{notation}[Reeb shift notation]\label{remark:reeb_shift}
Given a Legendrian $\Lambda$ in a contact manifold with a chosen contact form, let $\Lambda^t$ denote the time-$t$ Reeb shift of $\Lambda$. If additionally $\Lambda$ (asymptotically) bounds a Lagrangian disk $D$ in a Liouville manifold, we let $D^t$ denote a Lagrangian disk filling $\Lambda^t$ obtained by extending the time-$t$ Reeb flow into the cylindrical end of the Liouville manifold via a Hamiltonian isotopy.      
\end{notation}

More generally, we have the following:

\begin{claim}\label{claim:disk_isotopy}
If $\Lambda$ (asymptotically) bounds a Lagrangian disk $D$ in a Liouville manifold, then any Legendrian isotopy $\Lambda_s$, $s\in [0,1]$, in the ideal contact boundary with $\Lambda_0=\Lambda$ may be extended to a Hamiltonian isotopy of $D$.
\end{claim}

\begin{proof}
    It is a standard fact that (a $C^0$-approximation of) the trace of a Legendrian isotopy generates a Lagrangian concordance in the symplectization, which is Hamiltonian isotopic to the identity concordance; see  e.g.\ \cite[Lemma 6.1]{EHK16} or \cite[Theorem 2.23]{casals2019fronts}. Stacking this concordance on top of $D$ gives the claim.
\end{proof}

\begin{figure}[ht]
	\begin{overpic}[scale=0.43]{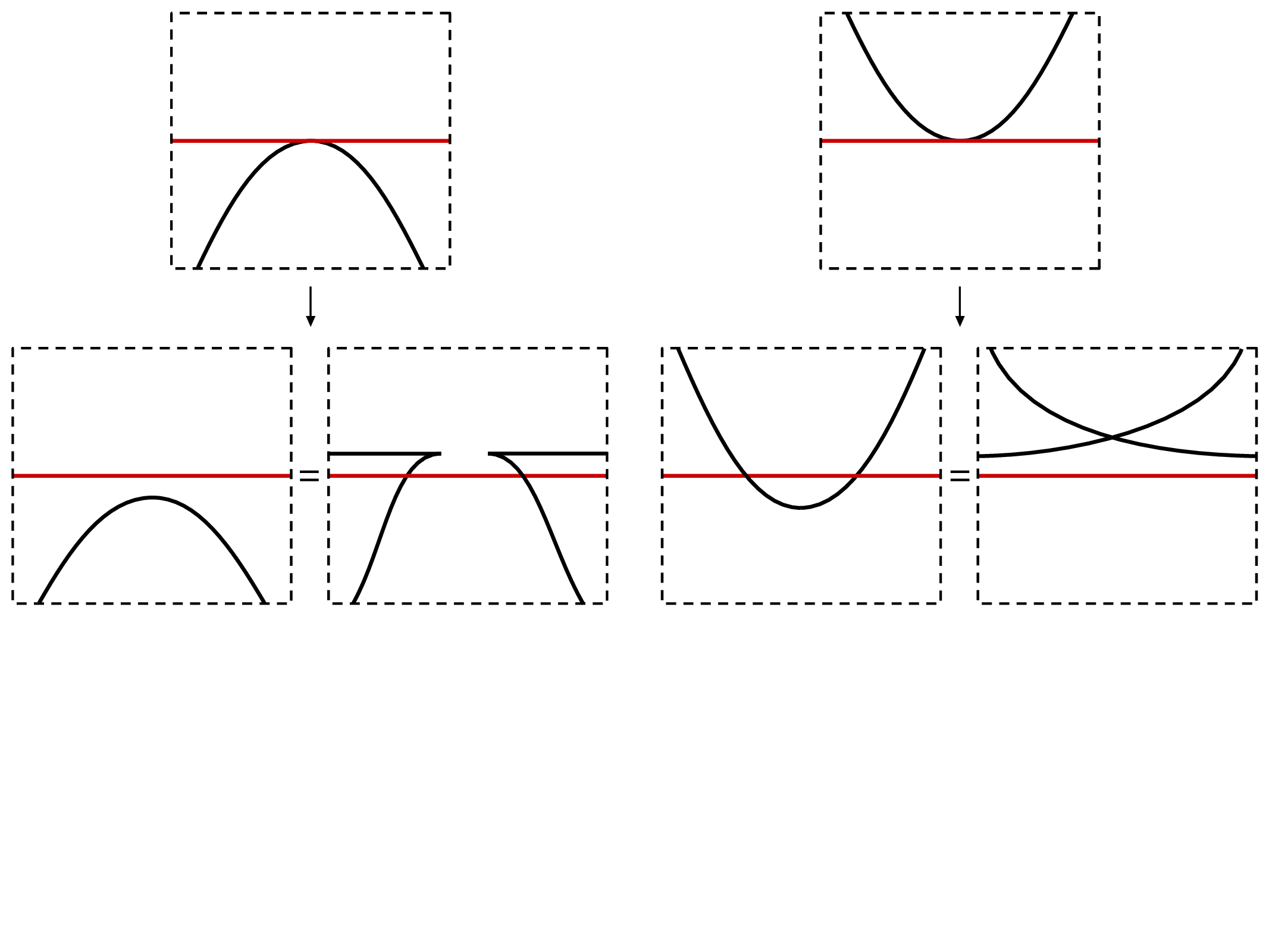}

    \put(14.5,61.5){\footnotesize \textcolor{darkred}{$\Lambda_0$}}
    \put(18.5,55){\footnotesize $\Lambda$}
    \put(66.5,61.5){\footnotesize \textcolor{darkred}{$\Lambda_0$}}
    \put(70.5,70){\footnotesize $\Lambda$}

  \put(18,35){\footnotesize \textcolor{darkred}{$(-1)$}}
  \put(43,35){\footnotesize \textcolor{darkred}{$(-1)$}}
  \put(69.25,35){\footnotesize \textcolor{darkred}{$(-1)$}}
  \put(94,35){\footnotesize \textcolor{darkred}{$(-1)$}}

  \put(2,35){\footnotesize \textcolor{darkred}{$\Lambda_0$}}
  \put(27,35){\footnotesize \textcolor{darkred}{$\Lambda_0$}}
  \put(53.25,35){\footnotesize \textcolor{darkred}{$\Lambda_0$}}
  \put(78,35){\footnotesize \textcolor{darkred}{$\Lambda_0$}}

   \put(6,28.5){\footnotesize $\Lambda^{-\epsilon}$} 
   \put(33,42){\footnotesize $(\Lambda \uplus \Lambda_0)^{\epsilon}$} 
   \put(55.5,45){\footnotesize $\Lambda^{-\epsilon}$} 
   \put(84,44){\footnotesize $(\Lambda \uplus \Lambda_0)^{\epsilon}$} 
  
	\end{overpic}
    \vskip-3.75cm
	\caption{The statement of \cref{prop:handleslide}: Legendrian handlesliding $\Lambda^{-\epsilon}$ up across contact $(-1)$-surgery along $\Lambda_0$ to produce $(\Lambda \uplus \Lambda_0)^{\epsilon}$.}
	\label{fig:slide-sum}
\end{figure}

While \cref{prop:handleslide} describes the effect of handleslides on the Legendrians, we also need to keep track of Lagrangian disk fillings under such slides. The following lemma, together with \cref{claim:disk_isotopy}, suffices for proving the triviality of trivial bypass attachments below.

\begin{lemma}\label{lemma:disk_slide}
Let $W^{2n}$ be a Weinstein domain, $\Lambda \subset \partial W$ a standard Legendrian unknot, and $D\subset W$ a standard Lagrangian disk filling of $\Lambda$. Let $W'$ be the Weinstein domain obtained by attaching a critical handle along a Legendrian $\Lambda_0$ as in \cref{fig:cocore}. Then the completion of $D$ is Hamiltonian isotopic in the completion of $W'$ to the completion of the cocore of the handle.  
\end{lemma}

\begin{figure}[ht]
	\begin{overpic}[scale=0.43]{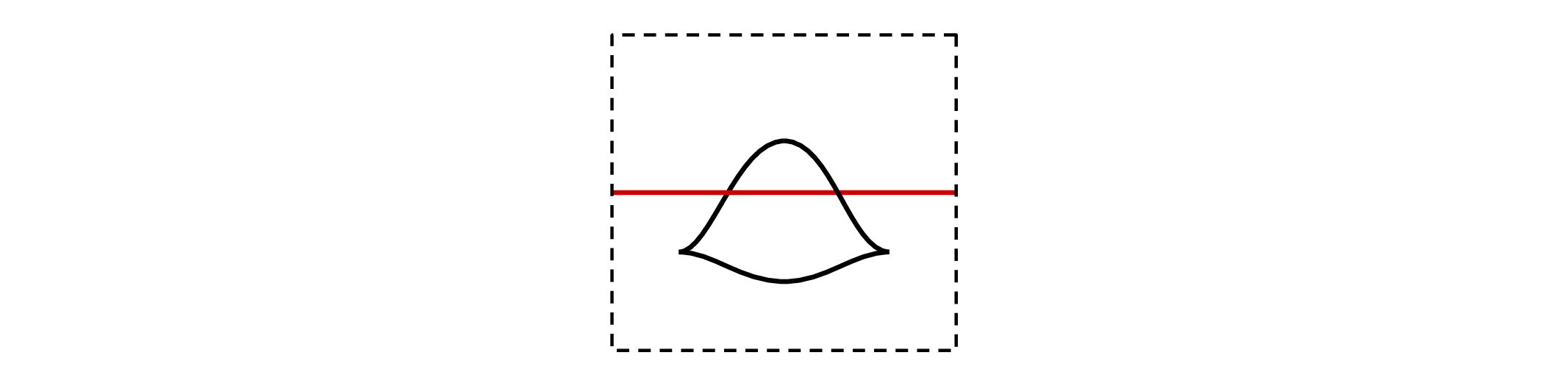}

    \put(35.5,11.5){\small \textcolor{darkred}{$\Lambda_0$}}
     \put(49,16.5){\small $\Lambda$}
    \put(62,11.5){\small \textcolor{darkred}{$(-1)$}}

	\end{overpic}
    
	\caption{The standard Lagrangian filling of the ``meridian'' standard unknot is isotopic to the cocore of a Weinstein handle.}
	\label{fig:cocore}
\end{figure}

\begin{proof}
We take the Weinstein handle attached along $\Lambda_0$ to be the standard one
\[
h= \{|x| \leq 1\} \times \{|y| \leq 1\} \subset \R^n_x \times \R^n_y \cong T^*\R^n_x
\]
with Liouville form $\lambda = -x\cdot dy - 2y\cdot dx$, so that the core is $K = \{|x| \leq 1\} \times \{y = 0\}$ and the cocore is $C = \{x = 0\}\times \{|y| \leq 1\}$. Writing $X$ and $X'$ for the Liouville vector fields on $W$ and $W'$, we have $X'|_h=-x\cdot \partial_x + 2y\cdot \partial_y$.

Since $D\subset W$ is a standard disk, we may assume that $D$ is the concatenation of:
\begin{itemize}
\item[(i)] a Lagrangian annulus $A_1$ corresponding to the trace of a Legendrian isotopy which shrinks the Legendrian unknot $\Lambda$; and 
\item[(ii)] an arbitrarily small Lagrangian disk $D_0$ which transversely intersects the core $K$, extended backwards along $X$.
\end{itemize}

\begin{figure}[ht]
	\begin{overpic}[scale=0.36]{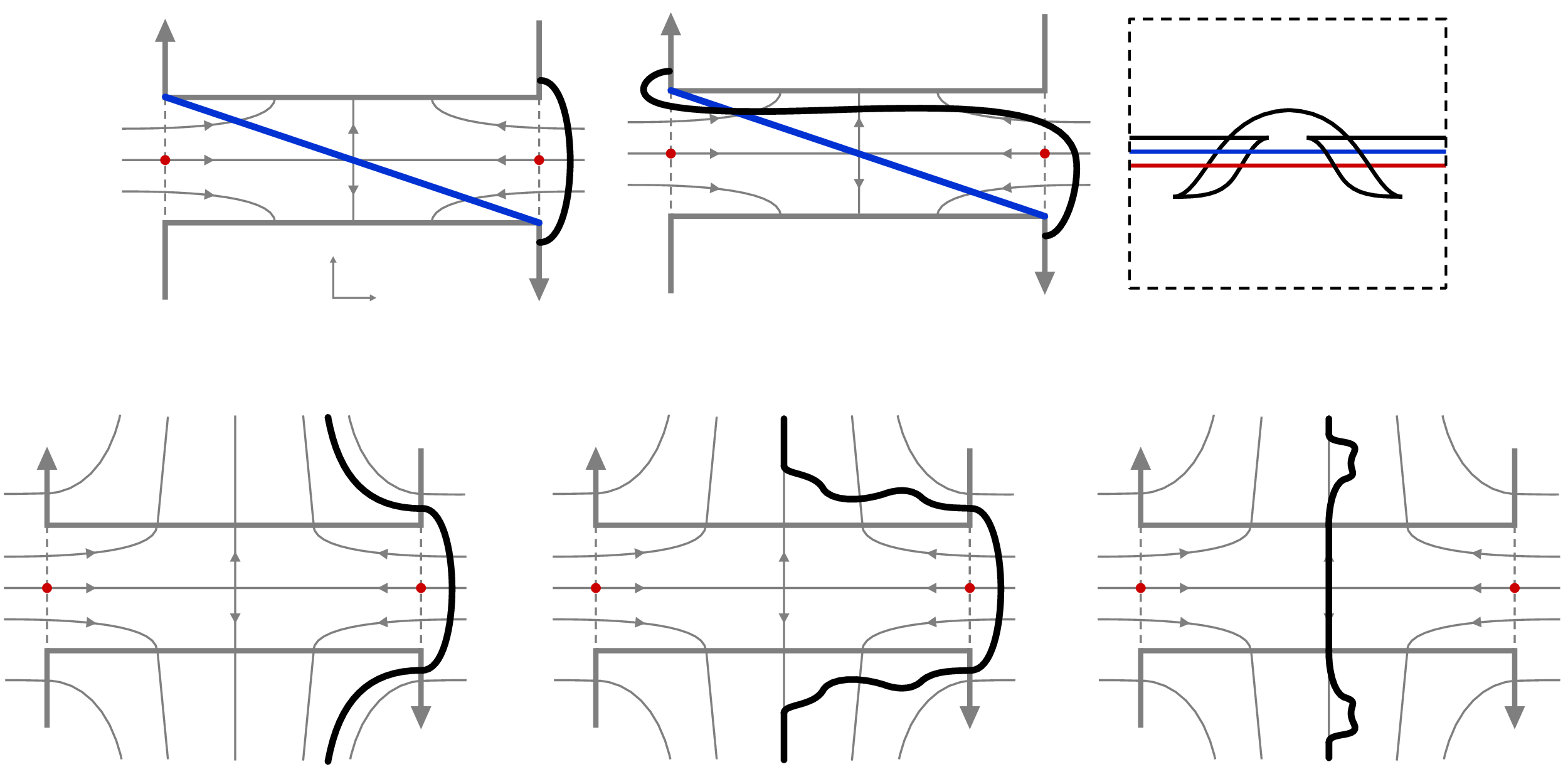}

    \put(81.5,36.5){\small \textcolor{darkred}{$\Lambda_0$}}
     \put(89,42.5){\small $\Lambda$}
    \put(93,38.5){\footnotesize \textcolor{darkred}{$(-1)$}}

    \put(14.5,45.5){\small \textcolor{darkblue}{$C_1$}}
    \put(35.5,46.5){\small \textcolor{black}{$D$}}

     \put(24.5,30.75){\tiny \textcolor{gray}{$x$}}
     \put(20.75,34.25){\tiny \textcolor{gray}{$y$}}
  
	\end{overpic}
    
	\caption{The proof of \cref{lemma:disk_slide}. The bottom row depicts an application of \cref{claim:disk_isotopy} and the effect of the Liouville flow on $D_0$.}
	\label{fig:cocore-proof}
\end{figure}

We now apply \cref{claim:disk_isotopy} to the trace of the Legendrian isotopy from $\bdry C$ to $\Lambda$ to obtain the annulus $A_2$. One way to see that $\bdry C$ and $\Lambda$ are Legendrian isotopic is to observe that the disks $C_t = \{x = -ty\}\subset h$, ${0\leq t \leq 1}$, are Lagrangian with Legendrian boundary $\bdry C_t\subset \bdry W'$ taking $\bdry C= \bdry C_0$ to $\bdry C_1$, and that the negative Reeb flow on $\bdry W'$ takes $\bdry C_1$ to $\Lambda$; see the top row of \cref{fig:cocore-proof}. 
We then isotop $D_0\cup A_1\cup A_2$ using the time-$\tau$ Liouville flow for large $\tau$, which in the handle is $(x,y)\mapsto (e^{-\tau}x, e^{2\tau}y)$, so that $D_0$ is moved arbitrarily close to the cocore disk $C$.
At the same time the Liouville flow pushes $A_1\cup A_2$ into the completion as an exact Lagrangian concordance (more or less) from the belt sphere to itself; see the bottom half of \cref{fig:cocore-proof}.
It follows that the completion of $D$ is Hamiltonian isotopic to the completion of the cocore.
\end{proof}

\subsubsection{Bypass attachment}

The following proposition summarizes a bypass attachment as a pair of contact handles. It originally appeared as a combination of \cite[Theorem 5.1.3, Proposition 8.3.2]{HH18}; here we include the restatement adapted from the preliminary material of \cite{breen2027bypass}.

\begin{proposition}[Bypass attachment]\label{prop:bypass_attachment}
Let $\Sigma^{2n} = R_+ \cup \Gamma \cup R_-$ be convex with bypass attachment data $(\Lambda_{-}, \Lambda_+; D_{-}, D_+)$. Then a pair of smoothly canceling contact $n$- and ($n+1$)-handles can be attached according to either of the following models:
\begin{enumerate}
    \item[($R_+$)] Attach a contact $n$-handle along $\Lambda_- \uplus \Lambda_+$ and a contact ($n+1$)-handle along $\tilde{D}_- \cup D_+^{-\epsilon}$, where $\tilde{D}_-$ is obtained by sliding $D_-^{\epsilon}$ down across the $(-1)$-surgery along $\Lambda_- \uplus \Lambda_+$. 
    \item[($R_-$)] Attach a contact $n$-handle along $\Lambda_- \uplus \Lambda_+$ and a contact ($n+1$)-handle along $D_-^{\epsilon} \cup \tilde{D}_+$, where $\tilde{D}_+$ is obtained by sliding $D_+^{-\epsilon}$ up across the $(-1)$-surgery along $\Lambda_- \uplus \Lambda_+$. 
\end{enumerate}
The two models, called the {\em $R_{\pm}$-centric models,} respectively, are identified by a handleslide of the ($n+1$)-handle across the $n$-handle; see \cref{fig:bypass_attachment}. Moreover, if the bypass attachment data is trivial, the resulting contact cobordism $\Sigma \times [0,1]$ has a vertically invariant contact structure.
\end{proposition}

\begin{figure}[ht]
	\begin{overpic}[scale=0.51]{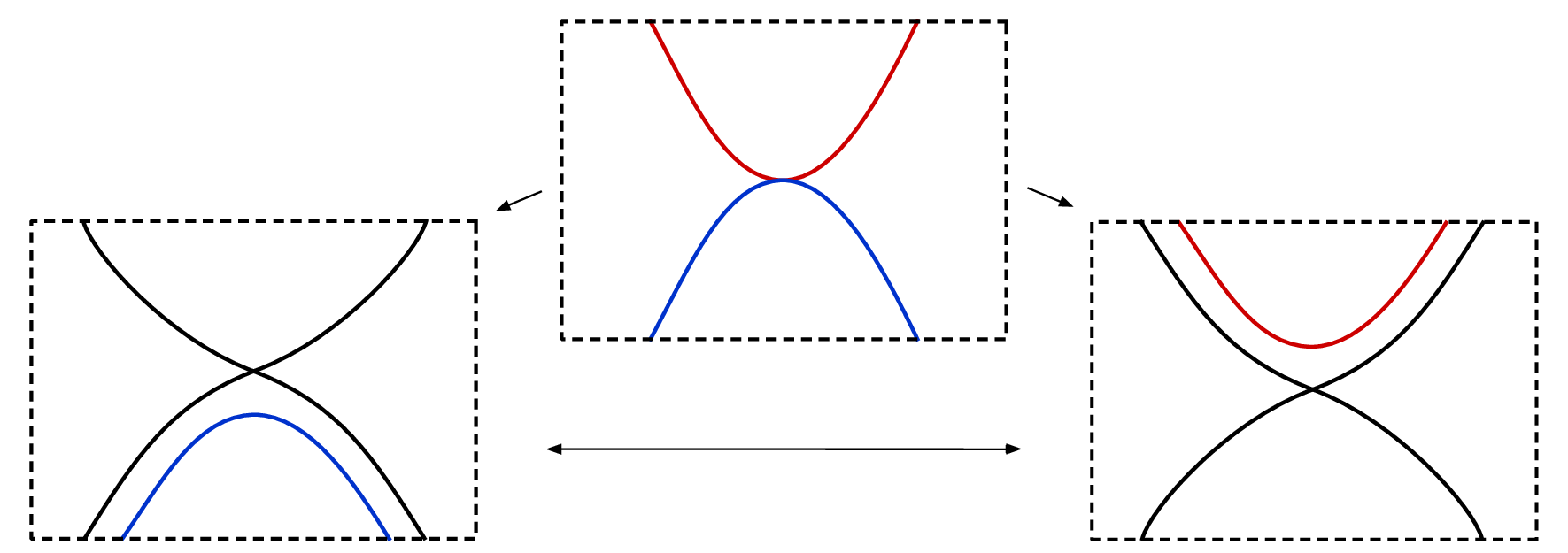}
	    \put(36.75,35){\small (Bypass attachment data)}
	    \put(45,29.5){\small {\color{darkred}  $\Lambda_-$}}
	    \put(45,18){\small \textcolor{darkblue}{$\Lambda_{+}$}}
	    \put(78,22){\small ($R_-$-centric)}
	    \put(92.75,12){\tiny $(-1)$}
	    \put(92.75,15){\tiny {\color{darkred}  $(+1)$}}
	    \put(79.5,17){\small {\color{darkred}  $\Lambda_-^{\epsilon}$}}
	    \put(79.5,4){\small $\Lambda_- \uplus \Lambda_+$}
	    \put(10.5,22){\small ($R_+$-centric)}
	    \put(25,10){\tiny $(-1)$}
	    \put(25,7){\tiny \textcolor{darkblue}{$(+1)$}}
	    \put(12,3.5){\small \textcolor{darkblue}{$\Lambda_+^{-\epsilon}$}}
	    \put(11.75,17){\small $\Lambda_- \uplus \Lambda_+$}
		\put(45,3.75){\small Handleslide}	
	\end{overpic}
	\caption{Two models of a bypass attachment viewed in the dividing set.}
	\label{fig:bypass_attachment}
\end{figure}

\begin{proof}
By \cref{prop:n_handle}, a contact $n$-handle $H_n$ can be attached to $\Sigma$ along $\Lambda_- \uplus \Lambda_+$ to yield a convex hypersurface $S=(\Sigma \setminus \partial_1 H_n) \cup \partial_2 H_n$. Let $S \setminus \Gamma_S = R_+ (S) \cup R_- (S)$ be its convex decomposition. Then $R_\pm(S)$ is obtained from $R_\pm(\Sigma)$ by a Weinstein handle attachment along $\Lambda_- \uplus \Lambda_+$, respectively, and $\Gamma_S$ is obtained from $\Gamma_{\Sigma}$ by a contact $(-1)$-surgery along $\Lambda_- \uplus \Lambda_+$.

Next we describe the attaching locus of the contact $(n+1)$-handle on $S$. We claim that by handlesliding $\Lambda_+^{-\epsilon}$ up across the contact $(-1)$-surgery along $\Lambda_- \uplus \Lambda_+$, one obtains a sphere Legendrian isotopic to $\Lambda_-^\epsilon$. (Note that we have chosen the surgery region to be disjoint from $\Lambda_{\pm}^{\mp \epsilon}$.) That this is true can be verified diagrammatically in the spin-symmetric front projections using the local model for a contact $(-1)$-handleslide in \cref{prop:handleslide}; this is shown in \cref{fig:bypass-slide}. In the top right of the figure, the nested boxes of spin symmetry should be interpreted as follows. Assume we are constructing fronts in a chart $\R^{n+1}_{(z,x_1, \dots,x_n)}$ and that the the horizontal coordinate in the diagram is $x_1$, oriented from left to right. Further assume that the nodal singularity of the front of $\Lambda_- \uplus \Lambda_+$ occurs at $x_1 = \cdots = x_n = 0$. Let $\mathcal{F}_{\mathrm{in}}$ denote the spin-symmetric front produced by the inner box and $U_{\mathrm{in}}$ its support neighborhood, with axis of symmetry centered at $x_1 = 1, x_2 = \cdots = x_n = 0$ and radius $\tfrac{1}{2}$. Let $\mathcal{F}_{\mathrm{out}}$ be the spin-symmetric front obtained by spinning $\{x_1\leq 0, x_2 = \cdots = x_n = 0$\} around the $z$-axis. Note that $\mathcal{F}_{\mathrm{out}}$ and $\mathcal{F}_{\mathrm{in}}$ agree along $\partial\overline{U}_{\mathrm{in}}$. The front depicted by the top right of \cref{fig:bypass-slide} is then $(\mathcal{F}_{\mathrm{out}} \cap U_{\mathrm{in}}^c) \cup (\mathcal{F}_{\mathrm{in}} \cap U_{\mathrm{in}})$.

\begin{figure}[ht]
	\begin{overpic}[scale=0.35]{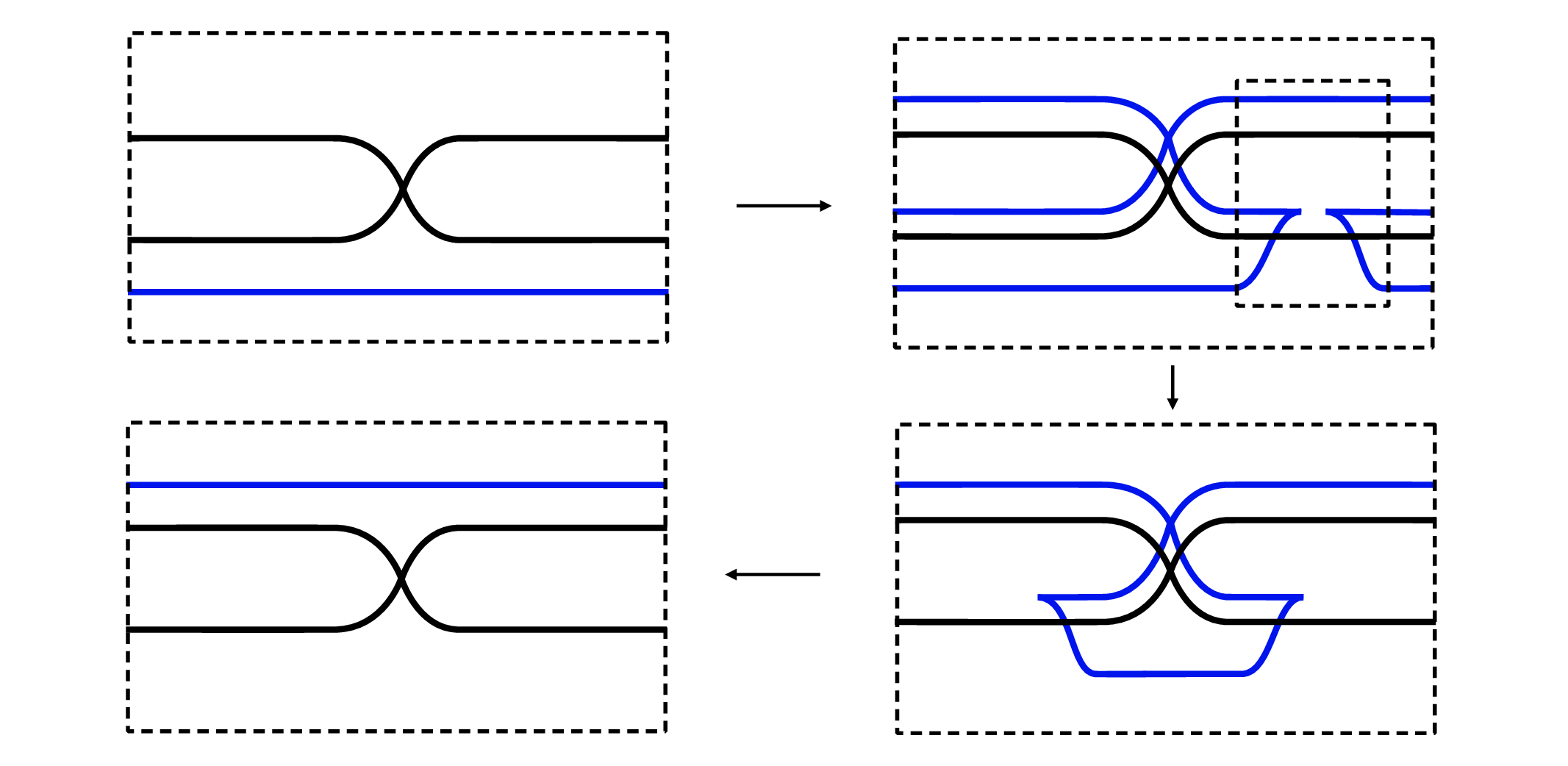}
        \put(10,13){\tiny $(-1)$}
        \put(10,38){\tiny $(-1)$}
        \put(59,13){\tiny $(-1)$}
        \put(59,38){\tiny $(-1)$}
    
	\end{overpic}
	\caption{Sliding $\Lambda_+^{-\epsilon}$ up across the contact $(-1)$-surgery along $\Lambda_- \uplus \Lambda_+$ to produce $\Lambda_-^{\epsilon}$.}
	\label{fig:bypass-slide}
\end{figure}

This Legendrian isotopy can be extended to a Hamiltonian isotopy of the Lagrangian disk $D_+^{-\epsilon}$, producing a Lagrangian disk $\tilde{D}_+ \subset R_+(S)$ with cylindrical end such that $\partial \tilde{D}_+ = \Lambda^{\epsilon}_- = \partial D^{\epsilon}_-$. After a perturbation of $S$ in its invariant neighborhood, we may take $D'_+$ to be Legendrian.  By \cref{prop:(n+1)_handle}, one can then attach a contact $(n+1)$-handle $H_{n+1}$ along the Legendrian sphere $\tilde{D}_+ \cup_{\Lambda^{\epsilon}_-} D^{\epsilon}_- \subset S$.

Note that, instead of sliding from $\Lambda^{-\epsilon}_+$ to $\Lambda^{\epsilon}_-$, one can also slide from $\Lambda^{\epsilon}_-$ to $\Lambda^{-\epsilon}_+$, which yields a Lagrangian disk $\tilde{D}_- \subset R_-(S)$ such that $\partial \tilde{D}_- = \Lambda^{-\epsilon}_+$. Then a contact $(n+1)$-handle can be attached along $D^{-\epsilon}_+ \cup_{\Lambda^{-\epsilon}_+} \tilde{D}_- \subset S$. The reader can verify that these two contact $(n+1)$-handle attachments are contactomorphic, and give the $R_{\mp}$-centric models respectively. For clarity, \cref{fig:centric-models-dim3} depicts these two bypass attachment models in dimension $3$.

\begin{figure}[ht]
	\begin{overpic}[scale=0.43]{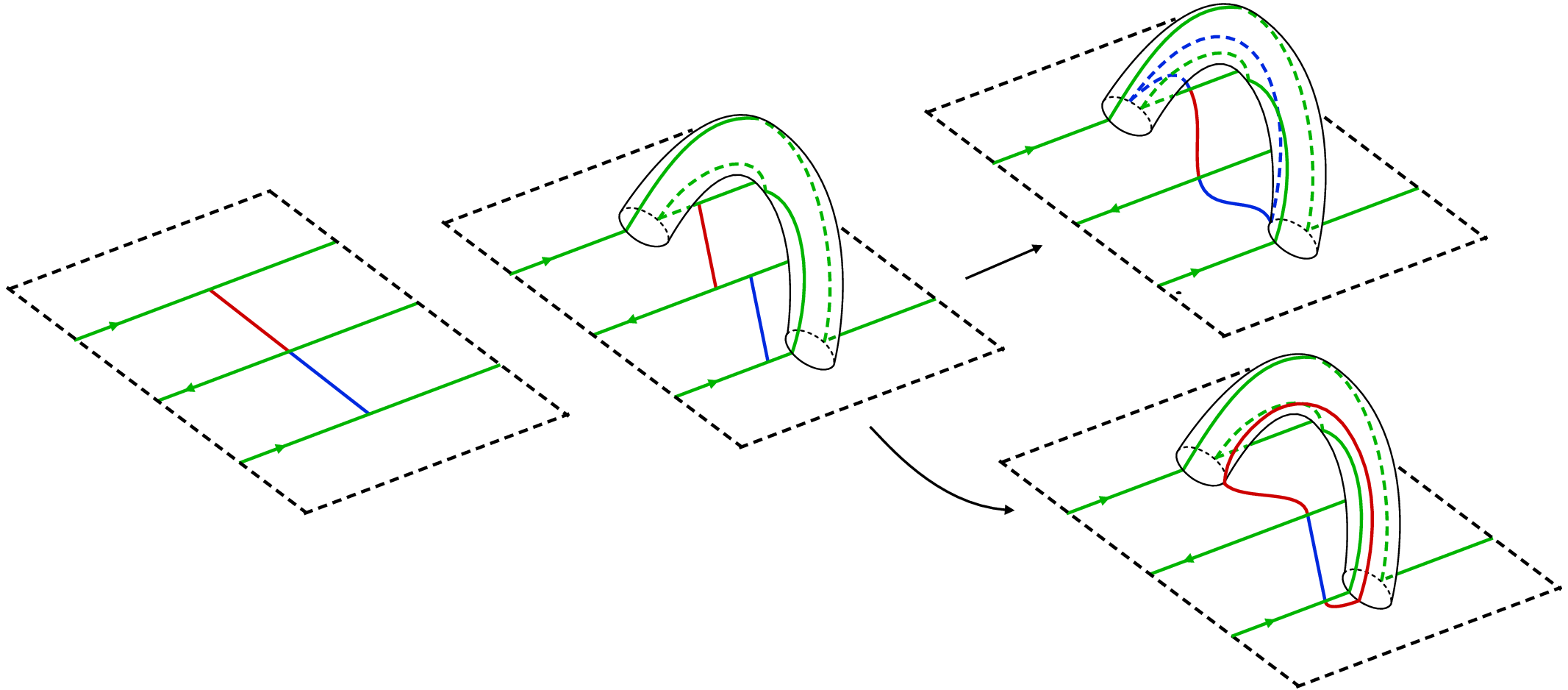}
	   \put(19,18){\tiny \textcolor{darkblue}{$D_+$}}
       \put(15.75,24.25){\tiny \textcolor{darkred}{$D_-$}}

       \put(44.75,22){\tiny \textcolor{darkblue}{$D_+^{-\epsilon}$}}
       \put(45.75,28.5){\tiny \textcolor{darkred}{$D_-^{\epsilon}$}}

       \put(76.75,29){\tiny \textcolor{darkblue}{$\tilde{D}_+$}}
       \put(76.75,35.5){\tiny \textcolor{darkred}{$D_-^{\epsilon}$}}

       \put(80,7){\tiny \textcolor{darkblue}{$D_+^{-\epsilon}$}}
       \put(79.75,13.25){\tiny \textcolor{darkred}{$\tilde{D}_-$}}

       \put(51,12){\tiny $R_+$-centric}
       \put(63.5,25){\tiny $R_-$-centric}
	\end{overpic}
	\caption{The $R_{\pm}$-centric models of a bypass attachment in dimension $3$.}
	\label{fig:centric-models-dim3}
\end{figure}

Next we show that the handles are smoothly canceling. It suffices show that up to smooth (but not necessarily contact) isotopy, the belt sphere of $H_n$ geometrically intersects the attaching locus of $H_{n+1}$ in one point. Up to smooth isotopy, the attaching sphere of $H_{n+1}$ can be obtained by taking the union of $D_+^{-\epsilon}$, $D_-^{\epsilon}$ and the trace of the handleslide from $\Lambda_+^{-\epsilon}$ to $\Lambda_-^{\epsilon}$ in $\Gamma_{\partial_2 H_n} \subset \Gamma_S$, where $\partial_2 H_n$ is defined in \cref{subsubsec:n_handle}. Clearly $D_+^{-\epsilon} \cup D_-^{\epsilon}$ is disjoint from the belt sphere of $H_n$. Then the assertion follows from the observation that the trace of the handleslide intersects the belt sphere of $H_n$ in precisely one point.

Finally, we show that trivial bypass data results in a vertically invariant contact structure. Specifically, we will show that a trivial bypass attachment results in a POBD stabilization of an appropriately defined local POBD, and then will appeal to \cref{lem:positive_stabilization} (or rather, its proof). We assume trivial bypass data of type (TB1), so that $\Lambda_+$ is the standard Legendrian unknot bounding the standard Lagrangian disk $D_+$, and $\Lambda_+$ is below $\Lambda_-$; see the top left of \cref{fig:triv-pobd}. Furthermore, we will consider an $R_-$-centric bypass attachment. Type (TB2) can be treated in a similar manner. 

Consider a small standard cornered Weinstein neighborhood $N(D_-)\subset R_-$ of $D_-$ and a small ``tubular neighborhood'' $N(\Lambda_-) \subset R_+$ of $\Lambda_-$ in the sense that $\overline{N(\Lambda_-)}$ is a (half-)tubular neighborhood of $\Lambda_-$ in the manifold with boundary $\overline{R}_+$. Since $\Lambda_+$ is the standard unknot bounding a standard disk $D_+$, we may assume that $D_+ \subset N(\Lambda_-)$. As a bypass attachment is a semi-local operation which modifies the contact structure only on a neighborhood of $D_+ \cup D_-$, it suffices to restrict to the locally-defined hypersurface $\Sigma_0 := \overline{N(D_-)} \cup \overline{N(\Lambda_+)}$ and show that the bypass attachment produces a vertically invariant structure on $\Sigma_0 \times [0,1]$. 

To that end, we construct a sutured model for a vertically invariant collar neighborhood $\Sigma_0 \times [-\delta,0]$ of $\Sigma_0 = \Sigma_0 \times \{0\}$  as follows. Let $N(\Gamma)\subset \Sigma$ be a small tubular neighborhood of $\Gamma$, and set $N_0(D_-) = N(D_-) \setminus N(\Gamma), N_0(\Lambda_-) = N(\Lambda_-) \setminus N(\Gamma)$. Then consider 
\[
\left(N_0(D_-)\times [-1, 0]\right)\, \cup \, \left(N_0(\Lambda_-) \times [0,1]\right),
\]
with contact form $\alpha = \tfrac{1}{\epsilon}\, dt + \lambda$, where $t$ is the $[-1,1]$-interval coordinate and $\lambda$ is the Liouville form on $N_0(D_-)$ inherited from that of $R_-$. Under this model, the suture neighborhood $N(\Gamma)\cap \Sigma_0$ is identified with $\partial_{\mathrm{out}} N_0(D_-) \times [-1,1]$, $N_0(\Lambda_-) \subset R_+$ is identified with $N_0(\Lambda_-) \times \{1\}$, and $N_0(D_-) \subset R_-$ is identified with $N_0(D_-) \times \{-1\}$; see the top right of \cref{fig:triv-pobd}. The model additionally identifies $\partial_{\mathrm{out}} N_0(D_-) \times \{\pm 1\}$ contactomorphically with $\Gamma \cap \Sigma_0$, which in turn is identified with $\partial_{\mathrm{out}} N_0(D_-) \times \{0\}$. Moreover, as the characteristic foliation on $\partial_{\mathrm{out}} N_0(D_-) \times [0,1]$ has negative $\Gamma$-Reeb holonomy from $t=-1$ to $t=1$ (as depicted by the thin, dashed lines in \cref{fig:triv-pobd}), $\left(D_- \cap N_0(D_-)\right) \times \{-1\}$ is naturally identified with a positive Reeb-shift $D_-^{\epsilon}$ of $D_-$ (where both are viewed on $t=0$) and $\left(D_+ \cap N(\Lambda_-)\right) \times \{1\}$ is naturally identified with a negative Reeb-shift $D_+^{-\epsilon}$ (viewed on $t=0$). Finally, note that the sutured model may be viewed as a locally-defined POBD with empty monodromy.

\begin{figure}[ht]
	\begin{overpic}[scale=0.43]{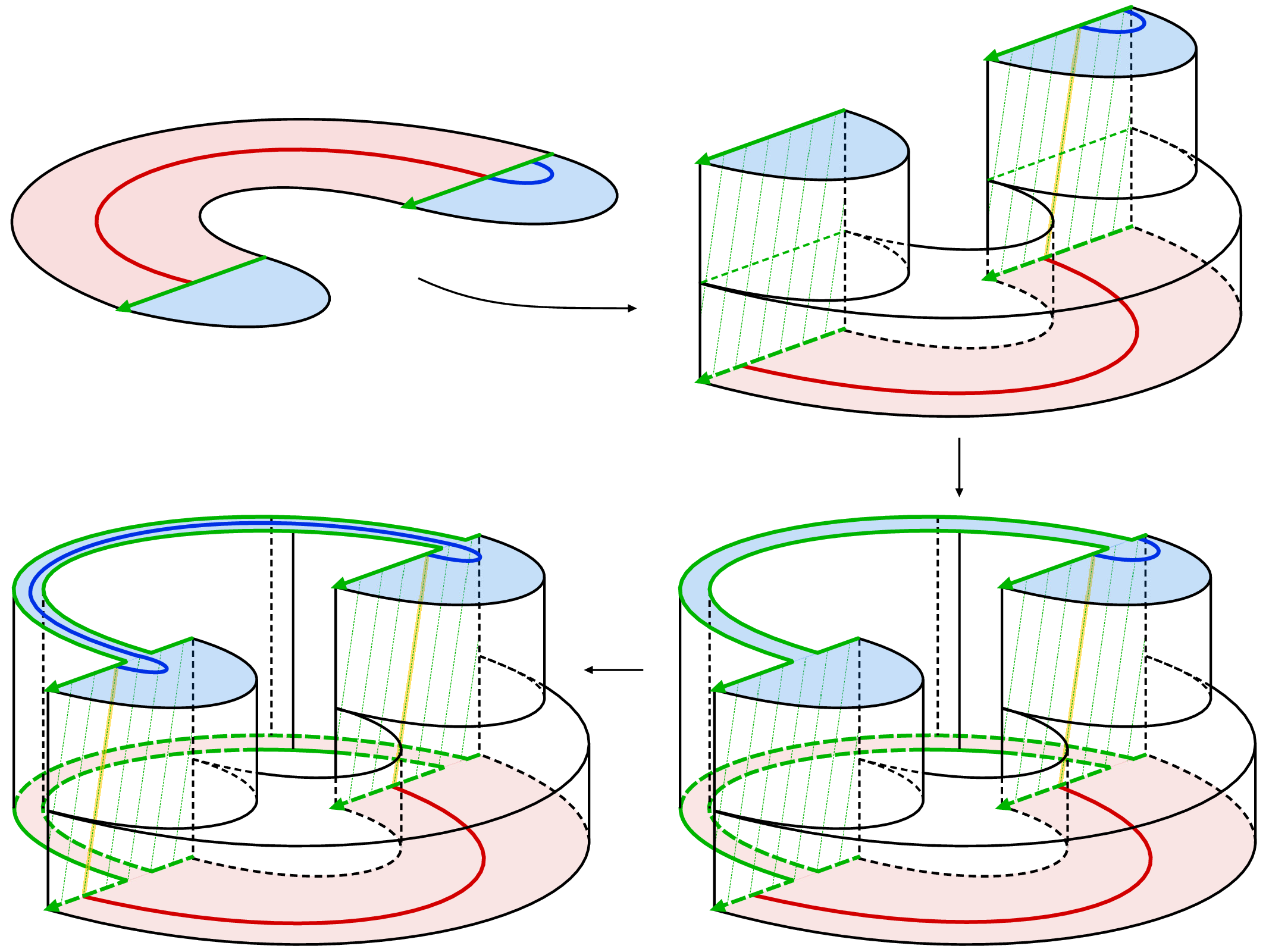}
      \put(2,50){$\Sigma_0$}
	  \put(77,38){\tiny attach $h\times [-1,1]$}
      \put(4,59){\footnotesize \textcolor{darkred}{$D_-$}}
      \put(64,46.5){\footnotesize \textcolor{darkred}{$D_-^{\epsilon}$}}
      \put(37,58.5){\footnotesize \textcolor{darkblue}{$D_+$}}
      \put(81,74.5){\footnotesize \textcolor{darkblue}{$D_+^{-\epsilon}$}}
      \put(30,35.5){\footnotesize \textcolor{darkblue}{$\tilde{D}_+$}}
	\end{overpic}
	\caption{A trivial bypass attachment as a (local) POBD stabilization.}
	\label{fig:triv-pobd}
\end{figure}

The trivial bypass attachment along $(\Lambda_+, \Lambda_-; D_+, D_-)$ first involves a contact $n$-handle attached along $\Lambda_- \uplus \Lambda_+ \cong \Lambda_-$, given in sutured form as $h\times [-1,1]$ in the bottom right of \cref{fig:triv-pobd}. Then on $\{t=1\}$ we may isotop $\Lambda_+^{-\epsilon}$ up across $h\times \{1\}$ to obtain $\tilde{D}_+ \times\{1\}$, so that the characteristic foliation on the suture neighborhood identifies $\partial \tilde{D}_+ \times \{1\}$ with $\Lambda_-^{\epsilon} \times \{-1\}$. By \cref{lemma:handle-to-POBD}, the contact $(n+1)$-handle attachment then introduces a partially defined monodromy from a small neighborhood $S := n(\tilde{D}_+)$ of $\tilde{D}_+$ to a small neighborhood $n(D_-^{\epsilon}) \subset N(D_-)$. Note that, by \cref{lemma:disk_slide}, $\tilde{D}_+$ is in fact Hamiltonian isotopic to the cocore disk of $h\times \{1\}$.

Let $L' \subset h$ denote the Lagrangian core disk of the Weinstein handle $h$, so that $L' \cup D_-$ is a Lagrangian sphere. Then, up to a Reeb shift to account for the sutured model and viewing a parallel copy of $S$ in the level $\{t=-1\}$, we claim that $\tau_{L' \cup D_-}(S)= n(D_-^{\epsilon})$. To see this, consider the model Weinstein $n$-handle as in the proof of \cref{lemma:disk_slide}: 
\[
h= \{|x| \leq 1\} \times \{|y| \leq 1\} \subset \R^n_x \times \R^n_y \cong T^*\R^n_x
\]
with Liouville form $\lambda = -x\cdot dy - 2y\cdot dx$. Then $L' = \{|x| \leq 1\}\times \{y=0\}$, and by the preceding discussion, we may take $\tilde{D}_+ = \{x=0\}\times \{|y| \leq 1\}$ as the cocore. Also choose $\epsilon>0$ small and let $h^{\op{ext}}= \{|x| \leq 1+2\epsilon\} \times \{|y| \leq 1\}$ be a slight extension of $h$ into the Weinstein domain to which $h$ is being attached. Here we have $h^{\op{ext}}\cap D_-=\{1\leq |x|\leq 1+2\epsilon\}\times\{y=0\}$. Performing a symplectic Dehn twist of $\tilde{D}_+$ around $L' \cup D_-$ is achieved by Polterovich surgery on $L'\cup D_-$ and $\tilde{D}_+$; see \cite[Proposition 8.6]{seidel1999lagrangianspheres}. The result is 
\[
\tau_{L' \cup D_-}(\tilde{D}_+)= (L' \cup D_-) \# \tilde{D}_+ = D_- \cup (L' \#\tilde{D}_+) = (D_-\setminus h^{\op{ext}}) \cup A,
\]
where $A =(D_-\cap h^{\op{ext}}) \cup (L' \#\tilde{D}_+)\subset h^{\op{ext}}$ is the Lagrangian annulus defined as follows. Let $t\mapsto (t,f(t))$ for $t\in [0,1+2\epsilon]$ be a parametric curve in $\R^2$ so that $f(t) = -1+t$ for $t\in [0,\epsilon]$, $f(t) = 0$ for $t\in[1+\epsilon,1+2\epsilon]$, and
\[
A = \{(tx, f(t)x)\, :\, |x| = 1, t\in [0,1+2\epsilon]\}.
\]
See the middle of \cref{fig:cocore-twist}. Note that, restricted to $0<|x|\leq 1+2\epsilon$, $A$ has the form $y = \tfrac{f(t)}{t}x$  and is graphical over the core of $h$. We may modify $f(t)$ so that $f(t)$ is close to $-1$ for $t\in [0, 1]$. Applying the Hamiltonian isotopy induced by the Hamiltonian vector field $-y\cdot \bdry_x$ which takes the cocore disk $\{x=0\}$ to the diagonal disk $\{x=-y\}$ then takes $\tau_{L' \cup D_-}(\tilde{D}_+) = D_- \cup A$ to a Reeb shift of $D_-$; see the right side of \cref{fig:cocore-twist}. This proves the claim.

\begin{figure}[ht]
	\begin{overpic}[scale=0.35]{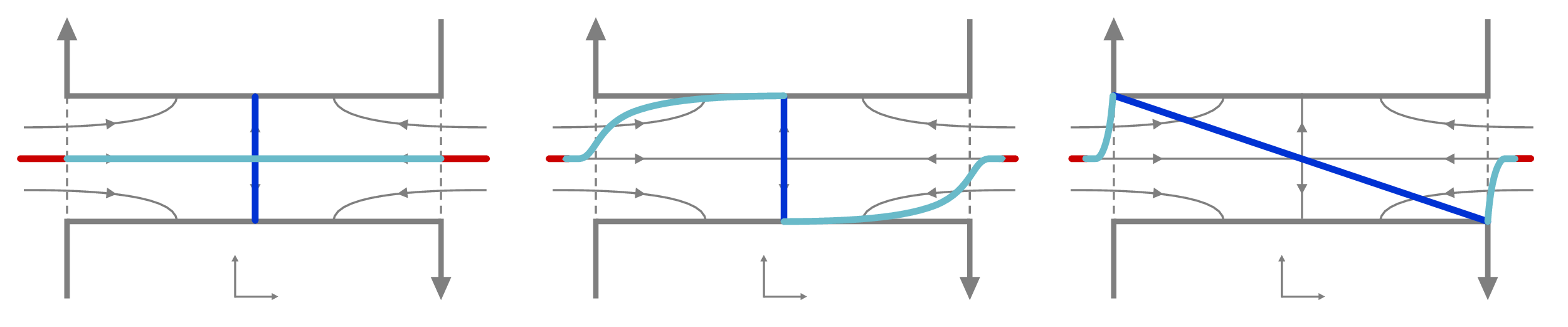}
     \put(14.5,5.25){\tiny \textcolor{gray}{$y$}}
     \put(18,1.5){\tiny \textcolor{gray}{$x$}}
     \put(0,6){\footnotesize \textcolor{darkred}{$D_-$}}
     \put(10,8.5){\footnotesize \textcolor{Cyan}{$L'$}}
     \put(15.5,16){\footnotesize \textcolor{darkblue}{$\tilde{D}_+$}}

     \put(48.25,5.25){\tiny \textcolor{gray}{$y$}}
     \put(51.75,1.5){\tiny \textcolor{gray}{$x$}}
     \put(46,8.5){\footnotesize \textcolor{Cyan}{$A$}}

     \put(81.25,5.25){\tiny \textcolor{gray}{$y$}}
     \put(84.75,1.5){\tiny \textcolor{gray}{$x$}}
	\end{overpic}
	\caption{Dehn twisting a cocore via Lagrangian surgery.}
	\label{fig:cocore-twist}
\end{figure}

The trivial bypass attachment therefore induces a positive stabilization of the (local) POBD, hence by (the proof of) \cref{lem:positive_stabilization} induces a vertically invariant contact structure on the bypass cobordism.
\end{proof}

Finally, we observe that the proof of the triviality above suggests the more general principle that a bypass attachment to a POBD produces a new POBD. This is an immediate consequence of \cref{lemma:handle-to-POBD} and the fact that a bypass attachment is by definition a pair of contact handles. 

\begin{corollary}\label{lem:bypass to POB}
If $(M, \xi)$ is supported by a POBD and $(M', \xi')$ is obtained from $(M, \xi)$ by a bypass attachment, then $(M',\xi')$ is supported by a POBD. 
\end{corollary}

\noindent It is of course possible to interpret the specific effects of of the handles involved on the POBD using \cref{lemma:handle-to-POBD}, but for the purpose of \cref{part:OBD} it suffices to acknowledge that a bypass produces some POBD.

\bibliography{references}
\bibliographystyle{amsalpha}

\end{document}